\documentclass[a4, 11pt]{article}

\usepackage[latin1]{inputenc}
\usepackage[T1]{fontenc}
\usepackage[english]{babel}
\usepackage{layout}

\usepackage[active]{srcltx}
\usepackage{amsmath,amsfonts,amssymb,epsf}
\usepackage{eepic,epic,graphicx,amsmath,amssymb,xypic}
\xyoption{curve}
\usepackage{verbatim}
\usepackage{soul}
\usepackage{bm}
\usepackage{lscape}
\usepackage{dsfont}
\usepackage{lmodern}
\usepackage{multicol}
\usepackage{colortbl}
\usepackage{wrapfig}
\usepackage{graphicx}
\usepackage[svgnames]{xcolor}
\usepackage{stmaryrd}
\usepackage{listings}
\usepackage{hobsub}
\usepackage{hobsub-generic}
\usepackage{hobsub-hyperref}
\usepackage{yhmath}
\usepackage{array} 
\usepackage{color}
\usepackage[top=3cm, bottom=3.1cm, left=2.65cm, right=2.65cm]{geometry}

\usepackage{amsthm}
\usepackage{amsmath}
\usepackage{amssymb}
\usepackage{stmaryrd} 
\usepackage{mathrsfs}
\usepackage{ifthen,fp}
\usepackage{fancyhdr}
\usepackage{hyperref}
\usepackage{lastpage}
\usepackage{tikz,tkz-tab, tkz-fct}
\usetikzlibrary{calc}
\usepackage{tkz-euclide}
\usepackage{frcursive}
\usetikzlibrary[patterns]
\usepackage[french]{minitoc}

\def\II{{\mathds 1}}
\def\R{{\mathbb R}}

\def\P{{\mathbb P}}

\def\CCCC{{\mathscr C}}

\def\PP{{\mathcal P}}

\def\FF{{\mathcal F}}
\def\DD{{\mathcal D}}
\def\EE{{\mathcal E}}

\def\BB{{\mathcal B}}
\def\AA{{\mathcal A}}
\def\NN{{\mathcal N}}

\def\MM{{\mathcal M}}

\def\D{{\mathbb D}}

\def\N{{\mathbb N}}
\def\E{{\mathbb E}}

\def\K{{\mathbb K}}

\def\Q{{\mathbb Q}}
\def\P{{\mathbb P}}

\def\MM{{\mathcal M}}

\def\Supp{{\rm{Supp} \,}}
\def\Hess{{\rm{Hess}}}
\def\Span{{\rm{Span}}}

\def\Cov{{\rm{Cov}}}
\def\id{{\rm{id}}}

\def\LL{{\mathcal{L}}}
\def\AA{{\mathcal{A}}}
\newcommand{\dd}{\mathrm{d}}

\makeatletter
\newcounter{subsubsubsection}[subsubsection]
\renewcommand\thesubsubsubsection{\@roman\c@subsubsubsection}
\newcommand\subsubsubsection{\@startsection{subsubsubsection}{4}{\z@}%
                                     {-3.25ex\@plus -1ex \@minus -.2ex}%
                                     {1.5ex \@plus .2ex}%
                                     {\normalfont\small\bfseries}}
\newcommand*\l@subsubsubsection{\@dottedtocline{3}{5.2em}{1em}}
\newcommand*{\subsubsubsectionmark}[1]{}
\makeatother

\definecolor{jaune clair}{rgb}{0, 0.9,0.5}
\definecolor{gris clair}{gray}{0.75}
\definecolor{Mon_rouge}{rgb}{1.,0.,0.}
\definecolor{Mon_violet}{rgb}{0.4,0.,0.6}
\definecolor{Mon_violet2}{rgb}{0.498039,0.,1.}
\definecolor{Mon_orange}{rgb}{1.,0.4980392156862745,0.}
\definecolor{Mon_vert}{rgb}{0.,0.8,0.} 
\definecolor{Mon_vert2}{rgb}{0.,0.4,0.2}
\definecolor{Mon_vert3}{rgb}{0.2,0.6,0.}
\definecolor{Mon_vert4}{rgb}{0.,0.39215686274509803,0.}
\definecolor{Mon_bleu}{rgb}{0.,0.,0.8}
\definecolor{Mon_bleu_coalescence}{rgb}{0.49,0.49,1}

\definecolor{qqqqff}{rgb}{0,0,1}
\definecolor{ffqqqq}{rgb}{1,0,0}
\definecolor{xdxdff}{rgb}{0.49019607843137253,0.49019607843137253,1}
\definecolor{ffxfqq}{rgb}{1,0.4980392156862745,0}

\newtheorem{commun}{Commun}[section]

\newtheorem{Thm}[commun]{Theorem}
\newtheorem{Prop}[commun]{Proposition}
\newtheorem{Def}[commun]{Definition}
\newtheorem{Lem}[commun]{Lemma}
\newtheorem{Cor}[commun]{Corollary}
\newtheorem{Rem}[commun]{Remark}

\definecolor{Mon_rouge}{rgb}{1.,0.,0.}
\definecolor{Mon_violet}{rgb}{0.4,0.,0.6}
\definecolor{Mon_orange}{rgb}{1.,0.4980392156862745,0.}
\definecolor{Mon_vert}{rgb}{0.,0.8,0.}
\definecolor{Mon_bleu}{rgb}{0.,0.,0.8}

\definecolor{bleu clair}{rgb}{0, 0,0.8}
\definecolor{gris clair}{gray}{0.75}

\begin{document}

\pagestyle{fancy}
\renewcommand{\rightmark}{}
\renewcommand{\leftmark}{}
\cfoot{\textbf{\thepage/\pageref{LastPage}}}

\vspace{0.5cm}

\begin{center}
\LARGE \textbf{Existence, uniqueness and ergodicity for the centered Fleming-Viot process} \\ \vspace{0.25cm}
\large \textsc{Nicolas Champagnat}$^{\textsc{*}}$, 
\qquad \textsc{Vincent Hass}$^{\textsc{*}}$\footnote{Corresponding author at: IECL - Site de Nancy Facult\'{e} des Sciences et Technologies Campus, Boulevard des Aiguillettes 54506 Vand\oe uvre-l\`{e}s-Nancy, Cedex, France. \\
 \textit{E-mail addresses:} nicolas.champagnat@inria.fr (Nicolas Champagnat), vincent.hass@inria.fr (Vincent Hass).} \\ \vspace{0.5cm} 
\textsc{\today} \\  \vspace{0.5cm}
\textit{{}$^{\textsc{*}}$ Universit\'{e} de Lorraine, CNRS, Inria, IECL, F-54000 Nancy, France}
\end{center}

\begin{center}
\rule{10cm}{.5pt} 
\end{center}

\textbf{Abstract.}  Motivated by questions of ergodicity for shift invariant \textsc{Fleming-Viot} process, we consider the centered \textsc{Fleming-Viot} process $\left(Z_{t} \right)_{t\geqslant 0}$ defined by $Z_{t} := \tau_{-\left\langle \id, Y_{t}  \right\rangle} \sharp \, Y_{t}$, where $\left(Y_{t}\right)_{t\geqslant 0}$ is the original \textsc{Fleming-Viot} process. Our goal is to characterise the centered \textsc{Fleming-Viot} process with a martingale problem. 
To establish the existence of a solution to this martingale problem, we exploit the original \textsc{Fleming-Viot} martingale problem and asymptotic expansions. The proof of uniqueness is based on a weakened version of the duality method, allowing us to prove uniqueness for initial conditions admitting finite moments. We also provide counter examples showing that our approach based on the duality method cannot be expected to give uniqueness for more general initial conditions. Finally, we establish ergodicity properties with exponential convergence in total variation for the centered \textsc{Fleming-Viot} process and characterise the invariant measure.  \\



\begin{center}
\rule{10cm}{.5pt} 
\end{center}

\normalsize

\section{Introduction\label{Section_1_Introduction}}
\textsc{Fleming} and \textsc{Viot} have introduced in {\color{blue} \cite{fleming_no_1979}} a probability-measure-valued stochastic process modeling the dynamics of the distribution of allelic frequencies in a \emph{selectively neutral} genetic population as influenced by mutation and random genetic drift:  the original \textsc{Fleming-Viot} process. The initial model of {\color{blue} \cite{fleming_no_1979}} was progressively enriched with further mechanisms of Darwinian evolution: selection {\color{blue} \cite{fleming_no_1979, Ethier_Kurtz_1987, Ethier_Kurtz_1993, Dawson, Etheridge}}, recombination {\color{blue} \cite{EK_convergence_1994, Ethier_Kurtz_1993}} or the effect of an environment {\color{blue} \cite{HE_Hui_2010}}. \textsc{Fleming} and \textsc{Viot} characterise in {\color{blue} \cite{fleming_no_1979}} the law of their process as a solution of a \textsc{Stroock-Varadhan} measure-valued martingale problem {\color{blue} \cite{stroock_multidimensional_2006}} both in the selective neutral case and the case with selection. To obtain the existence of such solution on a compact metric space, their method is based on discretisation of the mutation operator and tightness arguments. 
An alternative approach is used, in the studies {\color{blue} \cite{moran_wandering_1975, moran_wandering_1976, Kingman_coherent_1976, kesten_number_1980, shiga_wandering_1982}} based on the \textsc{Otha-Kimura} model {\color{blue} \cite{ohta_model_1973, ohta_simulation_1974}} and in the references {\color{blue} \cite{fleming_no_1979, dawson_wandering_1982, Ethier_Kurtz_1993, Etheridge_2}} based on its continuous-time version: the \textsc{Moran} model (also called \emph{continuous-state stepwise mutation model}). If we denote by $N$ the population size, these authors construct a particle process whose limiting behavior is analysed under the assumptions that the mutation step is proportional to $1/\sqrt{N}$ and on the time scale $(Nt)_{t\geqslant 0}$. In {\color{blue} \cite{HE_Hui_2010}}, another particle process, based on the lookdown construction {\color{blue} \cite{donnelly_countable_1996}} is used to show the existence of the \textsc{Fleming-Viot} process in a random environment. This lookdown construction also allows to analyse sample path properties of the process and has been used in numerous references since then, such as {\color{blue} \cite{donnelly_particle_1999} \cite[\color{black} Chapter 5]{Etheridge_2}}.  \\
\indent In {\color{blue} \cite{fleming_no_1979}}, uniqueness of the solution to the \textsc{Fleming-Viot} martingale problem in the selectively neutral case, is proved using uniqueness of moments of certain finite-dimensional distributions and arguments on semigroup. However, in the case where natural selection acts, the previous method fails, but the result can be obtained from a version of the \textsc{Cameron-Martin-Girsanov} formula {\color{blue} \cite[\color{black} Chapter 10]{Dawson} \cite[\color{black} Theorem 5.1]{dawson_geostochastic_1978}}. See also {\color{blue}\cite{Ethier_Shiga_unbounded_selection_2002}} for an application of this method in the case of unbounded selection function. In most references such as {\color{blue} \cite{Etheridge_2, dawson_wandering_1982, Ethier_Kurtz_1981,  Ethier_Kurtz_1987, Ethier_Kurtz_1993, Ethier_Shiga_unbounded_selection_2002}}, under a variety of assumptions, the \emph{duality method} {\color{blue} \cite[\color{black} Proposition 4.4.7]{Ethier_markov_1986}} is used to prove the uniqueness of the \textsc{Fleming-Viot} process. The idea is to relate the distribution of the original  process with that of a simple process, called \emph{dual process}. This leads to a duality relation which ensures that two solutions to the martingale problem have the same 1-dimensional marginal laws. Uniqueness of the solution to the martingale problem then follows from  \textsc{Markov}'s property {\color{blue} \cite[\color{black} Theorem 4.4.2]{Ethier_markov_1986}}. Other methods are used in some references: {\color{blue} \cite{dawson_resolvent_1995}} makes use of resolvent estimates;  {\color{blue} \cite{Overbeck_et_al_1995, Overbeck_1995}} prove existence and uniqueness of \textsc{Fleming-Viot} processes with \emph{unbounded selection} intensity functions by using \textsc{Dirichlet}'s forms.\\
\indent Questions of ergodicity of the \textsc{Fleming-Viot} process were also the subject of many works. Let $E$ be a Polish space and $\BB(E)$ the \textsc{Borel} $\sigma-$field on $E$. Let us recall from { {\color{blue} \cite[\color{black} Section 5]{Ethier_Kurtz_1993}}} that an $E-$valued \textsc{Markov} process $\left(Z_{t}\right)_{t\geqslant 0}$ {  with a unique stationary distribution $\pi$} is \emph{weakly ergodic} if for all {  bounded} continuous functions $f$ on $E$, {  for all initial condition $x_{0} \in \E$,}
\begin{equation}
\lim_{t\to + \infty}{\E_{{ x_{0}}}\left(f\left(Z_{t} \right) \right) = \int_{E}^{}{f(x){ \pi}(\dd x)}}
\label{Eq_Ergodicity_Weak}
\end{equation}
and \emph{strongly ergodic} if 
\begin{equation}
\lim_{t\to + \infty}{\sup_{B \in \BB(E)}{\left|\P_{{ x_{0}}}\left(Z_{t} \in B \right) - { \pi}(B) \right|}} = 0, \quad {  x_{0}} \in E.
\label{Eq_Ergodicity_Strong}
\end{equation}
If $E$ is compact, for mutation operators $A$ whose closure generates a \textsc{Feller} semigroup on the space of  continuous functions and such that there is a unique probability measure $\nu_{0}$ on $E$ satisfying $\int_{E}^{}{Af(x)\nu_{0}(\dd x)} = 0$, some ergodicity results for the \textsc{Fleming-Viot} process are obtained in {\color{blue} \cite{Ethier_Kurtz_1993}}.  More precisely, in the selectively neutral case and without recombination, a simple proof of weak ergodicity of the \textsc{Fleming-Viot} process is given using duality arguments whereas coupling arguments provide an approach to strong ergodicity. These results were extended in {\color{blue} \cite{EK_coupling_1998}} to models with recombination and in {\color{blue} \cite{Ethier_Shiga_unbounded_selection_2002}} to models with unbounded selection, with the additional tool of \textsc{Dawson}'s \textsc{Girsanov}-type formula for strong ergodicity. In the special case where the mutation operator of the \textsc{Fleming-Viot} process has the form \begin{equation}
Af(x) = \frac{\theta}{2}\int_{E}^{}{\left(f(y) - f(x) \right)P(x, \dd y)}, \quad \theta \in (0, +\infty), \ f \in \DD(A), \label{Intro_1}
\end{equation}
it is proved in {\color{blue} \cite{Ethier_Kurtz_1993}} that the \textsc{Fleming-Viot} process has a reversible stationary distribution if $P(x,\dd y) = \nu(\dd y)$ { for some probability measure $\nu$ on $E$} (see {\color{blue} \cite{li_shiga_yao_1999}} for a converse result). For the mutation operator (\ref{Intro_1}), it is proved in {\color{blue} \cite{Ethier_Kurtz_1981, Ethier_Kurtz_1987}} and {\color{blue}\cite[\color{black} Theorem 8.2.1]{Dawson}} that the \textsc{Fleming-Viot} process is purely atomic for every time, in other words the solutions of the martingale problem take values in the set of purely atomic probability measures. In {\color{blue} \cite{EK_convergence_1994}}, the ergodicity result of {\color{blue} \cite{EK_coupling_1998}} was extended to the weak atomic topology. \\
\indent However, if we consider the case where the mutation operator is the Laplacian on $\R^{d}$, there exists no stationary distribution {\color{blue} \cite{Ethier_Kurtz_1993, li_shiga_yao_1999}}, {\color{blue} \cite[\color{black} Problem 11 p.450]{Ethier_markov_1986}}. Instead the process exhibits a wandering phenomenon {\color{blue} \cite{dawson_wandering_1982, Birkner_Blath_2009}}. Nevertheless, {\color{blue} \cite{shiga_wandering_1982, Ethier_Kurtz_1993}} considered the \textsc{Fleming-Viot} process shifted by minus its empirical mean and established existence of a unique invariant measure and weak ergodicity for this process using moment and duality arguments. More precisely in {\color{blue} \cite{Ethier_Kurtz_1993}}, thanks to some estimates of the original \textsc{Fleming-Viot} dual process and the finiteness of all moments of the \textsc{Fleming-Viot} process shifted by minus its empirical mean for any time $t$, the authors obtain an expression for these in the asymptotic $t \to + \infty$. Then, by {  tightness} arguments and characterisation of the limit, the result follows. In {\color{blue} \cite{shiga_wandering_1982}}, an analoguous approach is used for the continuous-state stepwise mutation model. \\
 \indent In this paper we are interested in the \textsc{Fleming-Viot} process shifted by minus its empirical mean, which we call \emph{centered \textsc{Fleming-Viot} process}. As in previous works it is natural to ask questions of existence, uniqueness and ergodicity {  when the mutation operator is the Laplacian on $\R^d$}. Moreover, the study of this process was motivated by biological questions in adaptive dynamics. The theory of adaptative dynamics {\color{blue} \cite{Metz}} is based on biological assumptions of \emph{rare} and \emph{small mutations} and of \emph{large population} under which an ODE approximating the population evolutionary dynamics, the \emph{Canonical Equation of Adaptive Dynamics} (CEAD) was proposed {\color{blue}\cite{Dieckmann_Law_1996}}. Two mathematical approaches were developed to give a proper mathematical justification of this theory: a \emph{deterministic} one {\color{blue} \cite{Diekmann_Jabin_Mischler_Perthame, Perthame_Barles_2008, Lorz_Mirrahimi_Perthame_2011}}, and a \emph{stochastic} one {\color{blue} \cite{Champagnat, Champagnat_PTRF_2011, Champagnat_Jabin_Meleard_JMPA_2014}}. Despite their success, the proposed approaches are criticised by biologists {\color{blue}\cite{Waxman, Perthame_Gauduchon_2010}}. Among the biological assumptions of adaptive dynamics, the assumption of rare mutations is the most critised as unrealistic. In order to solve this problem, we propose to apply an asymptotic of small mutations and large population, but \emph{frequent mutations}. After conveniently scaling the population state, this leads to a \emph{slow-fast dynamics} {\color{blue} \cite{Papanicolaou_1977, Kurtz}}, where the fast dynamics appears to be given by a discrete version of the centered \textsc{Fleming-Viot} process {\color{blue} \cite{Champagnat_Hass_CEAD_2023}}. This explains why we are interested in ergodicity properties of such processes. \\
\indent To establish the existence of the centered \textsc{Fleming-Viot} process, we characterise it as a solution of a measure-valued martingale problem that we called the \emph{centered \textsc{Fleming-Viot} martingale problem}. Our method is to exploit the original \textsc{Fleming-Viot} martingale problem and asymptotic expansions. An additional difficulty occurs in our case since we need to apply the original \textsc{Fleming-Viot} martingale problem to \emph{predictable test functions}. This requires to extend the martingale problem to such test functions  using \emph{regular conditional probabilities}. The proof of uniqueness of the solution of the centered \textsc{Fleming-Viot} martingale problem is based on duality methods as in the previous works. However, additional difficulties occur in our case since bounds on the dual process are much harder to obtain and the duality identity can only be proved in a weakened version. In particular, our uniqueness result only holds for initial conditions admitting finite moments. We also provide a counter example showing that our uniqueness result is optimal in the sense that we cannot expect to obtain uniqueness for more general initial conditions using the duality approach.   Finally, we obtain strong ergodicity properties of the centered \textsc{Fleming-Viot} process that extend the weak ergodicity results obtained in {\color{blue} \cite{shiga_wandering_1982, Ethier_Kurtz_1993}}. To this aim, we construct the centered version of the \textsc{Moran} process and we prove that it converges in law to the centered \textsc{Fleming-Viot} process. Exploiting the relationship between the \textsc{Moran} model and the \textsc{Kingman} coalescent, we obtain a result of exponential ergodicity in total variation for the centered \textsc{Moran} model uniformly in the number of particles. This result is propagated to the centered \textsc{Fleming-Viot} process by coupling arguments. Using another strategy proposed by {\color{blue} \cite{shiga_wandering_1982, Ethier_Kurtz_1993}}, based on the \textsc{Donnelly-Kurtz} modified look-down {\color{blue} \cite{donnelly_particle_1999}} we give a characterisation of the unique invariant measure of the centered \textsc{Fleming-Viot} process. \\
\indent This paper is organised as follows. In Section \ref{Section_Existence} we define the martingale problem for the centered \textsc{Fleming-Viot} process and establish an existence result. We give also some equivalent extensions to the centered \textsc{Fleming-Viot} martingale problem and some properties of the centered \textsc{Fleming-Viot} process. In Section \ref{Section_Unicite} we prove uniqueness to the centered \textsc{Fleming-Viot} martingale problem for initial conditions admitting finite moments and we discuss this assumption. In Section \ref{Section_Ergo}, we establish exponential convergence in total variation for the centered \textsc{Fleming-Viot} process to its unique invariant measure and provide a characterisation of this measure based on the \textsc{Donnelly-Kurtz} modified look-down. Finally in Sections \ref{Section_preuve_Existence} and \ref{Section_preuve_Unicite}, we prove respectively the main results of existence and uniqueness of the solution of the centered \textsc{Fleming-Viot} martingale problem. The paper ends with appendices gathering technical lemmas for the existence proof {  and extensions to the centered \textsc{Fleming-Viot} martingale problem}.

\section{Existence for the centered \textsc{Fleming-Viot} process\label{Section_Existence}}

In this section, our aim is to define the martingale problem for the centered \textsc{Fleming-Viot} process and to establish an existence result.  This result is stated in Section \ref{Sous_section_2_1}. Then, we give in Section \ref{Sous_section_2_2}, the framework and ideas of the proof. We end this section by giving some interesting results about the centered \textsc{Fleming-Viot} process: it satisfies the \textsc{Markov} property (Section \ref{Sous_section_2_4_1}), admits finite moments (Section \ref{Sous_section_2_4_2}) and has compact support (Section \ref{Sous_section_2_4_3}). 

\subsection{Centered \textsc{Fleming-Viot} martingale problem and main result \label{Sous_section_2_1}}

The centered \textsc{Fleming-Viot} process is a measure-valued diffusion in 
{ \[\MM_{1}^{c,2}\left(\R^{d}\right) := \left\{ \mu \in \MM_{1}(\R^d) \left| \phantom{1^{1^{1^{1}}}} \hspace{-0.6cm} \right. \int_{\R^{d}}^{}{\left\|x\right\|^{2}\mu(\dd x)} < \infty, \int_{\R^{d}}^{}{x\mu(\dd x)}  = 0\right\} \]}
which is endowed with the trace of the topology of weak convergence on {  $\MM_{1}\left(\R^{d}\right)$}, the set of probability measures on $\R^d$ endowed with the topology of weak convergence making it a Polish space {\color{blue} \cite{Billingsley}}. We consider the filtered probability space $\left({ \widetilde{\Omega}_{d}}, \widetilde{\FF}, \left(\widetilde{\FF}_{t}\right)_{t\geqslant 0} \right)$ where 
{  \[\forall d \in \N^{\star}, \quad \widetilde{\Omega}_{d} := \left\{ X \in \CCCC\left(\left[0, +\infty\right), \MM_{1}^{c,2}\left(\R^{d}\right) \right) \left| \phantom{1^{1^{1^{1}}}} \hspace{-0.6cm} \right. \forall T>0, \ \sup_{0\leqslant t\leqslant T}{\int_{\R^{d}}^{}{\left\|x \right\|^{2}X_{t}(\dd x)} < \infty } \right\}\]}
is endowed with the trace of the \textsc{Skorohod} topology on { $\Omega_d:=\CCCC([0, +\infty), \MM_{1}(\R^{d}) )$, the set of continuous functions from $[0,+\infty)$ to $\MM_{1}(\R^{d})$}, $\widetilde{\FF}$ is the trace of the \textsc{Borel} $\sigma-$field $\FF$ on $\Omega_d$ and $(\widetilde{\FF}_{t})_{t\geqslant 0}$ is the trace of the \textsc{Borel} filtration $(\FF_t)_{t\geqslant 0}$ on $\Omega_d$. Let us denote {  $\widetilde{\Omega} := \widetilde{\Omega}_{1}$}. \\
\indent We introduce several notations that we use repeatedly in the sequel. For a measurable real bounded function $f$ and a measure $\nu \in { \MM_{1}\left(\R^{d}\right)}$, we denote $\left\langle f, \nu \right\rangle := \int_{  \R^{d}}^{}{f(x)\nu(\dd x)}.$ We denote by { $\id = (\id_{1}, \cdots, \id_{d})$ the identity function on $\R^d$ where for all $k \in \{1, \cdots, d \}$, $\id_{k} : \R^{d} \to \R, x = (x_{1}, \cdots, x_{d}) \mapsto x_{k}$}.  We denote $\N := \left\{0,1,2, \cdots \right\}$ and $\N^{\star} := \N \setminus{\{0\}}$.  {  If 
$\DD$ is a domain of $\R^{d}$, then for all $\ell \in \N$, we denote by $\CCCC^{\ell}(\DD,\R)$ the space of functions of class $\CCCC^{\ell}$ from $\DD$ to $\R$. For $\ell \in \N$, we denote by $\CCCC^{\ell}_{b}\left(\R^{d},\R\right)$ the space of real bounded functions of class $\CCCC^{\ell}\left(\R^{d},\R\right)$ with bounded derivatives.} {  For all $F \in \CCCC^{2}\left(\R, \R \right)$ and $g \in \CCCC^{2}_{b}\left(\R^{d}, \R\right)$ we denote for all $\nu \in \MM_{1}\left(\R^{d}\right)$, $F_{g}(\nu) := F\left(\left\langle g, \nu \right\rangle \right)$.
We denote by $\Delta$ the Laplacian operator on $\R^{d}$ and by $x^{t}$ the transpose vector of $x\in \R^{d}$. For any function $f$ whose second partial derivatives exist, we denote by $\Hess(f) := \left(\partial_{ij}^{2}f \right)_{1\leqslant i,j \leqslant d}$ the Hessian matrix of $f$. 

\begin{Def} A probability measure on $\P_{\mu} \in \MM_{1}(\widetilde{\Omega}_{d})$ is said to solve the  \emph{centered \textsc{Fleming-Viot} martingale problem} with resampling rate $\gamma \in (0,+\infty)$ and  with initial condition $\mu \in \MM_{1}^{c,2}(\R^{d})$ if the canonical process $\left(X_{t} \right)_{t\geqslant0}$ on $\widetilde{\Omega}_{d}$ satisfies $\P_{\mu}(X_{0} = \mu) = 1$ and for each $F\in \CCCC^{2}(\R, \R)$ and $g \in \CCCC^{2}_{b}\left(\R^{d},\R\right)$,
\begin{equation}
 \widehat{M}_{t}^{F,d}(g)  :=  F_{g}\left(X_{t}\right)  - F_{g}\left(X_{0}\right)   - \int_{0}^{t}{\LL_{\rm FVc}^{d}F_{g}\left(X_{s}\right)\dd s} \label{PB_Mg_Z_Multi_d}
 \end{equation}
with for all $\varpi \in \MM_{1}^{c,2}\left(\R^{d}\right)$, 
\begin{equation}
\begin{aligned}
\LL_{\rm FVc}^{d}F_{g}(\varpi) & := F'\left(\left\langle g, \varpi \right\rangle \right)\left(\left\langle\frac{\Delta g}{2}, \varpi \right\rangle + \gamma \left[\left\langle \id^{t}\left\langle \Hess\left(g \right), \varpi \right\rangle \id, \varpi \right\rangle - 2\left\langle \nabla g^t \id, \varpi \right\rangle \right] \right) \\
& \hspace{-1.75cm} + \gamma F''\left(\left\langle g, \varpi \right\rangle\right)\left(\left\langle g^{2}, \varpi \right\rangle - \left\langle g, \varpi \right\rangle^{2} + \left\langle \nabla g, \varpi \right\rangle^{t}\left\langle \id \, \id^{t}, \varpi \right\rangle \left\langle \nabla g, \varpi \right\rangle - 2\left\langle \nabla g, \varpi \right\rangle^t \left\langle g \times \id, \varpi \right\rangle \right) \\ & \hspace{-1.25cm} =F'\left(\left\langle g, \varpi \right\rangle \right)\left(\sum_{i\, = \, 1}^{d}{\left\langle \frac{\partial_{ii}^{2}g}{2}, \varpi \right\rangle} + \gamma \left[\sum_{i,j \, = \, 1}^{d}{\left\langle \partial_{ij}^{2}g, \varpi \right\rangle \left\langle \id_{i}\times \id_{j}, \varpi \right\rangle} - 2\sum_{i\, = \, 1}^{d}{\left\langle \partial_{i}g \times \id_{i}, \varpi \right\rangle} \right] \right) \\
& \hspace{-1.75cm} + \gamma F''\left(\left\langle g, \varpi \right\rangle\right)\left(\left\langle g^{2}, \varpi \right\rangle - \left\langle g, \varpi \right\rangle^{2}  + \sum_{i, j \, = \, 1}^{d}{\left\langle \partial_{i}g, \varpi \right\rangle \left\langle \partial_{j}g, \varpi \right\rangle \left\langle \id_{i}\times \id_{j}, \varpi \right\rangle} \right. \\
& \hspace{8.25cm} - \left. 2\sum_{i\, = \, 1}^{d}{\left\langle \partial_{i}g, \varpi \right\rangle \left\langle g \times \id_{i}, \varpi \right\rangle} \right)
\end{aligned}
\label{Eq_Generateur_L_FVc_Multi_d}
\end{equation}
is a continuous $\P_{\mu}-$martingale in $L^{2}\left(\widetilde{\Omega}_{d} \right)$ with quadratic variation process 
\begin{equation}
\begin{aligned}
    \left\langle \widehat{M}^{F,d}(g) \right\rangle_{t} & = 2\gamma \int_{0}^{t}{\left[F'\left(\left\langle g, X_{s} \right\rangle \right)\right]^{2}\left[\left\langle g^{2}, X_{s}  \right\rangle - \left\langle g, X_{s}  \right\rangle^{2}  \right. } \\
		 & \hspace{1.5cm} + \left. \left\langle \nabla g, X_{s}  \right\rangle^{t}\left\langle \id \, \id^{t}, X_{s} \right\rangle \left\langle \nabla g, X_{s} \right\rangle - 2\left\langle \nabla g, X_{s} \right\rangle^t \left\langle g \times \id, X_{s} \right\rangle \right] \dd s.
\end{aligned}
\label{Variation_quad_FV_recentre_Multi_d}
\end{equation}
    \label{Def_PB_Mg_Z_Multi_d}
\end{Def}
}

{ 
\begin{Rem}
    Defining
    \[ \MM_{1}^{2}(\R^{d}) := \Big\{\nu \in \MM_{1}(\R^{d}) \,\Big|\, \left\langle \left\|\id \right\|^{2}, \nu \right\rangle  <\infty \Big\} \]
    the above definition takes a simpler form in the case $d=1$. Indeed, defining
    for all $\nu\in\MM_{1}^{2}(\R)$
    \[M_{2}\left(\nu \right) := \int_{\R}^{}{\left|x - \left\langle  \id, \nu \right\rangle \right|^{2} \nu(\dd x)},  \]
    we notice that $M_2(\nu)=\langle \id^2,\nu\rangle$ if $\nu\in\MM_1^{c,2}(\R)$. Then, with the notations of {\rm Definition~\ref{Def_PB_Mg_Z_Multi_d}}, the generator of the centered \textsc{Fleming-Viot} process is given by
    \begin{equation}   
    \begin{aligned}
    \LL_{\rm FVc}F_{g}(\varpi) & :=\LL^1_{\rm FVc}F_{g}(\varpi)= F'\left(\left\langle g, \varpi \right\rangle \right)\left(\left\langle \frac{g''}{2}, \varpi \right\rangle + \gamma\left[\left\langle g'', \varpi \right\rangle M_{2}(\varpi) - 2\left\langle g' \times \id , \varpi \right\rangle \right] \right)\\
    & \quad + \gamma F''\left(\left\langle g, \varpi \right\rangle  \right)\left(\left\langle g^{2}, \varpi  \right\rangle - \left\langle g, \varpi \right\rangle^{2} + \left\langle g', \varpi \right\rangle^{2}M_{2}\left(\varpi \right) - 2\left\langle g', \varpi \right\rangle\left\langle g \times \id, \varpi   \right\rangle    \right).
    \end{aligned}
    \label{Eq_Generateur_L_FVc}
    \end{equation}
\end{Rem}
}


{  The  original \textsc{Fleming-Viot} process is a measure-valued diffusion in { $\MM_{1}\left(\R^{d}\right)$}. We consider the filtered probability space $\left({ \Omega_{d}}, \FF, \left(\FF_{t}\right)_{t\geqslant 0} \right)$ as defined above. }
We recall that a probability measure $\P_{\nu}^{\rm FV} \in { \MM_{1}\left(\Omega_d\right)}$ is said to solve the  \emph{original \textsc{Fleming-Viot} martingale problem} {  with resampling rate $\gamma$ and}  initial condition $\nu \in { \MM_{1}(\R^d)}$ if the canonical process $\left(Y_{t} \right)_{t\geqslant 0}$ on { $\Omega_d$} satisfies $\P_{\nu}^{\rm FV}(Y_{0} = \nu) = 1$ and for each $F\in \CCCC^{2}(\R, \R)$ and $g \in \CCCC^{2}_{b}({ \R^d},\R)$, 
 \begin{equation}
 \begin{aligned}
M_{t}^{F}(g) & := F\left(\left\langle g, Y_{t}  \right\rangle \right)  - F\left(\left\langle g, Y_{0}  \right\rangle \right)  - \int_{0}^{t}{F'\left(\left\langle g, Y_{s}  \right\rangle \right) \left\langle \frac{\Delta g}{2}, Y_{s}   \right\rangle \dd s }   \\ 
& \hspace{1cm} - \gamma\int_{0}^{t}{F''\left(\left\langle g, Y_{s}  \right\rangle \right)\left[\left\langle g^{2}, Y_{s}  \right\rangle - \left\langle g, Y_{s} \right\rangle^{2}  \right] \dd s }  
\end{aligned}
 \label{PB_Mg_Y}
 \end{equation}
is a {  continuous} square integrable  $\P_{\nu}^{\rm FV}-$martingale whose bracket satisfies for all $G, H \in \CCCC^{2}(\R, \R)$ and for all $g, h \in \CCCC^{2}_{b}(\R^d, \R)$,
\begin{equation}
\left\langle M^{G}(g), M^{H}(h) \right\rangle_{t}  = 2\gamma \int_{0}^{t}{G'\left(\left\langle g, Y_{s} \right\rangle\right)H'\left(\left\langle h, Y_{s} \right\rangle\right)\left[\left\langle gh, Y_{s} \right\rangle - \left\langle g, Y_{s} \right\rangle \left\langle h, Y_{s}\right\rangle \right]\dd s}. 
\label{Crochet_M_G_M_H_general}
\end{equation}
{  In the case $d=1$, we will denote $\Omega := \Omega_{1}$.}

In the population genetics literature, the terms involving the first order derivative $F'$ in~\eqref{PB_Mg_Y} describe the effect of the mutation whereas the ones involving the second order derivative $F''$ describe the effect of the random genetic drift. It is well-known that, for all $\nu \in { \MM_{1}(\R^d)}$, there exists a unique probability measure $\P_{\nu}^{\rm FV} \in { \MM_{1}\left(\Omega_d \right)}$ satisfying the previous martingale problem~{\rm{(\ref{PB_Mg_Y})}} {\color{blue}\cite[\color{black} Theorem 3]{fleming_no_1979}}.

\begin{Rem} \begin{itemize}
    \item[{\rm\textbf{(1)}} ] The additional terms in the martingale problem {\rm{(\ref{PB_Mg_Z_Multi_d})}} with respect to the martingale problem {\rm{(\ref{PB_Mg_Y})}} describe the impact of centering and ensure that at all times the centered \textsc{Fleming-Viot} process remains $\MM_{1}^{c,2}({ \R^d})-$valued.  
     \item[{\rm\textbf{(2)}} ] Note that a major difficulty comes from the presence of terms in { $\langle \id_{i}\times \id_{j}, \mu\rangle$, $i,j \in \{1,\cdots, d \}$ in~{\rm (\ref{Eq_Generateur_L_FVc_Multi_d})} or} $M_{2}(\mu)$ in {\rm(\ref{Eq_Generateur_L_FVc})}. We will see in particular that this leads to the creation of particles in the dual process.
\end{itemize}
\label{Rem_effet_recentrage}
\end{Rem}

{  Let us define by $\tau_{\alpha}$ the translation operator of vector $\alpha \in \R^{d}$. For all $\mu \in \MM_{1}\left(\R^{d}\right)$, for all $A \in \BB\left(\R^{d}\right)$, \[\tau_{\alpha} \sharp \, \mu(A) := \mu\left(\tau_{\alpha}^{-1}(A) \right) = \mu\left(\left\{x - \alpha \left| \phantom{1^{1^{1^{1}}}} \hspace{-0.6cm} \right. x \in A \right\} \right), \] 
where $\sharp$ is the \emph{pushforward operator}.} \\

The main result of this section is the following:

\begin{Thm} For all $\mu \in { \MM_{1}^{c,2}\left(\R^{d}\right)}$, there exists a probability measure $\P_{\mu} \in { \MM_{1}\left(\widetilde{\Omega}_{d} \right)}$ satisfying the martingale problem of { {\rm{Definition~\ref{Def_PB_Mg_Z_Multi_d}}}}, given by the law of the process $\left(Z_{t} \right)_{t\geqslant 0}$ defined by \begin{equation}
 Z_{t} := \tau_{-\left\langle \id, Y_{t}  \right\rangle} \sharp \, Y_{t} := Y_{t}\left( \cdot + \left\langle \id, Y_{t}  \right\rangle \right), \qquad t \geqslant 0  \label{Ecriture_Z}
\end{equation}
where $\left(Y_{t}\right)_{t\geqslant 0}$ is the original \textsc{Fleming-Viot} process.
\label{Prop_PB_Mg_Z}
\end{Thm}

{  Note that, for all probability measure with finite first-order moment, $\langle \id,\mu\rangle$ is the mean value (or barycenter) of $\mu$. In addition, $\left\langle \id, Z_{t} \right\rangle = 0$ for all $t\geqslant 0$. As a result, the process $\left(Z_{t} \right)_{t\geqslant 0}$ corresponds to the original \textsc{Fleming-Viot} process centered by its mean value, hence its name of \emph{centered \textsc{Fleming-Viot} process}.}

\subsection{Sketch of proof of Theorem \ref{Prop_PB_Mg_Z} \label{Sous_section_2_2}}

{  The proof of Theorem \ref{Prop_PB_Mg_Z}, for any dimension $d \in \N^{\star}$ is much more complicated, but it is treated in the same way as for dimension $1$. In the following, only the case where $d=1$ is detailed for greater clarity.} The proof is based on the original \textsc{Fleming-Viot} martingale problem  {\rm{(\ref{PB_Mg_Y})}}. {  Below, we give the main ideas of the proof. Full details are provided in Section \ref{Section_preuve_Existence}.}

\subsubsection{Framework and objective of the proof}

Let $\nu \in \MM_{1}^{2}(\R)$ and $\P_{\nu}^{\rm FV}$ the unique solution to the original \textsc{Fleming-Viot} martingale problem (\ref{PB_Mg_Y}). As the support of the original \textsc{Fleming-Viot} process is compact at all positive times $\P_{\nu}^{\rm FV}-$a.s. {\color{blue} \cite{liu_support_2015}}, $\P_{\nu}^{\rm FV}\left(\CCCC^{0}\left(\left[0,+\infty\right), \MM_{1}^{2}(\R) \right) \right) = 1$.  
Moreover, as $t \mapsto \left\langle \id, Y_{t} \right\rangle$ is continuous $\P_{\nu}^{\rm FV}-$a.s. ({\rm{Lemma \ref{Lem_Mart_id_id_bien_def} \textbf{(2)}}}), we deduce that, for all $t \geqslant 0$, $Z_{t}$ given by (\ref{Ecriture_Z}), is well defined and is a random variable on $\widetilde{\Omega}$. \\

When the dependency of $Y$ and $Z$ on the initial condition $\nu$ of $Y$ is important, we shall use the notation $\left(Y_{t}^{\nu} \right)_{t\geqslant 0}$ and $\left(Z_{t}^{\nu} \right)_{t\geqslant 0}$ instead of $\left(Y_{t} \right)_{t\geqslant 0}$ and $\left(Z_{t} \right)_{t\geqslant 0}$. Our goal is to prove that the law of the process $\left(Z_{t}\right)_{t \geqslant 0}$ denoted by $\P_{\tau_{-\left\langle \id, \nu\right\rangle}\sharp \, \nu}^{\rm FVc}$ solves the martingale problem of {\rm{Definition \ref{Def_PB_Mg_Z_Multi_d}}} with initial condition $\tau_{-\left\langle \id, \nu\right\rangle}\sharp \, \nu$.  Note that the notation $\P_{\tau_{-\left\langle \id, \nu\right\rangle}\sharp \, \nu}^{\rm FVc}$ is justified 
because the original \textsc{Fleming-Viot} process is invariant by translation: 

\begin{Prop}
Let $\nu \in \MM_{1}^{{  2}}(\R)$ and $a \in \R$. Then, the law of $Z_{t}^{\nu}$ defined by {\rm(\ref{Ecriture_Z})} is the same as the law of $Z_{t}^{\tau_{a} \sharp \,  \nu}$.
\label{Prop_notation_P_nu_c}
\end{Prop}

\begin{proof}
By translation invariance of the original \textsc{Fleming-Viot} process, the process $\left(\tau_{-a}\sharp\, Y_{t}^{\tau_{a}\sharp \, \nu} \right)_{t\geqslant 0}$ has the same law as the process $\left(Y_{t}^{\nu} \right)_{t\geqslant 0}$. Now, \begin{align*}
Z_{t}^{\tau_{a}\sharp \, \nu} = \tau_{-\left\langle \id, Y_{t}^{\tau_{a}\sharp \,  \nu} \right\rangle} \sharp \, Y_{t}^{\tau_{a}\sharp \,  \nu} = \tau_{-\left\langle \id, \tau_{-a}\sharp\, Y_{t}^{\tau_{a}\sharp \,  \nu} \right\rangle } \sharp\, \left(\tau_{-a} \sharp\, Y_{t}^{\tau_{a}\sharp \,  \nu}\right).
\end{align*}
Thus, $\left(Z_{t}^{\tau_{a} \sharp \nu} \right)_{t\geqslant 0}$ has the same law as $\left( \tau_{-\left\langle \id, Y_{t}^{\nu} \right\rangle } \sharp \, Y_{t}^{\nu} \right)_{t\geqslant 0} = \left(Z_{t}^{\nu} \right)_{t\geqslant 0}$.
\end{proof}

\subsubsection{Outline of the proof}

We restrict to the time interval $[0,T]$ for $T>0$ arbitrary. By standard arguments, it is sufficient to prove that, for all $F \in \CCCC^{2}(\R, \R)$ and $g \in \CCCC^{2}_{b}(\R, \R)$, 
{ \begin{equation*}
F_{g}\left(Z_{t} \right) - F_{g}\left(Z_{0} \right) - \int_{0}^{t}{\LL_{\rm FVc}F_{g}\left(Z_{s} \right)\dd s}
\label{Mart_outline_the_proof}
\end{equation*}}
is a {  continuous} $\P_{\nu}^{\rm FV}-$martingale, $\nu \in \MM_{1}^{2}(\R)$ {  where $\LL_{\rm FVc}$ is given by (\ref{Eq_Generateur_L_FVc})}. We start by assuming $F, g \in \CCCC^{4}_{b}(\R, \R)$ and we seek for the \textsc{Doob}'s semi-martingale decomposition of 
\begin{align*}
F_{g}(Z_{t }) := F\left(\left\langle g, Z_{t } \right\rangle \right) = F\left(\left\langle g\circ\tau_{-\left\langle \id, Y_{t } \right\rangle}, Y_{t } \right\rangle \right),
\end{align*}
using the original \textsc{Fleming-Viot} martingale problem~(\ref{PB_Mg_Y}).  However, $F\left(\left\langle g\circ\tau_{-\left\langle \id, Y_{t }  \right\rangle}, Y_{t } \right\rangle \right)$
 does not take the form $H\left(\left\langle h, Y_{t }  \right\rangle \right)$ with \emph{deterministic} $h$. Therefore, we cannot apply (\ref{PB_Mg_Y}) directly. To get over this difficulty, we consider for $t \in [0,T]$, 
 an increasing sequence $0 = t_{0}^{n} < t_{1}^{n} < \dots < t_{p_{n}}^{n} = T$ of subdivisions of $[0,T]$ whose mesh tends to 0. We can observe that 
 \begin{equation*}
 \begin{aligned}
F_{g}(Z_{t }) - F_{g}(Z_{0}) & = \sum\limits_{i \, = \, 0}^{p_{n}-1}{\left\{F_{g}(Z_{t^{n}_{i+1} \wedge t }) - F_{g}(Z_{t^{n}_{i} \wedge t }) \right\}}  \\
& \hspace{0cm}  = \sum\limits_{i \, = \, 1}^{p_{n} - 1}{\left\{F\left(\left\langle g\circ\tau_{-\left\langle \id, Y_{t_{i+1}^{n} \wedge t  }  \right\rangle}, Y_{t_{i+1}^{n} \wedge t  }  \right\rangle \right) - F\left(\left\langle g\circ\tau_{-\left\langle  \id, Y_{t_{i\phantom{+1} \hspace{-0.25cm}}^{n} \wedge t  } \right\rangle}, Y_{t_{i}^{n} \wedge t  }  \right\rangle \right) \right\}.}
\end{aligned}
 \label{somme_differences_termes}
 \end{equation*}
Using asymptotic expansions (see Lemma \ref{Lem_FVr_p_variables} with $p = d = 1$) of the terms in the previous sum, we prove that 
\begin{equation}
\begin{aligned}
F_{g}(Z_{t }) - F_{g}(Z_{0}) & = \sum\limits_{i\, = \, 0}^{p_{n}-1}{} \left\{{\rm{\textbf{(A)}}}_{i} + {\rm{\textbf{(B)}}}_{i} + O\left(\left|\left\langle \id , Y_{t_{i+1}^{n} \wedge t  } - Y_{t_{i}^{n} \wedge t  } \right\rangle\right|^{3} \right) \phantom{\sum_{k}^{2}{}}\right.  \\
& \hspace{+1.25cm} + \left. O\left( \sum\limits_{k\, = \, 0}^{2}{\left|\left\langle g^{(k)}\circ\tau_{-\left\langle \id,  Y_{t_{i}^{n} \wedge t  } \right\rangle}, Y_{t_{i+1}^{n} \wedge t  } - Y_{t_{i}^{n} \wedge t  } \right\rangle\right|^{3}} \right)\right\}, 
\end{aligned}
\label{Existence_PB_mg}
\end{equation}
where 
\begin{align}
{\rm{\textbf{(A)}}}_{i} = \  & F'\left(\left\langle g\circ\tau_{-\left\langle \id , Y_{t_{i}^{n} \wedge t  } \right\rangle},  Y_{t_{i}^{n} \wedge t  } \right\rangle \right) \left\{\left\langle g\circ\tau_{-\left\langle  \id, Y_{t_{i}^{n} \wedge t  } \right\rangle}, Y_{t_{i+1}^{n} \wedge t  } - Y_{t_{i}^{n} \wedge t  }   \right\rangle   \right. \notag  \\   
& \hspace{-1.35cm}   - \left\langle \id, Y_{t_{i+1}^{n} \wedge t  } - Y_{t_{i}^{n} \wedge t  }  \right\rangle \left[\left\langle g'\circ\tau_{-\left\langle \id, Y_{t_{i}^{n} \wedge t  }  \right\rangle},  Y_{t_{i}^{n} \wedge t} \right\rangle + \left\langle g'\circ\tau_{-\left\langle \id, Y_{t_{i}^{n} \wedge t  }  \right\rangle}, Y_{t_{i+1}^{n} \wedge t  } - Y_{t_{i}^{n}\wedge t  } \right\rangle \right]     \label{Morceau_Ai_F_prime} \\ 
& \hspace{-1.35cm} + \left. \frac{1}{2} \left\langle \id, Y_{t_{i+1}^{n} \wedge t  } - Y_{t_{i}^{n} \wedge t  } \right\rangle^{2} \left\langle g''\circ\tau_{-\left\langle \id, Y_{t_{i}^{n} \wedge t  }  \right\rangle},  Y_{t_{i}^{n} \wedge t  } \right\rangle \right\}, \notag \\
{\rm{\textbf{(B)}}}_{i} = \  &  \frac{F''\left(\left\langle g\circ\tau_{-\left\langle \id,  Y_{t_{i}^{n} \wedge t  } \right\rangle}, Y_{t_{i}^{n} \wedge t  }  \right\rangle \right)}{2}  \left\{
\left\langle g\circ\tau_{-\left\langle \id, Y_{t_{i}^{n} \wedge t  } \right\rangle},  Y_{t_{i+1}^{n} \wedge t  } - Y_{t_{i}^{n} \wedge t  } \right\rangle^{2}   \right. \notag \\ 
&  \hspace{-1.35cm} + \left\langle \id, Y_{t_{i+1}^{n} \wedge t  } - Y_{t_{i}^{n} \wedge t  } \right\rangle^{2} \left\langle g'\circ\tau_{-\left\langle \id, Y_{t_{i}^{n}\wedge t  } \right\rangle}, Y_{t_{i}^{n} \wedge t  } \right\rangle^{2} \label{Morceau_Bi_F_seconde} \\ 
 &  \hspace{-1.35cm} - \left. 2\left\langle g\circ\tau_{-\left\langle \id , Y_{t_{i}^{n} \wedge t  }\right\rangle}, Y_{t_{i+1}^{n} \wedge t  } - Y_{t_{i}^{n} \wedge t  } \right\rangle\left\langle \id , Y_{t_{i+1}^{n} \wedge t  } - Y_{t_{i}^{n} \wedge t  } \right\rangle \left\langle g'\circ\tau_{-\left\langle \id, Y_{t_{i}^{n} \wedge t  } \right\rangle}, Y_{t_{i}^{n} \wedge t  } \right\rangle^{\phantom{2}} \hspace{-0.2cm} \right\}, \notag 
\end{align}
and where $g^{(j)}$, $j \in \{0, 1, 2 \}$, denotes the $j^{\rm{th}}$ derivative of $g$. {  Note that the proposed decomposition in (\ref{Existence_PB_mg}) is intended such that each of the terms in (\ref{Morceau_Ai_F_prime}) and (\ref{Morceau_Bi_F_seconde}) is either $\FF_{t_{i}^{n}\wedge t}-$adapted or, exhibit increments of $\left(Y_{t}\right)_{t\geqslant 0}$ between $t_{i}^{n}\wedge t$ and $t_{i+1}^{n}\wedge t$.} Several steps are described in Section \ref{Section_preuve_Existence} to obtain the semi-martingale decomposition of each term of the previous sum. {  Note that the last terms of ${\rm{\textbf{(A)}}}_{i}$ and ${\rm{\textbf{(B)}}}_{i}$ bring out the terms in $M_{2}(\mu)$.}  By making the step of the subdivision tend towards 0, we obtain the expected result.  

\subsection{Some properties of the centered \textsc{Fleming-Viot} process in dimension $1$\label{Sous_section_2_4}}

\subsubsection{\textsc{Markov}'s property \label{Sous_section_2_4_1}}

Due to the translation invariance property of the original \textsc{Fleming-Viot} process, we can prove that the centered \textsc{Fleming-Viot} process is homogeneous \textsc{Markov}.
 
 \begin{Prop}
 The centered \textsc{Fleming-Viot} process $\left(Z_{t} \right)_{t\geqslant 0}$ defined by {\rm{(\ref{Ecriture_Z})}} satisfies the  homogeneous \textsc{Markov} property: for all measurable bounded function $f$,
 \[\forall \mu \in \MM_{1}^{c,2}(\R),\quad  \forall t, s >0, \quad \quad \E_{\mu}\left(f(Z_{t+s}) \left| \phantom{\frac{1}{2} } \hspace{-0.2cm} \FF_{t} \right.\right) = \E_{Z_{t}}\left(f(Z_{s})\right) \quad \P_{\mu}{\rm{-a.s}}.\] 
 \end{Prop}

\begin{proof} Let $\mu \in \MM_{1}^{c,2}(\R)$ and $f$ a measurable bounded function. Let $t, s >0$. Using the \textsc{Markov} property of the original \textsc{Fleming-Viot} process $\left(Y_{t} \right)_{t\geqslant 0}$ we obtain $\P_{\mu}-$a.s.,
\begin{align*}
\E_{\mu}\left(f(Z_{t+s}) \left| \phantom{\frac{1}{2} } \hspace{-0.2cm} \FF_{t} \right. \right)&  = \E_{\mu}\left(f(\tau_{-\left\langle \id, Y_{t+s} \right\rangle} \sharp \, Y_{t+s}) \left| \phantom{\frac{1}{2} } \hspace{-0.2cm} \FF_{t} \right. \right)  = \E_{\mu}\left(g(Y_{t+s}) \left| \phantom{\frac{1}{2} } \hspace{-0.2cm} \FF_{t} \right. \right) = \E_{Y_{t}}\left(g(Y_{s}) \right)
\end{align*}
where the bounded measurable map $g$ is defined on $\MM_{1}^{2}(\R)$ by $g(\nu) := f\left(\tau_{-\left\langle \id, \nu \right\rangle} \sharp \, \nu \right)$. By invariance by translation of the original \textsc{Fleming-Viot} process $\left(Y_{t} \right)_{t\geqslant 0}$ we obtain under the distribution $\P_{\mu}$: 
\begin{align*}
\E_{Y_{t}}\left(g\left(Y_{s}\right) \right)& = \E_{\tau_{-\left\langle \id, Y_{t} \right\rangle} \sharp \, Y_{t}}\left(g\left(\tau_{\left\langle  \id, Y_{t} \right\rangle} \sharp \, Y_{s}\right) \right) = \E_{\tau_{-\left\langle \id, Y_{t} \right\rangle} \sharp \, Y_{t}}\left(g\left( Y_{s} \right) \right) = \E_{Z_{t}}\left(g\left( Y_{s} \right) \right)   = \E_{Z_{t}}\left(f\left( Z_{s} \right) \right). \qedhere
\end{align*} 
\end{proof}

\subsubsection{Moments and some martingales \label{Sous_section_2_4_2}} 
\begin{Prop}
Let $\mu \in \MM_{1}^{c,2}(\R)$ and let $\P_{\mu}$ be a distribution on $\widetilde{\Omega}$ satisfying {\rm{(\ref{PB_Mg_Z_Multi_d})}} and such that $X_{0} $ is equal in law to $\mu$. Let $T > 0$ and $k \in \N \, \backslash \, \{0,1 \}$ be fixed. 
\begin{itemize}
\item[{\rm{\textbf{{\rm{\textbf{(1)}}}}}}]   If $\left\langle \left|\id \right|^{k}, \mu \right\rangle  <\infty$, there exist two constants $C_{k,T}, \widetilde{C}_{k,T} >0$, such that any stochastic process $\left(X_{t} \right)_{0\leqslant t \leqslant T}$ whose law is $\P_{\mu}$ satisfies
\begin{align}
{\rm{\textbf{\rm{\textbf{(a)}}}}} \qquad &  \sup_{t\in[0,T]}{\E_{  \mu}\left(\left\langle \left|\id\right|^{k}, X_{t}\right\rangle\right)} \leqslant  C_{k,T}\left(1 + \left\langle \left|\id \right|^{k}, \mu \right\rangle  \right)
 \label{Moments_Z_id_k}  \\
{\rm{\textbf{\rm{\textbf{(b)}}}}} \qquad & \forall \alpha>0, \quad  \P_{\mu}\left(\sup_{t\in[0,T]}{\left\langle \left|\id\right|^{k}, X_{t}\right\rangle } \geqslant \alpha \right) \leqslant \frac{\widetilde{C}_{k,T}\left(1 + \left\langle \left|\id \right|^{k}, \mu \right\rangle \right)}{\alpha}. \notag 
\end{align}
\item[{\rm{\textbf{{\rm{\textbf{(2)}}}}}}] If $\left\langle \left|\id \right|^{k}, \mu \right\rangle  <\infty$, respectively $\left\langle \left|\id \right|^{k+1}, \mu \right\rangle <\infty$, the process $\left(\widehat{M}_{t}^{\id}\left(\id^{k}\right)\right)_{0\leqslant t \leqslant T}$
defined by \begin{align*}
\widehat{M}_{t}^{\id}\left(\id^{k} \right) & := \left\langle \id^{k}, X_{t} \right\rangle - \left\langle \id^{k}, X_{0} \right\rangle - \int_{0}^{t}{\left\langle \frac{k(k-1)}{2} \id^{k-2}, X_{s}\right\rangle \dd s} \\
& \hspace{1cm} -\gamma \int_{0}^{t}{\left[\left\langle k(k-1)\id^{k-2}, X_{s} \right\rangle M_{2}(X_{s}) - 2k\left\langle \id^{k}, X_{s} \right\rangle \right] \dd s} 
\end{align*}
  is a continuous $\P_{\mu}-$local martingale, respectively a continuous $\P_{\mu}-$martingale.  Moreover, if   $\left\langle \left|\id \right|^{2k}, \mu \right\rangle  <\infty$, then 
  $\left(\widehat{M}^{\id}\left(\id^{k}\right) \right)_{0\leqslant t \leqslant T}$ is a martingale in $L^{2}\left(\widetilde{\Omega} \right)$ whose quadratic variation is given by
\begin{align*}
\left\langle \widehat{M}^{\id}\left(\id^{k}\right) \right\rangle_{t}  & = 2\gamma \int_{0}^{t}{\left[\left\langle \id^{2k}, X_{s}  \right\rangle - \left\langle \id^{k}, X_{s}  \right\rangle^{2} + k\left\langle \id^{k-1}, X_{s}\right\rangle^{2}M_{2}(X_{s})  \right.} \\
 & \hspace{3cm} -\left. 2k\left\langle \id^{k-1}, X_{s} \right\rangle \left\langle \id^{k+1}, X_{s} \right\rangle \right]\dd s.
\end{align*}
\end{itemize}
\label{Prop_moments_non_bornes_RECENTRE}
\end{Prop}
\vspace{-0.35cm}
{  Note that the properties for $k = 1$ fail because of the term $M_{2}$ in the expression of $\LL_{\rm FVc}$. Note also that this results entails that $Z_t$ is a continuous martingale.}

\begin{proof} 
\textbf{Step 1. Proof of (1)(a).} We prove only the case $k\geqslant 3$: the case $k = 2$, which is simpler because some terms disappear, is treated in the same way. Let $t \in [0,T]$. We consider a sequence of functions $\left(g_{n}\right)_{n\in \N}$ of class $\CCCC^{2}(\R, \R)$ with compact support satisfying: 
\begin{center}
\begin{tabular}{ll}
\textbf{(i)} \ for all $n\in \N$, $|g_{n}| \leqslant |\id|$,  \qquad \qquad &  \textbf{(iii)} \ $g_{n} = \id$ on $[-n,n]$, \\ \textbf{(ii)} \ $\displaystyle{\lim_{n\to + \infty}{\left\|g''_{n}\right\|_{\infty}}= 0}$, & \textbf{(iv)} \ $g_{n}'$ is uniformly bounded on $\R$.
\end{tabular} 
\end{center}
We consider the sequence of functions $\left(h_{n}\right)_{n\in \N}$ defined by $h_{n} := \sqrt{1 + g_{n}^{2}}$ and we deduce from the properties of $g_{n}$ that for all $n \in \N$, $h_{n}$ is a non-negative function with compact support, 
that for all $k \in \N$, 
\[\left(h_{n}^{k} \right)' = kg_{n}g_{n}'h_{n}^{k-2} \qquad  {\rm{and}} \qquad  \left(h_{n}^{k} \right)'' = k\left(g_{n}' \right)^{2}h_{n}^{k-4}\left(h_{n}^{2} + (k-2)g_{n}^{2} \right) + kg_{n}g_{n}''h_{n}^{k-2}, \]
$h_{n} = h := \sqrt{1 + \id^{2}}$ on the compact set $[-n,n]$ and $h_{n} \leqslant h$ on $\R$. We consider for all $A \in \N$ and $\ell \in \N$,  the stopping time $\tau_{A,\ell} := \inf{\left\{ t \geqslant 0\left. \phantom{1^{1^{1^{1}}}} \hspace{-0.6cm} \right|\left\langle \left|\id \right|^{\ell}, X_{t}\right\rangle  \geqslant A  \right\}}$. Noting that for all $t \in [0,T]$, $n \in \N$ and $k \geqslant 3$, 
$\left\langle h_{n}^{k-2},  X_{t} \right\rangle \leqslant \left\langle h_{n}^{k},  X_{t} \right\rangle^{\frac{k-2}{k}}   \leqslant \left\langle h^{k},  X_{t} \right\rangle$ and $\left\langle h_{n}^{k-2},  X_{t} \right\rangle \left\langle \id^{2},  X_{t} \right\rangle \leqslant \left\langle h^{k},  X_{t} \right\rangle$ from \textsc{H\"{o}lder}'s inequality, we deduce from the martingale problem (\ref{PB_Mg_Z_Multi_d}) that there exists constants $C_{1}(k), C_{2}(k,A)>0$ such that
{ 
\begin{align*}
\E_{  \mu}\left(\left\langle h_{n}^{k}, X_{t\wedge \tau_{A,k}} \right\rangle \right)&  
\leqslant \left\langle h^{k}, \mu \right\rangle  + C_{1}(k)\E_{  \mu}\left(\int_{0}^{t\wedge \tau_{A,k}}{\left\langle h^{k}, X_{s} \right\rangle \dd s} \right) + C_{2}(k,A)\left\|g_{n}'' \right\|_{\infty}.
\end{align*}}
By \textsc{Fatou}'s lemma we obtain when $n \to + \infty$, 
\begin{align*}
\E_{  \mu}\left(\left\langle h^{k}, X_{t\wedge \tau_{A,k}} \right\rangle \right) \leqslant \left\langle h^{k}, \mu \right\rangle  + C_{1}(k)\E_{  \mu}\left(\int_{0}^{t\wedge \tau_{A,k}}{\left\langle h^{k}, X_{s} \right\rangle \dd s} \right).
\end{align*}
By \textsc{Gronwall}'s lemma, we deduce that 
\begin{equation}
\E_{  \mu}\left(\left\langle \left|\id \right|^{k}, X_{t\wedge \tau_{A,k}} \right\rangle \right) \leqslant  \E_{  \mu}\left(\left\langle h^{k}, X_{t\wedge \tau_{A,k}} \right\rangle \right) \leqslant \left\langle h^{k}, \mu \right\rangle  \exp\left(C_{1}(k) t \right).
\label{Gronwall_moments}
\end{equation}
In particular, this implies that the sequence $\left(\tau_{A,k} \right)_{A\in \N}$ converges $\P_{\mu}-$a.s. to infinity. Indeed, for all $\widetilde{T} >0$, we have \[\P_{\mu}\left(\sup_{A\in \N}{\tau_{A,k}} < \widetilde{T} \right) \leqslant \frac{\sup_{t\in[0,T]}{\E_{  \mu}\left(\left\langle \left|\id \right|^{k}, X_{t\wedge \widetilde{T} \wedge \tau_{A,k}} \right\rangle \right)}}{A}  \] which tends to $0$ when $A \to + \infty$. We deduce by \textsc{Fatou}'s lemma, when $A \to + \infty$, the first announced result. \\

\textbf{Step 2. Proof of (1)(b).} Let $\alpha >0$. From the martingale problem (\ref{PB_Mg_Z_Multi_d}), we deduce that 
\begin{align*}
\P_{\mu}\left(\sup_{t\in \left[0, T\wedge \tau_{A,k} \right]}{\left\langle h_{n}^{k}, X_{t} \right\rangle} \geqslant \alpha \right) & \leqslant \P_{\mu}\left(\left\langle h_{n}^{k}, \mu \right\rangle \geqslant \frac{\alpha}{5} \right) + \P_{\mu}\left(\int_{0}^{T\wedge \tau_{A,k}}{\left\langle \frac{\left(h_{n}^{k} \right)''}{2}, X_{s} \right\rangle \dd s} \geqslant \frac{\alpha}{5} \right) \\
& \quad + \P_{\mu}\left(\gamma \int_{0}^{T\wedge \tau_{A,k}}{\left\langle \left(h_{n}^{k} \right)'', X_{s} \right\rangle M_{2}(X_{s}) \dd s} \geqslant \frac{\alpha}{5} \right) \\
& \quad + \P_{\mu}\left(2\gamma \int_{0}^{T\wedge \tau_{A,k}}{\left\langle \left(h_{n}^{k} \right)' \times \id, X_{s} \right\rangle M_{2}(X_{s}) \dd s} \geqslant \frac{\alpha}{5} \right) \\
& \quad + \P_{\mu}\left(\sup_{t\in \left[0, T\wedge \tau_{A,k} \right]}{\left|\widehat{M}^{\id}_{t}\left(h_{n}^{k} \right) \right|} \geqslant \frac{\alpha}{5}\right)
\end{align*}
The \textsc{Doob} maximal inequality allows us to write \[\P_{\mu}\left(\sup_{t\in \left[0, T\wedge \tau_{A,k} \right]}{\left|\widehat{M}^{\id}_{t}\left(h_{n}^{k} \right) \right|} \geqslant \frac{\alpha}{5}\right) \leqslant \frac{5\E_{  \mu}\left(\left(\widehat{M}^{\id}_{T\wedge \tau_{A,k}}\left(h_{n}^{k} \right)\right)_{+} \right)}{\alpha}. \]
From the martingale problem (\ref{PB_Mg_Z_Multi_d}) and the computations of Step 1, we deduce that 
\begin{align*}
\E_{  \mu}\left(\left|\widehat{M}^{\id}_{T\wedge \tau_{A,k}}\left(h_{n}^{k} \right)\right| \right) & \leqslant 2 \left\langle h^{k}, \mu \right\rangle + 2 C_{1}(k)\E_{  \mu}\left(\int_{0}^{T\wedge \tau_{A,k}}{\left\langle h^{k}, X_{s}\right\rangle \dd s} \right) + 2C_{2}(k,A)\left\|g_{n}'' \right\|_{\infty} \\
& \leqslant 2\left\langle h^{k}, \mu \right\rangle \left[ 1 + \exp\left(C_{1}(k)T \right) \right] + 2C_{2}(k,A)\left\|g_{n}'' \right\|_{\infty},
\end{align*}
where we use the \textsc{Fubini-Tonelli} theorem and the relation (\ref{Gronwall_moments}). It follows, from \textsc{Markov}'s inequality, that there exists a constant $C_{k} >0$ such that  
\begin{align*}
\P_{\mu}\left(\sup_{t\in \left[0, T\wedge \tau_{A,k} \right]}{\left\langle h_{n}^{k}, X_{t} \right\rangle} \geqslant \alpha \right) & \leqslant \frac{C_{k}}{\alpha}\left[\left\langle h^{k}, \mu \right\rangle  \left[ 1 + \exp\left(C_{1}(k)T \right) \right] + C_{2}(k,A)\left\|g_{n}'' \right\|_{\infty} \right].
\end{align*}
By applying the dominated convergence theorem twice, successively when $n \to + \infty$ then when $A \to + \infty$, there exists a constant $\widetilde{C}_{k,T}>0$ such that \[\P_{\mu}\left(\sup_{t\in \left[0, T \right]}{\left\langle h^{k}, X_{t} \right\rangle} \geqslant \alpha \right)  \leqslant \frac{\widetilde{C}_{k,T}\left\langle h^{k}, \mu \right\rangle }{\alpha},\]
and thus the announced result. \\

\textbf{Step 3. $\widehat{M}^{\id}\left(\id^{k}\right)$ is a continuous local martingale.} From the properties of $\left(g_{n} \right)_{n\in \N}$, note that there exists a constant $\widehat{C}_{k} > 0$ such that for all $n \in \N, \quad \left|g_{n}^{k}\right| \leqslant \left|\id \right|^{k}$ and $\left|\left(g_{n}^{k} \right)'' \right| \leqslant \widehat{C}_{k}\left(1 + \left|\id \right|^{k-1} \right)$. It follows from the martingale problem (\ref{PB_Mg_Z_Multi_d}), the properties of $\left(g_{n} \right)_{n\in \N}$ and the dominated convergence theorem for conditional expectation that 
\begin{align*}
\widehat{M}^{\id}_{t\wedge \tau_{A, 2}}\left(\id^{k} \right) & := \lim_{n \to + \infty}{\widehat{M}_{t\wedge \tau_{A,2}}^{\id}\left(g_{n}^{k}\right)} \\
& = \left\langle \id^{k}, X_{t\wedge \tau_{A, 2}} \right\rangle - \left\langle \id^{k}, X_{0} \right\rangle - \int_{0}^{t\wedge \tau_{A, 2}}{ \frac{k(k-1)}{2}\left\langle \id^{k-2}, X_{s}\right\rangle \dd s} \\ 
&  \qquad - \gamma \int_{0}^{t\wedge \tau_{A, 2}}{\left[k(k-1)\left\langle  \id^{k-2}, X_{s} \right\rangle M_{2}(X_{s}) - 2k\left\langle \id^{k}, X_{s} \right\rangle \right]\dd s}
\end{align*} 
is a continuous $\P_{\mu}-$martingale and thus $\left(\widehat{M}^{\id}\left(\id^{k}\right)\right)_{0\leqslant t \leqslant T}$ is a continuous $\P_{\mu}-$local martingale. When $\left\langle \left|\id \right|^{k+1}, \mu \right\rangle < \infty$, using the inequality for all $t \in [0,T]$, \[\left\langle \left|\id\right|^{k-1},  X_{t} \right\rangle \left\langle \id^{2},  X_{t} \right\rangle \leqslant \left\langle \left|\id\right|^{k+1},  X_{t} \right\rangle,\] the same computation applies replacing $t\wedge \tau_{A,2}$ by $t$ to obtain that $\left(\widehat{M}^{\id}\left(\id^{k}\right)\right)_{0\leqslant t \leqslant T}$ is a continuous $\P_{\mu}-$martingale.\\

\textbf{Step 4. $L^{2}-$martingale and quadratic variation.}  As soon as $\left\langle \left|\id \right|^{k+1}, \mu \right\rangle < \infty$, $\widehat{M}^{\id}\left(\id^{k} \right) \in L^{2}\left(\widetilde{\Omega} \right)$ as a straightforward consequence of \textsc{H\"{o}lder}'s inequality. It follows from (\ref{Variation_quad_FV_recentre_Multi_d}) that for all $n \in \N$, for all $t \in [0,T]$, 
\begin{align*}
\left\langle \widehat{M}^{\id}\left(g_{n}^{k} \right) \right\rangle_{t} & = 2\gamma \int_{0}^{t}{\left(\left\langle g_{n}^{2k}, X_{s} \right\rangle - \left\langle g_{n}^{k}, X_{s} \right\rangle^{2}  + \left\langle \left(g_{n}^{k}\right)', X_{s} \right\rangle^{2}M_{2}(X_{s}) \right.} \\
& \hspace{3cm} \left. - 2\left\langle \left(g_{n}^{k} \right)', X_{s} \right\rangle\left\langle g_{n}^{k} \times \id, X_{s} \right\rangle \right) \dd s.
\end{align*}
For all $n \in \N$, the process \[N_{t , n} := \left[\widehat{M}_{t }^{\id}\left(g_{n}^{k}\right) \right]^{2} - \left\langle \widehat{M}^{\id}\left(g_{n}^{k}\right) \right\rangle_{t }\]
is a $\P_{\mu}-$local martingale. As $\widehat{M}^{\id}\left(\id^{k} \right)$ is bounded on $[0, T\wedge \tau_{A,2k}]$ for all $A \in \N$, then for all $n \in \N$, $\left(N_{t\wedge \tau_{A,2k} , n} \right)$ is a martingale.
From the relation (\ref{Moments_Z_id_k}) with $2k$, the dominated convergence theorem for conditional expectation implies as above that $\P_{\mu}-$a.s.
\begin{multline*}
\lim_{n\to + \infty}{N_{t\wedge \tau_{A,2k} ,n}} = \left[\widehat{M}_{t\wedge \tau_{A,2k} }^{\id}\left(\id^{k}\right) \right]^{2}  - 2\gamma \int_{0}^{t\wedge \tau_{A,2k} }{\left(\left\langle \id^{2k}, X_{s} \right\rangle - \left\langle \id^{k}, X_{s} \right\rangle^{2}   \right.} \\
+ \left. k\left\langle \id^{k-1}, X_{s} \right\rangle^{2}M_{2}(X_{s}) - 2k \left\langle \id^{k-1}, X_{s} \right\rangle \left\langle \id^{k+1}, X_{s} \right\rangle \right) \dd s
\end{multline*}
is a $\P_{\mu}-$martingale and we deduce the quadratic variation announced.\qedhere
\end{proof}

\subsubsection{Compact support \label{Sous_section_2_4_3}}
{  For all $\nu \in \MM_{1}(\R)$, we denote by $\Supp \nu$ the support of $\nu$. The historical reference of compact support property of the original \textsc{Fleming-Viot} process is {\color{blue} \cite[\color{black} Theorem 7.1]{dawson_wandering_1982}} where the authors proved that $\Supp Y_{t}$ is a.s. compact for each fixed $t>0$. We will used a slightly stronger version based on {\color{blue} \cite{liu_support_2015}}.}

\begin{Prop}
For all $\mu \in \MM_{1}^{c,2}(\R)$, for all $\varepsilon >0$, $\bigcup_{\varepsilon\leqslant s \leqslant t}{\Supp{Z_{s}}}$ is compact $\P_{\mu}-${\rm a.s.} Further, if  $\Supp{Z_{0}}$ is compact, then $\overline{\bigcup_{0\leqslant s \leqslant t}{\Supp{Z_{s}}}}$ is compact for all $t >0$, $\P_{\mu}-${\rm a.s.} 
\label{Prop_compacite_support_Z}
\end{Prop}
\begin{proof} 
It is proved in {\color{blue} \cite{liu_support_2015}} that the support of the $\Lambda-$\textsc{Fleming-Viot} process associated to a $\Lambda-$coalescent which comes down from infinity, is compact at all positive times. Our case corresponds to \textsc{Kingman}'s coalescent. In addition, they prove that, given that the initial condition $\nu$ has compact support, \[\overline{\bigcup_{0\leqslant s < t}{\Supp{Y_{s}}}} \] is compact for all $t>0$, $\P_{\nu}^{\rm FV}-$a.s. \textsc{Markov}'s property then entails that, if $\nu \in \MM_{1}^{c,2}(\R)$, $\bigcup_{\varepsilon \leqslant s \leqslant t}{\Supp{Y_{s}}}$ is compact for all $0<\varepsilon < t$, $\P_{\nu}^{\rm FV}-$a.s. Hence, the same is true for $Z_{t} = \tau_{-\left\langle \id, Y_{t}\right\rangle}\sharp \, Y_{t}$. \qedhere
\end{proof}

\section{Uniqueness for the centered \textsc{Fleming-Viot} process \label{Section_Unicite}}

As for the original \textsc{Fleming-Viot} martingale problem, we will prove uniqueness to the martingale problem (\ref{PB_Mg_Z_Multi_d}) by relying on the \textit{duality method} {\color{blue} \cite{Etheridge_2, Dawson, Ethier_Kurtz_1987, Ethier_Kurtz_1993}}. Additional difficulties occur in our case since bounds on the dual process are harder to obtain and the duality identity cannot be proved in its usual form.  In particular, we can prove uniqueness only for initial conditions admitting finite moments.

\subsection{Main result}

\begin{Thm} The centered \textsc{Fleming-Viot} martingale problem 
{  {\rm(\ref{PB_Mg_Z_Multi_d})}}
has a unique~solution if its initial condition has all their moments finite.
\label{Thm_unicite_FVr}
\end{Thm}

The reason why we need to assume finite initial moments will be explained at the end of Section \ref{Section_3_2_2}. In particular, we will see in Remark \ref{Remarque_contre_exemple} that we cannot hope to prove uniqueness for more general initial conditions using our duality method. 

\subsection{Notations and outline of the uniqueness proof \label{Section_3_2}}
\subsubsection{Martingale problem for polynomials\label{Section_3_2_1}}
In this section, we give an extension to the martingale problem (\ref{PB_Mg_Z_Multi_d}) which will be useful to prove uniqueness of the solution of the martingale problem of the centered \textsc{Fleming-Viot}. {  The heuristic leading to this extension is given in Appendix \ref{Sous_section_2_3}. Let us introduce, for all $n \in \N^{\star}, \mu \in \MM_{1}^{c,2}(\R^{d})$ and $f \in \CCCC^{2}_{b}((\R^{d})^{n}, \R)$, polynomial test functions: 
\begin{equation}
P_{f,n}(\mu) := \left\langle f, \mu^{n} \right\rangle := \int_{\R^{d}}^{}{\cdots \int_{\R^{d}}^{}{f(x_{1}, \cdots, x_{n})\mu(\dd x_{1})\cdots \mu(\dd x_{n})}}
\label{Eq_Fonction_polynome_mu}
\end{equation}
 where $\mu^{n}$ is the $n-$fold product measure of $\mu$.} 

{ 
\begin{Def}
The probability measure $\P_{\mu} \in \MM_{1}\left(\widetilde{\Omega}_{d} \right)$ is said to solve the \emph{centered \textsc{Fleming-Viot} martingale problem for {  polynomials}} {  with resampling rate $\gamma$ and} with initial condition $\mu \in \MM_{1}^{c,2}\left(\R^{d}\right)$, if the canonical process $\left(X_{t} \right)_{t\geqslant 0}$ on $\widetilde{\Omega}_{d}$ satisfies $\P_{\mu}(X_{0} = \mu) = 1$, for all $n \in \N^{\star}$, and for each  $f \in \CCCC_{b}^{2}((\R^{d})^{n}, \R)$, 
\begin{equation}
\widehat{M}^{(n),d}_{t}(f)  := \left\langle f, X_{t}^{n} \right\rangle - \left\langle f, X_{0}^{n}  \right\rangle - \int_{0}^{t}{\LL_{\rm FVc}^{d}{  P_{f, n}\left(X_{s} \right)} \dd s}  \\
\label{PB_Mg_dual_FVr_Multi_d} 
\end{equation}
with for all  $\varpi \in \MM_{1}^{c,2}\left(\R^{d}\right)$ and $f \in \CCCC_{b}^{2}\left(\left(\R^{d}\right)^{n}, \R\right)$,
\begin{equation}
\begin{aligned}
\LL_{\rm FVc}^{d}{  P_{f, n}(\varpi)} :=  \left\langle B^{(n),d}f, \varpi^{n} \right\rangle & + \gamma \sum\limits_{\substack{i, j\, = \, 1 \\ j \, \neq \, i} }^{n}{\left[\left\langle \Phi_{i,j}f, \varpi^{n-1}   \right\rangle - \left\langle f, \varpi^{n} \right\rangle \right] }     + \gamma \sum\limits_{i,j\, = \, 1}^{n}{\left\langle K_{i,j}f , \varpi^{n+1}  \right\rangle }
\end{aligned}
 \label{Gene_n_copie_FVr_Multi_d}
\end{equation}
is a continuous $\P_{\mu}-$martingale in $L^{2}\left(\widetilde{\Omega}_{d} \right)$ where, 
\begin{equation}
B^{(n), d}f(x) := \frac{\Delta f(x)}{2} - 2\gamma \sum_{k\, = \, 1}^{d}{\left(\nabla f(x) \cdot \varepsilon_{k} \right)\left(x\cdot \varepsilon_{k}\right)}
    \label{Eq_Bn_Multi_d}
\end{equation}
where $\varepsilon_{k} = \left(e_{k}, \cdots, e_{k}\right) \in (\R^{d})^{n}$ and $e_{k}$ is the $k^{\rm th}$ vector of the canonical basis of $\R^{d}$, and for all $i, j \in \left\{1, \cdots, n \right\}$, 
\begin{itemize}
{ \item[$\bullet$] $\Phi_{i,j} : \CCCC^{2}_{b}((\R^{d})^{n}, \R) \longrightarrow \CCCC^{2}_{b}((\R^{d})^{n-1}, \R)$, with $i \neq j$, is the function obtained from $f$ by inserting the vector $x_{i}$ between $x_{j-1}$ and $x_{j}$ when $i<j$ and by inserting the vector $x_{i-1}$ between $x_{j-1}$ and $x_{j}$ when $i>j$: 
\begin{equation}
\begin{aligned}
 \Phi_{i,j}f\left(x_{1}, \cdots, x_{n-1}\right)  &  = f\left(x_{1}, \cdots,  x_{j-1}, x_{i}, x_{j}, x_{j+1}, \cdots, x_{n-1} \right) & \qquad i<j \\
 \Phi_{i,j}f\left(x_{1}, \cdots, x_{n-1}\right)  &  = f\left(x_{1}, \cdots,  x_{j-1}, x_{i-1}, x_{j}, x_{j+1}, \cdots, x_{n-1} \right) & \qquad i>j
\end{aligned}
\label{Expression_Phi_ij}
\end{equation}}
\item[$\bullet$] $K_{i,j} : \CCCC^{2}_{b}((\R^{d})^{n}, \R) \longrightarrow \CCCC^{2}((\R^{d})^{n+1}, \R)$ is defined as 
\begin{equation}
 K_{i,j}f(x_{1}, \cdots, x_{n}, x_{n+1}) :=  \id\left(x_{n+1}\right)^{t}\left(\frac{\partial^{2} f}{\partial x_{i,k} \partial x_{j,\ell}}\left(x_{1}, \cdots, x_{n} \right) \right)_{1\leqslant k, \ell \leqslant d} \id\left(x_{n+1}\right).
 \label{Operateur_K_ij}
\end{equation}
\end{itemize}
\label{Def_PB_Mg_dual_FVr_Multi_d} 
\end{Def}
}

{  Recall that the centering effect in the \textsc{Fleming-Viot} process introduces additional terms in the infinitesimal generator (see (\ref{PB_Mg_Y_p_variables})). These terms, involving second moments of the measure, cannot be included in the $B^{(n),d}f$ term of \eqref{Gene_n_copie_FVr_Multi_d} because they do not have the proper dimensionality when expressed as integrals with respect to $\mu$. This is why we write them in the form $\langle K_{i,j}f, \mu^{n+1} \rangle$. As will appear in Section~\ref{Section_3_2_2} below, this term leads to the creation of particles in the dual process, contrary to the standard \textsc{Fleming-Viot} process.  This prevents us to apply the usual arguments and explains why we can only prove a weak version of the duality relation in Theorem~\ref{Thm_Identite_dualite_affaiblie} below.}

\begin{Thm}
For all $\mu \in \MM_{1}^{c,2}({ \R^{d}})$, the probability measure $\P_{\mu}$ constructed in {\rm{Theorem \ref{Prop_PB_Mg_Z}}}, satisfies the martingale problem of  {\rm{Definition~\ref{Def_PB_Mg_dual_FVr_Multi_d}}}. 
\label{Prop_PB_Mg_Z_dual}
\end{Thm}

\begin{proof} We can deduce the result from the original \textsc{Fleming-Viot} martingale problem for {  polynomials} {\rm{\color{blue} \cite{Ethier_Kurtz_1993}}} given below, following exactly the same arguments as for the proof of {\rm{Theorem \ref{Prop_PB_Mg_Z}}}. The probability measure $\P_{\nu}^{\rm FV} \in \MM_{1}\left({ \Omega_{d}} \right)$ is said to solve the \emph{original \textsc{Fleming-Viot} martingale problem for {  polynomials}} {  with resampling rate $\gamma$ and} with initial condition $\nu \in \MM_{1}({ \R^{d}})$, if the canonical process $\left(Y_{t} \right)_{t\geqslant 0}$ on { $\Omega_{d}$} satisfies $\P_{\nu}^{\rm FV}(Y_{0} = \nu) = 1$, for all $n \in \N^{\star}$, and for each $f \in \CCCC_{b}^{2}({ (\R^{d})^{n}}, \R)$,
\begin{equation}
M^{(n),d}_{t}(f)  := \left\langle f, Y_{t}^{n}  \right\rangle - \left\langle f, Y_{0}^{n}  \right\rangle - \int_{0}^{t}{\LL_{\rm FV}^{{ d}}{ P_{f,n}\left(Y_{s} \right)} \dd s}  \\
\label{PB_Mg_dual_FV} 
\end{equation}
with for all $\varpi \in \MM_{1}({ \R^{d}})$ 
\begin{equation*}
\LL_{\rm FV}^{{ d}}{ P_{f,n}\left(\varpi \right)} =  \left\langle \frac{\Delta f}{2}  , \varpi^{n} \right\rangle  + \gamma \sum\limits_{i\, = \, 1}^{n}{\sum\limits_{\substack{j \, = \, 1 \\ j \, \neq \, i} }^{n}{\left(\left\langle \Phi_{i,j}f, \varpi^{n-1} \right\rangle - \left\langle f, \varpi^{n}  \right\rangle \right)} }
\label{Gene_n_copie_FV}
\end{equation*}
is a $\P_{\nu}^{\rm FV}-$martingale. By {\color{blue} \cite[ {\color{black} Theorem 3.2}]{Ethier_Kurtz_1993}}, the solution $\P^{\rm FV}_{\nu}$ of the martingale problem {\rm{(\ref{PB_Mg_Y})}} is the unique solution to the previous martingale problem. \qedhere
\end{proof}

\subsubsection{Duality\label{Section_3_2_2}}

Our proof of Theorem \ref{Thm_unicite_FVr} is based on the \emph{duality method} as proposed in {\color{blue} \cite{Ethier_Kurtz_1987, Ethier_Kurtz_1993}}. 
From {\rm{(\ref{Gene_n_copie_FVr_Multi_d})}}, the operator $\LL_{\rm FVc}^{{  d}}$ applied on the function {  $P_{f,n}$ defined in (\ref{Eq_Fonction_polynome_mu})} with fixed $f$ {  and $n\in \N^{\star}$}, satisfies the following identity: 
\begin{equation}
\begin{aligned}
 \LL_{\rm FVc}^{{ d}}{ P_{f,n}(\mu)} & = \left\langle B^{(n){ ,d}}f, \mu^{n} \right\rangle  + \gamma \sum\limits_{i\, = \, 1}^{n}{\sum\limits_{\substack{j\, = \, 1 \\ j \, \neq \, i} }^{n}{\left[\left\langle \Phi_{i,j}f, \mu^{n-1} \right\rangle - \left\langle f, \mu^{n} \right\rangle \right] }  }   \\ 
 & \hspace{0.75cm} + \gamma \sum\limits_{i\, = \, 1}^{n}{\sum\limits_{j\, = \, 1  }^{n}{\left[\left\langle K_{i,j}f, \mu^{n+1} \right\rangle - \left\langle f, \mu^{n} \right\rangle \right] }  } + \gamma n^{2}\left\langle f, \mu^{n} \right\rangle  \\
& =: \widetilde{\LL}^{\star}_{f}{ P_{f,n}(\mu)} + \gamma n^{2}\left\langle f, \mu^{n} \right\rangle 
\end{aligned}
 \label{Identite_Generateurs}
\end{equation}
We note that $\widetilde{\LL}^{\star}_{f}$ can be seen as an operator acting on the function $f \mapsto { P_{f,n}(\mu)}$ with fixed $\mu$. The operator $\widetilde{\LL}^{\star}_{f}$ can be interpreted as the generator of a stochastic process on the state space $\bigcup_{n\in \N^{\star}}^{}{\CCCC^{2}({ (\R^{d})^{n}}, \R)}$. Following \textsc{Ethier-Kurtz}'s works {\color{blue} \cite{Ethier_Kurtz_1987, Ethier_Kurtz_1993}}, this suggests to introduce a dual process $\left(\xi_{t} \right)_{t\geqslant 0}$, of generator $\widetilde{\LL}^{\star}_{f}$ and to prove a duality relation of the form:
\begin{equation}
\forall t \geqslant 0, \quad \E\left(\left\langle \xi_{0}, X^{M(0)}_{t}  \right\rangle\right) = \E\left(\left\langle  \xi_{t}, X_{0}^{M(t)}   \right\rangle\exp\left(\gamma \int_{0}^{t}{M^{2}(u) \dd u} \right) \right)
 \label{Identite_dualite_forte}
\end{equation}
where $M := \left(M(t) \right)_{t\geqslant 0}$ is a \textsc{Markov}'s {  birth and death} process in $\N$ whose transition rates {  $q_{i,j}$ from $i$ to $j$} are given by:
\begin{equation*}
{\rm{\textbf{\rm{\textbf{(1)}}}}} \ \ q_{n, n+1} = \gamma n^{2} \qquad \quad {\rm{\textbf{\rm{\textbf{(2)}}}}} \ \ q_{n,n-1} = \gamma n(n-1) \qquad \quad {\rm{\textbf{\rm{\textbf{(3)}}}}} \ \ q_{i,j} = 0 \ {\rm{otherwise}}.
 \label{Taux_de_saut}
\end{equation*}
It is known that the relation  (\ref{Identite_dualite_forte}) implies uniqueness {\color{blue}\cite[\color{black} Theorem 4.4.2]{Ethier_markov_1986}}. However in our situation, it is difficult to obtain the strong version (\ref{Identite_dualite_forte}). For technical reasons, we will obtain only a weakened version. Therefore, the proof will be divided in two large steps. \\

\textbf{Step 1. Construction of the dual process $\left(\xi_{t} \right)_{t\geqslant 0}$.} The relation (\ref{Identite_Generateurs}) suggests that the dual process $\left(\xi_{t} \right)_{t\geqslant 0}$ jumps, for all $i, j \in \{1, \cdots, n\}$ from  $f \in \CCCC^{2}_{b}({ (\R^{d})^{n}}, \R)$ to $K_{i,j}f \in \CCCC^{2}({ (\R^{d})^{n+1}}, \R)$ at rate $\gamma$ and if $i \neq j$,  from $f \in \CCCC^{2}_{b}({ (\R^{d})^{n}}, \R)$ to $\Phi_{i,j}f \in \CCCC^{2}_{b}({ (\R^{d})^{n-1}}, \R)$ at rate $\gamma$. {  Moreover, note that if $n = 1$, the dual process can only jump from $f \in \CCCC_{b}^{2}(\R^{d}, \R)$ to $K_{i,j}f$.} Between jumps, this dual process evolves according to the  semi-group of operator $\left(T^{(n){ , d}}(t) \right)_{t \geqslant 0}$ associated to the generator $B^{(n){ , d}}$ given by {\rm{(\ref{Eq_Bn_Multi_d})}}. We will give in Section \ref{Section_preuve_Unicite} an explicit expression of the  semi-group $\left(T^{(n){ , d}}(t) \right)_{t \geqslant 0}$ defined as an integral against Gaussian kernels. We define the dual process as follows: 
\begin{Def}
For all $M(0) \in \N^{\star}$, for all $\xi_{0} \in \CCCC^{2}_{b}({ (\R^{d})^{M(0)}}, \R )$, 
\begin{multline}
\xi_{t}  := T^{\left(M\left(\tau_{n} \right) \right){ , d}}\left(t-\tau_{n}\right)\Lambda_{n} T^{\left(M\left(\tau_{n-1} \right) \right){ , d}}\left(\tau_{n}-\tau_{n-1}\right)\Lambda_{n-1} \cdots  \Lambda_{1}  T^{\left(M\left(0 \right) \right){ , d}}\left(\tau_{1}\right)\xi_{0}, \\
   \tau_{n} \leqslant t < \tau_{n+1}, \ n \in \N , \label{Processus_dual_explicite}
\end{multline}
where $\left(\tau_{n}\right)_{n\in \N}$ is the sequence of jump times of the birth-death process $M$ with $\tau_{0} = 0$ and where $\left(\Lambda_{n}\right)_{n\in \N}$ is a sequence of random operators.  These are conditionally independent given $M$ and satisfy for all $k \in \N$, $n\geqslant 1$ and $1\leqslant i \neq j \leqslant n$, 
\begin{equation}
\P\left(\Lambda_{k} = \Phi_{i,j} \left| \phantom{1^{1^{1^{1}}}} \hspace{-0.7cm} \right. \left\{M\left(\tau_{k}^{-} \right) = n , M\left(\tau_{k} \right) = n-1\right\} \right)  = \frac{1}{n(n-1)} 
\label{Saut_Phi_i_j}
\end{equation}
and for all $n \geqslant 1$ and $1\leqslant i, j \leqslant n$, 
\begin{equation}
\P\left(\Lambda_{k} = K_{i,j} \left| \phantom{1^{1^{1^{1}}}} \hspace{-0.7cm} \right. \left\{M\left(\tau_{k}^{-} \right) = n , M\left(\tau_{k} \right) = n+1\right\} \right)  = \frac{1}{n^{2}}.
 \label{Saut_K_i_j}
\end{equation} 
Moreover, the random times $\left(\tau_{k} - \tau_{k-1}\right)_{k\geqslant 1}$ are independent conditionally to $M\left(\tau_{k-1} \right) = n$ and of exponential law of parameter $\gamma n^{2} + \gamma n(n-1)$.
\label{Def_Processus_Dual}
\end{Def}
{  Note that particle creation for the centered \textsc{Fleming-Viot} dual process is possible unlike the original \textsc{Fleming-Viot} dual process. It comes from the $K_{i,j}$ operator which has appeared under the centering effect. Note that these terms correspond to the ones with the factor { $\langle \id_{i} \times \id_{j}, \mu \rangle$, $i,j \in \{1,\cdots, n\}$ in  (\ref{Eq_Generateur_L_FVc_Multi_d})} or $M_{2}(\mu)$ factor in (\ref{Eq_Generateur_L_FVc}). Indeed, for all $f \in \CCCC_{b}^{2}((\R^{d})^{n}, \R)$ and all $\mu \in \MM_{1}^{c,2}(\R^{d})$, 
\[\sum_{i,j\, = \, 1}^{n}{\left\langle K_{i,j}f, \mu^{n+1} \right\rangle} = \sum_{i,j \, = \, 1}^{n}{\sum_{k, \ell \, = \, 1}^{d}{\left\langle \partial_{x_{i,k}x_{j,\ell}}^{2}f, \mu^{n} \right\rangle \left\langle \id_{k}\times \id_{\ell}, \mu \right\rangle}} \]

The terms $\Phi_{i,j}$ are present both in the original and centered \textsc{Fleming-Viot}. Because of the operators $K_{i,j}$, difficulties will arise to get bounds on the dual process (see Section \ref{Sous_section_6_2}).} Note {  also} that $M$ is a non-explosive process:
\begin{equation}
\forall T >0, \qquad \P\left(\sup_{t\in [0,T] }{M(t)} < +\infty \right) = \P\left(\lim_{n\to + \infty}{\tau_{n}} = + \infty \right) = 1.  \label{Non-explosion}
\end{equation}
Indeed, we note that for the choice of $\mu_{i} := \gamma i(i-1)$ and $\lambda_{i} := \gamma i^{2} $, $i\geqslant 1$ in {\color{blue} \cite[\color{black} Theorem 2.2.]{Bans}}, we have $$\frac{\mu_{i} \cdots \mu_{2}}{\lambda_{i} \cdots \lambda_{2}\lambda_{1}} = \frac{1}{2\gamma i} >0,  $$ so that \[\sum\limits_{i\geqslant 1}^{}{\left(\frac{1}{\lambda_{i}} + \frac{\mu_{i}}{\lambda_{i}\lambda_{i-1}} + \cdots + \frac{\mu_{i} \cdots \mu_{2}}{\lambda_{i} \cdots \lambda_{2}\lambda_{1}}  \right)} \geqslant \sum\limits_{i\geqslant 1}^{}{\frac{\mu_{i} \cdots \mu_{2}}{\lambda_{i} \cdots \lambda_{2}\lambda_{1}}} = +\infty.\]
Hence, $M$ is non-explosive. \\

\textbf{Step 2. Weakened duality relation.} We consider fixed $M(0) \in \N^{\star}$,  $\xi_{0} \in \CCCC^{2}_{b}({ (\R^{d})^{M(0)}}, \R)$ and $\left(X_{t}\right)_{t\geqslant 0}$ a stochastic process whose law $\P_{\mu}$ is a solution of the martingale problem (\ref{PB_Mg_Z_Multi_d}) with $\mu \in \MM_{1}^{c,2}({ \R^{d}})$. We introduce a dual process $\left(\xi_{t} \right)_{t\geqslant 0}$ independent of $\left(X_{t} \right)_{t\geqslant 0}$ built on the same probability space (enlarging it if necessary). We shall denote by $\P_{(\mu, \xi_{0})}$, the law of $\left(\left(X_{t}, \xi_{t} \right)\right)_{t\geqslant 0}$ on this probability space. For any $k \in \N$, we introduce the stopping time \begin{equation}
\theta_{k} := \inf{\left\{ t \geqslant 0\left. \phantom{1^{1^{1^{1}}}} \hspace{-0.6cm} \right| M(t)  \geqslant k \quad {\rm{or}} \quad \exists s \in [0,t], \ \left\langle  \xi_{s}, X_{t-s}^{M(s)}  \right\rangle \geqslant k \right\}}.  \label{Theta_k}
\end{equation}

\begin{Thm} Given any $\left(X_{t}\right)_{t\geqslant 0}$, $\left(\xi_{t}\right)_{t\geqslant 0}$ as above, we have the \emph{weakened duality identity}: for all $k \in \N$ and any stopping time $\theta$ such that $\theta \leqslant \theta_{k}$, 
\begin{equation}
\begin{aligned}
& \forall t \geqslant 0,  \quad \E_{\left(\mu, \xi_{0} \right)}\left(\left\langle \xi_{0}, X^{M(0)}_{t\wedge \theta}  \right\rangle\right)= \E_{\left(\mu, \xi_{0} \right)}\left(\left\langle \xi_{t\wedge \theta}, X_{0}^{M(t\wedge \theta)}   \right\rangle\exp\left(\gamma \int_{0}^{t\wedge \theta}{M^{2}(u) \dd u} \right) \right).
\end{aligned}
 \label{Identite_dualite_affaiblie}
\end{equation}
\label{Thm_Identite_dualite_affaiblie}
\end{Thm}
Note that this result holds true for any initial measure $\mu \in \MM_{1}^{c,2}({ \R^{d}})$.
The stopping time $\theta_{k}$ ensures that each of the quantities involved in  (\ref{Identite_dualite_affaiblie}) are bounded and thus that their expectations are finite. 
This is where we need stronger assumptions on $\mu$.
\begin{Lem} Assume that $\mu \in \MM_{1}^{c,2}({ \R^{d}})$ has all its moments finite. Then, the stopping time $\theta_{k}$ defined by {\rm{(\ref{Theta_k})}}  satisfies $\lim_{k\to +\infty}{\theta_{k}} = +\infty$, $\P_{\left(\mu, \xi_{0}\right)}-{\rm{a.s}}$.  
\label{Lemme_Theta_k}
\end{Lem} 

We will see in Remark \ref{Remarque_contre_exemple} that the assumption on $\mu$ is optimal in the following sense: even if $\xi_{0}$ is bounded, $\xi_{t}$ may have polynomial growth of any exponent $k$ in some direction of ${ (\R^{d})}^{M(t)}$ so that $\left\langle \left\|\xi_{t} \right\|, \mu^{M(t)} \right\rangle$ is infinite if $\mu$ has infinite $k^{\rm{th}}$ moment. This shows that we cannot expect to have $\theta_{k} \to +\infty$ when $k \to + \infty$ if $\mu$ has not all its moments finite. \\ 

The proofs of {\rm{Theorem \ref{Thm_Identite_dualite_affaiblie}}} and {\rm{Lemma \ref{Lemme_Theta_k}}} are respectively given in {\rm{Sections \ref{Sous_section_6_2}}} and {\rm{\ref{Sous_section_6_3}}}. Once they are proved, the proof of uniqueness can be completed as follows.

\subsubsection{Proof of Theorem \ref{Thm_unicite_FVr} from Theorem \ref{Thm_Identite_dualite_affaiblie} and Lemma \ref{Lemme_Theta_k}}

We rely on the \textsc{Ethier-Kurtz}'s result {\color{blue} \cite[\color{black} Theorem 4.4.2]{Ethier_markov_1986}}: to get the desired result, i.e. the uniqueness of the martingale problem (\ref{PB_Mg_Z_Multi_d}), it is sufficient to verify that if we give ourselves two solutions to the martingale problem (\ref{PB_Mg_Z_Multi_d}) {  satisfying the weakened duality identity, then} they have the same $1-$dimensional marginal laws. 

Let $\left(X_{t} \right)_{t\geqslant 0}$ and $\left(\widetilde{X}_{t} \right)_{t\geqslant 0}$ be two solutions of the martingale problem (\ref{PB_Mg_Z_Multi_d}) with the same initial condition $\mu \in \MM_{1}^{c,2}({ \R^{d}})$ which has all its moments finite. Let $\left(\xi_{t} \right)_{t\geqslant 0}$ be the dual process with initial condition $\xi_{0}\in \CCCC^{2}_{b}({ (\R^{d})}^{M(0)}, \R)$ with $M(0) \in \N^{  \star}$. We suppose that these three processes are built on the same probability space and independent of each other. We denote {  for all $k\in \N$,} \[\widetilde{\theta}_{k} := \inf{\left\{ t \geqslant 0\left. \phantom{1^{1^{1^{1}}}} \hspace{-0.6cm} \right| M(t)  \geqslant k \quad {\rm{or}} \quad \exists s \in [0,t], \ \left\langle  \xi_{s}, \widetilde{X}_{t-s}^{M(s)}  \right\rangle \geqslant k \right\}}. \] These processes satisfy, {  for all $k \in \N$, $\xi_{0} \in \CCCC_{b}^{2}({ (\R^{d}})^{M(0)}, \R )$} the weakened duality identity (\ref{Identite_dualite_affaiblie}):
\begin{equation*}
\begin{aligned}
\forall t \geqslant 0,\qquad \E_{\left(\mu, \xi_{0} \right)}\left(\left\langle  \xi_{0}, X^{M(0)}_{t\wedge \theta_{k} \wedge \widetilde{\theta}_{k}} \right\rangle\right) & \\
& \hspace{-3.5cm} = \E_{\left(\mu, \xi_{0} \right)}\left(\left\langle \xi_{t\wedge \theta_{k} \wedge \widetilde{\theta}_{k}}, \mu^{M(t\wedge \theta_{k} \wedge \widetilde{\theta}_{k})}  \right\rangle\exp\left(\gamma \int_{0}^{t\wedge \theta_{k} \wedge \widetilde{\theta}_{k}}{M^{2}(u) \dd u} \right) \right) \\ 
& \hspace{-3.5cm} = \E_{\left(\mu, \xi_{0} \right)}\left(\left\langle \xi_{0}, \widetilde{X}^{M(0)}_{t\wedge \theta_{k} \wedge \widetilde{\theta}_{k}}  \right\rangle\right). 
\end{aligned}
 \label{Difference_identite_dualite}
\end{equation*}
From {\rm{Lemma \ref{Lemme_Theta_k}}}, since $\left(X_{t}\right)_{t\geqslant 0}$ and $\left(\widetilde{X}_{t}\right)_{t\geqslant 0}$ have continuous paths for the topology of weak convergence, we have $\P_{\left(\mu, \xi_{0}\right)}-$a.s., \[\lim_{k\to +\infty}{\left\langle  \xi_{0}, X_{t\wedge \theta_{k} \wedge \widetilde{\theta}_{k}}^{{  M(0)}} \right\rangle} = \left\langle\xi_{0}, X_{t}^{{  M(0)}} \right\rangle \quad  {\rm{and}} \quad \lim_{k\to +\infty}{\left\langle  \xi_{0}, \widetilde{X}_{t\wedge \theta_{k} \wedge \widetilde{\theta}_{k}}^{{  M(0)}} \right\rangle} = \left\langle\xi_{0}, \widetilde{X}_{t}^{{  M(0)}} \right\rangle.\]
Therefore, we deduce from the dominated convergence theorem {  and (\ref{Eq_Fonction_polynome_mu})}, that for all  $\xi_{0} \in \CCCC^{2}_{b}\left({ (\R^{d})}^{M(0)}, \R\right)$, \[\forall t \geqslant 0, \qquad { \E_{\left(\mu, \xi_{0} \right)}\left(P_{\xi_{0}, M(0)}\left(X_{t} \right) \right) =  \E_{\left(\mu, \xi_{0} \right)}\left(P_{\xi_{0}, M(0)}\left(\widetilde{X}_{t} \right) \right)}.\]
{  As the set of test functions $\left\{P_{f,n} \left|\phantom{1^{1^{1}}} \hspace{-0.5cm} \right.  f \in \CCCC_{b}^{2}({ (\R^{d})}^{n}, \R), n\in \N^{\star} \right\}$ is $\MM_{1}\left(\MM_{1}({ \R^{d}}) \right)-$convergence determining {\color{blue} \cite[\color{black} Lemma 2.1.2]{Dawson}}, it is $\MM_{1}\left(\MM_{1}({ \R^{d}}) \right)-$separating {\color{blue} \cite[\color{black} Chapter 3, Section 4, p.112]{Ethier_markov_1986}},} it follows that for any $t\geqslant 0$, $X_{t}^{{  M(0)}}$ and $\widetilde{X}_{t}^{{  M(0)}}$ have the same law. In particular, for the choice $M(0) := 1$, {\color{blue} \cite[\color{black} Theorem 4.4.2]{Ethier_markov_1986}} ensures uniqueness to the martingale problem (\ref{PB_Mg_Z_Multi_d}).\hfill $\square$

\section{Ergodicity for the centered \textsc{Fleming-Viot} process \label{Section_Ergo}}

In this section, we establish ergodicity properties with exponential convergence in total variation for the centered \textsc{Fleming-Viot} process $\left(Z_{t} \right)_{t\geqslant 0}$. {  Note that in the case of the original \textsc{Fleming-Viot} process, ergodicity fails without the centering property {\color{blue} \cite[\color{black} Section 9.1]{Ethier_Kurtz_1993}}. Standard duality arguments would provide weak ergodicity estimates (see (\ref{Eq_Ergodicity_Weak}) and {\color{blue} \cite{Ethier_Kurtz_1993}}). However, using a coupling argument based on the \textsc{Moran} process and its relationship with the \textsc{Kingman} coalescent, it is possible to obtain strong ergodicity bounds (see (\ref{Eq_Ergodicity_Strong})) as done below. In addition, this will provide an explicit construction of the invariant probability measure from the  \textsc{Donnelly-Kurtz} modified look-down {\color{blue} \cite{donnelly_countable_1996, donnelly_particle_1999}}.}
For all $\mu, \nu \in \MM_{1}({ \R^{d}})$, we denote by \[\|\mu - \nu \|_{TV} := \frac{1}{2}\sup_{\left\|f \right\|_{\infty} \leqslant 1}{\left|\left\langle f, \mu \right\rangle - \left\langle f, \nu \right\rangle \right|},\]
the total variation distance between $\mu$ and $\nu$. {  The main result of this section is the following.}

\begin{Thm} There exists a unique invariant probability measure $\pi$ for $\left(Z_{t}\right)_{t\geqslant 0}$ and constants $\alpha, \beta \in (0,+\infty)$ such that for all $\mu \in \MM_{1}^{c,2}({ \R^{d}})$, for all $T \geqslant 0$,   \[ \left\| \P_{\mu}\left(Z_{T} \in \cdot \right) - \pi \right\|_{TV} \leqslant \alpha \exp\left(-\beta T \right).\]
\label{Thm_Ergodicite_VT}
\end{Thm}

The main part of this section is devoted to the proof of this result (Sections \ref{Ergodicite_Sect_4.1} to  \ref{Ergodicite_Sect_4.4}) and in Section \ref{Ergodicite_Sect_4.6}, we characterise the invariant measure of the centered \textsc{Fleming-Viot} process. In Section \ref{Ergodicite_Sect_4.1} we construct the centered \textsc{Moran} process and we establish its  convergence in law to the centered \textsc{Fleming-Viot} process. In Section \ref{Ergodicite_Sect_4.2} we construct, backward in time, the \textsc{Moran} process, its centered version and we exploit its relationship with the \textsc{Kingman} coalescent in order to prove in Section \ref{Ergodicite_Sect_4.4} an exponential coupling in total variation for the \textsc{Moran} process. We finally deduce the main result announced by letting the number of particles go to infinity. 

\subsection{\textsc{Moran}'s models and \textsc{Fleming-Viot}'s processes \label{Ergodicite_Sect_4.1}} 

In {\color{blue} \cite{Etheridge_2, dawson_wandering_1982, fleming_no_1979, Dawson}}, the authors construct the original \textsc{Fleming-Viot} process as a scaling limit of a particle process: the \textsc{Moran} process. The aim of this section is to construct the version of the centered \textsc{Moran} process and to establish its convergence in law to the centered \textsc{Fleming-Viot} process. \\ 

We consider the \textsc{Moran} particle process $Y^{N}$ defined by \[Y^{N}_{t}  := \frac{1}{N}\sum\limits_{i\, = \, 1}^{N}{\delta_{X_{i}(t)}}  \] with state space $\MM_{1,N}({ \R^{d}})$, the set of probability measures on ${ \R^{d}}$ consisting of $N$ atoms of mass $1/N$. Moreover, if $\left(X_{i}(0)\right)_{i\in \N^{\star}}$ is exchangeable, then for all $t > 0$, $\left(X_{i}(t)\right)_{i\in \N^{\star}}$ is exchangeable {\color{blue} \cite[\color{black} Theorem 6.1]{Ethier_Kurtz_1993}}. The infinitesimal generator of the ${ \R^{d}}-$measure-valued process $Y^{N}$ is given {  for all $n\in \N^{\star}, f \in \CCCC^{2}_{b}({ (\R^{d})^{n}}, \R )$, $\mu_{N} \in \MM_{1,N}({ \R^{d}})$ by
\begin{align*}
L_{N}P_{f,n}\left(\mu_{N} \right) & := P_{\frac{\Delta f}{2}, n}\left(\mu_{N} \right) \\
& \hspace{1cm} + \gamma N(N-1)\int_{{ \R^{d}}}^{}{\int_{{ \R^{d}}}^{}{\left[P_{f,n}\left(\mu_{N} - \frac{\delta_{x}}{N} + \frac{\delta_{y}}{N} \right) - P_{f,n}\left(\mu_{N} \right) \right]\mu_{N}(\dd x)\mu_{N}(\dd y)}}.
\end{align*}}

The first term of the generator describes the effect of the mutation according to the Laplacian operator. The second term describes the sampling replacement mechanism: at rate $\gamma$ (the sampling rate) an individual of type $x$ is immediately replaced by one of type $y$. Note that the population size remains constant over time. \\

We recall the following convergence result {\color{blue} \cite[\color{black}Theorem 2.7.1]{Dawson}}: for all initial condition $Y_{0}^{N} = \frac{1}{N} \sum_{i\, = \, 1}^{N}{\delta_{X_{i}}} \in \MM_{1,N}({ \R^{d}})$ with $\left(X_{i} \right)_{1\leqslant i \leqslant N}$ exchangeable ${ \R^{d}}-$valued random variables such that $Y_{0}^{N}$ converges in law to $\mu \in \MM_{1}({ \R^{d}})$ as $N \to +\infty$, the \textsc{Moran} process $\left(Y^{N}_{t}\right)_{t\geqslant 0}$  converges in law on 
{ $\Omega_{d}$}, as $N \to + \infty$, to the original \textsc{Fleming-Viot} process $\left(Y_{t} \right)_{t\geqslant 0}$ defined as the solution to the martingale problem {  (\ref{PB_Mg_dual_FV})}. \\ 

We denote $\MM_{1, N}^{c,2}({ \R^{d}}) := \left\{ \mu_{N} \in \MM_{1,N}({ \R^{d}})  \left| \phantom{1^{1^{1^{1}}}} \hspace{-0.7cm} \right. \left\langle { \left\|\id\right\|^{2}}, \mu_{N} \right\rangle < \infty, \  \left\langle \id, \mu_{N} \right\rangle = 0 \right\}$, and we define the centered \textsc{Moran} process $\left(Z^{N}_{t} \right)_{t\geqslant 0}$  by  \[Z^{N}_{t} := \tau_{-\left\langle \id, Y_{t}^{N} \right\rangle} \sharp \, Y_{t}^{N}, \qquad t\geqslant 0.\]

The main result of this section is the following:

\begin{Prop} For all initial condition $Z_{0}^{N} := \frac{1}{N}\sum_{i\, = \, 1}^{N}{\delta_{X_{i}}} \in \MM_{1, N}^{c,2}({ \R^{d}})$ with $\left(X_{i} \right)_{1\leqslant i \leqslant  N}$  exchangeable ${ \R^{d}}-$valued random variables such that $Z_{0}^{N}$ converges in law to $Z_{0} \in \MM_{1}^{c,2}({ \R^{d}})$ as $N \to + \infty$ and satisfying $\sup_{N\in \N^{\star}}{\E\left(\left\langle {  \left\|\id\right\|^{2}}, Z_{0}^{N} \right\rangle\right)} < \infty$, the centered \textsc{Moran} process $\left(Z^{N}_{t} \right)_{t\geqslant 0}$ converges in law on $\CCCC\left([0, +\infty), \MM_{1}^{c,2}({ \R^{d}}) \right)$, as $N \to +\infty$, to the centered \textsc{Fleming-Viot} process $\left(Z_{t} \right)_{t\geqslant 0}$ solution of the martingale problem 
{  {\rm{(\ref{PB_Mg_dual_FVr_Multi_d})}}} with initial condition $Z_{0}$.
\label{Prop_0_Genealogie}
\end{Prop}

A difficulty in proving this result lies in the fact that $\mu \mapsto \tau_{-\left\langle \id, \mu \right\rangle} \sharp \, \mu$ {  is not} continuous on $\MM_{1}({ \R^{d}})$ because $\id$ is not bounded. Hence we need to carry out several approximations and control carefully the approximation error. In order to prove this proposition, we need to introduce some notations. For all ${ \R^{d}}$-valued function $f$ on ${ \R^{d}}$, the \textsc{Lipschitz} seminorm is defined by $\|f \|_{L} = \sup_{x\neq y}{ \frac{\left\|f(x) - f(y) \right\|}{\left\|x - y \right\|}}$. We denote by \[BL_{1}({ \R^{d}}) := \left\{ f : { \R^{d}} \to { \R^{d}} \left| \phantom{1^{1^{1^{1}}}} \hspace{-0.7cm} \right. \left\|f \right\|_{BL} \leqslant 1 \right\} \]
where $\left\|f \right\|_{BL} := \left\|f \right\|_{L} + \left\| f\right\|_{\infty}$ For all $\mu, \nu \in \MM_{1}({ \R^{d}})$, we denote by 
\[d_{FM}(\mu, \nu) := \sup_{f \in BL_{1}({ \R^{d}})}{\left\| \left\langle f, \mu \right\rangle - \left\langle f, \nu \right\rangle \right\|},\]
the \textsc{Fortet-Mourier} distance. Recall that $\MM_{1}\left({ \R^{d}}\right)$ endowed with the weak topology is complete for the distance of \textsc{Fortet-Mourier} {\color{blue} \cite[\color{black} Corollary 11.5.5]{dudley_real_2002}}. Let $\Lambda$ denote the class of strictly increasing, continuous mappings of $[0,T]$ onto itself. For a given metric space $E$, we denote by $\D\left([0,T], E \right)$, the space of right continuous and left limited (c\`{a}d-l\`{a}g) functions from $[0,T]$ to $E$. For $x,y \in \D\left([0,T], \MM_{1}({ \R^{d}}) \right)$, we define the distance $d_{0}(x,y)$ by:
\[d_{0}(x,y) := \inf_{\lambda \in \Lambda}\left\{ \sup_{t\in[0,T]}{d_{FM}(x\circ \lambda (t), y(t))} + \sup_{s<t}\left|\log\left(\frac{\lambda(t) -\lambda(s)}{t-s} \right) \right| \right\}. \] 
From {\color{blue} \cite[\color{black} Theorem 12.2 and Remark of page 121]{Billingsley}}, $\left(\D\left([0,T], \MM_{1}({ \R^{d}}) \right), d_{0}\right)$ is a Polish space {  when $\MM_{1}({ \R^{d}})$ is endowed with the \textsc{Fortet-Mourier} distance} and the topology induced by the distance $d_{0}$ is the \textsc{Skorohod} topology.  
{  The following lemma is a basic property of the original \textsc{Fleming-Viot}. It can be proved similarly as Proposition \ref{Prop_moments_non_bornes_RECENTRE}.}
\begin{Lem}
Let $T >0$ and $k \in \N$ fixed. There exists a constant $C_{k, T}>0$  independent of $N$ such that for all $Y_{0}^{N} \in \MM_{1,N}({ \R^{d}})$ satisfying $\sup_{N \in \N}{\E\left(\left\langle { \left\|\id \right\|^{k}}, Y_{0}^{N} \right\rangle\right)} < \infty$, the \textsc{Moran} process $\left(Y_{t}^{N} \right)_{0\leqslant t\leqslant T}$ satisfies \[\forall \alpha > 0, \qquad  \P_{Y_{0}^{N}}\left(\sup_{t\in [0,T]}{\left\langle { \left\| \id\right\|^{k}}, Y_{t}^{N}\right\rangle} \geqslant \alpha \right) \leqslant \frac{C_{k,T} \sup_{N \in \N^{\star}}{\E\left(\left\langle { \left\|\id \right\|^{k}}, Y_{0}^{N} \right\rangle\right)}}{\alpha}.\] 
\label{Moments_Moran}
\end{Lem}

\begin{proof}[Proof of Proposition \ref{Prop_0_Genealogie}] We want to establish that \[\forall g \in \CCCC_{b}\left(\D\left([0,T], \MM_{1}({ \R^{d}}) \right), \R\right), \quad \lim_{N\to +\infty}{\E\left(g\left(Z^{N} \right) \right)} = \E\left(g\left(Z \right) \right).\]
Let $\varepsilon>0$. We consider the two following maps $F$ and $F_{\varepsilon}$ from $\D\left([0,T], \MM_{1}^{2}({ \R^{d}}) \right)$ to  $\D\left([0,T], \MM_{1}({ \R^{d}}) \right)$ defined by $F(y)(t) := \tau_{-\left\langle \id, y(t)\right\rangle} \, \sharp \, y(t)$ and $F_{\varepsilon}(y)(t) := \tau_{-\left\langle h_{\varepsilon}, y(t)\right\rangle} \, \sharp \, y(t)$ where $h_{\varepsilon}$ is a map from ${ \R^{d}}$ to ${ \R^{d}}$ defined by { \[h_{\varepsilon}(x) := \left\{\begin{array}{lll}
x & {\rm{if}} & \|x\| \leqslant \frac{1}{\varepsilon} \vspace{0.15cm}\\
\frac{\bm{1}}{\varepsilon\|\bm{1}\|} & {\rm{if}} & \exists i \in \{1, \cdots, d \}, \ x_{i} > \frac{1}{\varepsilon} \vspace{0.15cm} \\
-\frac{\bm{1}}{\varepsilon\|\bm{1}\|} & {\rm{if}} & \exists i \in \{1, \cdots, d \}, \ x_{i} < -\frac{1}{\varepsilon}
\end{array}\right. .\] 
where $\bm{1} \in \R^{d}$ designates the vector whose coordinates are all $1$.} \\

 \textbf{Step 1. Continuity of $F_{\varepsilon}$.} In this step, we want to establish that  \[F_{\varepsilon} \in \CCCC_{b}\left(\D\left([0,T], \MM_{1}({ \R^{d}}) \right), \D\left([0,T], \MM_{1}({ \R^{d}}) \right)\right). \]
 To obtain this, it is equivalent to prove that if for all $n \in \N$, $y_{n}, y \in \D\left([0,T], \MM_{1}({ \R^{d}}) \right)$ and $\lim_{n\to + \infty} d_{0}(y_{n}, y) = 0$, we have $\lim_{n\to + \infty}{d_{0}\left(F_{\varepsilon}(y_{n}), F_{\varepsilon}(y)\right)} = 0$. As, $\lim_{n\to + \infty} d_{0}(y_{n}, y) = 0$, there exists $n_{0} \in \N$  such that for all $n \geqslant n_{0}$, there exists $\lambda_{n} \in \Lambda$ satisfying 
\begin{equation}
\sup_{t\in[0,T]}{d_{FM}(y_{n}\circ \lambda_{n} (t), y(t))} + \sup_{s<t}\left|\log\left(\frac{\lambda_{n}(t) -\lambda_{n}(s)}{t-s} \right) \right| \leqslant \frac{\varepsilon^{2}}{2}.
\label{Convergence_avec_d0}
\end{equation}
Note that \[d_{0}\left(F_{\varepsilon}(y_{n}), F_{\varepsilon}(y)\right) \leqslant \sup_{t\in[0,T]}{d_{FM}(F_{\varepsilon}\left(y_{n}\right)\left(\lambda_{n} (t)\right), F_{\varepsilon}(y)(t))} + \sup_{s<t}\left|\log\left(\frac{\lambda_{n}(t) -\lambda_{n}(s)}{t-s} \right) \right|. \]
Now, for all $t \in [0,T]$, 
\begin{align*}
d_{FM}\left(F_{\varepsilon}(y_{n})(\lambda_{n}(t)), F_{\varepsilon}(y)(t) \right) & = \sup_{f \in BL_{1}({ \R^{d}})}{\left\|\left\langle f \circ \tau_{-\left\langle h_{\varepsilon}, y_{n} \circ \lambda_{n}(t) \right\rangle}, y_{n}\circ \lambda_{n}(t) \right\rangle - \left\langle f \circ \tau_{-\left\langle  h_{\varepsilon}, y(t) \right\rangle}, y(t) \right\rangle  \right\|} \\
& \leqslant \sup_{f \in BL_{1}({ \R^{d}})}{\left\|\left\langle f \circ \tau_{-\left\langle h_{\varepsilon}, y_{n} \circ \lambda_{n}(t) \right\rangle}, y_{n}\circ \lambda_{n}(t) - y(t) \right\rangle \right\|} \\
& \hspace{1cm} + \sup_{f \in BL_{1}({ \R^{d}})}{\left\|\left\langle f \circ \tau_{-\left\langle h_{\varepsilon}, y_{n} \circ \lambda_{n}(t) \right\rangle} - f\circ\tau_{-\left\langle h_{\varepsilon}, y(t) \right\rangle}, y(t) \right\rangle \right\|}.
\end{align*}
On the one hand, as $f \in BL_{1}({ \R^{d}})$, it follows that $f \circ \tau_{-\left\langle h_{\varepsilon}, y_{n} \circ \lambda_{n}(t) \right\rangle} \in BL_{1}({ \R^{d}})$ and thus \[\sup_{f \in BL_{1}({ \R^{d}})}{\left\|\left\langle f \circ \tau_{-\left\langle h_{\varepsilon}, y_{n} \circ \lambda_{n}(t) \right\rangle}, y_{n}\circ \lambda_{n}(t) - y(t) \right\rangle \right\|} \leqslant d_{FM}\left(y_{n}\circ\lambda_{n}(t), y(t)\right). \]
On the other hand, as $f$ and $\varepsilon h_{\varepsilon}$ are in $BL_{1}({ \R^{d}})$, we have 
\begin{align*}
\left\|\left\langle f \circ \tau_{-\left\langle h_{\varepsilon}, y_{n} \circ \lambda_{n}(t) \right\rangle} - f\circ\tau_{-\left\langle h_{\varepsilon}, y(t) \right\rangle}, y(t) \right\rangle \right\| & \leqslant \left\|- \left\langle h_{\varepsilon}, y_{n} \circ \lambda_{n}(t)\right\rangle  + \left\langle h_{\varepsilon}, y(t)\right\rangle \right\| \\
& = \frac{1}{\varepsilon}\left\|\left\langle \varepsilon h_{\varepsilon}, y_{n}\circ \lambda_{n}(t) - y(t) \right\rangle \right\| \\
& \leqslant \frac{1}{\varepsilon}d_{FM}\left(y_{n} \circ\lambda_{n}(t), y(t) \right).
\end{align*}
It follows from (\ref{Convergence_avec_d0}) that, 
\begin{align*}
d_{0}\left(F_{\varepsilon}(y_{n}), F_{\varepsilon}(y)\right) & \leqslant \left(1 + \frac{1}{\varepsilon} \right) \sup_{t\in [0,T]}{d_{FM}\left(y_{n}\circ \lambda_{n}(t), y(t)\right)} + \sup_{s<t}\left|\log\left(\frac{\lambda_{n}(t) -\lambda_{n}(s)}{t-s} \right) \right| \leqslant \varepsilon. 
\end{align*}

\textbf{Step 2. Control in distance $d_{0}$ of the difference between $F\left(Y^{N} \right)$ and $F_{\varepsilon}\left(Y^{N} \right)$.} We consider the \textsc{Moran} process $\left(Y_{t}^{N} \right)_{0\leqslant t \leqslant T}$  started from $Y_{0}^{N} = Z_{0}^{N}$ and the original \textsc{Fleming-Viot} process $\left(Y_{t} \right)_{0\leqslant t \leqslant T}$  started from $Y_{0} = Z_{0}$. In this step, we consider the events \[\Omega_{\varepsilon, N} := \left\{ \sup_{t\in [0,T]}{\left\langle \left\|\id\right\|^{2}, Y^{N}_{t}\right\rangle}  \leqslant \frac{2}{\sqrt{\varepsilon}} \right\} \qquad {\rm{and}} \qquad \Omega_{\varepsilon, \infty} := \left\{ \sup_{t\in [0,T]}{\left\langle \left\|\id\right\|^{2}, Y_{t}\right\rangle}  \leqslant \frac{2}{\sqrt{\varepsilon}} \right\}.\] 
As $\sup_{N\in \N}{\E\left(\left\langle \left\|\id\right\|^{2}, Z_{0}^{N} \right\rangle \right)} < \infty$, it follows from Lemma \ref{Moments_Moran} (respectively Proposition \ref{Prop_moments_non_bornes_RECENTRE}) that there exists a constant $\widetilde{C}_{T}>0$, independent of $N$,  such that $\P_{Y_{0}^{N}}\left( \Omega_{\varepsilon,N}\right) \geqslant 1 - \widetilde{C}_{T} \sqrt{\varepsilon}$ (respectively $\P_{Y_{0}}\left(\Omega_{\varepsilon, \infty} \right) \geqslant 1 - \widetilde{C}_{T} \sqrt{\varepsilon}$). Moreover, on $\Omega_{\varepsilon,N}$, for all $t \in [0,T]$,
\begin{align*}
d_{0}\left(F\left(Y^{N} \right), F_{\varepsilon}\left(Y^{N} \right) \right) & \leqslant \sup_{t\in [0,T]}{d_{FM}\left(F\left(Y^{N} \right)(t), F_{\varepsilon}\left(Y^{N} \right)(t) \right)} \\
& \leqslant \sup_{t \in [0,T]}{\sup_{f\in BL_{1}({ \R^{d}})}{\left\|f \circ \tau_{-\left\langle \id, Y_{t}^{N} \right\rangle} - f \circ \tau_{-\left\langle h_{\varepsilon}, Y_{t}^{N} \right\rangle} \right\|_{\infty}}} \\
& \leqslant \sup_{t\in [0,T]}{\left\|\left\langle \left\|h_{\varepsilon} - \id \right\|, Y^{N}_{t} \right\rangle \right\|_{\infty}} \\
& \leqslant \frac{\varepsilon}{2}\sup_{t\in [0,T]}{\left\|\left\langle \left\|\id\right\|^{2}, Y_{t}^{N} \right\rangle\right\|_{\infty}} \\
& \leqslant \sqrt{\varepsilon},
\end{align*}
where we used the inequality $\left\|h_{\varepsilon} - \id \right\| \leqslant \frac{\varepsilon}{2}\left\|\id\right\|^{2}$.  
Similarly on $\Omega_{\varepsilon, \infty}$, \[d_{0}\left(F\left(Y \right), F_{\varepsilon}\left(Y \right) \right) \leqslant \sqrt{\varepsilon}. \] 

\textbf{Step 3. Conclusion.} We want to prove that for all $g \in \CCCC_{b}\left(\D\left([0,T], \MM_{1}({ \R^{d}}) \right), \R \right)$, $$\lim_{N \to + \infty}{\E\left(g\left(Z^{N} \right) \right)} = \E\left(g(Z) \right).$$ Thanks to the \textsc{Portmanteau} theorem  {\color{blue} \cite[\color{black} Theorem 2.1]{Billingsley}}, it is sufficient to prove this for all $g$ $1-$\textsc{Lipschitz}. As $\left(Y^{N}\right)_{N \in \N^{\star}}$ converges in law to $Y$, we deduce that, for $N$ large enough, \[\left|\E\left(g \circ F_{\varepsilon} \left(Y^{N} \right) \right) - \E\left(g \circ F_{\varepsilon} \left(Y\right) \right) \right| \leqslant \sqrt{\varepsilon}.\]
Using that $g$ is $1-$\textsc{Lipschitz}, and the inequalities of Step 2, it follows that 
\begin{align*}
\left|\E\left(g \left(Z^{N} \right) \right) - \E\left(g \left(Z \right) \right) \right| & = \left|\E\left(g \circ F \left(Y^{N} \right) \right) - \E\left(g \circ F \left(Y \right) \right) \right| \\
& \leqslant \left|\E\left(g \circ F_{\varepsilon} \left(Y^{N} \right) \right) - \E\left(g \circ F_{\varepsilon} \left(Y\right) \right) \right| \\
& \qquad  + \left|\E\left(g \circ F \left(Y^{N} \right) \right) - \E\left(g \circ F_{\varepsilon} \left(Y^{N}\right) \right) \right| \\
& \qquad  + \left|\E\left(g \circ F_{\varepsilon}  \left(Y \right) \right) - \E\left(g \circ F \left(Y\right) \right) \right| \\
& \leqslant \sqrt{\varepsilon} + 2\left\|g \right\|_{\infty}\left[\P_{Y_{0}^{N}}\left(\Omega_{\varepsilon, N}^{c} \right) + \P_{Y_{0}}\left(\Omega_{\varepsilon, \infty}^{c} \right) \right] \\
& \qquad + \sqrt{\varepsilon}\left[\P_{Y_{0}^{N}}\left(\Omega_{\varepsilon, N} \right) + \P_{Y_{0}}\left(\Omega_{\varepsilon, \infty} \right)\right] \\
& \leqslant \left(3 + 2\left\|g \right\|_{\infty}\widetilde{C}_{T} \right)\sqrt{\varepsilon}.
\end{align*}
The announced result follows and completes the proof. \qedhere
\end{proof}

\subsection{Backward construction of \textsc{Moran}'s process and \textsc{Kingman}'s coalescent\label{Ergodicite_Sect_4.2}} In this section, we exploit the well-known relationship between the \textsc{Moran} model and the \textsc{Kingman} coalescent to obtain in Proposition \ref{Prop_2_Genealogie} a result of exponential ergodicity in total variation for the centered \textsc{Moran} model uniformly in $N$. The genealogy of a sample from a population evolving according to the \textsc{Moran} model of Section \ref{Ergodicite_Sect_4.1} is exactly determined by \textsc{Kingman}'s coalescent with  coalescence rate { $2\gamma$}. The state of the population at the final time is constructed from the ancestral positions by following the genealogy and adding mutations on the genealogical tree of the sample according to a standard Brownian motion. 

{  More precisely, on the probability space $\left(\widehat{\Omega}, \widehat{\FF}, \widehat{\P} \right)$, for any $T>0$, given the sample path $\left(k_{N,t} \right)_{0\leqslant t \leqslant T}$ of a \textsc{Kingman}'s coalescent in $\D\left([0,T], \Pi_{N} \right)$ with $\Pi_{N}$ the set of partitions of $\{1, \cdots,  N\}$, given an independent family of i.i.d.\ standard $d-$dimensional Brownian motions $(W^{(B)})_{B\subset\{1,\cdots,N\}}$ and given an independent uniformly distributed random permutation $\sigma$ of $\{1,\cdots,N\}$, we can construct the random variable
\begin{equation}
\label{eq:couplage-Moran-Kingman}
    \widehat{Y}^{N, \mu_{N}}_{T} := \frac{1}{N}\sum\limits_{i\, = \, 1}^{N}{\delta_{u_{i}}},
\end{equation}
where for all $i \in \left\{1, \cdots, N \right\}$, 
\begin{equation}
u_{i} := u_{i}^{\mu_{N}} := x_{a(\sigma(i))} + \int_{0}^{T}{\dd W_{s}^{\left(B\left(s , i \right)\right)}},
\label{Def_u_i}
\end{equation}
where $B(s,i)$ is the block of the partition $k_{N,s}$ containing $i$ and where $a(i)$ is the number of the block of $k_{N,T}$ containing $i$, according to some arbitrary order (e.g.\ lexicographic).}

The well-known backward construction of the \textsc{Moran} model {\color{blue} \cite[\color{black}Section 1.2]{Etheridge_genealogical_2019}, \cite[\color{black}Section 2.8]{Etheridge}} entails the following result.
\begin{Prop}
For all initial condition $\mu_{N} \in \MM_{1,N}({ \R^{d}})$ {  and all $T>0$},  $Y^{N}_{T} \overset{{\rm{law}}}{=} \widehat{Y}^{N, \mu_{N}}_{T}$ where $Y_{0}^{N} = \mu_{N}$.\label{Prop_1_Genealogie}
\end{Prop}

\begin{figure}[!h]
\centering		
\begin{minipage}[h]{.49\textwidth} \centering
 \definecolor{ffqqff}{rgb}{1.,0.,1.} 
\definecolor{ffxfqq}{rgb}{0.490,0.490,1} 
\definecolor{ttzzqq}{rgb}{0.294,0.,0.51} 
\definecolor{qqqqff}{rgb}{0.2,0.6,0.} 
\begin{tikzpicture}[line cap=round,line join=round,>=triangle 45,x=0.5cm,y=0.5cm]
\clip(0.0,0.75) rectangle (15.4,7.75);
\draw [>=stealth,->,line width=1pt] (0.8,1.5)-- (15.4,1.5);
\draw [line width=1.75pt,color=ffxfqq] (1.,3.)-- (14.5,3.);
\draw [line width=1.75pt,color=ffxfqq] (1.,7.)-- (14.5,7.);

\draw [>=stealth,->,line width=1pt,color=qqqqff] (12.5,3.) -- (12.5,5.);
\draw [>=stealth,->,line width=1pt,color=qqqqff] (2.,6.) -- (2.,5.);
\draw [>=stealth,->,line width=1pt,color=qqqqff] (4.,4.) -- (4.,6.);
\draw [>=stealth,->,line width=1.75pt,color=ffxfqq] (9.,3.) -- (9.,6.);
\draw [>=stealth,->,line width=1.75pt,color=ffxfqq] (10.5,3.) -- (10.5,4.);
\draw [>=stealth,->,line width=1.75pt,color=ffxfqq] (13.5,7.) -- (13.5,5.);
\draw [line width=1.75pt,color=ffxfqq] (14.5,5.)-- (13.5,5.);
\draw [line width=1pt] (13.5,5.)-- (1.,5.);
\draw [line width=1.75pt,color=ffxfqq] (14.5,6.)-- (9.,6.);
\draw [line width=1pt] (9.,6.)-- (1.,6.);
\draw [line width=1.75pt,color=ffxfqq] (14.5,4.)-- (10.5,4.);
\draw [line width=1pt] (10.5,4.)-- (1.,4.);


\draw [line width=1.pt, color = Mon_orange] (10.5,4)-- (10.9,4.3);
\draw [line width=1.pt, color = Mon_orange] (10.9,4.3)-- (11.3,3.7);
\draw [line width=1.pt, color = Mon_orange] (11.3,3.7)-- (11.7,4.3);
\draw [line width=1.pt, color = Mon_orange] (11.7,4.3)-- (12.1,3.7);
\draw [line width=1.pt, color = Mon_orange] (12.1,3.7)-- (12.5,4.3);
\draw [line width=1.pt, color = Mon_orange] (12.5,4.3)-- (12.9,3.7);
\draw [line width=1.pt, color = Mon_orange] (12.9,3.7)-- (13.3,4.3);
\draw [line width=1.pt, color = Mon_orange] (13.3,4.3)-- (13.7,3.7);
\draw [line width=1.pt, color = Mon_orange] (13.7,3.7)-- (14.1,4.3);
\draw [line width=1.pt, color = Mon_orange] (14.1,4.3)-- (14.5,4);

\draw [line width=1.pt, color = Mon_orange] (1,7)-- (1.4,7.3);
\draw [line width=1.pt, color = Mon_orange] (1.4,7.3)-- (1.8,6.7);
\draw [line width=1.pt, color = Mon_orange] (1.8,6.7)-- (2.2,7.3);
\draw [line width=1.pt, color = Mon_orange] (2.2,7.3)-- (2.6,6.7);
\draw [line width=1.pt, color = Mon_orange] (2.6,6.7)-- (3,7.3);
\draw [line width=1.pt, color = Mon_orange] (3,7.3)-- (3.4,6.7);
\draw [line width=1.pt, color = Mon_orange] (3.4,6.7)-- (3.8,7.3);
\draw [line width=1.pt, color = Mon_orange] (3.8,7.3)-- (4.2,6.7);
\draw [line width=1.pt, color = Mon_orange] (4.2,6.7)-- (4.6,7.3);
\draw [line width=1.pt, color = Mon_orange] (4.6,7.3)-- (5,6.7);
\draw [line width=1.pt, color = Mon_orange] (5,6.7)-- (5.4,7.3);
\draw [line width=1.pt, color = Mon_orange] (5.4,7.3)-- (5.8,6.7);
\draw [line width=1.pt, color = Mon_orange] (5.8,6.7)-- (6.2,7.3);
\draw [line width=1.pt, color = Mon_orange] (6.2,7.3)-- (6.6,6.7);
\draw [line width=1.pt, color = Mon_orange] (6.6,6.7)-- (7,7.3);
\draw [line width=1.pt, color = Mon_orange] (7,7.3)-- (7.4,6.7);
\draw [line width=1.pt, color = Mon_orange] (7.4,6.7)-- (7.8,7.3);
\draw [line width=1.pt, color = Mon_orange] (7.8,7.3)-- (8.2,6.7);
\draw [line width=1.pt, color = Mon_orange] (8.2,6.7)-- (8.6,7.3);
\draw [line width=1.pt, color = Mon_orange] (8.6,7.3)-- (9,6.7);
\draw [line width=1.pt, color = Mon_orange] (9,6.7)-- (9.4,7.3);
\draw [line width=1.pt, color = Mon_orange] (9.4,7.3)-- (9.8,6.7);
\draw [line width=1.pt, color = Mon_orange] (9.8,6.7)-- (10.2,7.3);
\draw [line width=1.pt, color = Mon_orange] (10.2,7.3)-- (10.6,6.7);
\draw [line width=1.pt, color = Mon_orange] (10.6,6.7)-- (11,7.3);
\draw [line width=1.pt, color = Mon_orange] (11,7.3)-- (11.4,6.7);
\draw [line width=1.pt, color = Mon_orange] (11.4,6.7)-- (11.8,7.3);
\draw [line width=1.pt, color = Mon_orange] (11.8,7.3)-- (12.2,6.7);
\draw [line width=1.pt, color = Mon_orange] (12.2,6.7)-- (12.6,7.3);
\draw [line width=1.pt, color = Mon_orange] (12.6,7.3)-- (13,6.7);
\draw [line width=1.pt, color = Mon_orange] (13,6.7)-- (13.4,7.3);
\draw [line width=1.pt, color = Mon_orange] (13.4,7.3)-- (13.5,7);

\draw [line width=1.pt, color = Mon_orange] (1,3)-- (1.4,3.3);
\draw [line width=1.pt, color = Mon_orange] (1.4,3.3)-- (1.8,2.7);
\draw [line width=1.pt, color = Mon_orange] (1.8,2.7)-- (2.2,3.3);
\draw [line width=1.pt, color = Mon_orange] (2.2,3.3)-- (2.6,2.7);
\draw [line width=1.pt, color = Mon_orange] (2.6,2.7)-- (3,3.3);
\draw [line width=1.pt, color = Mon_orange] (3,3.3)-- (3.4,2.7);
\draw [line width=1.pt, color = Mon_orange] (3.4,2.7)-- (3.8,3.3);
\draw [line width=1.pt, color = Mon_orange] (3.8,3.3)-- (4.2,2.7);
\draw [line width=1.pt, color = Mon_orange] (4.2,2.7)-- (4.6,3.3);
\draw [line width=1.pt, color = Mon_orange] (4.6,3.3)-- (5,2.7);
\draw [line width=1.pt, color = Mon_orange] (5,2.7)-- (5.4,3.3);
\draw [line width=1.pt, color = Mon_orange] (5.4,3.3)-- (5.8,2.7);
\draw [line width=1.pt, color = Mon_orange] (5.8,2.7)-- (6.2,3.3);
\draw [line width=1.pt, color = Mon_orange] (6.2,3.3)-- (6.6,2.7);
\draw [line width=1.pt, color = Mon_orange] (6.6,2.7)-- (7,3.3);
\draw [line width=1.pt, color = Mon_orange] (7,3.3)-- (7.4,2.7);
\draw [line width=1.pt, color = Mon_orange] (7.4,2.7)-- (7.8,3.3);
\draw [line width=1.pt, color = Mon_orange] (7.8,3.3)-- (8.2,2.7);
\draw [line width=1.pt, color = Mon_orange] (8.2,2.7)-- (8.6,3.3);
\draw [line width=1.pt, color = Mon_orange] (8.6,3.3)-- (9,3);

\draw [line width=1.pt, color = Mon_orange] (9,3)-- (9.4,3.3);
\draw [line width=1.pt, color = Mon_orange] (9.4,3.3)-- (9.8,2.7);
\draw [line width=1.pt, color = Mon_orange] (9.8,2.7)-- (10.2,3.3);
\draw [line width=1.pt, color = Mon_orange] (10.2,3.3)-- (10.5,3);

\draw [line width=1.pt, color = Mon_orange] (13.5,5)-- (13.9,5.3);
\draw [line width=1.pt, color = Mon_orange] (13.9,5.3)-- (14.3,4.7);
\draw [line width=1.pt, color = Mon_orange] (14.3,4.7)-- (14.5,5);

\begin{scriptsize}
\draw[color=black] (0.35,3) node {$x_{1}$};
\draw[color=black] (0.35,4) node {$x_{2}$};
\draw[color=black] (0.35,5) node {$x_{3}$};
\draw[color=black] (0.35,6) node {$x_{4}$};
\draw[color=black] (0.35,7) node {$x_{5}$};

\draw[color=black] (15.15,3) node {$u_{1}$};
\draw[color=black] (15.15,4) node {$u_{2}$};
\draw[color=black] (15.15,5) node {$u_{3}$};
\draw[color=black] (15.15,6) node {$u_{4}$};
\draw[color=black] (15.15,7) node {$u_{5}$};

\draw [color=black] (1.,1.5)-- ++(0pt,0 pt) -- ++(0pt,0 pt) ++(0pt,-2.5pt) -- ++(0 pt,5.0pt);
\draw[color=black] (1,1) node {$0$};
\draw [color=black] (14.5,1.5)-- ++(0pt,0 pt) -- ++(0pt,0 pt) ++(0pt,-2.5pt) -- ++(0 pt,5.0pt);
\draw[color=black] (14.5,1) node {$T$};
\draw [color=qqqqff] (2.,1.5)-- ++(0pt,0 pt) -- ++(0pt,0 pt) ++(0pt,-2.5pt) -- ++(0 pt,5.0pt);
\draw[color=qqqqff] (2,1) node {$t_{1}$};
\draw [color=qqqqff] (4.,1.5)-- ++(0pt,0 pt) -- ++(0pt,0 pt) ++(0pt,-2.5pt) -- ++(0 pt,5.0pt);
\draw[color=qqqqff] (4.1,1) node {$t_{2}$};
\draw [color=ffxfqq] (9.,1.5)-- ++(0pt,0 pt) -- ++(0pt,0 pt) ++(0pt,-2.5pt) -- ++(0 pt,5.0pt);
\draw[color=ffxfqq] (9,1) node {$t_{3}$};
\draw [color=ffxfqq] (10.5,1.5)-- ++(0pt,0 pt) -- ++(0pt,0 pt) ++(0pt,-2.5pt) -- ++(0 pt,5.0pt);
\draw[color=ffxfqq] (10.5,1) node {$t_{4}$};
\draw [color=qqqqff] (12.4,1.5)-- ++(0pt,0 pt) -- ++(0pt,0 pt) ++(0pt,-2.5pt) -- ++(0 pt,5.0pt);
\draw[color=qqqqff] (12.5,1) node {$t_{5}$};
\draw [color=ffxfqq] (13.5,1.5)-- ++(0pt,0 pt) -- ++(0pt,0 pt) ++(0pt,-2.5pt) -- ++(0 pt,5.0pt);
\draw[color=ffxfqq] (13.5,1) node {$t_{6}$};
\end{scriptsize}
\end{tikzpicture}
\caption{Graphical representation of the \textsc{Moran} model with $N = 5$ and initial condition $\mu_{5} = \frac{1}{5}\sum_{i\, = \, 1}^{5}{\delta_{x_{i}}}$. {  The times $t_i$ are the successives times of resampling events, represented by arrows. The arrow $i\to j$ indicates that $i$ reproduced and $j$ died.}}
\label{FigMO3}
\end{minipage}
\hspace{1.3cm}
\begin{minipage}[h]{.41\textwidth}\centering
\definecolor{ffqqff}{rgb}{1.,0.,1.} 
\definecolor{ffxfqq}{rgb}{0.490,0.490,1} 
\definecolor{ttzzqq}{rgb}{0.294,0.,0.51} 
\definecolor{qqqqff}{rgb}{0.2,0.6,0.} 
\begin{tikzpicture}[line cap=round,line join=round,>=triangle 45,x=0.5cm,y=0.5cm]
\clip(16.7,1.6) rectangle (27.8,17.15);
\draw [>=stealth,->, line width=1pt] (18.,2.2) -- (18.,17.15);
\draw [line width=1pt,color=ffxfqq] (19.,2.5)-- (19.,6.5);
\draw [line width=1pt,color=ffxfqq] (19.,6.5)-- (21.,6.5);
\draw [line width=1pt,color=ffxfqq] (21.,6.5)-- (21.,2.5);
\draw [line width=1pt,color=ffxfqq] (20.,6.5)-- (20.,8.);
\draw [line width=1pt,color=ffxfqq] (20.,8.)-- (23.,8.);
\draw [line width=1pt,color=ffxfqq] (23.,8.)-- (23.,2.5);
\draw [line width=1pt,color=ffxfqq] (21.5,8.)-- (21.5,15.);

\draw [line width=1pt,color=ffxfqq] (21.5,15.)-- (21.5,16.);
\draw [line width=1pt,color=ffxfqq] (25.,2.5)-- (25.,3.5);
\draw [line width=1pt,color=ffxfqq] (27.,2.5)-- (27.,3.5);
\draw [line width=1pt,color=ffxfqq] (25.,3.5)-- (27.,3.5);
\draw [line width=1pt,color=ffxfqq] (26.,3.5)-- (26.,16.);


\draw [line width=1.pt, color = Mon_orange] (21.5,16)-- (21.8,15.6);
\draw [line width=1.pt, color = Mon_orange] (21.8,15.6)-- (21.2,15.2);
\draw [line width=1.pt, color = Mon_orange] (21.2,15.2)-- (21.8,14.8);
\draw [line width=1.pt, color = Mon_orange] (21.8,14.8)-- (21.2,14.4);
\draw [line width=1.pt, color = Mon_orange] (21.2,14.4)-- (21.8,14);
\draw [line width=1.pt, color = Mon_orange] (21.8,14)-- (21.2,13.6);
\draw [line width=1.pt, color = Mon_orange] (21.2,13.6)-- (21.8,13.2);
\draw [line width=1.pt, color = Mon_orange] (21.8,13.2)-- (21.2,12.8);
\draw [line width=1.pt, color = Mon_orange] (21.2,12.8)-- (21.8,12.4);
\draw [line width=1.pt, color = Mon_orange] (21.8,12.4)-- (21.2,12);
\draw [line width=1.pt, color = Mon_orange] (21.2,12)-- (21.8,11.6);
\draw [line width=1.pt, color = Mon_orange] (21.8,11.6)-- (21.2,11.2);
\draw [line width=1.pt, color = Mon_orange] (21.2,11.2)-- (21.8,10.8);
\draw [line width=1.pt, color = Mon_orange] (21.8,10.8)-- (21.2,10.4);
\draw [line width=1.pt, color = Mon_orange] (21.2,10.4)-- (21.8,10);
\draw [line width=1.pt, color = Mon_orange] (21.8,10)-- (21.2,9.6);
\draw [line width=1.pt, color = Mon_orange] (21.2,9.6)-- (21.8,9.2);
\draw [line width=1.pt, color = Mon_orange] (21.8,9.2)-- (21.2,8.8);
\draw [line width=1.pt, color = Mon_orange] (21.2,8.8)-- (21.8,8.4);
\draw [line width=1.pt, color = Mon_orange] (21.8,8.4)-- (21.5,8);

\draw [line width=1.pt, color = Mon_orange] (26,16)-- (26.3,15.6);
\draw [line width=1.pt, color = Mon_orange] (26.3,15.6)-- (25.7,15.2);
\draw [line width=1.pt, color = Mon_orange] (25.7,15.2)-- (26.3,14.8);
\draw [line width=1.pt, color = Mon_orange] (26.3,14.8)-- (25.7,14.4);
\draw [line width=1.pt, color = Mon_orange] (25.7,14.4)-- (26.3,14);
\draw [line width=1.pt, color = Mon_orange] (26.3,14)-- (25.7,13.6);
\draw [line width=1.pt, color = Mon_orange] (25.7,13.6)-- (26.3,13.2);
\draw [line width=1.pt, color = Mon_orange] (26.3,13.2)-- (25.7,12.8);
\draw [line width=1.pt, color = Mon_orange] (25.7,12.8)-- (26.3,12.4);
\draw [line width=1.pt, color = Mon_orange] (26.3,12.4)-- (25.7,12);
\draw [line width=1.pt, color = Mon_orange] (25.7,12)-- (26.3,11.6);
\draw [line width=1.pt, color = Mon_orange] (26.3,11.6)-- (25.7,11.2);
\draw [line width=1.pt, color = Mon_orange] (25.7,11.2)-- (26.3,10.8);
\draw [line width=1.pt, color = Mon_orange] (26.3,10.8)-- (25.7,10.4);
\draw [line width=1.pt, color = Mon_orange] (25.7,10.4)-- (26.3,10);
\draw [line width=1.pt, color = Mon_orange] (26.3,10)-- (25.7,9.6);
\draw [line width=1.pt, color = Mon_orange] (25.7,9.6)-- (26.3,9.2);
\draw [line width=1.pt, color = Mon_orange] (26.3,9.2)-- (25.7,8.8);
\draw [line width=1.pt, color = Mon_orange] (25.7,8.8)-- (26.3,8.4);
\draw [line width=1.pt, color = Mon_orange] (26.3,8.4)-- (25.7,8);
\draw [line width=1.pt, color = Mon_orange] (25.7,8)-- (26.3,7.6);
\draw [line width=1.pt, color = Mon_orange] (26.3,7.6)-- (25.7,7.2);
\draw [line width=1.pt, color = Mon_orange] (25.7,7.2)-- (26.3,6.8);
\draw [line width=1.pt, color = Mon_orange] (26.3,6.8)-- (25.7,6.4);
\draw [line width=1.pt, color = Mon_orange] (25.7,6.4)-- (26.3,6);
\draw [line width=1.pt, color = Mon_orange] (26.3,6)-- (25.7,5.6);
\draw [line width=1.pt, color = Mon_orange] (25.7,5.6)-- (26.3,5.2);
\draw [line width=1.pt, color = Mon_orange] (26.3,5.2)-- (25.7,4.8);
\draw [line width=1.pt, color = Mon_orange] (25.7,4.8)-- (26.3,4.4);
\draw [line width=1.pt, color = Mon_orange] (26.3,4.4)-- (25.7,4);
\draw [line width=1.pt, color = Mon_orange] (25.7,4)-- (26.3,3.6);
\draw [line width=1.pt, color = Mon_orange] (26.3,3.6)-- (26,3.5);

\draw [line width=1.pt, color = Mon_orange] (20,8)-- (20.3,7.6);
\draw [line width=1.pt, color = Mon_orange] (20.3,7.6)-- (19.7,7.2);
\draw [line width=1.pt, color = Mon_orange] (19.7,7.2)-- (20.3,6.8);
\draw [line width=1.pt, color = Mon_orange] (20.3,6.8)-- (20,6.5);

\draw [line width=1.pt, color = Mon_orange] (21,6.5)-- (21.3,6.1);
\draw [line width=1.pt, color = Mon_orange] (21.3,6.1)-- (20.7,5.7);
\draw [line width=1.pt, color = Mon_orange] (20.7,5.7)-- (21.3,5.3);
\draw [line width=1.pt, color = Mon_orange] (21.3,5.3)-- (20.7,4.9);
\draw [line width=1.pt, color = Mon_orange] (20.7,4.9)-- (21.3,4.5);
\draw [line width=1.pt, color = Mon_orange] (21.3,4.5)-- (20.7,4.1);
\draw [line width=1.pt, color = Mon_orange] (20.7,4.1)-- (21.3,3.7);
\draw [line width=1.pt, color = Mon_orange] (21.3,3.7)-- (20.7,3.3);
\draw [line width=1.pt, color = Mon_orange] (20.7,3.3)-- (21.3,2.9);
\draw [line width=1.pt, color = Mon_orange] (21.3,2.9)-- (21,2.5);

\draw [line width=1.pt, color = Mon_orange] (25,3.5)-- (25.3,3.1);
\draw [line width=1.pt, color = Mon_orange] (25.3,3.1)-- (24.7,2.7);
\draw [line width=1.pt, color = Mon_orange] (24.7,2.7)-- (25,2.5);

\draw[color=black,decorate,decoration={brace,raise=0.2cm}]
(25,3.5) -- (25,2.5) node[above=0.2cm,pos=0.5,sloped]{}; 

\draw[color=black,decorate,decoration={brace,raise=0.2cm}]
(21,6.5) -- (21,2.5) node[above=0.25cm,pos=0.5,sloped] {};

\draw[color=black,decorate,decoration={brace,raise=0.2cm}]
(20,8) -- (20,6.5) node[above=0.25cm,pos=0.5,sloped] {};

\draw[color=black,decorate,decoration={brace,raise=0.2cm}]
(21.5,16) -- (21.5,8) node[above=0.25cm,pos=0.5,sloped] {};

\draw[color=black,decorate,decoration={brace,raise=0.2cm}]
(26,16) -- (26,3.5) node[above=0.25cm,pos=0.5,sloped] {};

\begin{scriptsize}
\draw[color=black] (26.,3) node {\tiny $w_{1}$};
\draw[color=black] (22.,4.5) node {\tiny $w_{2}$};
\draw[color=black] (21.,7.25) node {\tiny $w_{3}$};
\draw[color=black] (22.5,12) node {\tiny $w_{4}$};
\draw[color=black] (27,9.75) node {\tiny $w_{5}$};
\end{scriptsize}

\begin{scriptsize}
\draw [color=black] (1.,0.5)-- ++(-2.5pt,0 pt) -- ++(5.0pt,0 pt) ++(-2.5pt,-2.5pt) -- ++(0 pt,5.0pt);
\draw [fill=blue] (19.,2.5) circle (2.5pt);
\draw[color=blue] (19,1.85) node {$u_{1}$};
\draw [fill=blue] (21.,2.5) circle (2.5pt);
\draw[color=blue] (21,1.85) node {$u_{2}$};
\draw [fill=blue] (23.,2.5) circle (2.5pt);
\draw[color=blue] (23,1.85) node {$u_{4}$};
\draw [fill=blue] (25.,2.5) circle (2.5pt);
\draw[color=blue] (25,1.85) node {$u_{3}$};
\draw [fill=blue] (27.,2.5) circle (2.5pt);
\draw[color=blue] (27,1.85) node {$u_{5}$};
\draw [color=black] (18.,2.5)-- ++(-2.5pt,0 pt) -- ++(5.0pt,0 pt) ++(-2.5pt,-0pt) -- ++(0 pt,0pt);
\draw[color=black] (17.1,2.5) node {$T$};
\draw [color=ffxfqq] (18.,3.5)-- ++(-2.5pt,0 pt) -- ++(5.0pt,0 pt) ++(-2.5pt,-0pt) -- ++(0 pt,0pt);
\draw[color=ffxfqq] (17.1,3.5) node {$t_{6}$};
\draw [color=qqqqff] (18.,4.6)-- ++(-2.5pt,0 pt) -- ++(5.0pt,0 pt) ++(-2.5pt,-0pt) -- ++(0 pt,0pt);
\draw[color=qqqqff] (17.1,4.5) node {$t_{5}$};
\draw [color=ffxfqq] (18.,6.5)-- ++(-2.5pt,0 pt) -- ++(5.0pt,0 pt) ++(-2.5pt,-0pt) -- ++(0 pt,0pt);
\draw[color=ffxfqq] (17.1,6.5) node {$t_{4}$};
\draw [color=ffxfqq] (18.,8.)-- ++(-2.5pt,0 pt) -- ++(5.0pt,0 pt) ++(-2.5pt,-0pt) -- ++(0 pt,0pt);
\draw[color=ffxfqq] (17.1,8) node {$t_{3}$};
\draw [color=qqqqff] (18.,13.)-- ++(-2.5pt,0 pt) -- ++(5.0pt,0 pt) ++(-2.5pt,-0pt) -- ++(0 pt,0pt);
\draw[color=qqqqff] (17.1,12.9) node {$t_{2}$};
\draw [color=qqqqff] (18.,15.)-- ++(-2.5pt,0 pt) -- ++(5.0pt,0 pt) ++(-2.5pt,-0pt) -- ++(0 pt,0pt);
\draw[color=qqqqff] (17.1,15) node {$t_{1}$};

\draw [color=black] (18.,16.)-- ++(-2.5pt,0 pt) -- ++(5.0pt,0 pt) ++(-2.5pt,-0pt) -- ++(0 pt,0pt);
\draw[color=black] (17.1,16) node {$0$};
\draw [fill=blue] (21.5,16.) circle (2.5pt);
\draw[color=blue] (21.5,16.75) node {$x_{\sigma(1)} = x_{1}$};
\draw [fill=blue] (26.,16.) circle (2.5pt);
\draw[color=blue] (26,16.75) node {$x_{\sigma(2)} = x_{5}$};
\end{scriptsize}
\end{tikzpicture}
\caption{\textsc{Kingman}'s genealogy $\left(k_{5, T-t}\right)_{0\leqslant t \leqslant T}$ under the \textsc{Moran} model on the left, tracing back from  time $T$ to time $0$.}
\label{FigMO4}
\end{minipage}
\end{figure}

We give an illustration of this result in Figure~\ref{FigMO3}, where mutations are denoted by
``\begin{tikzpicture}[line cap=round,line join=round,>=triangle 45,x=0.5cm,y=0.5cm]
\draw [line width=1.pt, color = Mon_orange] (1,6.7)-- (1.4,7.3);
\draw [line width=1.pt, color = Mon_orange] (1.4,7.3)-- (1.8,6.7);
\draw [line width=1.pt, color = Mon_orange] (1.8,6.7)-- (2.2,7.3);
\draw [line width=1.pt, color = Mon_orange] (2.2,7.3)-- (2.6,6.7);
\draw [line width=1.pt, color = Mon_orange] (2.6,6.7)-- (3,7.3);
\draw [line width=1.pt, color = Mon_orange] (3,7.3)-- (3.4,6.7);
\draw [line width=1.pt, color = Mon_orange] (3.4,6.7)-- (3.8,7.3);
\draw [line width=1.pt, color = Mon_orange] (3.8,7.3)-- (4.2,6.7);
\end{tikzpicture}''. Not all of them are shown for the sake of clarity. For example,
$u_2=x_1+w_4+w_3+w_2$, where $w_{2} := W_{T-t_{4}}^{(\{2\})} - W_{0}^{(\{2\})}$, $w_{3} := W_{T-t_{3}}^{(\{1,2\})} - W_{T-t_{4}}^{(\{1,2\})}$ and $w_{4} := W_{T}^{(\{1,2,4\})} - W_{T-t_{3}}^{(\{1,2,4\})}$.

%

{  We can then construct the centered version of the random variables $\widehat{Y}^{N, \mu_{N}}_{T}$ as follows:}
\[\widehat{Z}^{N,\mu_{N}}_{T} := \tau_{-\left\langle \id, \widehat{Y}_{T}^{N, \mu_{N}} \right\rangle} \sharp \, \widehat{Y}^{N,\mu_{N}}_{T}.\]
\begin{Cor}
For all initial condition $\mu_{N} \in \MM_{1,N}^{c,2}({ \R^{d}})$, $Z^{N}_{T} \overset{{\rm{law}}}{=} \widehat{Z}^{N, \mu_{N}}_{T} = \frac{1}{N}\sum_{i\, = \, 1}^{N}{\delta_{v_{i}}}  $
where, for all $i \in \left\{1, \cdots, N \right\}$, $v_{i} := v_{i}^{\mu_{N}} := u_{i} - \frac{1}{N}\sum_{j\, = \, 1}^{N}{u_{j}}$ and $Z_{0}^{N} = \mu_{N}$.\label{Corollaire_1_Genealogie}
\end{Cor}

\begin{figure}[!htb]
\minipage{0.42\textwidth}
  \centering
  \definecolor{xdxdff}{rgb}{0.490,0.490,1} 
  \definecolor{qqqqff}{rgb}{0.,0.,1.} 
\definecolor{qqzzqq}{rgb}{0.,0.6,0.} 
\definecolor{ffxfqq}{rgb}{1.,0.4980392156862745,0.} 
  \definecolor{ududff}{rgb}{0.30196078431372547,0.30196078431372547,1.}
\begin{tikzpicture}[line cap=round,line join=round,>=triangle 45,x=0.5cm,y=0.5cm]
\clip(-1.5,0.275) rectangle (9.25,10.15);
\draw [line width=1pt] (2.,1.)-- (2.,3.);
\draw [line width=1pt] (2.,3.)-- (4.,3.);
\draw [line width=1pt] (4.,1.)-- (4.,3.);
\draw [line width=1pt] (6.,1.)-- (6.,6.);
\draw [line width=1pt] (8.,1.)-- (8.,6.);
\draw [line width=1pt] (6.,6.)-- (8.,6.);
\draw [line width=1pt] (7.,6.)-- (7.,7.5);
\draw [line width=1pt] (3.,3.)-- (3.,7.5);
\draw [line width=1pt] (3.,7.5)-- (7.,7.5);
\draw [line width=1pt] (5.,7.5)-- (5.,9.);
\draw [>=stealth,->,line width=1pt] (1.,0.5) -- (1.,10.15);

\draw [line width=1.pt, color = Mon_orange] (5,9)-- (5.3,8.6);
\draw [line width=1.pt, color = Mon_orange] (5.3,8.6)-- (4.7,8.2);
\draw [line width=1.pt, color = Mon_orange] (4.7,8.2)-- (5.3,7.8);
\draw [line width=1.pt, color = Mon_orange] (5.3,7.8)-- (5,7.5);

\draw [line width=1.pt, color = Mon_orange] (2,3)-- (2.3,2.6);
\draw [line width=1.pt, color = Mon_orange] (2.3,2.6)-- (1.7,2.2);
\draw [line width=1.pt, color = Mon_orange] (1.7,2.2)-- (2.3,1.8);
\draw [line width=1.pt, color = Mon_orange] (2.3,1.8)-- (1.7,1.4);
\draw [line width=1.pt, color = Mon_orange] (1.7,1.4)-- (2,1);

\draw [line width=1.pt, color = Mon_orange] (4,3)-- (4.3,2.6);
\draw [line width=1.pt, color = Mon_orange] (4.3,2.6)-- (3.7,2.2);
\draw [line width=1.pt, color = Mon_orange] (3.7,2.2)-- (4.3,1.8);
\draw [line width=1.pt, color = Mon_orange] (4.3,1.8)-- (3.7,1.4);
\draw [line width=1.pt, color = Mon_orange] (3.7,1.4)-- (4,1);

\draw [line width=1.pt, color = Mon_orange] (6,6)-- (6.3,5.6);
\draw [line width=1.pt, color = Mon_orange] (6.3,5.6)-- (5.7,5.2);
\draw [line width=1.pt, color = Mon_orange] (5.7,5.2)-- (6.3,4.8);
\draw [line width=1.pt, color = Mon_orange] (6.3,4.8)-- (5.7,4.4);
\draw [line width=1.pt, color = Mon_orange] (5.7,4.4)-- (6.3,4);
\draw [line width=1.pt, color = Mon_orange] (6.3,4)-- (5.7,3.6);
\draw [line width=1.pt, color = Mon_orange] (5.7,3.6)-- (6.3,3.2);
\draw [line width=1.pt, color = Mon_orange] (6.3,3.2)-- (5.7,2.8);
\draw [line width=1.pt, color = Mon_orange] (5.7,2.8)-- (6.3,2.4);
\draw [line width=1.pt, color = Mon_orange] (6.3,2.4)-- (5.7,2);
\draw [line width=1.pt, color = Mon_orange] (5.7,2)-- (6.3,1.6);
\draw [line width=1.pt, color = Mon_orange] (6.3,1.6)-- (5.7,1.2);
\draw [line width=1.pt, color = Mon_orange] (5.7,1.2)-- (6,1);

\draw [line width=1.pt, color = Mon_orange] (8,6)-- (8.3,5.6);
\draw [line width=1.pt, color = Mon_orange] (8.3,5.6)-- (7.7,5.2);
\draw [line width=1.pt, color = Mon_orange] (7.7,5.2)-- (8.3,4.8);
\draw [line width=1.pt, color = Mon_orange] (8.3,4.8)-- (7.7,4.4);
\draw [line width=1.pt, color = Mon_orange] (7.7,4.4)-- (8.3,4);
\draw [line width=1.pt, color = Mon_orange] (8.3,4)-- (7.7,3.6);
\draw [line width=1.pt, color = Mon_orange] (7.7,3.6)-- (8.3,3.2);
\draw [line width=1.pt, color = Mon_orange] (8.3,3.2)-- (7.7,2.8);
\draw [line width=1.pt, color = Mon_orange] (7.7,2.8)-- (8.3,2.4);
\draw [line width=1.pt, color = Mon_orange] (8.3,2.4)-- (7.7,2);
\draw [line width=1.pt, color = Mon_orange] (7.7,2)-- (8.3,1.6);
\draw [line width=1.pt, color = Mon_orange] (8.3,1.6)-- (7.7,1.2);
\draw [line width=1.pt, color = Mon_orange] (7.7,1.2)-- (8,1);

\draw [line width=1.pt, color = Mon_orange] (3,7.5)-- (3.3,7.1);
\draw [line width=1.pt, color = Mon_orange] (3.3,7.1)-- (2.7,6.7);
\draw [line width=1.pt, color = Mon_orange] (2.7,6.7)-- (3.3,6.3);
\draw [line width=1.pt, color = Mon_orange] (3.3,6.3)-- (2.7,5.9);
\draw [line width=1.pt, color = Mon_orange] (2.7,5.9)-- (3.3,5.5);
\draw [line width=1.pt, color = Mon_orange] (3.3,5.5)-- (2.7,5.1);
\draw [line width=1.pt, color = Mon_orange] (2.7,5.1)-- (3.3,4.7);
\draw [line width=1.pt, color = Mon_orange] (3.3,4.7)-- (2.7,4.3);
\draw [line width=1.pt, color = Mon_orange] (2.7,4.3)-- (3.3,3.9);
\draw [line width=1.pt, color = Mon_orange] (3.3,3.9)-- (2.7,3.5);
\draw [line width=1.pt, color = Mon_orange] (2.7,3.5)-- (3.3,3.1);
\draw [line width=1.pt, color = Mon_orange] (3.3,3.1)-- (3,3);

\draw [line width=1.pt, color = Mon_orange] (7,7.5)-- (7.3,7.1);
\draw [line width=1.pt, color = Mon_orange] (7.3,7.1)-- (6.7,6.7);
\draw [line width=1.pt, color = Mon_orange] (6.7,6.7)-- (7.3,6.3);
\draw [line width=1.pt, color = Mon_orange] (7.3,6.3)-- (7,6);

\draw[color=black,decorate,decoration={brace,raise=0.2cm}]
(2,3) -- (2,1) node[above=0.25cm,pos=0.5,sloped] {};

\draw[color=black] (3,2) node {\tiny $w_{1}$};

\draw[color=black,decorate,decoration={brace,raise=0.2cm}]
(4,3) -- (4,1) node[above=0.25cm,pos=0.5,sloped] {};

\draw[color=black] (5,2) node {\tiny $w_{2}$};

\draw[color=black,decorate,decoration={brace,raise=0.2cm}]
(6,6) -- (6,1) node[above=0.25cm,pos=0.5,sloped] {};

\draw[color=black] (7,3.5) node {\tiny $w_{3}$};

\draw[color=black,decorate,decoration={brace,raise=0.2cm}]
(8,6) -- (8,1) node[above=0.25cm,pos=0.5,sloped] {};

\draw[color=black] (9,3.5) node {\tiny $w_{4}$};

\draw[color=black,decorate,decoration={brace,raise=0.2cm}]
(3,7.5) -- (3,3) node[above=0.25cm,pos=0.5,sloped] {};

\draw[color=black] (4,5.25) node {\tiny $w_{5}$};

\draw[color=black,decorate,decoration={brace,raise=0.2cm}]
(7,7.5) -- (7,6) node[above=0.25cm,pos=0.5,sloped] {};

\draw[color=black] (8,6.75) node {\tiny $w_{6}$};

\draw[color=black,decorate,decoration={brace,raise=0.2cm}]
(5,9) -- (5,7.5) node[above=0.25cm,pos=0.5,sloped] {};

\draw[color=black] (6,8.25) node {\tiny $w_{7}$};

\begin{scriptsize}
\draw[color=blue] (5,9.55) node {$x_{\sigma(a(1))}$};

\draw [fill=blue] (2.,1.) circle (2.5pt);
\draw[color=blue] (2,0.45) node {$v_{1}$};
\draw [fill=blue] (4.,1.) circle (2.5pt);
\draw[color=blue] (4,0.45) node {$v_{2}$};
\draw [fill=blue] (6.,1.) circle (2.5pt);
\draw[color=blue] (6,0.45) node {$v_{3}$};
\draw [fill=blue] (8.,1.) circle (2.5pt);
\draw[color=blue] (8,0.45) node {$v_{4}$};

\draw [fill=blue] (5.,9.) circle (2.5pt);
\draw [color=black] (1,1)-- ++(-2.5pt,0 pt) -- ++(5.0pt,0 pt) ++(-2.5pt,-2.5pt) -- ++(0 pt,5.0pt);
\draw[color=black] (0.25,1) node {$0$};

\draw [color=black] (1.,9.)-- ++(-2.5pt,0 pt) -- ++(5.0pt,0 pt) ++(-2.5pt,-2.5pt) -- ++(0 pt,5.0pt);
\draw[color=black] (0.45,9) node {$T$};

\end{scriptsize}
\end{tikzpicture}
\caption{Illustration of the centered \textsc{Moran} process where $\left| k_{4, T}\right| = 1$.}
\label{Fig81}
\endminipage \hspace{2cm}
\minipage{0.42\textwidth}
    \centering
    \definecolor{xdxdff}{rgb}{0.490,0.490,1} 
    \definecolor{qqqqff}{rgb}{0.,0.,1.} 
\definecolor{qqzzqq}{rgb}{0.,0.6,0.} 
\definecolor{ffxfqq}{rgb}{1.,0.4980392156862745,0.} 
    \definecolor{ududff}{rgb}{0.30196078431372547,0.30196078431372547,1.}
\begin{tikzpicture}[line cap=round,line join=round,>=triangle 45,x=0.5cm,y=0.5cm]
\clip(-1.5,0.275) rectangle (9.25,10.15);
\draw [line width=1pt] (2.,1.)-- (2.,3.);
\draw [line width=1pt] (2.,3.)-- (4.,3.);
\draw [line width=1pt] (4.,3.)-- (4.,1.);
\draw [line width=1pt] (6.,1.)-- (6.,7.);
\draw [line width=1pt] (6.,7.)-- (8.,7.);
\draw [line width=1pt] (8.,7.)-- (8.,1.);
\draw [line width=1pt] (7.,9.)-- (7.,7.);
\draw [line width=1pt] (3.,9.)-- (3.,3.);
\draw [>=stealth,->,line width=1pt] (1.,0.5) -- (1.,10.15);

\draw [line width=1.pt, color = Mon_orange] (2,3)-- (2.3,2.6);
\draw [line width=1.pt, color = Mon_orange] (2.3,2.6)-- (1.7,2.2);
\draw [line width=1.pt, color = Mon_orange] (1.7,2.2)-- (2.3,1.8);
\draw [line width=1.pt, color = Mon_orange] (2.3,1.8)-- (1.7,1.4);
\draw [line width=1.pt, color = Mon_orange] (1.7,1.4)-- (2,1);

\draw [line width=1.pt, color = Mon_orange] (4,3)-- (4.3,2.6);
\draw [line width=1.pt, color = Mon_orange] (4.3,2.6)-- (3.7,2.2);
\draw [line width=1.pt, color = Mon_orange] (3.7,2.2)-- (4.3,1.8);
\draw [line width=1.pt, color = Mon_orange] (4.3,1.8)-- (3.7,1.4);
\draw [line width=1.pt, color = Mon_orange] (3.7,1.4)-- (4,1);

\draw [line width=1.pt, color = Mon_orange] (3,9)-- (3.3,8.7);
\draw [line width=1.pt, color = Mon_orange] (3.3,8.7)-- (2.7,8.3);
\draw [line width=1.pt, color = Mon_orange] (2.7,8.3)-- (3.3,7.9);
\draw [line width=1.pt, color = Mon_orange] (3.3,7.9)-- (2.7,7.5);
\draw [line width=1.pt, color = Mon_orange] (2.7,7.5)-- (3.3,7.1);
\draw [line width=1.pt, color = Mon_orange] (3.3,7.1)-- (2.7,6.7);
\draw [line width=1.pt, color = Mon_orange] (2.7,6.7)-- (3.3,6.3);
\draw [line width=1.pt, color = Mon_orange] (3.3,6.3)-- (2.7,5.9);
\draw [line width=1.pt, color = Mon_orange] (2.7,5.9)-- (3.3,5.5);
\draw [line width=1.pt, color = Mon_orange] (3.3,5.5)-- (2.7,5.1);
\draw [line width=1.pt, color = Mon_orange] (2.7,5.1)-- (3.3,4.7);
\draw [line width=1.pt, color = Mon_orange] (3.3,4.7)-- (2.7,4.3);
\draw [line width=1.pt, color = Mon_orange] (2.7,4.3)-- (3.3,3.9);
\draw [line width=1.pt, color = Mon_orange] (3.3,3.9)-- (2.7,3.5);
\draw [line width=1.pt, color = Mon_orange] (2.7,3.5)-- (3.3,3.1);
\draw [line width=1.pt, color = Mon_orange] (3.3,3.1)-- (3,3);

\draw [line width=1.pt, color = Mon_orange] (7,9)-- (7.3,8.6);
\draw [line width=1.pt, color = Mon_orange] (7.3,8.6)-- (6.7,8.2);
\draw [line width=1.pt, color = Mon_orange] (6.7,8.2)-- (7.3,7.8);
\draw [line width=1.pt, color = Mon_orange] (7.3,7.8)-- (6.7,7.4);
\draw [line width=1.pt, color = Mon_orange] (6.7,7.4)-- (7,7);

\draw [line width=1.pt, color = Mon_orange] (6,7)-- (5.7,6.8);
\draw [line width=1.pt, color = Mon_orange] (5.7,6.8)-- (6.3,6.4);
\draw [line width=1.pt, color = Mon_orange] (6.3,6.4)-- (5.7,6);
\draw [line width=1.pt, color = Mon_orange] (5.7,6)-- (6.3,5.6);
\draw [line width=1.pt, color = Mon_orange] (6.3,5.6)-- (5.7,5.2);
\draw [line width=1.pt, color = Mon_orange] (5.7,5.2)-- (6.3,4.8);
\draw [line width=1.pt, color = Mon_orange] (6.3,4.8)-- (5.7,4.4);
\draw [line width=1.pt, color = Mon_orange] (5.7,4.4)-- (6.3,4);
\draw [line width=1.pt, color = Mon_orange] (6.3,4)-- (5.7,3.6);
\draw [line width=1.pt, color = Mon_orange] (5.7,3.6)-- (6.3,3.2);
\draw [line width=1.pt, color = Mon_orange] (6.3,3.2)-- (5.7,2.8);
\draw [line width=1.pt, color = Mon_orange] (5.7,2.8)-- (6.3,2.4);
\draw [line width=1.pt, color = Mon_orange] (6.3,2.4)-- (5.7,2);
\draw [line width=1.pt, color = Mon_orange] (5.7,2)-- (6.3,1.6);
\draw [line width=1.pt, color = Mon_orange] (6.3,1.6)-- (5.7,1.2);
\draw [line width=1.pt, color = Mon_orange] (5.7,1.2)-- (6,1);

\draw [line width=1.pt, color = Mon_orange] (8,7)-- (7.7,6.8);
\draw [line width=1.pt, color = Mon_orange] (7.7,6.8)-- (8.3,6.4);
\draw [line width=1.pt, color = Mon_orange] (8.3,6.4)-- (7.7,6);
\draw [line width=1.pt, color = Mon_orange] (7.7,6)-- (8.3,5.6);
\draw [line width=1.pt, color = Mon_orange] (8.3,5.6)-- (7.7,5.2);
\draw [line width=1.pt, color = Mon_orange] (7.7,5.2)-- (8.3,4.8);
\draw [line width=1.pt, color = Mon_orange] (8.3,4.8)-- (7.7,4.4);
\draw [line width=1.pt, color = Mon_orange] (7.7,4.4)-- (8.3,4);
\draw [line width=1.pt, color = Mon_orange] (8.3,4)-- (7.7,3.6);
\draw [line width=1.pt, color = Mon_orange] (7.7,3.6)-- (8.3,3.2);
\draw [line width=1.pt, color = Mon_orange] (8.3,3.2)-- (7.7,2.8);
\draw [line width=1.pt, color = Mon_orange] (7.7,2.8)-- (8.3,2.4);
\draw [line width=1.pt, color = Mon_orange] (8.3,2.4)-- (7.7,2);
\draw [line width=1.pt, color = Mon_orange] (7.7,2)-- (8.3,1.6);
\draw [line width=1.pt, color = Mon_orange] (8.3,1.6)-- (7.7,1.2);
\draw [line width=1.pt, color = Mon_orange] (7.7,1.2)-- (8,1);

\draw[color=black,decorate,decoration={brace,raise=0.2cm}]
(2,3) -- (2,1) node[above=0.25cm,pos=0.5,sloped] {};

\draw[color=black] (3,2) node {\tiny $w_{1}$};

\draw[color=black,decorate,decoration={brace,raise=0.2cm}]
(4,3) -- (4,1) node[above=0.25cm,pos=0.5,sloped] {};

\draw[color=black] (5,2) node {\tiny $w_{2}$};

\draw[color=black,decorate,decoration={brace,raise=0.2cm}]
(6,7) -- (6,1) node[above=0.25cm,pos=0.5,sloped] {};

\draw[color=black] (7,4) node {\tiny $w_{3}$};

\draw[color=black,decorate,decoration={brace,raise=0.2cm}]
(8,7) -- (8,1) node[above=0.25cm,pos=0.5,sloped] {};

\draw[color=black] (9,4) node {\tiny $w_{4}$};

\draw[color=black,decorate,decoration={brace,raise=0.2cm}]
(3,9) -- (3,3) node[above=0.25cm,pos=0.5,sloped] {};

\draw[color=black] (4,6) node {\tiny $w_{5}$};

\draw[color=black,decorate,decoration={brace,raise=0.2cm}]
(7,9) -- (7,7) node[above=0.25cm,pos=0.5,sloped] {};

\draw[color=black] (8,8) node {\tiny $w_{6}$};

\begin{scriptsize}
\draw[color=blue] (3,9.65) node {$x_{\sigma(a(1))}$};
\draw[color=blue] (7,9.65) node {$x_{\sigma(a(3))}$};

\draw [fill=blue] (2.,1.) circle (2.5pt);
\draw[color=blue] (2,0.45) node {$v_{1}$};
\draw [fill=blue] (4.,1.) circle (2.5pt);
\draw[color=blue] (4,0.45) node {$v_{2}$};
\draw [fill=blue] (6.,1.) circle (2.5pt);
\draw[color=blue] (6,0.45) node {$v_{3}$};
\draw [fill=blue] (8.,1.) circle (2.5pt);
\draw[color=blue] (8,0.45) node {$v_{4}$};
\draw [fill=blue] (7.,9.) circle (2.5pt);
\draw [fill=blue] (3.,9.) circle (2.5pt);
\draw [color=black] (1.,1.)-- ++(-2.5pt,0 pt) -- ++(5.0pt,0 pt) ++(-2.5pt,-2.5pt) -- ++(0 pt,5.0pt);
\draw[color=black] (0.25,1) node {$0$};
\draw [color=black] (1.,9.)-- ++(-2.5pt,0 pt) -- ++(5.0pt,0 pt) ++(-2.5pt,-2.5pt) -- ++(0 pt,5.0pt);
\draw[color=black] (0.5,9) node {$T$};
\end{scriptsize}
\end{tikzpicture}
\caption{Illustration of the centered \textsc{Moran} process where $\left| k_{4, T}\right| = 2$.}  
\label{Fig82}
\endminipage
\end{figure}

We illustrate this result with the Figures \ref{Fig81} and \ref{Fig82}. In Figure \ref{Fig81}, we can observe that {  $u_{1} = x_{\sigma(a(1))} + w_{1} + w_{5} + w_{7}$, $u_{2} = x_{\sigma(a(1))} + w_{2} + w_{5} + w_{7}$, $u_{3} = x_{\sigma(a(1))} + w_{3} + w_{6} + w_{7}$, $u_{4} = x_{\sigma(a(1))} + w_{4} + w_{6} + w_{7}$, so that
\begin{gather*}
    v_{1} = \frac{1}{2}\left[w_{5} - w_{6} \right] + \frac{3}{4}w_{1} - \frac{1}{4}\left[w_{2} + w_{3} + w_{4} \right],\qquad v_{2} = \frac{1}{2}\left[w_{5} - w_{6} \right] + \frac{3}{4}w_{2} - \frac{1}{4}\left[w_{1} + w_{3} + w_{4} \right] \\ v_{3} = -\frac{1}{2}\left[w_{5} - w_{6} \right] + \frac{3}{4}w_{3} - \frac{1}{4}\left[w_{1} + w_{2} + w_{4} \right],\qquad v_{4} = -\frac{1}{2}\left[w_{5} - w_{6} \right] + \frac{3}{4}w_{4} - \frac{1}{4}\left[w_{1} + w_{2} + w_{3} \right]
\end{gather*}}
and for Figure \ref{Fig82} that {  $u_{1} = x_{\sigma(a(1))} + w_{1} + w_{5}$, $u_{2} = x_{\sigma(a(1))} + w_{2} + w_{5}$, $u_{3} = x_{\sigma(a(3))} + w_{3} + w_{6}$, $u_{4} = x_{\sigma(a(3))} + w_{4} + w_{6}$ so that 
\begin{align*}
v_{1} & =  \frac{1}{2}\left[x_{\sigma(a(1))} - x_{\sigma(a(3))} + w_{5} - w_{6} \right] + \frac{3}{4}w_{1} - \frac{1}{4}\left[w_{2} + w_{3} + w_{4} \right], \\
 v_{2} & =  \frac{1}{2}\left[x_{\sigma(a(1))} - x_{\sigma(a(3))} + w_{5} - w_{6} \right] + \frac{3}{4}w_{2} - \frac{1}{4}\left[w_{1} + w_{3} + w_{4} \right], \\
 v_{3} & =  -\frac{1}{2}\left[x_{\sigma(a(1))} - x_{\sigma(a(3))} + w_{5} - w_{6} \right] + \frac{3}{4}w_{3} - \frac{1}{4}\left[w_{1} + w_{2} + w_{4} \right], \\
v_{4} & = -\frac{1}{2}\left[x_{\sigma(a(1))} - x_{\sigma(a(3))} + w_{5} - w_{6} \right] + \frac{3}{4}w_{4} - \frac{1}{4}\left[w_{1} + w_{2} + w_{3} \right].
\end{align*} }
We notice that, when there is just one ancestral lineage {  as in Figure \ref{Fig81}}, the random variable $\widehat{Z}^{N,\mu_{4}}_{T}$ {  does not depend on} the ancestral positions $x_1,\cdots,x_4${ : this is the centering effect}. In general, when $n = \left|k_{N,T} \right| = 1$, 
\begin{equation}
  v_{i}  = \int_{0}^{T}{\dd W_{s}^{\left(B\left(s^{-}, i \right)\right)}} - \frac{1}{N}\sum\limits_{j \, = \, 1}^{N}{\int_{0}^{T}{\dd W_{s}^{\left(B\left(s^{-}, j \right)\right)}}}.
\label{Ecriture_v_i}
\end{equation}
 {  This property is fundamental to implement coupling arguments leading to strong ergodicity.} 




\subsection{Coupling arguments with two distinct initial conditions and proof of Theorem \ref{Thm_Ergodicite_VT} \label{Ergodicite_Sect_4.4}} In this section, we want to couple centered \textsc{Moran}'s processes from different initial conditions but with the same \textsc{Kingman} genealogy and the same mutations in order to establish the following exponential ergodicity result. Then, we prove Theorem \ref{Thm_Ergodicite_VT}.
 
\begin{Prop} For all $\mu_{N}, \nu_{N} \in \MM_{1,N}^{c,2}({ \R^{d}})$, for all $T \geqslant 0$,  there exist constants $\alpha, \beta \in (0,+\infty)$, independent of $\mu_{N}, \nu_{N}, T$ and $N$ such that \[ \left\| \widehat{\P}\left(\widehat{Z}_{T}^{N, \mu_{N}} \in \cdot \right) - \widehat{\P}\left(\widehat{Z}_{T}^{N, \nu_{N}} \in \cdot \right) \right\|_{TV} \leqslant \alpha \exp\left(-\beta T \right).\]
In particular, for all $N \in \N^{\star}$ there exists a unique invariant probability measure $\pi_{N}$ for the centered \textsc{Moran} process $\left(Z_{t}^{N} \right)_{t\geqslant 0}$ such that for all $\mu_{N} \in \MM_{1,N}^{c,2}({ \R^{d}})$, for all $T \geqslant 0$, \[ \left\| \P_{\mu_{N}}\left(Z_{T}^{N} \in \cdot \right) -  \pi_{N}\right\|_{TV} \leqslant \alpha \exp\left(-\beta T \right).\]\label{Prop_2_Genealogie}
\end{Prop}
\begin{Rem}
The previous result is true for all deterministic initial conditions, so also for any random initial conditions.
\end{Rem}

\begin{proof}[Proof of Proposition \ref{Prop_2_Genealogie}] \textbf{Step 1. Coupling.} {  Let us denote $\mu_{N}:= \frac{1}{N}\sum_{i\, = \, 1}^{N}{\delta_{x_i}}$ and $\nu_{N} := \frac{1}{N}\sum_{i\, = \, 1}^{N}{\delta_{y_i}}$.
We construct the random variables $\widehat{Y}^{N, \mu_{N}}_{T}$ and $\widehat{Y}^{N, \nu_{N}}_{T}$ as in~\eqref{eq:couplage-Moran-Kingman} from the same realisations of $\left(k_{N,t} \right)_{0\leqslant t \leqslant T}$, $(W^{(B)})_{B\subset\{1,\ldots,N\}}$ and $\sigma$.}
This allows us to construct on the same probability space the two random variables \[\widehat{Z}_{T}^{N, \mu_{N}}:= \frac{1}{N}\sum\limits_{i\, = \, 1}^{N}{\delta_{v_{i}^{\mu_{N}}}} \qquad {\rm{and}} \qquad \widehat{Z}_{T}^{N, \nu_{N}} := \frac{1}{N}\sum\limits_{i\, = \, 1}^{N}{\delta_{v_{i}^{\nu_{N}}}}\] such that $Z^{N,\mu_{N}}_{T} \overset{{\rm{law}}}{=} \widehat{Z}^{N, \mu_{N}}_{T}$ and $Z^{N, \nu_{N}}_{T} \overset{{\rm{law}}}{=} \widehat{Z}^{N, \nu_{N}}_{T}$. \\

\textbf{Step 2. Control in total variation.} From (\ref{Ecriture_v_i}), on the event $\left\{ \left|k_{N, T}\right| = 1 \right\}$, we have that for all $i \in \left\{1, \cdots, N \right\}$, $v_{i}^{\mu_{N}} = v_{i}^{\nu_{N}}$ a.s. and from  {\color{blue} \cite{Lindvall}} we deduce that 
\begin{align*}
\left\|\widehat{\P}\left(\widehat{Z}_{T}^{N, \mu_{N}} \in \cdot \right) - \widehat{\P}\left(\widehat{Z}_{T}^{N, \nu_{N}} \in \cdot \right) \right\|_{TV}&  \leqslant \widehat{\P}\left(\widehat{Z}_{T}^{N, \mu_{N}} \neq \widehat{Z}_{T}^{N, \nu_{N}} \right) \\
& = 1 - \K_{N,T}\left(\left|k_{N,T} \right| = 1 \right)
\end{align*}
{ where we denote by $\K_{N,T}$ the law of the \textsc{Kingman} $N-$coalescent with coalescence rate $2\gamma$ on $[0,T]$.} We denote by $H_{N} := \sum_{k\, = \, 2}^{N}{T_{k}}$ the height of the \textsc{Kingman} $N-$coalescent where $\left(T_{k} \right)_{2\leqslant k \leqslant N}$ are independent random variables such that $T_{k}$ follows an exponential law of parameter ${ 2\gamma} \binom{k}{2}$ {\color{blue} \cite[\color{black} Lemma 2.20]{Etheridge}}. Now, $\K_{N,T} \left(\left|k_{N,T} \right| = 1\right) \geqslant \K_{\infty,T} \left(\left|k_{\infty,T} \right| = 1\right) = \K_{\infty,T} \left(H_{\infty} \leqslant T \right)$ and by the  exponential \textsc{Tchebychev} inequality we have \[\K_{\infty,T} \left(H_{\infty} > T \right) \leqslant \inf_{\lambda \, \in \, (0,{ 2\gamma})}{\frac{\E\left(\exp\left(\lambda H_{\infty} \right) \right)}{\exp\left(\lambda T\right)}}.\]
Note that for all $ \lambda \in \, (0,{ 2\gamma})$, 
\begin{equation}
\E\left(\exp\left(\lambda H_{\infty} \right) \right) = \prod\limits_{k\, = \, 2}^{+\infty}{\E\left(\exp\left(\lambda T_{k} \right) \right)} = \frac{{ 2\gamma}}{{ 2\gamma}-\lambda}\prod\limits_{k\, = \, 3}^{+\infty}{\frac{1}{1 - \frac{\lambda}{\gamma k(k-1)}}},
    \label{Eq_exponential_Tchebychev_inequality}
\end{equation}
where the last product is convergent. We deduce that \[\K_{\infty,T} \left(H_{\infty} > T \right) \leqslant C \inf_{\lambda \, \in \, (0,{ 2\gamma})}{\frac{1}{({ 2\gamma}-\lambda)\exp(\lambda T)}} = { 2}C\gamma\exp(1)T\exp(-{ 2\gamma} T),\]
where $C := \prod_{k\, = \, 3}^{+\infty}{\frac{1}{1-\frac{1}{\gamma k(k-1)}}}$. The result follows for $\alpha := { 2}C\gamma\exp(1)T$ and $\beta := { 2\gamma}$. \qedhere
\end{proof}

\begin{proof}[Proof of Theorem \ref{Thm_Ergodicite_VT}]
Classically, it is sufficient to check that there exists constants $\alpha, \beta \in \R_{+}$ such that for all $\mu, \nu \in \MM_{1}^{c,2}({ \R^{d}})$, for all $T \geqslant 0$,   \[ \left\| \P_{\mu}\left(Z_{T} \in \cdot \right) - \P_{\nu}\left(Z_{T} \in \cdot \right) \right\|_{TV} \leqslant \alpha \exp\left(-\beta T \right).\]
From \textsc{Lusin}'s theorem {\color{blue} \cite[\color{black} Corollary of  Theorem 2.24]{rudin_real_2013}}, Proposition \ref{Prop_2_Genealogie}  and Corollary \ref{Corollaire_1_Genealogie} there exists two constants $\alpha, \beta \in (0, + \infty)$ such that for all $\mu_{N}, \nu_{N} \in \MM_{1,N}^{c,2}({ \R^{d}})$, for all $T \geqslant 0$, 
\begin{align*}
\sup_{\substack{\left\|f \right\|_{\infty} \leqslant 1 \\ f : { \R^{d} \to \R} \, {\rm continuous}}}{\left|\E_{\mu_{N}}\left(f\left(Z_{T}^{N} \right) \right) - \E_{\nu_{N}}\left(f\left(Z_{T}^{N} \right) \right) \right|} &= \sup_{\left\|f \right\|_{\infty} \leqslant 1}{\left|\E_{\mu_{N}}\left(f\left(Z_{T}^{N} \right) \right) - \E_{\nu_{N}}\left(f\left(Z_{T}^{N} \right) \right) \right|}  \\
& \leqslant \alpha \exp(-\beta T).  
\end{align*}
Now, let be fixed two deterministic initial conditions $\mu, \nu \in \MM_{1}^{c,2}({ \R^{d}})$ and consider an i.i.d. sample $\left(X_{i} \right)_{1\leqslant i \leqslant N}$ of distribution $\mu$ and an i.i.d. sample $\left(\widetilde{X}_{i} \right)_{1\leqslant i \leqslant N}$ of distribution $\nu$. Then, we construct two initial conditions $\mu_{N} := \frac{1}{N}\sum_{i\, = \, 1}^{N}{\delta_{X_{i}}}$ and $\nu_{N} := \frac{1}{N}\sum_{i\, = \, 1}^{N}{\delta_{\widetilde{X}_{i}}}$ such that $\mu_{N}$ and $\nu_{N}$ converge in law respectively to $\mu$ and $\nu$. We define $\widetilde{\mu}_{N} := \tau_{-\left\langle \id, \mu_{N} \right\rangle}  \sharp \, \mu_{N}$ and $\widetilde{\nu}_{N} := \tau_{-\left\langle \id, \nu_{N} \right\rangle}  \sharp \, \nu_{N}$ such that $\widetilde{\mu}_{N}, \widetilde{\nu}_{N} \in \MM_{1,N}^{c,2}({ \R^{d}})$. By construction, the assumptions of exchangeability of the random variables $\left(X_{i} \right)_{1\leqslant i \leqslant N}$ and $\left(\widetilde{X}_{i} \right)_{1\leqslant i \leqslant N}$ are satisfied, $\widetilde{\mu}_{N}$ and $\widetilde{\nu}_{N}$ converge in law respectively to $\mu$ and $\nu$ and we have \[\E\left(\left\langle \id^{2}, \widetilde{\mu}_{N} \right\rangle \right) = \E\left(\left\langle \id^{2}, \mu_{N} \right\rangle - \left\langle \id, \mu_{N} \right\rangle^{2} \right) = \left(1 - \frac{1}{N} \right){ \E\left(X_{1}^{2}\right)} . \] 
Then, we deduce from Proposition \ref{Prop_0_Genealogie} that for all $f \in \CCCC_{b}({  \R^{d}}, \R)$ satisfying $\left\|f \right\|_{\infty} \leqslant 1$, for all $T \geqslant 0$, 
\[\left|\E_{\mu} f\left(Z_{T} \right) - \E_{\nu} f\left(Z_{T} \right) \right| = \lim_{N\to + \infty}{\left|\E_{\mu_{N}}\left(f\left(Z_{T}^{N} \right) \right) - \E_{\nu_{N}}\left(f\left(Z_{T}^{N} \right) \right) \right| \leqslant \alpha\exp(-\beta T)} \]
which concludes the proof. \qedhere 
\end{proof}

\subsection{Characterisation of the invariant probability measure $\pi$\label{Ergodicite_Sect_4.6}}

{ The main result of this section is Theorem \ref{Convergence_Z_infini_check} of Section \ref{Look-down_2}, which gives a characterisation of the invariant measure $\pi$}  thanks to an adaptation of \textsc{Donnelly-Kurtz}'s modified look-down construction. Let us begin by giving a convergence result of the invariant probability measure $\pi_{N}$ to the invariant probability measure~$\pi$. 

\begin{Lem}
The sequence of laws $\left(\pi_{N} \right)_{N \in \N^{\star}}$ converges in law to $\pi$ in $\MM_{1}\left(\MM_{1}({ \R^{d}}) \right)$.
\label{Lemme_1_Genealogie}
\end{Lem}

\begin{proof}  Let $T \geqslant 0$, $\mu_{N} \in \MM_{1,N}^{c,2}({ \R^{d}})$ and $\mu \in \MM_{1}^{c,2}({ \R^{d}})$ such that $\mu_{N}$ converges in law to $\mu$. From Proposition \ref{Prop_2_Genealogie} and Theorem \ref{Thm_Ergodicite_VT}, we have for all $f \in \CCCC_{b}(\MM_{1}({ \R^{d}}), \R)$,
\begin{align*}
\left|\left\langle f, \pi_{N} \right\rangle - \left\langle f, \pi \right\rangle \right| & \leqslant \left|\left\langle f, \pi_{N} \right\rangle - \E_{\mu_{N}}\left(f\left(\widehat{Z}_{T}^{N, \mu_{N}} \right) \right) \right| + \left|\E_{\mu_{N}}\left(f\left(\widehat{Z}_{T}^{N, \mu_{N}} \right) \right) - \E_{\mu}\left(f\left( Z_{T}\right) \right) \right| \\
& \qquad  + \left|\E_{\mu}\left(f\left( Z_{T}\right) \right) - \left\langle f, \pi \right\rangle \right| \\
& \leqslant 2 \left\|f \right\|_{\infty} \alpha \exp\left(-\beta T \right) + \left|\E_{\mu_{N}}\left(f\left(\widehat{Z}_{T}^{N, \mu_{N}} \right) \right) - \E_{\mu}\left(f\left( Z_{T}\right) \right) \right|.
\end{align*}
The announced result follows from Proposition \ref{Prop_0_Genealogie}. \qedhere
\end{proof}

{ \subsubsection{Characterisation of $\pi_{N}$ from the modified look-down construction \label{Look-down_1}}}
We consider the probability space $\left(\check{\Omega}, \check{\FF}, \check{\P} \right)$ where we define the modified look-down process on $(-\infty, 0]$ as a population dynamics on the set $\N$ of levels where 
one individual is assigned to each level. To each pair of levels $(i,j) \in \N^{2}$ with $1 \leqslant i < j$, we assign an independent \textsc{Poisson} processes $\left(N_{ij}(t)\right)_{t\geqslant 0}$ with intensity { $2\gamma$} and to each level $i \in \N^{\star}$, we assign an independent standard Brownian motion $\left(B_{i}(t) \right)_{t\leqslant 0}$ on $\R_{-}$. Jointly with the modified look-down is constructed for all $N \in \N^{\star}$, the so-called $N-$look-down process whose evolution is given as follows:

\begin{figure}[!h]
\centering		
\begin{minipage}[h]{.515\textwidth} \centering
\centering		
\definecolor{ccqqqq}{rgb}{0.8,0.,0.} 
 \definecolor{ffqqff}{rgb}{1.,0.,1.} 
\definecolor{ffxfqq}{rgb}{0.490,0.490,1} 
\definecolor{ttzzqq}{rgb}{0.294,0.,0.51} 
\definecolor{wwqqzz}{rgb}{0.4,0.,0.6} 
\definecolor{qqqqff}{rgb}{0.2,0.6,0.} 
\definecolor{qqwuqq}{rgb}{0.,0.39215686274509803,0.}
\begin{tikzpicture}[line cap=round,line join=round,>=triangle 45,x=0.575cm,y=0.575cm]
\clip(-3.25,-0.75) rectangle (12,5.25);
\draw [>=stealth,->,line width=2.5pt, color = gray] (1.,1.25) -- (1.,2.);
\draw [shift={(0.75,1.25)},line width=2.5pt, color = gray]  plot[domain=-1.5715003927507514:0.00830175753034396,variable=\t]({1.*0.25000006196362157*cos(\t r)+0.*0.25000006196362157*sin(\t r)},{0.*0.25000006196362157*cos(\t r)+1.*0.25000006196362157*sin(\t r)});

\draw [shift={(0.75,2.25)},line width=1pt, color = gray]  plot[domain=-1.5715003927507514:0.00830175753034396,variable=\t]({1.*0.25000006196362157*cos(\t r)+0.*0.25000006196362157*sin(\t r)},{0.*0.25000006196362157*cos(\t r)+1.*0.25000006196362157*sin(\t r)});

\draw [shift={(0.75,3.25)},line width=1pt, color = gray]  plot[domain=-1.5715003927507514:0.00830175753034396,variable=\t]({1.*0.25000006196362157*cos(\t r)+0.*0.25000006196362157*sin(\t r)},{0.*0.25000006196362157*cos(\t r)+1.*0.25000006196362157*sin(\t r)});

\draw [shift={(0.75,4.25)},line width=1pt, color = gray]  plot[domain=-1.5715003927507514:0.00830175753034396,variable=\t]({1.*0.25000006196362157*cos(\t r)+0.*0.25000006196362157*sin(\t r)},{0.*0.25000006196362157*cos(\t r)+1.*0.25000006196362157*sin(\t r)});

\draw [shift={(1.25,2.75)},line width=1pt, color = gray]  plot[domain=1.5766858572921507:3.141592653589793,variable=\t]({1.*0.2500043358838501*cos(\t r)+0.*0.2500043358838501*sin(\t r)},{0.*0.2500043358838501*cos(\t r)+1.*0.2500043358838501*sin(\t r)});

\draw [shift={(1.25,3.75)},line width=1pt, color = gray]  plot[domain=1.5766858572921507:3.141592653589793,variable=\t]({1.*0.2500043358838501*cos(\t r)+0.*0.2500043358838501*sin(\t r)},{0.*0.2500043358838501*cos(\t r)+1.*0.2500043358838501*sin(\t r)});

\draw [shift={(1.25,4.75)},line width=1pt, color = gray]  plot[domain=1.5766858572921507:3.141592653589793,variable=\t]({1.*0.2500043358838501*cos(\t r)+0.*0.2500043358838501*sin(\t r)},{0.*0.2500043358838501*cos(\t r)+1.*0.2500043358838501*sin(\t r)});

\draw [line width=1pt, color = gray] (1.,2.25)-- (1.,2.75);
\draw [line width=1pt, color = gray] (1.,3.25)-- (1.,3.75);
\draw [line width=1pt, color = gray] (1.,4.25)-- (1.,4.75);

\draw [line width=2.5pt, color = gray] (1,2)-- (2.75,2);
\draw [line width=1pt, color = gray] (2.75,2)-- (11,2);

\draw [shift={(2.75,3.25)},line width=1pt, color = Mon_orange]  plot[domain=-1.5715003927507514:0.00830175753034396,variable=\t]({1.*0.25000006196362157*cos(\t r)+0.*0.25000006196362157*sin(\t r)},{0.*0.25000006196362157*cos(\t r)+1.*0.25000006196362157*sin(\t r)});

\draw [>=stealth,->,line width=2.5pt, color = Mon_orange] (3.,2.25) -- (3.,3.);

\draw [shift={(2.75,2.25)},line width=2.5pt, color = Mon_orange]  plot[domain=-1.5715003927507514:0.00830175753034396,variable=\t]({1.*0.25000006196362157*cos(\t r)+0.*0.25000006196362157*sin(\t r)},{0.*0.25000006196362157*cos(\t r)+1.*0.25000006196362157*sin(\t r)});
\draw [line width=1pt, color = ffxfqq] (-0.5,2.)-- (1,2.);

\draw [line width=2.5pt, color = Mon_orange] (3,3.)-- (5.75,3.);
\draw [line width=1pt, color = Mon_orange] (5.75,3.)-- (6,3.);

\draw [line width=1pt, color = ffxfqq] (-0.5,1.)-- (0.75,1.);
\draw [line width=1pt, color = ffxfqq] (0.75,1.)-- (11.,1.);

\draw [shift={(3.25,3.75)},line width=1pt, color = Mon_orange]  plot[domain=1.5766858572921507:3.141592653589793,variable=\t]({1.*0.2500043358838501*cos(\t r)+0.*0.2500043358838501*sin(\t r)},{0.*0.2500043358838501*cos(\t r)+1.*0.2500043358838501*sin(\t r)});
\draw [line width=1pt, color = Mon_orange] (3.,3.25)-- (3.,3.75);
\draw [line width=1pt, color = ffxfqq] (-0.5,3.)-- (1.25,3.);
\draw [line width=1pt, color = gray] (1.25,3.)-- (3,3.);

\draw [line width=1pt, color = Mon_orange] (3.25,4.)-- (5.75,4.);
\draw [line width=1pt, color = Mon_orange] (5.75,4.)-- (6.25,4.);

\draw [shift={(2.75,4.25)},line width=1pt, color = Mon_orange]  plot[domain=-1.5515546232470658:0.,variable=\t]({1.*0.250046287535086*cos(\t r)+0.*0.250046287535086*sin(\t r)},{0.*0.250046287535086*cos(\t r)+1.*0.250046287535086*sin(\t r)});
\draw [shift={(3.25,4.75)},line width=1pt, color = Mon_orange]  plot[domain=1.5707963267948966:3.141592653589793,variable=\t]({1.*0.25*cos(\t r)+0.*0.25*sin(\t r)},{0.*0.25*cos(\t r)+1.*0.25*sin(\t r)});

\draw [line width=1pt, color = ffxfqq] (-0.5,4.)-- (1.25,4.);
\draw [line width=1pt, color = gray] (1.25,4.)-- (3.25,4.);

\draw [line width=1pt, color = Mon_orange] (3.,4.25)-- (3.,4.75);
\draw [line width=1pt, color = Mon_orange] (3.25,5.)-- (6.25,5.);

\draw [>=stealth,->,line width=1.5pt, color = qqqqff] (6.,1.25) -- (6.,3);
\draw [shift={(5.75,1.25)},line width=1.5pt, color = qqqqff]  plot[domain=-1.5707963267948966:0.,variable=\t]({1.*0.25*cos(\t r)+0.*0.25*sin(\t r)},{0.*0.25*cos(\t r)+1.*0.25*sin(\t r)});
\draw [shift={(5.75,3.25)},line width=2.5pt, color = qqqqff]  plot[domain=-1.5707963267948966:0.,variable=\t]({1.*0.25*cos(\t r)+0.*0.25*sin(\t r)},{0.*0.25*cos(\t r)+1.*0.25*sin(\t r)});
\draw [shift={(6.25,3.75)},line width=2.5pt, color = qqqqff]  plot[domain=1.5707963267948966:3.141592653589793,variable=\t]({1.*0.25*cos(\t r)+0.*0.25*sin(\t r)},{0.*0.25*cos(\t r)+1.*0.25*sin(\t r)});
\draw [shift={(5.75,4.25)},line width=1pt, color = qqqqff]  plot[domain=-1.5707963267948966:0.,variable=\t]({1.*0.25*cos(\t r)+0.*0.25*sin(\t r)},{0.*0.25*cos(\t r)+1.*0.25*sin(\t r)});
\draw [shift={(6.25,4.75)},line width=1pt, color = qqqqff]  plot[domain=1.5707963267948966:3.141592653589793,variable=\t]({1.*0.25*cos(\t r)+0.*0.25*sin(\t r)},{0.*0.25*cos(\t r)+1.*0.25*sin(\t r)});

\draw [line width=1pt, color = qqqqff] (6,3.)-- (11.,3.);
\draw [line width=2.5pt, color = qqqqff] (6.,3.25)-- (6.,3.75);
\draw [line width=2.5pt, color = qqqqff] (6.25,4.)-- (7.75,4.);
\draw [line width=1pt, color = qqqqff] (7.75,4.)-- (8,4.);
\draw [line width=1pt, color = qqqqff] (6.,4.25)-- (6.,4.75);
\draw [line width=1pt, color = qqqqff] (6.25,5.)-- (7.75,5.);
\draw [line width=1pt, color = qqqqff] (7.75,5.)-- (8.25,5.);

\draw [>=stealth,->,line width=1.5pt, color = ccqqqq] (8.,3.25) -- (8.,4.);
\draw [shift={(7.75,3.25)},line width=1.5pt, color = ccqqqq ]  plot[domain=-1.5707963267948966:0.,variable=\t]({1.*0.25*cos(\t r)+0.*0.25*sin(\t r)},{0.*0.25*cos(\t r)+1.*0.25*sin(\t r)});

\draw [shift={(7.75,4.25)},line width=2.5pt, color = ccqqqq]  plot[domain=-1.5707963267948966:0.,variable=\t]({1.*0.25*cos(\t r)+0.*0.25*sin(\t r)},{0.*0.25*cos(\t r)+1.*0.25*sin(\t r)});
\draw [shift={(8.25,4.75)},line width=2.5pt, color = ccqqqq]  plot[domain=1.5707963267948966:3.141592653589793,variable=\t]({1.*0.25*cos(\t r)+0.*0.25*sin(\t r)},{0.*0.25*cos(\t r)+1.*0.25*sin(\t r)});

\draw [line width=1pt, color = ccqqqq] (8,4.)-- (11.,4.);
\draw [line width=2.5pt, color = ccqqqq] (8.,4.25)-- (8.,4.75);
\draw [line width=2.5pt, color = ccqqqq] (8.25,5.)-- (11.,5.);

\draw [line width=1pt, color = ffxfqq] (-0.5,5.)-- (1.25,5.);
\draw [line width=1pt, color = gray] (1.25,5.)-- (3.25,5.);

\draw [>=stealth,->,line width=1pt] (-1.,0.) -- (12,0.);

\begin{scriptsize}

\draw (-2.5,1) node {$B_{1}(t)$}; 
\draw (-2.5,2) node {$B_{2}\left(t\right)$}; 
\draw (-2.5,3) node {$B_{3}\left(t\right)$}; 
\draw (-2.5,4) node {$B_{4}\left(t\right)$}; 
\draw (-2.5,5) node {$B_{5}\left(t\right)$}; 

\draw[color=black] (11.5,1) node {$\check{u}_{1}$};
\draw[color=black] (11.5,2) node {$\check{u}_{2}$};
\draw[color=black] (11.5,3) node {$\check{u}_{3}$};
\draw[color=black] (11.5,4) node {$\check{u}_{4}$};
\draw[color=black] (11.5,5) node {$\check{u}_{5}$};

\draw[color=black] (-1,1) node {$1$};
\draw[color=black] (-1,2) node {$2$};
\draw[color=black] (-1,3) node {$3$};
\draw[color=black] (-1,4) node {$4$};
\draw[color=black] (-1,5) node {$5$};

\draw [color=black] (11,0)-- ++(0pt,0 pt) -- ++(0pt,0 pt) ++(0pt,-2.5pt) -- ++(0 pt,5.0pt);
\draw[color=black] (11,-0.5) node {$0$};
\draw [color=ccqqqq] (8,0)-- ++(0pt,0 pt) -- ++(0pt,0 pt) ++(0pt,-2.5pt) -- ++(0 pt,5.0pt);
\draw[color=ccqqqq] (8,-0.5) node {$-t_{1}$};
\draw [color=qqqqff] (6,0)-- ++(0pt,0 pt) -- ++(0pt,0 pt) ++(0pt,-2.5pt) -- ++(0 pt,5.0pt);
\draw[color=qqqqff] (6,-0.5) node {$-t_{2}$};
\draw [color=Mon_orange] (3,0)-- ++(0pt,0 pt) -- ++(0pt,0 pt) ++(0pt,-2.5pt) -- ++(0 pt,5.0pt);
\draw[color=Mon_orange] (3,-0.5) node {$-t_{3}$};
\draw [color=gray] (1,0)-- ++(0pt,0 pt) -- ++(0pt,0 pt) ++(0pt,-2.5pt) -- ++(0 pt,5.0pt);
\draw[color=gray] (1,-0.5) node {$-t_{4}$};
\draw [color=black] (-0.5,0)-- ++(0pt,0 pt) -- ++(0pt,0 pt) ++(0pt,-2.5pt) -- ++(0 pt,5.0pt);
\draw[color=black] (-0.5,-0.5) node {$-T$};
\end{scriptsize}
\end{tikzpicture}
\caption{Graphical representation of the modified look-down process with $N = 5$.}
\label{Fig_Modified_Look_Down}
\end{minipage}
\hspace{1.3cm}
\begin{minipage}[h]{.375\textwidth}\centering
\definecolor{ccqqqq}{rgb}{0.8,0.,0.} 
\definecolor{wwqqzz}{rgb}{0.4,0.,0.6} 
\definecolor{qqwuqq}{rgb}{0.,0.39215686274509803,0.}
\definecolor{ffqqff}{rgb}{1.,0.,1.} 
\definecolor{ffxfqq}{rgb}{0.490,0.490,1} 
\definecolor{ttzzqq}{rgb}{0.294,0.,0.51} 
\definecolor{qqqqff}{rgb}{0.2,0.6,0.} 
\begin{tikzpicture}[line cap=round,line join=round,>=triangle 45,x=0.42cm,y=0.42cm]
\clip(-4.45,-0.25) rectangle (9.5,13.35);
\draw [>=stealth,<-, line width=1pt] (0,0.25) -- (0,13.25);

\draw [line width=1pt,color=ffxfqq] (1,1)-- (1,6);
\draw [line width=1pt,color=qqqqff] (3,1)-- (3,4);
\draw [line width=1pt,color=ccqqqq] (5,1)-- (5,4);
\draw [line width=1pt,color=gray] (7,1)-- (7,9);
\draw [line width=2.5pt,color=ccqqqq] (9,1)-- (9,4);
\draw [line width=2.5pt,color=qqqqff] (9,4)-- (9,6);
\draw [line width=2.5pt,color=Mon_orange] (9,6)-- (9,9);
\draw [line width=1pt,color=ffxfqq] (2.5,6)-- (2.5,11);
\draw [line width=1pt,color=qqqqff] (4,4)-- (4,6);
\draw [line width=2.5pt,color=gray] (8,9)-- (8,11);
\draw [line width=1pt,color=ffxfqq] (5.25,11)-- (5.25,12.5);

\draw [line width=1pt,color=qqqqff] (3,4)-- (4,4);
\draw [line width=1pt,color=ccqqqq] (4,4)-- (5,4);

\draw [line width=1pt,color=qqqqff] (2.5,6)-- (4,6);
\draw [line width=1pt,color=ffxfqq] (1,6)-- (2.5,6);

\draw [line width=1pt,color=ffxfqq] (2.5,11)-- (5.25,11);
\draw [line width=2.5pt,color=gray] (5.25,11)-- (8,11);

\draw [line width=2.5pt,color=Mon_orange] (8,9)-- (9,9);
\draw [line width=1pt,color=gray] (7,9)-- (8,9);

\begin{scriptsize}
\draw [fill=blue] (1,1) circle (2.5pt);
\draw[color=blue] (1,0.35) node {$\check{u}_{1}$};
\draw [fill=blue] (3,1) circle (2.5pt);
\draw[color=blue] (3,0.35) node {$\check{u}_{3}$};
\draw [fill=blue] (5,1) circle (2.5pt);
\draw[color=blue] (5,0.35) node {$\check{u}_{4}$};
\draw [fill=blue] (7,1) circle (2.5pt);
\draw[color=blue] (7,0.35) node {$\check{u}_{2}$};
\draw [fill=blue] (9,1) circle (2.5pt);
\draw[color=blue] (9,0.35) node {$\check{u}_{5}$};
\draw [color=black] (0,1)-- ++(-2.5pt,0 pt) -- ++(5.0pt,0 pt) ++(-2.5pt,-0pt) -- ++(0 pt,0pt);
\draw[color=black] (-0.9,1) node {$0$};
\draw [color=ccqqqq] (0,4)-- ++(-2.5pt,0 pt) -- ++(5.0pt,0 pt) ++(-2.5pt,-0pt) -- ++(0 pt,0pt);
\draw[color=ccqqqq] (-0.9,4) node {$-t_{1}$};
\draw [color=qqqqff] (0,6)-- ++(-2.5pt,0 pt) -- ++(5.0pt,0 pt) ++(-2.5pt,-0pt) -- ++(0 pt,0pt);
\draw[color=qqqqff] (-0.9,6) node {$-t_{2}$};
\draw [color=Mon_orange] (0,9)-- ++(-2.5pt,0 pt) -- ++(5.0pt,0 pt) ++(-2.5pt,-0pt) -- ++(0 pt,0pt);
\draw[color=Mon_orange] (-0.9,9) node {$-t_{3}$};
\draw [color=gray] (0,11)-- ++(-2.5pt,0 pt) -- ++(5.0pt,0 pt) ++(-2.5pt,-0pt) -- ++(0 pt,0pt);
\draw[color=gray] (-2.45,11) node {$-\check{T}_{coal}^{5} = -t_{4}$};
\draw [color=black] (0,12.5)-- ++(-2.5pt,0 pt) -- ++(5.0pt,0 pt) ++(-2.5pt,-0pt) -- ++(0 pt,0pt);
\draw[color=black] (-0.9,12.5) node {$-T$};
\end{scriptsize}
\end{tikzpicture}
\caption{\textsc{Kingman}'s genealogy $\left(\check{k}_{5, t}\right)_{0\leqslant t \leqslant T}$ under the modified look-down model on the left, tracing back from time $0$ to time $-T$.}
\label{Fig_Look-down_genealogy}
\end{minipage}
\end{figure}

\begin{itemize}
\item[\textbf{(1)}] \textbf{Birth/Death rule.} Each jump time $t_{k}$ of one of the \textsc{Poisson} process $\left(N_{ij} \right)_{1\leqslant i < j \leqslant N}$ corresponds to a reproduction event at backward time $-t_{k}$.  When the time $t_{k}$ is the jump time of the \textsc{Poisson} process $N_{ij}$, we put an arrow from $i$ to $j$ as illustrated in Figure \ref{Fig_Modified_Look_Down} which means that the individual at level $i$ puts a child at level $j$. The offspring at level $j$ adopts the current spatial position of its parent at level $i$. The parent level and position do not change. Individuals previously at level $\ell \in \{j, \cdots, N-1 \}$ are shifted one level up to $\ell + 1$ and the individual at level $N$ dies.
\item[\textbf{(2)}] \textbf{Spatial motion.} Between reproduction events, individuals' spatial positions at each level $i$ evolve according to the standard {  $d$-dimensional} Brownian motion $B_{i}(-t)$. As explain below, we will fix the position of the individual at level $1$ at coalescence time to $0$.
\end{itemize}
Note that the $N-$modified look-down process is simply the first $N$ levels of the $(N+k)-$modified look-down for any $k \in \N^{\star}$. In other words, the modified look-down construction can be done with an infinite population as a projective limit of the so-called (infinite) modified look-down. From {\color{blue} \cite{donnelly_countable_1996, donnelly_particle_1999}}, the genealogy $\left(\check{k}_{N,t} \right)_{t\geqslant 0}$ in backward time since time $0$ of a sample from a population evolving according to the $N-$modified look-down is exactly determined by \textsc{Kingman}'s $N-$coalescent with coalescence rate { $2\gamma$}. In Figure \ref{Fig_Look-down_genealogy} we give the \textsc{Kingman} genealogy associated to the $5-$modified look-down  of the Figure \ref{Fig_Modified_Look_Down}. \\ 

We denote by $a(i,t)$, $i \in \left\{1, \cdots, N \right\}$, $t \in (-\infty, 0]$,
 the ancestor level of the individual at level $i$ at time $t$. For example, in Figure \ref{Fig_Modified_Look_Down}, for all $t \in \, ]-t_{3}, -t_{4}]$, $a(5,t) = 2$ and for all $t \in \, ]-t_{2}, -t_{4}]$, $a(3,t) = 1$. Let us consider the random variables
\begin{align*}
\check{T}_{coal}^{N} & := \inf{\left\{ T \geqslant 0\left. \phantom{1^{1^{1^{1}}}} \hspace{-0.6cm} \right| a(i, -T) = 1, \ \forall i \in \left\{1, \cdots, N \right\} \right\}}, \\
 \check{T}_{coal}^{\infty} & := \inf{\left\{ T \geqslant 0\left. \phantom{1^{1^{1^{1}}}} \hspace{-0.6cm} \right| a(i, -T) = 1, \ \forall i \in \N^{\star} \right\}},
\end{align*}
which can be interpreted respectively as the coalescence time (i.e. the first time where $|\check{k}_{N,t}|= 1$) of the \textsc{Kingman} $N-$coalescent $\left(\check{k}_{N,t} \right)_{t\geqslant 0}$ and 
the \textsc{Kingman} coalescent $\left(\check{k}_{\infty,t} \right)_{t\geqslant 0}$. 
Note that, for all $N \in \N^{\star}$, $\check{T}_{coal}^{N} \leqslant \check{T}_{coal}^{\infty}$ $\check{\P}-$a.s. {  In the next proposition}, we establish that $\check{T}^{\infty}_{coal}$ admits moments of any order. 
{ \begin{Prop} For all $k \in \N$, there exists a constant $C_{k}>0$ such that $\E\left(\left(\check{T}_{coal}^{\infty} \right)^{k} \right) \leqslant C_{k}$.
\label{Prop_Genealogie_Moments}
\end{Prop}
\begin{proof}
Note that $\check{T}^{\infty}_{coal} = \sum_{k \, = \, 2}^{+\infty}{T_{k}}$ where $\left(T_{k} \right)_{k \geqslant 2}$ are independent random variables such that $T_{k}$ follows an exponential law of parameter ${ 2\gamma} \binom{k}{2}$. Hence the result follows from (\ref{Eq_exponential_Tchebychev_inequality}). \qedhere
\end{proof}}

We shall be interested in the spatial position $\check{u}_{i}$ of the individual at level $i \in \N^{\star}$ at time $0$ assuming that the position of its ancestor at backward time $-\check{T}_{coal}^{\infty}$ is $0$. For example, if we assume that, in Figure \ref{Fig_Look-down_genealogy}, $\check{T}_{coal}^{\infty} = \check{T}_{coal}^{5} = t_{4}$,  then the spatial position of the individual at level $5$ at backward time $-\check{T}_{coal}^{\infty}$, represented by the curve in bold in Figures \ref{Fig_Modified_Look_Down} and \ref{Fig_Look-down_genealogy}, is 
\[\check{u}_{5} := B_{2}(-t_{4}) - B_{2}(-t_{3}) + B_{3}(-t_{3}) - B_{3}(-t_{2}) + B_{4}(-t_{2}) - B_{4}(-t_{1}) + B_{5}(-t_{1}). \] 
 Similarly, $\check{u}_{1} := B_{1}(-t_{4}), \check{u}_{2} := B_{2}\left(-t_{4} \right),  \check{u}_{3} := B_{1}(-t_{4}) - B_{1}(-t_{2}) + B_{3}\left(-t_{2} \right),  \check{u}_{4} := B_{1}\left(-t_{4} \right) - B_{1}\left(-t_{2} \right) + B_{3}\left(-t_{2} \right) - B_{3}\left(-t_{1} \right) + B_{4}\left(-t_{1} \right)$. In general, we define for all $i \in \N^{\star}$, the random variable \begin{equation}
\check{u}_{i} :=  \int_{-\check{T}_{coal}^{\infty}}^{0}{\dd B_{a\left(i, t\right)}(t)}. 
\label{Def_u_i_check}
\end{equation} 

{  In view of (\ref{Ecriture_v_i}), it is natural to introduce for all $N, i \in \N^{\star}$, 
\begin{equation}
\check{v}_{i}^{N} := \check{u}_{i} - \frac{1}{N}\sum\limits_{j\, = \, 1}^{N}{\check{u}_{j}}.
    \label{Ecriture_v_i_check}
\end{equation}}
Let us define respectively the empirical distribution of $\left(\check{u}_{i}\right)_{1\leqslant i \leqslant N}$ and its centered version by \[\check{Y}_{coal}^{N} := \frac{1}{N}\sum_{i\, = \, 1}^{N}{\delta_{\check{u}_{i}}} \qquad {\rm and} \qquad \check{Z}_{coal}^{N} := \frac{1}{N}\sum_{i\, = \, 1}^{N}{\delta_{\check{v}_{i}^{N}}}.\]

{  For any fixed $T >0$, for all $i, j \in \left\{1, \cdots, N \right\}$, let us consider $T_{ij}$ the coalescence time between the individuals $i$ and $j$ at time $T$ in the process $\left(k_{N, t} \right)_{0\leqslant t \leqslant T}$ and $\check{T}_{ij}$ the coalescence time between individuals at level $i$ and $j$ at time $0$ in the process $\left(\check{k}_{N, t} \right)_{t\geqslant 0}$.} 

\begin{Prop}
The measure-valued random variable  $\check{Z}^{N}_{coal}$ has the law $\pi_{N}$.
\label{Prop_2_5_Genealogie}
\end{Prop}

\begin{proof} The proof consists in establishing for all $f : \MM_{1}^{c,2}({ \R^{d}}) \to \R$ measurable real bounded function, $\E\left(f\left(\check{Z}^{N}_{coal} \right) \right) = \int_{\MM_{1}^{c, 2}({ \R^{d}})}^{}{f(\mu)\pi_{N}(\dd \mu)}$. From Proposition \ref{Prop_2_Genealogie}, it is sufficient to establish for all $\mu_{N} \in \MM_{1, N}^{c, 2}({ \R^{d}})$, that  $\lim_{T \to + \infty}{\left|\E\left(f\left( \check{Z}^{N}_{coal}\right) \right) - \E\left(f\left( \widehat{Z}^{N, \mu_{N}}_{T}\right) \right)  \right|} = 0$. \\ Let $f : \MM_{1}^{c,2}({ \R^{d}}) \to \R$ be a measurable real bounded function. Note that, 
\begin{align*}
\left|\E\left(f\left( \check{Z}^{N}_{coal}\right) \right) - \E\left(f\left( \widehat{Z}^{N, \mu_{N}}_{T}\right) \right)  \right| & \leqslant \left|\E\left(f\left( \check{Z}^{N}_{coal}\right)\II_{\left\{\left|\check{k}_{N,T} \right| = 1 \right\}} \right) - \E\left(f\left( \widehat{Z}^{N, \mu_{N}}_{T}\right)\II_{\left\{\left|k_{N,T} \right| = 1 \right\}} \right)  \right| \\
& \qquad + \left\|f \right\|_{\infty} \left[\check{\P}\left(\left|\check{k}_{N,T} \right| > 1 \right) + \widehat{\P}\left(\left|k_{N,T} \right| > 1 \right) \right].
\end{align*}

{  For all $\mu_{N} \in \MM_{1, N}^{c, 2}({ \R^{d}})$, conditionally to $k_{N} := \left(k_{N,t} \right)_{0\leqslant t \leqslant T}$, on the event $\left\{\left|k_{N,T} \right| = 1 \right\}$, we have from (\ref{Def_u_i}) that for all $i, j \in \left\{1, \cdots, N \right\}$, the covariance matrix $\Cov\left(u_{i}, u_{j} \left| \, k_{N} \right.\right)$ is equal to  $(T - T_{ij})I$ where $I$ designates the identity matrix of size $d$. Conditionally to $k_{N}$, on the event $\left\{\left|k_{N,T} \right| = 1 \right\}$, it follows from (\ref{Ecriture_v_i}) that $\left(v_{i}\right)_{1\leqslant i \leqslant N} \sim \NN^{({ Nd})}\left(0_{\R^{{ Nd}}}, \Sigma \right)$  where $\Sigma$ has the block decomposition $\left(\Sigma_{ij} \right)_{1\leqslant i, j \leqslant N}$ where for all $i,j \in \{1, \cdots, N\}$, $\Sigma_{ij}$ is the $d\times d$ matrix defined by  
\begin{equation*}
  \Sigma_{ij} := \Cov\left(v_{i}, v_{j} \left| \, k_{N} \right. \right) = \left(\frac{1}{N}\sum_{k\, = \, 1}^{N}{\left(T_{ik} + T_{jk} \right)} - \left( T_{ij} + \frac{1}{N^{2}}\sum_{k, \ell \, = \, 1}^{N}{T_{k \ell}}  \right)\right)I.
\end{equation*}

In similar way as previously, we have from (\ref{Def_u_i_check}) and (\ref{Ecriture_v_i_check}) that, conditionally to $\check{k}_{N} := \left(\check{k}_{N,t} \right)_{t \geqslant 0}$, on the event $\left\{\left|\check{k}_{N,T} \right| = 1 \right\}$, we have for all $i, j \in \left\{1, \cdots, N \right\}$
\begin{equation}
    \Cov\left(\check{u}_{i}, \check{u}_{j} \left| \, \check{k}_{N} \right. \right) = \left(\check{T}_{coal}^{\infty} - \check{T}_{ij}\right)I.
    \label{Eq_Cov_u_i_check} 
\end{equation}
Hence, conditionally to $\check{k}_{N}$, on the event $\left\{\left|\check{k}_{N,T} \right| = 1 \right\}$, $\left( \check{v}_{i}^{N}\right)_{1\leqslant i \leqslant N} \sim \NN^{({ Nd})}\left(0_{\R^{{ Nd}}}, \check{\Sigma} \right)$ where $\check{\Sigma} := \left(\check{\Sigma}_{ij} \right)_{1\leqslant i, j \leqslant N}$ and for all $i, j \in \left\{1, \cdots, N \right\}$, $\check{\Sigma}_{ij}$ is a matrix of size $d\times d$ defined by
\begin{equation}
   \check{\Sigma}_{ij} := \Cov\left(\check{v}_{i}^{N}, \check{v}_{j}^{N} \left| \, \check{k}_{N} \right. \right) = \left(\frac{1}{N}\sum_{k\, = \, 1}^{N}{\left(\check{T}_{ik} + \check{T}_{jk} \right)} - \left(\check{T}_{ij} + \frac{1}{N^{2}}\sum_{k, \ell \, = \, 1}^{N}{\check{T}_{k \ell}}  \right)\right)I .
\label{Sigma_Genealogie}
\end{equation}
Hence,} it follows that for all $\mu_{N} \in \MM_{1, N}^{c, 2}({ \R^{d}})$, \[\widehat{Z}^{N, \mu_{N}}_{T}\II_{\left\{\left|k_{N,T} \right| = 1 \right\}} \overset{{\rm law}}{=} \check{Z}^{N}_{coal} \II_{\left\{\left|\check{k}_{N,T} \right| = 1 \right\}}.\]
As established in Step 2 of the proof of Proposition \ref{Prop_2_Genealogie}, $$\lim_{T\to + \infty}{\check{\P}\left(\left|\check{k}_{N,T} \right| > 1 \right)} = \lim_{T \to + \infty}{\widehat{\P}\left(\left|k_{N,T} \right| > 1 \right)} = 0,$$ which concludes the proof. \qedhere
\end{proof}

In the next proposition, we establish the exchangeability property of the family $\left(\check{u}_{i} \right)_{i\in \N^{\star}}$ which will allow to apply the \textsc{De Finetti} representation theorem. 

\begin{Prop}
\begin{itemize}
\item[\textbf{{\rm{\textbf{(1)}}}}]  The family $\left(\check{u}_{i} \right)_{i\in \N^{\star}}$ is exchangeable. 
\item[\textbf{{\rm{\textbf{(2)}}}}]  
There exists a random variable measure-valued $\check{Y}_{coal}^{\infty} : \check{\Omega} \to \MM_{1}({ \R^{d}})$ such that $\left(\check{Y}_{coal}^{N} \right)_{N \in \N^{\star}}$ converges $\check{\P}-${\rm a.s.} when $N\to + \infty$ to $\check{Y}_{coal}^{\infty}$ in $\MM_{1}({ \R^{d}})$ which is equipped with the weak topology. Moreover, given $\check{Y}_{coal}^{\infty}$, $\left(\check{u}_{i} \right)_{i\in \N^{\star}}$ is i.i.d. of law $\check{Y}_{coal}^{\infty}$. 
\end{itemize} 
\label{Prop_3_Genealogie}
\end{Prop}
\begin{proof} \textbf{{\rm{\textbf{(1)}}}} From {\color{blue} \cite[\color{black} Proof of Theorem 2.2]{donnelly_countable_1996}}, it is enough to show for each $N \in \N^{\star}$, $\left(\check{u}_{i} \right)_{1\leqslant i \leqslant N}$ is exchangeable. Let $\sigma : \N^{\star} \to \N^{\star}$ be a finite permutation, that is to say a bijection that leaves all but finitely many points unchanged. The well-known backward construction of the modified look-down process {\color{blue} \cite{donnelly_countable_1996, donnelly_particle_1999}} entails that $\left(\check{k}_{\infty, t}\right)_{t\geqslant 0} \overset{\rm law}{=} \left(\check{k}_{\infty, t}^{\sigma}\right)_{t\geqslant 0}$ where $\check{k}_{\infty, t}^{\sigma}$ is the partition obtained by applying the permutation $\sigma$ to $\check{k}_{\infty, t}$. Therefore, for any permutation $\sigma : \{1, \cdots, N \} \to \{1, \cdots, N \}$ extended by $\id$ to $\N^{\star}$, it is sufficient to prove that $\left(\left(\check{u}_{i} \right)_{1\leqslant i \leqslant N} \left| \, \check{k}_{\infty} \right. \right) \overset{\rm law}{=} \left(\left(\check{u}_{\sigma(i)} \right)_{1\leqslant i \leqslant N} \left| \, \check{k}_{\infty}^{\sigma} \right. \right)$ to obtain the announced result. \\ 

 We define $\check{T}_{coal}^{\infty, \sigma} := \check{T}_{coal}^{\infty}\left(\check{k}_{\infty}^{\sigma} \right)$ and for all $i, j \in \left\{1, \cdots, N \right\}$, $\check{T}_{ij}^{\sigma} := \check{T}_{ij}\left(\check{k}_{\infty}^{\sigma} \right)$. Note that $\check{T}_{coal}^{\infty, \sigma} = \check{T}_{coal}^{\infty}$ $\check{\P}-$a.s. and it follows from the fact $\left(\check{k}_{\infty, t}\right)_{t\geqslant 0} \overset{\rm law}{=} \left(\check{k}_{\infty, t}^{\sigma}\right)_{t\geqslant 0}$ that $$\left(\left(\check{T}_{ij}^{\sigma}, \check{T}_{coal}^{\infty, \sigma} \right)\right)_{1\leqslant i, j \leqslant N} \overset{\rm law}{=} \left(\left(\check{T}_{ij}, \check{T}_{coal}^{\infty} \right)\right)_{1\leqslant i, j \leqslant N}.$$ 
 {  Hence, we conclude from (\ref{Eq_Cov_u_i_check}).} 

\noindent \textbf{(2)} As the family $\left(\check{u}_{i} \right)_{i\in \N^{\star}}$ is exchangeable, the announced result follows from \textsc{De Finetti}'s representation theorem {\color{blue} \cite[\color{black} Theorem 12.26 and Remark 12.27]{klenke_probability_2013}}.
\end{proof}

We conclude this section with a corollary which will be useful to characterise the probability measure $\pi$.

\begin{Cor}
\begin{itemize}
\item[\textbf{{\rm{\textbf{(1)}}}}] For all $k \in \N$, the random variable $\check{Y}_{coal}^{\infty}$ satisfies $\left\langle \left\|\id\right\|^{k} , \check{Y}_{coal}^{\infty}\right\rangle < \infty$ $\check{\P}-${\rm a.s.} 
\item[\textbf{{\rm{\textbf{(2)}}}}] The limit $\check{u}_{\infty} := \lim_{N\to+\infty}{\frac{1}{N}\sum_{j\, = \, 1}^{N}{\check{u}_{j}} } $
exists $\check{\P}-${\rm a.s.} and satisfies $\check{u}_{\infty} = \left\langle \id, \check{Y}^{\infty}_{coal} \right\rangle$.
\end{itemize}
\label{Corollaire_0_Genealogie}
\end{Cor}

\begin{proof}
\textbf{(1)} For all $M \in (0, + \infty)$, let us consider $\left|\id \right|_{M}$ the truncation function of $\id$ at level $M$ defined by $ \left|\id \right|_{M} = \left\|\id\right\|$ on $[-M,M]^{ d}$ and $ \left|\id \right|_{M} = M$ on $\R^{ d} \, \backslash \, [-M,M]^{ d}$. By \textsc{Fatou}'s lemma, we obtain that \[\E\left(\left\langle \left|\id \right|_{M}^{k}, \check{Y}_{coal}^{\infty} \right\rangle \right) \leqslant \varliminf_{N \to + \infty}{\E\left(\left\langle \left|\id \right|_{M}^{k}, \check{Y}_{coal}^{N} \right\rangle \right)} \leqslant \varliminf_{N\to + \infty}{\frac{1}{N}\sum_{i\, = \, 1}^{N}{\left( 1 +  \E\left(\left\|\check{u}_{i}\right\|^{2k} \right) \right)} }. \]
Now, classical moment results for Gaussian random variables show that for all $n \in \N$, $\E\left(G^{2n} \right) = \frac{(2n)!}{2^{n}n!}\sigma^{2n}$ for $G \sim \NN(0,\sigma^{2})$. The announced result follows from 
{ (\ref{Eq_Cov_u_i_check})}. \\

\textbf{(2)} From Proposition \ref{Prop_3_Genealogie}, given $\check{Y}_{coal}^{\infty}$, $\left(\check{u}_{i} \right)_{i\in \N^{\star}}$ is i.i.d. Therefore, the announced almost surely existence limit follows from the Strong Law of Large Numbers. Moreover, 
\[\left\langle \id, \check{Y}^{\infty}_{coal} \right\rangle = \lim_{N\to + \infty}{\left\langle \id, \check{Y}^{N}_{coal} \right\rangle} = \lim_{N\to + \infty}{\frac{1}{N}\sum_{j\, = \, 1}^{N}{\check{u}_{j}}} = \check{u}_{\infty}, \qquad \check{\P}-{\rm a.s.} \qedhere \]
\end{proof}

\subsubsection{Characterisation of the invariant probability measure $\pi$ \label{Look-down_2}}

{  We can now conclude with the characterisation of the invariant measure $\pi$:} now we define the random variable $\check{Z}_{coal}^{\infty} \in \MM_{1}({ \R^{d}})$ as 
\begin{equation*}
\check{Z}^{\infty}_{coal} := \tau_{-\left\langle \id, \check{Y}^{\infty}_{coal} \right\rangle} \, \sharp \check{Y}^{\infty}_{coal}. 
\label{Z_infty_coal}
\end{equation*}
The following proposition establishes the convergence of $\left(\check{Z}^{N}_{coal} \right)_{N \in \N^{\star}}$ to $\check{Z}_{coal}^{\infty}$. 
Let us recall the following well-known fact useful for the proof below: a straightforward adaptation of the proof of {\color{blue} \cite[\color{black} Lemma 2.1.2]{Dawson}} allows us to obtain that 
the algebra of polynomials \[\Span\left(\left\{ { P_{f, n}(\mu)} \left. \phantom{1^{1^{1^{1}}}} \hspace{-0.6cm} \right| f : { (\R^{d})}^{n} \to \R {\rm{\ uniformly \ continuous}}, \mu \in \MM_{1}({ \R^{d}}), { n\in \N^{\star}} \right\} \right)\] is convergence determining in $\MM_{1}\left(\MM_{1}({ \R^{d}})\right)$ {  where $P_{f,n}$ is given by (\ref{Eq_Fonction_polynome_mu})}. 

\begin{Thm} \begin{itemize}
\item[\textbf{\rm{\textbf{(1)}}}] The sequence of random variables $\left(\check{Z}_{coal}^{N} \right)_{N \in \N^{\star}}$ converges $\check{\P}-${\rm a.s.} when $N \to + \infty$ to $\check{Z}^{\infty}_{coal}$ in $\MM_{1}({ \R^{d}})$ for the weak convergence topology.   
\item[\textbf{\rm{\textbf{(2)}}}] The random variable $\check{Z}_{coal}^{\infty}$ has the law $\pi$.
\end{itemize}
\label{Convergence_Z_infini_check}
\end{Thm}

\begin{proof}
\textbf{(1)} From the previous reminder, it is sufficient to prove that for all $n \in \N^{\star}$, for all $f : { (\R^{d})}^{n} \to \R$ uniformly  continuous, { $\lim_{N\to + \infty}{P_{f,n}\left(\check{Z}_{coal}^{N} \right)} = P_{f,n}\left(\check{Z}_{coal}^{\infty} \right)$}. With an argument similar to the proof of Proposition \ref{Prop_0_Genealogie}, we obtain that for all $n \in \N^{\star}$, for all $f : { (\R^{d})}^{n} \to \R$ uniformly  continuous, { $P_{f\circ \tau_{-\left\langle \id, \cdot \right\rangle},n}$ from $\MM_{1}^{2}({ \R^{d}})$ to $\R$}
 is continuous. Let $n \in \N^{\star}$ and $f : { (\R^{d})}^{n} \to \R$ uniformly continuous. Now,  $\check{Y}^{\infty}_{coal} \in \MM_{1}^{2}({ \R^{d}})$ from Corollary \ref{Corollaire_0_Genealogie} and 
 \[{ P_{f\circ\tau_{-\left\langle \id, \check{Y}^{N}_{coal}  \right\rangle}, n}\left(\check{Y}^{N}_{coal} \right) = P_{f, n}\left(\check{Z}^{N}_{coal} \right)} \qquad {\rm and} \qquad  { P_{f\circ\tau_{-\left\langle \id, \check{Y}^{\infty}_{coal}  \right\rangle}, n}\left(\check{Y}^{\infty}_{coal} \right) = P_{f, n}\left(\check{Z}^{\infty}_{coal} \right)}\]
 From Proposition \ref{Prop_3_Genealogie},  { $\lim_{N\to + \infty}{P_{f\circ\tau_{-\left\langle \id, \check{Y}^{N}_{coal}  \right\rangle},n}\left(\check{Y}^{N}_{coal} \right)} = P_{f\circ\tau_{-\left\langle \id, \check{Y}^{\infty}_{coal}  \right\rangle}, n}\left(\check{Y}^{\infty}_{coal} \right)$} $\check{\P}-$a.s. which concludes the proof. \\
 
\textbf{(2)} Let $n \in \N^{\star}$. As for all $N \in \N^{\star}$ and $f : { (\R^{d})}^{n} \to \R$ uniformly continuous, { \begin{align*}
& \left|\E\left(P_{f,n}\left(\check{Z}_{coal}^{\infty}\right)\right) - \int_{\MM_{1}({ \R^{d}})}^{}{P_{f,n}(\mu) \pi(\dd \mu)} \right| \\
& \hspace{0.5cm} \leqslant \E\left(\left| P_{f,n}\left(\check{Z}_{coal}^{\infty} \right) - P_{f,n}\left(\check{Z}_{coal}^{N} \right) \right| \right) + \left| \E\left(P_{f,n}\left(\check{Z}_{coal}^{N} \right) \right) - \int_{\MM_{1}({ \R^{d}})}^{}{P_{f,n}\left(\mu \right) \pi_{N}(\dd \mu)} \right| \\
& \hspace{2cm} + \left|\int_{\MM_{1}({ \R^{d}})}^{}{P_{f,n}\left(\mu \right) \left[\pi(\dd \mu) - \pi_{N}(\dd \mu)\right]} \right|,
\end{align*}}
the announced result follows from Theorem \ref{Convergence_Z_infini_check} \textbf{(1)}, Proposition \ref{Prop_2_5_Genealogie} and Lemma \ref{Lemme_1_Genealogie} when $N~\to~+\infty$. \qedhere 
\end{proof}

{ \subsubsection{Application: computation of the second moment of the invariant measure}}

{  We assume here that $d = 1$.} The last characterisation of the probability measure $\pi$ is suitable to make explicit computations. The next corollary gives an expression of the second moment under~$\pi$.

\begin{Cor}
We have $\int_{\MM_{1}^{c, 2}(\R)}^{}{M_{2}\left(\mu\right)\pi\left(\dd \mu\right)} = 1/{ 2\gamma}$. 
\end{Cor}
\begin{proof} \textbf{Step 1. Uniform bound in $N$ of $\E\left(M_{2k}\left(\check{Z}^{N}_{coal} \right) \right), {  k \in \N^{\star}}$.} In this step, we want to establish \[\sup_{N \in \N^{\star}}{\E\left(\left\langle \left|\id\right|^{2k}, \check{Z}^{N}_{coal} \right\rangle \right)} < \infty.\]
Let $N \in \N^{\star}$. From Proposition \ref{Prop_3_Genealogie} \textbf{(1)}, 
\begin{align*}
 \E\left(\left\langle \left|\id\right|^{2k}, \check{Z}^{N}_{coal} \right\rangle \right) & = \E\left( \left. \E\left(\frac{1}{N}\sum\limits_{i\, = \, 1}^{N}{\left|\check{u}_{i} - \frac{1}{N}\sum\limits_{j\, = \, 1}^{N}{\check{u}_{j}} \right|^{2k}} \right|  \left(\check{T}_{m\ell} \right)_{1\leqslant m, \ell \leqslant N}, \check{T}^{\infty}_{coal} \right)\right) \\
 & = \E\left(\E\left( \left|\check{v}_{1}^{N} \right|^{2k}\left| \left(\check{T}_{m\ell} \right)_{1\leqslant m, \ell \leqslant N}, \check{T}^{\infty}_{coal} \right.\right) \right).
\end{align*}
From { (\ref{Sigma_Genealogie}) and Proposition 
\ref{Prop_Genealogie_Moments}}, we obtain that
\begin{align*}
\E\left(\E\left( \left|\check{v}_{1}^{N} \right|^{2k}\left| \left(\check{T}_{m,\ell} \right)_{1\leqslant m, \ell \leqslant N}, \check{T}^{\infty}_{coal} \right.\right) \right) = \frac{(2k)!}{2^{k}k!}\E\left(\check{\Sigma}_{ii}^{2k} \right) \leqslant \frac{(2k)!}{k!}\E\left(\left(\check{T}^{\infty}_{coal}\right)^{2k} \right) <\infty,
\end{align*} 
and the announced result follows. \\

\textbf{Step 2. Convergence result of $\E\left(M_{2}\left(\check{Z}^{N}_{coal} \right) \right)$ to $\E\left(M_{2}\left(\check{Z}^{N}_{coal} \right) \right)$.} Note that for all $N \in \N^{\star}$, $M \in (0, + \infty)$,
\begin{align*}
\left|\E\left(M_{2}\left(\check{Z}^{N}_{coal} \right) \right) - \E\left(M_{2}\left(\check{Z}^{\infty}_{coal} \right) \right) \right| & \leqslant {\textbf{\rm{\textbf{(A)}}}}_{N, M} + {\textbf{\rm{\textbf{(B)}}}}_{N, M} + {\textbf{\rm{\textbf{(C)}}}}_{N, M},
\end{align*}
where
\begin{align*}
{\textbf{\rm{\textbf{(A)}}}}_{N, M} & := \left|\E\left(M_{2}\left(\check{Z}^{N}_{coal} \right) \right) - \E\left(\left\langle \left|\id \right|^{2}_{M},  \check{Z}^{N}_{coal} \right\rangle \right) \right|, \\
{\textbf{\rm{\textbf{(B)}}}}_{N, M} & := \left|\E\left(\left\langle \left|\id \right|^{2}_{M},  \check{Z}^{N}_{coal} \right\rangle \right) - \E\left(\left\langle \left|\id \right|^{2}_{M},  \check{Z}^{\infty}_{coal} \right\rangle \right)\right|, \\
{\textbf{\rm{\textbf{(C)}}}}_{N, M} & := \left|\E\left(\left\langle \left|\id \right|^{2}_{M},  \check{Z}^{\infty}_{coal} \right\rangle \right) - \E\left(M_{2}\left(\check{Z}^{\infty}_{coal} \right) \right)  \right|.
\end{align*}

From the inequality $\left| \id^{2} - \id^{2}_{M}\right| \leqslant \frac{2}{3\sqrt{3} M} \left|\id \right|^{3}$, then the \textsc{H\"{o}lder} inequality, we obtain that 
\begin{align*}
{\textbf{\rm{\textbf{(A)}}}}_{N,M}  \leqslant \E\left(\left\langle \left| \id^{2} - \id_{M}^{2} \right|, \check{Z}^{N}_{coal} \right\rangle \right) \leqslant \frac{2}{3\sqrt{3}M}\E\left(\left\langle \left|\id \right|^{3}, \check{Z}^{N}_{coal} \right\rangle \right)  \leqslant \frac{2}{3\sqrt{3}M}\E\left(\left\langle \id^{4}, \check{Z}^{N}_{coal} \right\rangle^{\frac{3}{4}} \right).
\end{align*}
From Step 1, we deduce that for all $N \in \N^{\star}$, \[{\textbf{\rm{\textbf{(A)}}}}_{N,M} \leqslant \frac{2}{3\sqrt{3}M}\left(1 + \sup_{N \in \N^{\star}}{\E\left(\left\langle \id^{4}, \check{Z}^{N}_{coal} \right\rangle \right)} \right) < \infty.\]
In similar way, we obtain that for all $N \in \N^{\star}$, ${\textbf{\rm{\textbf{(C)}}}}_{N,M} \leqslant \frac{2}{3\sqrt{3}M}\left(1 + \E\left(\left\langle \id^{4}, \check{Z}^{\infty}_{coal}\right\rangle \right) \right)$ where $\E\left(\left\langle \id^{4}, \check{Z}^{\infty}_{coal}\right\rangle \right) < \infty$ from Corollary \ref{Corollaire_0_Genealogie} \textbf{(1)}. By the monotone convergence theorem, we deduce that for all $N \in \N^{\star}$, $\E\left(M_{2}\left(\check{Z}^{N}_{coal} \right) \right) = \lim_{M \to + \infty}{\E\left(\left\langle \left|\id \right|_{M}^{2}, \check{Z}^{N}_{coal} \right\rangle \right)}$ and $\E\left(M_{2}\left(\check{Z}^{\infty}_{coal} \right) \right) = \lim_{M \to + \infty}{\E\left(\left\langle \left|\id \right|_{M}^{2}, \check{Z}^{\infty}_{coal} \right\rangle \right)}$. From Theorem \ref{Convergence_Z_infini_check}, for all $M \in (0, +\infty)$, $\lim_{N \to +\infty}{{\textbf{\rm{\textbf{(B)}}}}_{N,M}} = 0$. Hence $\lim_{N\to+ \infty}{\E\left(M_{2}\left(\check{Z}^{N}_{coal} \right) \right)} = \E\left(M_{2}\left(\check{Z}^{\infty}_{coal} \right) \right)$. \\

\textbf{Step 3. Conclusion.}
Note that for all $N \in \N^{\star}$, \begin{align*}
M_{2}\left(\check{Z}_{coal}^{N} \right) = \left\langle \id^{2}, \check{Z}_{coal}^{N} \right\rangle = \frac{1}{N}\sum\limits_{i\, = \, 1}^{N}{\check{u}_{i}^{2}} - \frac{1}{N^{2}}\left[\sum\limits_{i\, = \, 1}^{N}{\check{u}_{i}^{2}} + 2 \sum\limits_{1\leqslant i < j \leqslant N}^{}{\check{u}_{i}\check{u}_{j}} \right].
\end{align*}
From 
{ (\ref{Eq_Cov_u_i_check})}, we have for all $i, j \in \left\{1, \cdots, N \right\}$, $\E\left(\check{u}_{i}\check{u}_{j}\right) = \E\left(\E\left(\check{u}_{i}\check{u}_{j} \left|  \, \check{k}_{N} \right. \right)\right) = \E\left(\check{T}_{coal}^{\infty} - \check{T}_{ij} \right)$ where $\check{T}_{ij}$ is an exponential random variable with parameter { $2\gamma$} if $i\neq j$. 
Therefore  $\E\left(M_{2}\left(\check{Z}_{coal}^{N} \right) \right) = \frac{N-1}{N} \times \frac{1}{{ 2\gamma}}$. By Step 2, we deduce that $\E\left(M_{2}\left(\check{Z}^{\infty}_{coal} \right) \right) = 1/{ 2\gamma}$ which completes the proof. \qedhere
\end{proof}

\section{Proof of Theorem \ref{Prop_PB_Mg_Z} \label{Section_preuve_Existence}}

{  The proof of Theorem \ref{Prop_PB_Mg_Z} in any dimension $d \in \N^{\star}$ is identical to that in  dimension $1$ but is much more cumbersome to write. For the sake of clarity, only the proof in the case $d = 1$ is presented. We leave the details for the multidimensioanl case to the reader.} \\

We divide the proof of the main result into { 3} steps,  each of which will constitute a section (Sections \ref{Etape_2_Existence} to { \ref{Etape_4_Existence}}). We recall that the aim of this proof is to prove that the law $\P_{\tau_{-\left\langle \id, \nu \right\rangle}  \sharp \, \nu}^{\rm FVc}$ of the process $\left(Z_{t}\right)_{0\leqslant t \leqslant T}$ defined by \[\forall t \geqslant 0, \qquad Z_{t} := \tau_{-\left\langle \id, Y_{t} \right\rangle} \sharp \, Y_{t}, \] under $\P_{\nu}^{\rm FV}$ for $\nu \in \MM_{1}^{2}(\R)$ is solution of the martingale problem {\rm{(\ref{PB_Mg_Z_Multi_d})}}.  
{  By standard density arguments, it is sufficient to consider the case $F, g \in \CCCC^{4}_{b}(\R, \R)$ to obtain the announced result with $F \in \CCCC^{2}(\R, \R)$ and $g \in \CCCC^{2}_{b}(\R, \R)$.}
In {\rm{(\ref{Existence_PB_mg})}} there are essentially two types of terms: $\left\langle \id, Y_{t} - Y_{s} \right\rangle$ and $\left\langle g^{(j)}\circ\tau_{-\left\langle \id , Y_{s} \right\rangle} , Y_{t} - Y_{s} \right\rangle$, $j\in\{0, 1, 2 \}$, $s\leqslant t$. In Sections \ref{Etape_2_Existence} and \ref{Etape_3_Existence}, we prove that the two previous quantities admit a \textsc{Doob}'s semi-martingale decomposition. In Section \ref{Etape_4_Existence} we handle all the terms in  (\ref{Existence_PB_mg})  involving respectively the first and second derivatives of $F$. {  We end this section dealing with} the different error terms involved in (\ref{Existence_PB_mg}).  {  By classical arguments, we can prove that the martingale involved in (\ref{PB_Mg_Z_Multi_d}) is square integrable and we can compute its martingale bracket (\ref{Variation_quad_FV_recentre_Multi_d}).} 

\subsection{\textsc{Doob}'s semi-martingale decomposition of $\left\langle \id, Y_{t} - Y_{s}  \right\rangle$, $s\leqslant t$ \label{Etape_2_Existence}} 

In {\rm{(\ref{PB_Mg_Y})}}, $M^{\id}(g)$ is well-defined only for $g \in \CCCC^{2}_{b}(\R, \R)$. The expression  makes sense for more general functions $g$. The goal of this section is to prove that, for any $k \in \N$,  $M^{\id}\left(\id^{k} \right)$ is the martingale part in the \textsc{Doob} semi-martingale decomposition of $\left\langle \id^{k}, Y_{t} \right\rangle$. In particular, 
\begin{equation*}
\left\langle \id , Y_{s } - Y_{t_{i}^{n} \wedge t} \right\rangle = M^{\id}_{s}(\id) - M^{\id}_{t_{i}^{n}\wedge t}(\id), \qquad s \geqslant t_{i}^{n}\wedge t  \label{Martingale_id_id_Etape_2}
\end{equation*}
is a $\P_{\nu}^{\rm FV}-$martingale. 

\begin{Lem} Let $\nu \in \MM_{1}(\R)$ 
and let $\P_{\nu}$ be a distribution on $\Omega$ satisfying {\rm{(\ref{PB_Mg_Y})}} and such that $Y_{0} $ is equal in law to $\nu$. Let $T > 0$ and $k \in \N^{\star}$ be fixed. 
\begin{itemize}
\item[{\rm{\textbf{{\rm{\textbf{(1)}}}}}} ] If $\left\langle \left|\id \right|^{k}, \nu \right\rangle  <\infty$, then there exist two constants $C_{k,T}, \widetilde{C}_{k,T} >0$, such that any stochastic process $\left(Y_{t} \right)_{0\leqslant t \leqslant T}$ whose law is $\P_{\nu}$ satisfies
\begin{align*}
{\rm{\textbf{\rm{\textbf{(a)}}}}} \qquad & \sup_{t\in[0,T]}{\E_{\nu}\left(\left\langle \left|\id\right|^{k}, Y_{t}\right\rangle \right)}  \leqslant  C_{k, T}\left(1 + \left\langle \left|\id \right|^{k}, \nu \right\rangle \right),  \\
{\rm{\textbf{\rm{\textbf{(b)}}}}} \qquad & \forall \alpha >0, \quad \P_{\nu}\left(\sup_{t\in[0,T]}{\left\langle \left|\id\right|^{k}, Y_{t}\right\rangle } \geqslant \alpha \right) \leqslant \frac{\widetilde{C}_{k, T}\left(1 + \left\langle \left|\id \right|^{k}, \nu \right\rangle \right)}{\alpha}. 
\end{align*}
\item[{\rm{\textbf{{\rm{\textbf{(2)}}}}}}] If $\left\langle \left|\id \right|^{k}, \nu \right\rangle  <\infty$, then the process $\left(M_{t}^{\id}\left(\id^{k}\right)\right)_{0\leqslant t \leqslant T}$ defined by \[M_{t}^{\id}\left(\id^{k} \right) := \left\langle \id^{k}, Y_{t} \right\rangle - \left\langle \id^{k}, Y_{0} \right\rangle - \int_{0}^{t}{\left\langle \frac{k(k-1)}{2} \id^{k-2}, Y_{s}\right\rangle \dd s}, 
\]
  is a continuous $\P_{\nu}-$martingale. Moreover, if $\left\langle \left|\id \right|^{2k}, \nu \right\rangle <\infty$, then $\left(M^{\id}\left(\id^{k}\right) \right)_{0\leqslant t \leqslant T}$ is a martingale in $L^{2}\left(\Omega \right)$ whose quadratic variation is given by
\begin{equation*}
\left\langle M^{\id}\left(\id^{k}\right) \right\rangle_{t}  = 2\gamma \int_{0}^{t}{\left[\left\langle \id^{2k}, Y_{s}  \right\rangle - \left\langle \id^{k}, Y_{s}  \right\rangle^{2} \right]\dd s}.
\end{equation*}
\end{itemize}
\label{Lem_Mart_id_id_bien_def}
\end{Lem}

\begin{proof} The proof is similar to that of Proposition \ref{Prop_moments_non_bornes_RECENTRE}. \qedhere
\end{proof}

\subsection{\textsc{Doob}'s semi-martingale decomposition of $\left\langle g^{(j)} \circ \tau_{-\left\langle \id, Y_{s}  \right\rangle}, Y_{t} - Y_{s}  \right\rangle$, $s \leqslant t$, $j\in\{0,1,2\}$ \label{Etape_3_Existence}} 

Equation  (\ref{Existence_PB_mg}) involves terms of the form $\left\langle g^{(j)}\circ\tau_{-\left\langle \id, Y_{t_{i}^{n} \wedge t} \right\rangle}, Y_{t_{i+1}^{n} \wedge t} - Y_{t_{i}^{n} \wedge t} \right\rangle$ with $j \in \{0,1, 2\}$. We wish to express, each of these terms using the martingale problem (\ref{PB_Mg_Y}). However, this  leads us to consider quantities of the form 
\begin{equation}
M^{\id}_{t_{i+1}^{n}\wedge t}\left(g^{(j)} \circ \tau_{-\left\langle \id, Y_{t_{i}^{n}\wedge t} \right\rangle} \right) - M^{\id}_{t_{i}^{n}\wedge t}\left(g^{(j)} \circ \tau_{-\left\langle \id, Y_{t_{i}^{n}\wedge t} \right\rangle} \right) \label{Difference_Mart_id}
\end{equation}
with $j \in \{0,1, 2\}$, which are not well defined at the moment. Indeed, in (\ref{Difference_Mart_id}) the input argument is a predictable \emph{random} function  of the process $\left(Y_{t} \right)_{0\leqslant t \leqslant T}$ while the martingale problem (\ref{PB_Mg_Y}) defines $M_{t}^{\id}(g)$ 
 only for \emph{deterministic} functions $g$. {\rm{Lemma \ref{Lem_Q_omega_Martingale}}} hereafter, allows us to give a precise meaning to  (\ref{Difference_Mart_id}) by extending the well-defined character of the martingales of (\ref{PB_Mg_Y}) to predictable input arguments. The proof of this technical lemma, {  is based on classical arguments using \emph{regular conditional probabilities}, that we recall in Appendix \ref{Appendix_Regular_Conditional_Proba} for completeness.}  
  
 \begin{Lem} Let $t^{\star} \in \R_{+}$ be a deterministic time and $h : \Omega \rightarrow \CCCC_{b}^{2}(\R, \R) $ be a measurable  function satisfying  the following property: 
 \begin{equation*}
\forall \, \omega, \omega' \in \Omega, \qquad h(\omega) = h(\omega') \qquad \qquad {\rm{if}} \quad \omega_{|_{\left[0, t^{\star}\right]}} = \omega'_{|_{\left[0, t^{\star}\right]}}.
\label{prop_fonc_h}
\end{equation*}
Then, the following process defined, for all $t \in [0,T]$, by
    \[\MM_{t}\left(\widetilde{\omega}\right) :=  M_{t}^{\id}\left(h\left(\widetilde{\omega}\right)\right)\left(\widetilde{\omega}  \right) - M_{t\wedge t^{\star}}^{\id}\left(h\left(\widetilde{\omega}\right) \right) \left(\widetilde{\omega}  \right)\] 
    is a $\P_{\nu}^{\rm FV}(\dd \widetilde{\omega})$ square integrable martingale whose quadratic variation is given by 
 \[\left\langle \MM\left(\widetilde{\omega} \right) \right\rangle_{t} = 2\gamma \int_{t\wedge t^{\star}}^{t}{\left[\left\langle h^{2}\left(\widetilde{\omega}_{s} \right), \widetilde{\omega}_{s} \right\rangle - \left\langle h\left(\widetilde{\omega}_{s} \right), \widetilde{\omega}_{s} \right\rangle^{2} \right]\dd s}.  \]
  \label{Lem_Q_omega_Martingale}
\end{Lem}

{\rm{Lemma \ref{Lem_Q_omega_Martingale}}} with $t^{\star} = t$ allows us to assert that {\rm{(\ref{Difference_Mart_id})}} is a $\P_{\nu}^{\rm FV}(\dd Y)-$martingale increment. Thus, we obtain for $j\in \{0,1,2 \}$:
\begin{align*}
& \left\langle g^{(j)}\circ\tau_{-\left\langle \id, Y_{t_{i}^{n}\wedge t} \right\rangle}, Y_{t_{i+1}^{n} \wedge t} - Y_{t_{i}^{n} \wedge t}  \right\rangle =  \int_{t_{i}^{n}\wedge t}^{t_{i+1}^{n}\wedge t}{\frac{1}{2}\left\langle g^{(j+2)}\circ\tau_{-\left\langle \id, Y_{t_{i}^{n}\wedge t} \right\rangle},  Y_{s} \right\rangle \dd s} \label{Decomp_semi_martingale_alea}  \\
 &  \hspace{6.25cm} + M_{t_{i+1}^{n}\wedge t}^{\id}\left(g^{(j)}\circ\tau_{-\left\langle \id, Y_{t_{i}^{n}\wedge t} \right\rangle} \right) - M_{t_{i}^{n}\wedge t}^{\id}\left(g^{(j)}\circ\tau_{-\left\langle \id, Y_{t_{i}^{n} \wedge t}  \right\rangle} \right) \notag 
\end{align*} 
where $\left(M_{s\wedge t}^{\id}\left(g^{(j)}\circ\tau_{-\left\langle \id, Y_{t_{i}^{n}\wedge t} \right\rangle} \right) - M_{t_{i}^{n}\wedge t}^{\id}\left(g^{(j)}\circ\tau_{-\left\langle \id, Y_{t_{i}^{n} \wedge t} \right\rangle} \right)\right)_{s\geqslant t_{i}^{n}}$ is a $\P_{\nu}^{\rm FV}$ square integrable martingale satisfying for all $s \geqslant t_{i}^{n}$,
\begin{equation*}
\begin{aligned}
& \left\langle M_{\cdot \wedge t}^{\id}\left(g^{(j)}\circ\tau_{-\left\langle \id, Y_{t_{i}^{n}\wedge t} \right\rangle} \right) - M_{t_{i}^{n}\wedge t}^{\id}\left(g^{(j)}\circ\tau_{-\left\langle  \id, Y_{t_{i}^{n} \wedge t} \right\rangle} \right) \right\rangle_{s} \\
& \hspace{0.75cm} = 2\gamma \int_{t_{i}^{n}\wedge t}^{s \wedge t}{\left[\left\langle \left(g^{(j)}\circ\tau_{-\left\langle \id, Y_{t_{i}^{n} \wedge t} \right\rangle}\right)^{2} , Y_{u} \right\rangle - \left\langle g^{(j)}\circ\tau_{-\left\langle \id, Y_{t_{i}^{n} \wedge t} \right\rangle}, Y_{u} \right\rangle^{2} \right]\dd u}. 
\end{aligned} 
\label{Martingale_h_i_n_Etape_3}
\end{equation*}

\subsection{Expressions of the terms of (\ref{Existence_PB_mg}) \label{Etape_4_Existence}} 
In the rest of this proof, we use the following notations to simplify the writing. We denote for all $s \geqslant 0$, $R(s) := \left\langle \id, Y_{s}\right\rangle$. Our goal is to prove the following lemma:
\begin{Lem}
When the mesh of the subdivision $0 = t_{0}^{n} < t_{1}^{n} <\dots < t_{p_{n}}^{n} = T$ of $[0,T]$ tends to $0$ when $n \to +\infty$, we have the following convergence in probability
\[\lim_{n\to + \infty}{\sum\limits_{i\, = \, 0}^{p_{n}-1}{{\rm{\textbf{\rm{\textbf{(A)}}}}}_{i}}}   = \int_{0}^{t}{F'\left(\left\langle g, Z_{s}  \right\rangle \right) \left( \left\langle \frac{g''}{2}, Z_{s}  \right\rangle + {\color{black} \gamma  \left\langle g'', Z_{s} \right\rangle M_{2}(Z_{s})   }- 2\gamma\left\langle g' \times \id, Z_{s} \right\rangle \right)} + {\rm{Mart_{t}}},\]
where $\left({\rm{Mart_{t}}}\right)_{0\leqslant t \leqslant T}$ is a $\P_{\nu}^{\rm FV}-$martingale. 
\label{Lemme_Termes_F'}
\end{Lem}
The proof of Lemma {\ref{Lemme_Termes_F'}} is based on the following decomposition of ${\rm{\textbf{\rm{\textbf{(A)}}}}}_{i}$ (given by the expression (\ref{Morceau_Ai_F_prime})) {  which makes use of} the \textsc{Doob} semi-martingale decompositions of Sections \ref{Etape_2_Existence} and \ref{Etape_3_Existence}.
We have
\[{\rm{\textbf{\rm{\textbf{(A)}}}}}_{i}  =  \sum\limits_{k\, = \, 1}^{6}{F'\left(\left\langle g\circ\tau_{-R\left(t_{i}^{n}\wedge t\right)}, Y_{t_{i}^{n} \wedge t} \right\rangle \right){\rm{\textbf{\rm{\textbf{(A)}}}}}^{k}_{i}}\] where 
\begin{align*}
{\rm{\textbf{\rm{\textbf{(A)}}}}}^{1}_{i} & = \int_{t_{i}^{n}\wedge t}^{t_{i+1}^{n}\wedge t}{\frac{1}{2}\left\langle g''\circ\tau_{-R\left(t_{i}^{n}\wedge t\right)}, Y_{s} \right\rangle \dd s}, \\
{\rm{\textbf{\rm{\textbf{(A)}}}}}^{2}_{i} & = M_{t_{i+1}^{n}\wedge t}^{\id}\left(g\circ\tau_{-R\left(t_{i}^{n}\wedge t\right) } \right) - M_{t_{i}^{n}\wedge t}^{\id}\left(g\circ\tau_{-R\left(t_{i}^{n}\wedge t\right) } \right),\\
{\rm{\textbf{\rm{\textbf{(A)}}}}}^{3}_{i} & = - \left[M_{t_{i+1}^{n}\wedge t}^{\id}\left(\id\right) - M_{t_{i}^{n}\wedge t}^{\id}\left( \id \right) \right]\left\langle g'\circ \tau_{-R\left(t_{i}^{n}\wedge t \right)}, Y_{t_{i}^{n}\wedge t} \right\rangle, \phantom{\left(\tau_{-\left\langle Y_{t_{i}^{n}\wedge t}, \id \right\rangle } \right)} \\
{\rm{\textbf{\rm{\textbf{(A)}}}}}^{4}_{i} & = - \left[M_{t_{i+1}^{n}\wedge t}^{\id}\left(\id\right) - M_{t_{i}^{n}\wedge t}^{\id}\left( \id \right) \right] \left[M_{t_{i+1}^{n}\wedge t}^{\id}\left(g'\circ\tau_{-R\left(t_{i}^{n}\wedge t\right) } \right) - M_{t_{i}^{n} \wedge t}^{\id}\left(g'\circ\tau_{-R\left(t_{i}^{n}\wedge t\right) } \right)\right], \\
{\rm{\textbf{\rm{\textbf{(A)}}}}}^{5}_{i} & = \frac{1}{2}  \left[M_{t_{i+1}^{n}\wedge t}^{\id}\left(\id \right) - M_{t_{i}^{n}\wedge t}^{\id}\left(\id \right) \right]^{2} \left\langle g''\circ \tau_{-R\left(t_{i}^{n}\wedge t \right)}, Y_{t_{i}^{n}\wedge t} \right\rangle, \\
{\rm{\textbf{\rm{\textbf{(A)}}}}}^{6}_{i} & = O \left( \left|t_{i+1}^{n}\wedge t - t_{i}^{n}\wedge t \right| \left|M_{t_{i+1}^{n}\wedge t}^{\id}(\id) - M_{t_{i}^{n}\wedge t}^{\id}(\id) \right|  \right).
\end{align*}
Note that we used the following inequality \begin{equation*}
\int_{t_{i}^{n} \wedge t}^{t_{i+1}^{n} \wedge t}{\frac{1}{2}\left\langle g^{(3)}\circ\tau_{-R\left(t_{i}^{n}\wedge t \right)}, Y_{s} \right\rangle \dd s} \leqslant \left\|g^{(3)} \right\|_{\infty} \left(t_{i+1}^{n} \wedge t - t_{i}^{n} \wedge t \right) \quad \P_{\nu}^{\rm FV}-{\rm{a.s.}}
 \label{Majoration_constante_universelle}
\end{equation*}  
to bound the term $\left[M^{\id}_{t_{i+1}^{n}}(\id) - M^{\id}_{t_{i}^{n}}(\id)\right]\int_{t_{i}^{n}\wedge t}^{t_{i+1}^{n}\wedge t}{\frac{1}{2}\left\langle g^{(3)} \circ \tau_{-R\left(t_{i}^{n}\wedge t \right)}, Y_{s} \right\rangle \dd s}$ by ${\rm{\textbf{\rm{\textbf{(A)}}}}}^{6}_{i}$. Our goal in the sequel is to write each of these six quantities as sums of finite variation terms, martingale terms and negligible terms and to study the limit of each of them.

\subsubsection{Decomposition and study of ${\rm{\textbf{\rm{\textbf{(A)}}}}}^{1}_{i}$ \label{Etape_4.1_Existence}} 
Note that, for any $i \in \{0, \cdots, p_{n}-1\}$, 
\begin{align*}
& {\rm{\textbf{\rm{\textbf{(A)}}}}}^{1}_{i} = \int_{t_{i}^{n} \wedge t}^{t_{i+1}^{n} \wedge t}{\frac{1}{2}\left\langle g''\circ\tau_{-R\left(s\right)}, Y_{s} \right\rangle \dd s} + \int_{t_{i}^{n}\wedge t}^{t_{i+1}^{n}\wedge t}{\frac{1}{2} \left\langle g''\circ\tau_{-R\left(t_{i}^{n}\wedge t\right)} - g''\circ\tau_{-R\left(s\right)}, Y_{s} \right\rangle \dd s}. 
\end{align*}
Using a argument of convergence of \textsc{Riemann}'s sums for the first term and {\rm{Lemma \ref{Lem_convergence_proba} \textbf{(2)}}} for the second one, we obtain {  in probability}, that 
\[\lim_{n\to+\infty}{\sum_{i\, = \, 0}^{p_{n} - 1}{F'\left(\left\langle g \circ \tau_{-R\left(t_{i}^{n}\wedge t \right)}, Y_{t_{i}^{n}\wedge t} \right\rangle \right){\rm{\textbf{\rm{\textbf{(A)}}}}}^{1}_{i}}} = \int_{0}^{t}{F'\left(\left\langle  g, Z_{s} \right\rangle \right)\left\langle \frac{g''}{2}, Z_{s} \right\rangle \dd s}.\]

\subsubsection{Martingale contribution of ${\rm{\textbf{\rm{\textbf{(A)}}}}}^{2}_{i}$ and ${\rm{\textbf{\rm{\textbf{(A)}}}}}^{3}_{i}$ \label{Etape_4.2_Existence}}
Note that $\sum_{i\, = \, 0}^{p_{n}- 1}{F'\left(\left\langle g\circ\tau_{-R\left(t_{i}^{n}\wedge t\right)}, Y_{t_{i}^{n}\wedge t} \right\rangle \right)}{\rm{\textbf{\rm{\textbf{(A)}}}}}^{3}_{i},$ is a stochastic integral  with respect to the square integrable martingale $\left( M^{\id}_{s}(\id)\right)_{0\leqslant s \leqslant T}$. Since $F'$ and $s \mapsto \left\langle g'\circ \tau_{-R\left(s \right)}, Y_{s} \right\rangle$ are bounded, we deduce that \[\lim_{n\to + \infty}\sum\limits_{i\, = \, 0}^{p_{n}- 1}{F'\left(\left\langle g\circ\tau_{-R\left(t_{i}^{n}\wedge t\right)}, Y_{t_{i}^{n}\wedge t} \right\rangle \right)}{\rm{\textbf{\rm{\textbf{(A)}}}}}^{3}_{i} \] is $\P_{\nu}^{\rm FV}-$martingale. The term \[\MM_{t}^{n} := \sum\limits_{i\, = \, 0}^{p_{n}- 1}{F'\left(\left\langle g\circ\tau_{-R\left(t_{i}^{n}\wedge t\right)}, Y_{t_{i}^{n}\wedge t} \right\rangle \right)}{\rm{\textbf{\rm{\textbf{(A)}}}}}^{2}_{i} \]
is a stochastic integral with respect to a martingale which depends on $n$. However, the same argument as above applies because $\left(\MM_{t}^{n}\right)_{0\leqslant t \leqslant T}$ is {  uniformly bounded in $L^{2}\left(\Omega \right)$ by {\rm{Lemma \ref{Lem_Q_omega_Martingale}}}, hence uniformly integrable.}

\subsubsection{Contributions of ${\rm{\textbf{\rm{\textbf{(A)}}}}}^{4}_{i}$ and ${\rm{\textbf{\rm{\textbf{(A)}}}}}^{5}_{i}$}

The contribution of the  next two terms corresponds to the terms due to the centering effect in the martingale problem (\ref{PB_Mg_Z_Multi_d}). \vspace{0.3cm} \\ 
\textbf{Study of the term ${\rm{\textbf{\rm{\textbf{(A)}}}}}^{4}_{i}$.} Using \textsc{It\^{o}}'s formula and the relation (\ref{Crochet_M_G_M_H_general}), we obtain that 
\begin{align*}
{\rm{\textbf{\rm{\textbf{(A)}}}}}^{4}_{i} &= {\rm{\textbf{\rm{\textbf{(A)}}}}}^{41}_{i} + {\rm{\textbf{\rm{\textbf{(A)}}}}}^{42}_{i} + {\rm{\textbf{\rm{\textbf{(A)}}}}}^{43}_{i}
\end{align*}
where
\begin{align*}
{\rm{\textbf{\rm{\textbf{(A)}}}}}^{41}_{i} & = -\int_{t_{i}^{n}\wedge t}^{t_{i+1}^{n}\wedge t}{\left[M_{s}^{\id}(\id) - M_{t_{i}^{n}\wedge t}^{\id}(\id)\right]\dd M_{s}^{\id}\left(g'\circ \tau_{-R\left(t_{i}^{n}\wedge t \right)}\right) }, \\
{\rm{\textbf{\rm{\textbf{(A)}}}}}^{42}_{i} & = -\int_{t_{i}^{n}\wedge t}^{t_{i+1}^{n}\wedge t}{\left[M_{s}^{\id}\left(g'\circ \tau_{-R\left(t_{i}^{n}\wedge t \right)}\right) - M_{t_{i}^{n}\wedge t}^{\id}\left(g'\circ \tau_{-R\left(t_{i}^{n}\wedge t \right)}\right)\right]\dd M_{s}^{\id}\left(\id\right) }, \\
{\rm{\textbf{\rm{\textbf{(A)}}}}}^{43}_{i} & = -2\gamma\int_{t_{i}^{n}\wedge t}^{t_{i+1}^{n}\wedge t}{\left[\left\langle \id \times g'\circ \tau_{-R\left(t_{i}^{n}\wedge t \right)}, Y_{s} \right\rangle - \left\langle \id , Y_{s} \right\rangle \left\langle g'\circ \tau_{-R\left(t_{i}^{n}\wedge t \right)} , Y_{s}\right\rangle \right]\dd s }.
\end{align*}
Using the same arguments as for ${\rm{\textbf{\rm{\textbf{(A)}}}}}^{2}_{i}$,  we deduce, in probability, that for $k \in \{41, 42 \}$, 
\[\lim_{n\to +\infty}{\sum\limits_{i\, = \, 0}^{p_{n}-1}{F'\left(\left\langle g\circ\tau_{-R\left(t_{i}^{n}\wedge t\right)}, Y_{t_{i}^{n} \wedge t} \right\rangle \right){\rm{\textbf{\rm{\textbf{(A)}}}}}^{k}_{i}} } \] 
is a $\P_{\nu}^{\rm FV}-$martingale. Moreover, we decompose the integral of ${\rm{\textbf{\rm{\textbf{(A)}}}}}^{43}_{i}$ in the following way: 
\begin{align*}
{\rm{\textbf{\rm{\textbf{(A)}}}}}^{43}_{i} & = {\rm{\textbf{\rm{\textbf{(A)}}}}}^{431}_{i} + {\rm{\textbf{\rm{\textbf{(A)}}}}}^{432}_{i} + {\rm{\textbf{\rm{\textbf{(A)}}}}}^{433}_{i}
\end{align*}
where
\begin{align*}
{\rm{\textbf{\rm{\textbf{(A)}}}}}^{431}_{i} & = -2\gamma \int_{t_{i}^{n}\wedge t}^{t_{i+1}^{n}\wedge t}{\left\langle \id \times g'\circ \tau_{-R\left(t_{i}^{n}\wedge t\right)} - \id \times g'\circ \tau_{-R\left(s\right)}, Y_{s} \right\rangle \dd s}, \\
{\rm{\textbf{\rm{\textbf{(A)}}}}}^{432}_{i} & = -2\gamma \int_{t_{i}^{n}\wedge t}^{t_{i+1}^{n}\wedge t}{\left\langle \id, Y_{s}  \right\rangle \left\langle g'\circ \tau_{-R\left(s\right)} -  g'\circ \tau_{-R\left(t_{i}^{n}\wedge t\right)}, Y_{s} \right\rangle \dd s}, \\
{\rm{\textbf{\rm{\textbf{(A)}}}}}^{433}_{i} & = -2\gamma \int_{t_{i}^{n}\wedge t}^{t_{i+1}^{n}\wedge t}{\left[\left\langle \id \times g'\circ \tau_{-R\left(s\right)}, Y_{s} \right\rangle - \left\langle \id, Y_{s}  \right\rangle \left\langle g'\circ \tau_{-R\left(s\right)}, Y_{s} \right\rangle \right]\dd s}.
\end{align*}
Using {\rm{Lemma \ref{Lem_convergence_proba}}}, we deduce, in probability, that for $k \in \{431, 432 \}$, \[\lim_{n\to +\infty}{\sum\limits_{i\, = \, 0}^{p_{n}-1}{F'\left(\left\langle g\circ\tau_{-R\left(t_{i}^{n}\wedge t\right)}, Y_{t_{i}^{n} \wedge t} \right\rangle \right){\rm{\textbf{\rm{\textbf{(A)}}}}}^{k}_{i}} } = 0,\]  
and we deduce from the convergence of \textsc{Riemann}'s sums that, $\P_{\nu}^{\rm FV}-$a.s. and hence in probability,
\[\lim_{n\to +\infty}{\sum\limits_{i\, = \, 0}^{p_{n}-1}{F'\left(\left\langle g\circ\tau_{-R\left(t_{i}^{n}\wedge t\right)}, Y_{t_{i}^{n} \wedge t} \right\rangle \right){\rm{\textbf{\rm{\textbf{(A)}}}}}^{433}_{i}} } =  -2\gamma\int_{0}^{t}{F'\left(\left\langle  g, Z_{s} \right\rangle \right) \left\langle g' \times \id, Z_{s} \right\rangle \dd s}.\]

\noindent \textbf{Study of the term ${\rm{\textbf{\rm{\textbf{(A)}}}}}^{5}_{i}$.} As ${\rm{\textbf{\rm{\textbf{(A)}}}}}^{5}_{i}$ satisfies the following decomposition: 
\begin{align*}
& {\rm{\textbf{\rm{\textbf{(A)}}}}}^{5}_{i} = \left(\int_{t_{i}^{n}\wedge t}^{t_{i+1}^{n}\wedge t}{\left[M^{\id}_{s}(\id) - M^{\id}_{t_{i}^{n}\wedge t}(\id) \right]\dd M_{s}^{\id}(\id)} \right. \\
& \hspace{4cm} + \left.  \gamma \int_{t_{i}^{n}\wedge t}^{t_{i+1}^{n}\wedge t}{\left[\left\langle \id^{2}, Y_{s} \right\rangle -\left\langle \id, Y_{s} \right\rangle^{2}  \right]\dd s}\right)\left\langle g''\circ \tau_{-R\left(t_{i}^{n}\wedge t \right)}, Y_{t_{i}^{n}\wedge t} \right\rangle
\end{align*}
and proceeding as for ${\rm{\textbf{\rm{\textbf{(A)}}}}}^{4}_{i}$ above, we obtain, in probability,  that
\[\lim_{n\to +\infty}{\sum\limits_{i\, = \, 0}^{p_{n}-1}{F'\left(\left\langle g\circ\tau_{-R\left(t_{i}^{n}\wedge t\right)}, Y_{t_{i}^{n} \wedge t} \right\rangle \right){\rm{\textbf{\rm{\textbf{(A)}}}}}^{5}_{i}} } = \gamma   \int_{0}^{T}{F'\left(\left\langle g, Z_{s\wedge t} \right\rangle\right)\left\langle g'', Z_{s\wedge t} \right\rangle M_{2}\left(Z_{s\wedge t}\right)\dd s} + {\rm{Mart}}_{t}^{(1)} \]
where $\left({\rm{Mart}}_{t}^{(1)} \right)_{0\leqslant t \leqslant T}$ is a $\P_{\nu}^{\rm FV}-$martingale.

\subsubsection{Study of the error term ${\rm{\textbf{\rm{\textbf{(A)}}}}}^{6}_{i}$}
From the inequality: for all $x, y \in \R_{+}$, $xy \leqslant \frac{2}{3}\left(x^{\frac{3}{2}} + y^{3} \right)$ and  {\rm{Lemma \ref{Lem_controle_erreur}}}, we deduce, in probability, that 
 \begin{multline*}
 \lim_{n\to + \infty}\sum\limits_{i\, = \, 0}^{p_{n} - 1}F'\left(\left\langle g \circ \tau_{-R\left(t_{i}^{n} \wedge t \right)}, Y_{t_{i}^{n} \wedge t} \right\rangle \right)\left|t_{i+1}^{n}\wedge t - t_{i}^{n} \wedge t \right| \left|M^{\id}_{t_{i+1}^{n}\wedge t}(\id) - M^{\id}_{t_{i}^{n} \wedge t}(\id) \right|  \\
 \leqslant \lim_{n\to + \infty}{\frac{2\left\|F' \right\|_{\infty}}{3} \left(\sum\limits_{i \, = \, 0}^{p_{n}-1}{\left| t_{i+1}^{n}\wedge t - t_{i}^{n}\wedge t\right|^{\frac{3}{2}}} + \sum\limits_{i \, = \, 0}^{p_{n}-1}{\left|M^{\id}_{t_{i+1}^{n}\wedge t}(\id) - M^{\id}_{t_{i}^{n} \wedge t}(\id) \right|^{3}}  \right)} = 0
 \end{multline*}
 and this completes the proof of {\rm{Lemma} \ref{Lemme_Termes_F'}}. \hfill $\square$ 

\subsubsection{Expressions of terms of (\ref{Existence_PB_mg}) involving $F''$ and error terms\label{Etape_5_Existence}}

{ Decomposing ${\rm{\textbf{(B)}}}_{i}$ in similar way as for ${\rm{\textbf{(A)}}}_{i}$, and using similar arguments, we can prove the following lemma:}

\begin{Lem}
When the mesh of the subdivision $0 = t_{0}^{n} < t_{1}^{n} <\dots < t_{p_{n}}^{n} = T$ of $[0,T]$ tends to $0$ when $n \to +\infty$, we obtain in probability that
\begin{multline*}
\lim_{n\to + \infty}{\sum\limits_{i\, = \, 0}^{p_{n}-1}{{\rm{\textbf{\rm{\textbf{(B)}}}}}_{i}}}   = \gamma\int_{0}^{t}{F''\left(\left\langle g, Z_{s} \right\rangle \right)\left[ \left\langle g^{2}, Z_{s} \right\rangle - \left\langle g,  Z_{s}\right\rangle^{2}  +  \left\langle g', Z_{s} \right\rangle^{2}M_{2}\left(Z_{s} \right)   \right.  } \\
  \left. \phantom{\left\langle g', Z_{s} \right\rangle^{2}} - 2\left\langle g', Z_{s} \right\rangle\left\langle g \times \id, Z_{s} \right\rangle \right]\dd s + \widehat{{\rm{Mart}}}_{t}
\end{multline*} 
where $\left({\widehat{{\rm{Mart}}}_{t}}\right)_{0\leqslant t \leqslant T}$ is a $\P_{\nu}^{\rm FV}-$martingale.
\label{Lemme_Termes_F''}
\end{Lem}

{  From Lemma {\rm{\ref{Lem_controle_erreur}}}, we deduce that the different error terms involved in the approximation (\ref{Existence_PB_mg}) converge in probability to $0$.}

\section{Proof of the results of Section \ref{Section_Unicite}\label{Section_preuve_Unicite}} 
{ 
Recall from (\ref{Eq_Bn_Multi_d}) that $B^{(n), d}$ is given for all $f \in \CCCC^{2}_{b}\left(\left(\R^{d}\right)^{n}, \R\right)$ by 
\[B^{(n), d}f(x) := \frac{\Delta f(x)}{2} - 2\gamma \sum_{k\, = \, 1}^{d}{\left(\nabla f(x) \cdot \varepsilon_{k} \right)\left(x\cdot \varepsilon_{k}\right)}\]
where $\varepsilon_{k} = 
\left(e_{k}, \cdots, e_{k}\right) \in \left(\R^{d}\right)^{n}$ and $e_{k}$ is the $k^{\rm th}$ vector of the canonical basis of $\R^{d}$. In this part, we justify the extension from the case $d = 1$ to the multidimensional case. To do this, we give a probabilistic interpretation of the semi-group $(T^{(n), d}(t))_{t\geqslant 0}$, generated by the operator $B^{(n), d}$, using the \textsc{Feynman-Kac} formula: for any $f \in \CCCC_{b}^{2}\left(\left(\R^{d} \right)^{n}, \R\right)$, \[T^{(n),d}(t)f(x) := \E_{x}f\left(X_{t}\right)\]
where $\left(X_{t} \right)_{t\geqslant 0}$ is solution to the following SDE: 
\[ X_{0} = x, \qquad \qquad \dd X_{t} = \dd B_{t} - 2\gamma b(X_{t})\dd t, \qquad X_{t} \in \R^{nd}, \quad t>0,\]
where $\left(B_{t}\right)_{t\geqslant 0} = (B_{t}^{(1)}, \cdots, B_{t}^{(n)} )_{t\geqslant 0}$ is a $nd-$standard Brownian motion, $b(x) = \left(b_{i,k}(x)\right)_{1 \leqslant i \leqslant n,\,1 \leqslant k \leqslant d} = (\sum_{j\, = \, 1}^{n}{x_{j,k}} )_{1 \leqslant i \leqslant n,\, 1 \leqslant k \leqslant d}$ where $x := (x_{1}, \cdots, x_{n} )^{t} \in (\R^{d})^{n}$ and $x_{i} := (x_{i,1}, \cdots, x_{i,d} )^{t} \in \R^{d}$. We shall denote $X_t=(X^{i,k}_t)_{1\leqslant i\leqslant n,\, 1\leqslant i\leqslant d}$. \\

Note that, for all $k \in \{1, \cdots, d \}$, $Y_{t}^{(k)} := (X_{t}^{1,k}, \cdots, X_{t}^{n,k} )$ are $d$ independent processes of semi-group $(T^{(n)}(t) )_{t\geqslant 0} := (T^{(n),1}(t) )_{t\geqslant 0}$, solutions to the following SDE: \[ \forall k \in \{1, \cdots, d\}, \qquad Y_{0}^{(k)} = x_{k}, \qquad \qquad \dd Y_{t}^{(k)} = \dd B_{t}^{(k)} - 2\gamma \left(Y_{t}^{(k)}\cdot \bm{1} \right)\bm{1}\dd t, \quad t>0\]
where $\bm{1} \in \R^{n}$ denotes the vector whose coordinates are all $1$ and $x \cdot y$ denotes the scalar product between $x$ and $y$ in $\R^{n}$. Considering a function $f\in \CCCC^{2}_{b}\left(\left(\R^{d} \right)^{n}, \R\right)$ with product form $f\left(x_{1}, \cdots, x_{n}\right) = \prod_{k\, = \, 1}^{n}{g_{k}\left(x_{k} \right)}$ with $g_{k} \in \CCCC^{2}_{b}\left(\R^{d}, \R \right)$, note that
\begin{equation}
    T^{(n),d}(t)f\left(x_{1}, \cdots, x_{n} \right) = \prod_{k\, = \, 1}^{d}{T^{(n), 1}(t)g_{k}(x_{k})} .
    \label{Eq_T_n_Multi_d}
\end{equation}
So, it is sufficient to study the case $d = 1$ since the general case can be deduced easily from (\ref{Eq_T_n_Multi_d}) following the same method as below. We leave the details to the reader.}

\subsection{Study of a semi-group \label{Sous_section_6_1}}

{  From now on, we assume that $d=1$.} In this section, we devote a specific study to the semi-group $(T^{(n)}(t))_{t\geqslant 0}$ generated by the operator {  $ B^{(n)} := B^{(n),1}$ defined for all $f \in \CCCC^{2}_{b}\left(\R^{n}, \R \right)$ by
\begin{equation}
B^{(n)}f(x) = \frac{\Delta f(x)}{2} - 2\gamma\left(\nabla f(x) \cdot \bm{1} \right)(x\cdot \bm{1}).
    \label{Operateur_B_n}
\end{equation}}
 In Section \ref{Sous_section_6_1_1}, we provide an explicit expression of $(t,x)\mapsto T^{(n)}(t)f(x)$ and prove that it is a strong solution to the semi-group PDE associated with $B^{(n)}$, by means of \textsc{Feynman-Kac}'s formula. With the aim of obtaining fairly fine bounds on this operator (see {\rm{Corollary \ref{Bornes_T_n_t}}}), we give all the necessary details. In Section \ref{Sous_section_6_1_2}, we give a MILD formulation of the martingale problem (\ref{PB_Mg_dual_FVr_Multi_d}) using the semigroup $(T^{(n)}(t) )_{t\geqslant 0}$ in  {\rm{Proposition \ref{Prop_Eq_MILD}}}. 

\subsubsection{Construction of the semi-group\label{Sous_section_6_1_1}}

For any real vector-valued function $f$ and $g$ of $L^{1}(\R^{n})$, we denote by $(f \ast g)(x) := \int_{\R^{n}}^{}{f(t)g(x-t)\dd t}$ the convolution product of $f$ and $g$. 
We denote by $\CCCC^{1,2}_{b}\left(\R_{+} \times \R^{n}, \R \right)$ the space of real functions on $\R_{+} \times \R^{n}$ of class $\CCCC^{1}(\R_{+}, \R)$ with respect to the first variable and of class $\CCCC^{2}_{b}(\R^{n}, \R)$ to the second variable.
\begin{Thm} The family of operators $\left(T^{(n)}(t) \right)_{t\geqslant 0}$ defined as:
\begin{align}
\forall t > 0, \ \forall x \in \R^{n}, \qquad T^{(n)}(t)f(x)   & := \int_{\R^{n}}^{}{f(u)g_{t,x}^{X}(u)\dd u}, \label{T_n_t_f} \\
 \forall x \in \R^{n}, \hspace{0.25cm} \qquad  T^{(n)}(0)f(x) & := f(x),  \notag 
\end{align}
where $g_{t,x}^{X}$ is the density of the Gaussian distribution $\NN^{(n)}\left(m_{t,x}, \Sigma_{t} \right)$ where $\Sigma_{t} := P\sigma_{t} P^{-1}$ and  $m_{t,x} := P\mu_{t,P^{-1}x} = x - \frac{\left(1-\exp\left(-2\gamma nt\right) \right)}{n}\left(x \cdot \bm{1} \right)\bm{1}$ 
 with \[ \mu_{t,y} := \begin{pmatrix}
y_{1}\exp\left(-2\gamma nt \right) \\ y_{2} \\ \vdots \\ y_{n} \end{pmatrix} \quad  \rm{and} \quad  \sigma_{t} := \begin{pmatrix}
e_{4}(t) & 0 & \dots  & \dots & 0 \\
0 & t & 0 & \dots & 0 \\
\vdots & 0 & \ddots & \ddots & \vdots \\
\vdots & \vdots & \ddots & \ddots & 0 \\
0 & 0 & \dots & 0 & t
\end{pmatrix},
 \] 
where $e_{4}(t) := \frac{1 - \exp(-4\gamma n t)}{4 \gamma n}$ and $P$ is an explicit change of orthonormal basis matrix defined in the proof below, is a semi-group of bounded operators on $L^{\infty}\left(\R^{n}\right)$. In addition, for all $f \in \CCCC^{2}_{b}(\R^{n}, \R)$, 
\begin{itemize}
\item[{\rm{\textbf{\rm{\textbf{(1)}}}}}] The application $(t,x) \mapsto T^{(n)}(t)f(x)$ is of class $\CCCC^{1,2}(\R_{+} \times \R^{n}, \R)$ and is a strong solution of the PDE 
\begin{align}
\forall t \geqslant  0,  \ \forall x \in \R^{n}, \qquad \partial_{t}u(t,x)   & = \frac{1}{2}\Delta u(t,x) -2\gamma \left(\nabla u(t,x)\cdot \bm{1} \right)\left(x\cdot\bm{1} \right)  \label{EDP_1} \\
 \forall x \in \R^{n}, \hspace{0.25cm} \qquad  u(0,x) & = f(x), \label{EDP_2}
\end{align}
 \end{itemize}
 and
\[\begin{array}{lrcl}
\textbf{\rm{\textbf{(2)}}} & \nabla T^{(n)}(t)f(x) & = & \left(\partial_{x_{i}}m_{t,x} \cdot \left(\nabla f \ast g_{t,0}^{X} \right)(m_{t,x}) \right)_{1\leqslant i \leqslant n}^{t} \vspace{0.1cm} \\
\textbf{\rm{\textbf{(3)}}} & \forall i, j \in \left\{1, \cdots, n \right\}, \quad \partial_{x_{i}x_{j}}^{2}T^{(n)}(t)f(x) & = & \left(\partial_{x_{j}}m_{t,x} \right)^{t}\left[\left(f \ast \Hess\left(g_{t, 0}^{X}\right) \right)(m_{t,x})\partial_{x_{i}}m_{t,x}\right].
\end{array}\]
\label{Thm_T_n_t}
\end{Thm}
As we will see in the proof, everything follows quite directly from the \textsc{Feynman-Kac} formula, except the fact that $(t,x) \mapsto T^{(n)}(t)f(x)$ is a strong solution of the PDE up to time $t = 0$. This technical point will be useful for the MILD formulation and this is why we make a detailed proof.

\begin{proof} In view of the operator $B^{(n)}$ given by (\ref{Operateur_B_n}), it is natural to define the semi-group $T^{(n)}(t)$ using the \textsc{Feynman-Kac} formula:  for any  $f \in \CCCC^{2}_{b}(\R^{n}, \R)$, \[T^{(n)}(t)f(x) := \E_{x}f(X_{t}) \]
 where $\left(X_{t} \right)_{t\geqslant 0}$ is solution to the following SDE: 
\begin{equation}
X_{0} = x, \qquad \qquad \dd X_{t} = \dd B_{t} - 2\gamma \left(X_{t}\cdot\bm{1}\right)\bm{1}\dd t, \qquad X_{t} \in \R^{n}, \quad t>0
 \label{EDS_1}
\end{equation}
where $\left(B_{t}\right)_{t\geqslant 0}$ is a $n-$standard Brownian motion and $x \in \R^{n}$. \\

In Step 1, we check that this definition of $T^{(n)}(t)$ coincides with the one given in the statement of Theorem \ref{Thm_T_n_t}. In Step 2, we verify that $(x,t) \mapsto \E_{x}f(X_{t})$ is indeed a solution of the PDE {\rm{(\ref{EDP_1})}} for all $t >0$. In Step 3, we treat the case $t = 0$. In Step 4, we prove the announced expressions of the derivatives of $T^{(n)}(t)f(x)$. \\ 

\textbf{Step 1. Change of basis in the SDE  (\ref{EDS_1}).} We consider the orthonormal basis $(v_{1}, \cdots, v_{n})$ of $\R^{n}$ defined by $v_{1} := \frac{1}{\sqrt{n}}(1 \ \cdots \ 1)^{t}$ and for $2 \leqslant i\leqslant n$, \begin{align*}
v_{i} & := \sqrt{\frac{i-1}{i} }\left(\underbrace{\frac{1}{i-1}, \cdots, \frac{1}{i-1}}_{i-1 \ \rm{terms}}, -1, 0, \cdots, 0  \right)^{t}. 
\end{align*}
We denote by $P$ the change of basis matrix from the canonical basis to the orthonormal basis $(v_{1}, \cdots, v_{n})$.
We define for all $t \geqslant 0$, $Z_{t} = P^{-1}X_{t}$, i.e. $Z_{t} := \left(Z_{t}^{(1)}, \cdots,  Z_{t}^{(n)}\right)$ where for all $i\in \{1, \cdots, n \}$, $Z_{t}^{(i)} := \left(X_{t}\cdot v_{i}\right)$. It is standard to check that $W_{t} := \left(W_{t}^{(1)}, \cdots, W_{t}^{(n)} \right)$ where for all $i\in \{1, \cdots, n \}$, $W_{t}^{(i)} := \left(B_{t}\cdot v_{i}\right)$ is a $n-$standard Brownian motion and that $\left(Z_{t} \right)_{t \geqslant 0}$ is solution to the SDE \begin{equation}
Z_{0} = y = P^{-1}x,  \qquad \qquad \left\{\begin{array}{rcll}
\dd Z_{t}^{(1)} & = & \dd W_{t}^{(1)} - 2\gamma n Z_{t}^{(1)} \dd t, &  \\
\dd Z_{t}^{(j)} & = & \dd W_{t}^{(j)}, & j \in \{2, \cdots, n \}
\end{array} \right. .
 \label{EDS_2}
 \end{equation} 
 All coordinates in (\ref{EDS_2}) are independent  and solve a one-dimensional SDE whose solution is explicit (\textsc{Ornstein-Uhlenbeck} for $Z^{(1)}$, standard Brownian motion for $Z^{(i)}, i\geqslant 2$). 
  It follows that  $Z_{t}$ is a Gaussian vector of law $\NN^{(n)}\left(\mu_{t,y}, \sigma_{t} \right)$. 
Therefore, for any $t >0$ and all $y \in \R^{n}$,  $\P_{Z_{t}}$ has a density with respect to the \textsc{Lebesgue} measure on $\R^{n}$ given by: 
\begin{equation}
g^{Z}_{t,y}(z_{1}, \cdots, z_{n}) = \frac{1}{(2 \pi)^{\frac{n}{2}}\sqrt{\det(\sigma_{t})} } \exp\left(-\frac{\left[z_{1} - y_{1}\exp\left(-2\gamma n t  \right) \right]^{2}}{2e_{4}(t)} - \frac{1}{2t}\sum\limits_{j\, = \, 2}^{n}{(z_{j} - y_{j})^{2}}   \right). 
\label{g_Z_preuve}
\end{equation}
Since, $X_{t} = PZ_{t}$, we deduce that for all $x \in \R^{n}$ and for all  $t>0$, $X_{t}$ follows the normal distribution $\NN^{(n)}\left(m_{t,x}, \Sigma_{t} \right)$, with density 
\begin{align}
g_{t, x}^{X}(r) & := g_{t, P^{-1}x}^{Z}\left(P^{-1}r\right)  = \frac{1}{(2 \pi)^{\frac{n}{2}}\sqrt{\det(\Sigma_{t})} }\exp\left(-\frac{(r-m_{t,x})^{t}\Sigma_{t}^{-1}(r-m_{t,x})}{2} \right). \label{Relation_g_X_et_g_Z}
\end{align}
Hence, $\E_{x}f\left(X_{t} \right)$ coincides with (\ref{T_n_t_f}). \\

\textbf{Step 2. $T^{(n)}(t)f$ is solution to (\ref{EDP_1}) on $(0, +\infty) \times \R^{n}$.} Without difficulty we verify that for any $y \in \R^{n}$, $g^{Z}_{t,y}$ satifies the following \textsc{Fokker-Planck} PDE: \begin{equation}
\forall t >0, \forall z \in \R^{n}, \qquad \partial_{t}g^{Z}_{t, y}(z)  = \frac{1}{2}\Delta_{y}g^{Z}_{t,y}(z) - 2 \gamma n \partial_{y_{1}}g^{Z}_{t,y}(z).  \label{EDP_pour_g_Z} 
\end{equation}
We deduce from (\ref{Relation_g_X_et_g_Z}) that \[\forall y \in \R^{n}, \forall r \in \R^{n}, \qquad \partial_{t}g_{t, y}^{Z}\left(P^{-1}r\right) = \partial_{t}g_{t,Py}^{X}(r),\]
and, for all $y, r \in \R^{n}$,  $\partial_{y_{i}}g_{t, y}^{Z}(P^{-1}r) = \sum\limits_{k\, = \, 1}^{n}{P_{ki}\partial_{x_{k}}g_{t, Py}^{X}(r)}$. In particular, 
\[\partial_{y_{1}}g_{t, y}^{Z}(P^{-1}r) = \frac{1}{\sqrt{n}}\left(\nabla_{x}g_{t,Py}^{X}(r)\cdot\bm{1}\right). \]
In an analogous way, we deduce that 
\begin{align*}
\Delta_{y}g_{t, y}^{Z}(P^{-1}r) & 
= \sum\limits_{i\, = \, 1}^{n}{\sum\limits_{k,\ell\, = \, 1}^{n}{P_{ki}P_{\ell i}\partial_{x_{\ell}, x_{k}}^{2} g_{t, Py}^{X}(r)} } = \Delta_{x}g_{t,Py}^{X}(r),
\end{align*}
because $P$ is an orthonormal matrix. From (\ref{EDP_pour_g_Z}) and since $\left(P^{-1}x \right)_{1} = \frac{\left(x\cdot\bm{1}\right)}{\sqrt{n}}$, we deduce that the density $g_{t,x}^{X}$ satisfies:
\begin{equation*}
\forall t>0, \forall x \in \R^{n}, \forall r \in \R^{n}, \qquad  \partial_{t}g_{t, x}^{X}(r) - \frac{1}{2}\Delta_{x}g_{t, x}^{X}(r) + 2\gamma (x\cdot\bm{1})\nabla_{x}g_{t, x}(r) = 0. \label{EDP_pour_g_X}
\end{equation*}
Now, the fact that for all $f \in L^{\infty}(\R^{n}, \R)$, \[T^{(n)}(t)f(x) = \int_{\R^{n}}^{}{f(r)g_{t,x}^{X}(r)\dd r} \] is $\CCCC^{1,2}_{b}\left((0, + \infty) \times \R^{n}, \R \right)$ and is a solution of   
(\ref{EDP_1}) on $(0, + \infty) \times \R^{n}$ follows from the theorem of differentiation under the integral sign. Note that, if $f$ is continuous, \[T^{(n)}(t)f(x) = \E_{x}f\left(X_{t}\right) \xrightarrow[t \to 0]{ } f(x) \] by the dominated convergence theorem which leads to (\ref{EDP_2}). \\

\textbf{Step 3. Verification of (\ref{EDP_1}) up to $t =0$.} Assume that $f \in \CCCC^{2}_{b}(\R^{n}, \R)$. This is equivalent to prove that for all $x \in \R^{n}$, \[\lim_{t\to 0}\frac{\int_{\R^{n}}^{}{f(u)g_{t,x}^{X}(u)\dd u} -f(x) }{t} =  \frac{1}{2}\Delta f(x) - 2\gamma \left(\nabla f(x) \cdot \bm{1} \right)\left(x \cdot \bm{1} \right).\]
Let $x \in \R^{n}$ be fixed. Using \textsc{Taylor}'s formula we obtain that \[\int_{\R^{n}}^{}{f(u)g_{t,x}^{X}(u)\dd u} -f(x) = {\textbf{\rm{\textbf{(A)}}}}_{t} + {\textbf{\rm{\textbf{(B)}}}}_{t} + {\textbf{\rm{\textbf{(C)}}}}_{t},\]
where 
\begin{center}
\begin{tabular}{ll}
${\textbf{\rm{\textbf{(A)}}}}_{t}  := \displaystyle{\int_{\R^{n}}^{}{\left([u-x] \cdot \nabla f(x) \right)g_{t,x}^{X}(u) \dd u},}$ \quad & ${\textbf{\rm{\textbf{(B)}}}}_{t}  := \displaystyle{\frac{1}{2}\int_{\R^{n}}^{}{(u-x)^{t}\Hess(f)(x)(u-x)g_{t,x}^{X}(u)\dd u},}$ \vspace{0.2cm}\\
${\textbf{\rm{\textbf{(C)}}}}_{t}  := \displaystyle{\int_{\R^{n}}^{}{R_{x}(u) g_{t,x}^{X}(u)\dd u},}$ &
\end{tabular}
\end{center}
where $R_{x}(u) := o\left(\left\|x-u \right\|_{2}^{2}\right)$. As $g_{t,x}^{X}$ is the density of $\NN^{(n)}\left(m_{t,x}, \Sigma_{t} \right)$, where 
\begin{equation}
 \begin{aligned}
x - m_{t,x}  = \left(1 - \exp\left( -2\gamma n t\right) \right) \frac{(x \cdot \bm{1})}{\sqrt{n}} \times \frac{\bm{1}}{\sqrt{n}} & \underset{t\to 0}{\sim} -2\gamma (x \cdot \bm{1}) t,  \\
 \forall i \in \{1, \cdots, n \}, \quad \left(\Sigma_{t} \right)_{ii} & \underset{t \to 0}{\sim} t, 
\end{aligned}
\label{m_t_x_Sigma_t}
\end{equation}
it follows that 
\begin{align*}
{\textbf{\rm{\textbf{(A)}}}}_{t} & 
=   \left(\nabla f(x) \cdot \left[m_{t,x} - x \right]\right) = -\frac{\left( 1 - \exp\left(-2\gamma nt \right)\right)}{n}\left(\nabla f(x) \cdot \bm{1} \right)\left( x \cdot \bm{1}\right). \\
{\textbf{\rm{\textbf{(B)}}}}_{t} & = \frac{1}{2} \int_{\R^{n}}^{}{\left(u - m_{t,x} \right)^{t}\Hess (f)(x) \left(u - m_{t,x} \right)g_{t,x}^{X}(u) \dd u} \\
& \hspace{1cm} + \frac{1}{2}\left(x - m_{t,x} \right)^{t}\Hess(f)(x) \left(x - m_{t,x} \right) \\
& = \frac{1}{2}\sum\limits_{i, j \, = \, 1}^{n}{\partial_{x_{i}x_{j}}^{2}{f(x)}\left(\Sigma_{t} \right)_{ij}} + \left(1 - \exp(-2\gamma n t) \right)^{2} \frac{\left(x \cdot \bm{1} \right)^{2}}{2n^{2}}\bm{1}^{t}\Hess(f)(x) \bm{1}.
\end{align*}
Therefore, \[\lim_{t\to 0}{\frac{{\textbf{\rm{\textbf{(A)}}}}_{t} + {\textbf{\rm{\textbf{(B)}}}}_{t}}{t} }  = - 2\gamma \left(\nabla f(x) \cdot \bm{1} \right)\left(x \cdot \bm{1} \right) + \frac{\Delta f(x)}{2}.\]
Now, it remains to manage the ${\textbf{\rm{\textbf{(C)}}}}_{t}$ error term. Note that, 
\begin{align*}
\forall \varepsilon >0, \exists \alpha >0, \forall u \in B(x,\alpha), \qquad \left|R_{x}(u)\right|  & \leqslant \varepsilon \left\|u-x \right\|_{2}^{2}. \\
\forall u \in \R^{n} \, \backslash \, B(x, \alpha), \qquad \left|R_{x}(u)\right| & \leqslant 2\left\|f \right\|_{\infty} + \left\|\nabla f \right\|_{\infty} \left\|u-x \right\|_{2} \\ 
& \hspace{1cm} + \frac{1}{2}\left\|\Hess f(x) \right\|_{\infty} \left\|u-x \right\|_{2}.
\end{align*}
Let $\varepsilon >0$, $\alpha >0$ and $t_{0} \geqslant 0$ such that for all $t \in [0, t_{0}]$,  $\left\|x - m_{t,x} \right\|_{2} \leqslant \frac{\alpha}{2}$. Let $t \in [0, t_{0}]$. Separating the domain of integration of the integral of ${\textbf{\rm{\textbf{(C)}}}}_{t}$ into $B(x,\alpha)$ and $\R^{n} \, \backslash \, B(x,\alpha)$, it follows from the \textsc{Young} and previous inequalities that there exists a constant $C >0$ such that
\begin{align*}
{\textbf{\rm{\textbf{(C)}}}}_{t} & \leqslant \varepsilon \int_{\R^{n}}^{}{\left\|u-x \right\|_{2}^{2}g_{t,x}^{X}(u)\dd u} + C \int_{\R^{n} \, \backslash \, B(x,\alpha)}{\left(1 + \left\| u-x \right\|_{2}^{2} \right)g_{t,x}^{X}(u) \dd u} \\
& \leqslant 2\varepsilon\sum_{i \, = \, 1}^{n}{\left(\Sigma_{t} \right)_{ii}} + 2(\varepsilon + C)\left\|x-m_{t,x}\right\|_{2}^{2} + 2C \int_{\R^{n} \, \backslash \, B(x,\alpha)}{\left(1 + \left\| u- m_{t,x} \right\|_{2}^{2} \right)g_{t,x}^{X}(u) \dd u}  
\end{align*}
Now, for the choice of $\alpha$ and then the \textsc{Markov} inequality, we obtain that 
\begin{align*}
 \int_{\R^{n} \, \backslash \, B(x,\alpha)}{\left(1 + \left\| u- m_{t,x} \right\|_{2}^{2} \right)g_{t,x}^{X}(u) \dd u} & \leqslant \left(1 + \frac{4}{\alpha^{2}} \right) \int_{\R^{n} \, \backslash \, B\left(m_{t,x},\frac{\alpha}{2} \right)}{\left\| u- m_{t,x} \right\|_{2}^{2} g_{t,x}^{X}(u) \dd u} \\
&  \leqslant  \left(1 + \frac{4}{\alpha^{2}} \right) \frac{\int_{\R^{n}}^{}{\left\|u-m_{t,x} \right\|_{2}^{4}g_{t,x}^{X}(u)\dd u}}{\left(\frac{\alpha}{2} \right)^{4}} \\
&  \leqslant \frac{16n}{\alpha^{4}}\left(1 + \frac{4}{\alpha^{2}} \right)\sum\limits_{i \, = \, 1}^{n}{\int_{\R^{n}}^{}{\left(u_{i} - \left(m_{t,x} \right)_{i} \right)^{4}g_{t,x}^{X}(u)\dd u}}. 
\end{align*}
As for all $i \in \{1, \cdots, n\}$, the fourth moment of a random variable of law $\NN\left(0, \left( \Sigma_{t}\right)_{ii} \right)$ is smaller than $3 \left(\Sigma_{t} \right)_{ii}^{2}$, it follows from (\ref{m_t_x_Sigma_t}) that there exists a constant $\widetilde{C} >0$ such that \[\frac{{\textbf{\rm{\textbf{(C)}}}}_{t}}{t} \leqslant 2\varepsilon n + \varepsilon  \widetilde{C} \]
for $t$ small enough and then the conclusion. \\

\textbf{Step 4. Expression of the derivatives of $T^{(n)}(t)f$.} Noting that for all $u \in \R^{n}$, $g_{t,x}^{X}(u) = g_{t, 0}^{X}(u-m_{t,x})$ and using the symmetry property of this density, we obtain that \[T^{(n)}(t)f(x) = \left(f\ast g_{t,0}^{X}\right)(m_{t,x}).\] By the chain rule formula, we deduce the properties {\rm{\textbf{\rm{\textbf{(2)}}}}} and for all $i, j \in \left\{1, \cdots, n \right\}$, 
\begin{align*}
& \partial_{x_{i}x_{j}}^{2}T^{(n)}(t)f(x) = \left(\partial_{x_{j}}m_{t,x} \right)^{t}\left[\left(f \ast \Hess\left(g_{t, 0}^{X}\right) \right)(m_{t,x})\partial_{x_{i}}m_{t,x}\right] \\
& \hspace{3.5cm} + \left(\partial_{x_{i}x_{j}}^{2}m_{t,x} \cdot \left( f \ast \nabla g_{t, 0}^{X}\right)(m_{t,x})\right).
\end{align*}
The property \textbf{(3)} follows.  \qedhere 
\end{proof}

The following corollary is useful for bounding the dual process in Section \ref{Sous_section_6_2}. 

\begin{Cor} Let $f \in \CCCC^{2}(\R^{n}, \R)$. We assume that there exists a constant $C_{1} >0$ such that for all $x \in \R^{n}$, \[|f(x)| \leqslant C_{1}\left(1 + \left\|x \right\|_{2}^{2n}  \right). \] Then, for all $t > 0$ and $x \in \R^{n}$, there exists two constants $C_{2}(t,n) >0$ locally bounded on $\R_{+} \times \N$ {  such that $t\mapsto C_{2}(t,n)$ is non-decreasing} and $C_{3}(t,n) >0$ locally bounded on $(0, + \infty) \times \N$ satisfying  \\
\[\begin{array}{lrcl}
\textbf{\rm{\textbf{(1)}}} & \left|T^{(n)}(t)f(x) \right| & \leqslant & C_{2}(t,n)\left(1 + \left\|x \right\|_{2}^{2n}  \right) \vspace{0.2cm} \\
\textbf{\rm{\textbf{(2)}}} & \left\|\left(\Hess\left(g_{t,0}^{X} \right) \ast f\right)\left(m_{t,x} \right) \right\| & \leqslant & C_{3}(t,n)\left(1 + \left\|x \right\|_{2}^{2n} \right)
\end{array}\]
\label{Bornes_T_n_t}
\end{Cor}
\begin{proof} \textbf{Step 1. Proof of (1).} From (\ref{g_Z_preuve}) and (\ref{Relation_g_X_et_g_Z}), note that for all $x, r \in \R^{n}$, $t \geqslant 0$, $g_{t,x}^{X}(r) =  \prod_{j\, = \, 1}^{n}{g_{t,\left[P^{-1}x\right]_{j}}^{Z_{j}}\left(\left[P^{-1}r\right]_{j}\right)},$
where 
\begin{equation}
\begin{aligned}
g_{t, y_{1}}^{Z^{(1)}}(z_{1}) & := \frac{1}{\sqrt{2\pi e_{4}(t)}}\exp\left(-\frac{ \left[z_{1} - y_{1}\exp\left(-2 \gamma nt \right)\right]^{2}}{2 e_{4}(t)}  \right), \\
g_{t, y_{j}}^{Z^{(j)}}(z_{j}) & := \frac{1}{\sqrt{2\pi t}}\exp\left(-\frac{\left[z_{j} - y_{j} \right]^{2}}{2t} \right), \qquad j \in \{ 2, \cdots, n\}. 
\end{aligned}
\label{g_Z_i_explicite}
\end{equation}
Since  $\|Pz\|_{2} = \|z\|_{2}$ for all $z \in \R^{n}$, we also have \[\int_{\R^{n}}^{}{\|u \|_{2}^{2n}g_{t,x}^{X}(u)\dd u} = \int_{\R^{n}}^{}{\|z \|_{2}^{2 n}g_{t,P^{-1}x}^{Z}(z)\dd z}.\]
Hence,
\begin{align*}
\int_{\R^{n}}^{}{\|u \|_{2}^{2n}g_{t,x}^{X}(u)\dd u} & \leqslant n^{n-1}\sum\limits_{i\, = \, 1}^{n}{ \int_{\R^{n}}^{}{z_{i}^{2n} \prod\limits_{j \, = \, 1}^{n}{g_{t, \left[P^{-1}x\right]_{j}}^{Z^{(j)}}(z_{j}) \dd z} }} =  n^{n-1}\sum\limits_{i\, = \, 1}^{n}{\E\left(\left[Z^{(i)}\right]^{2 n} \right)}.
\end{align*}
Classical moment bounds for Gaussian random variables show that $\E\left(G^{2n} \right) \leqslant C(n) t^{2n}$ for $G \sim \NN(0,t)$ and $C(n)>0$. Since $e_{4}(t) \leqslant t$ and using (\ref{g_Z_i_explicite}), we deduce that there exists two constants $\widetilde{C}_{1}(n)$ and 
$\widetilde{C}_{2}(t,n)$ such that for all $i \in \{1, \cdots, n \}$
\begin{align*}
\E\left(\left[Z^{(i)}\right]^{2 n} \right) & \leqslant \widetilde{C}_{1}(n) \left(\left(P^{-1}x \right)_{i}^{2n} + t^{2n} \right) \leqslant \widetilde{C}_{2}(t,n)\left(1 + \left\|P^{-1}x \right\|_{2}^{2n} \right).
\end{align*}
The result \textbf{(1)} follows. \\

\textbf{Step 2. Proof of (2).} Now, we want to control, for all $i, j \in \left\{1, \cdots, n \right\}$, \[\left|\left(\left(\Hess\left(g_{t,0}^{X} \right)\ast f \right)\left(m_{t,x} \right) \right)_{ij} \right|  \leqslant C_{1}(n)\int_{\R^{n}}^{}{\left|\partial_{r_{i}r_{j}}^{2}g_{t,0}^{X}(r) \right|\left(1 + \left\|m_{t,x} - r\right\|_{2}^{2n} \right)\dd r}.\]
 For all $k \in \left\{1, \cdots, n \right\}$, we consider: \[V_{k}(t) := \left\{ \begin{array}{ccc}
e_{4}(t) & {\rm{if}} & k = 1 \\
t  & {\rm{if}} & k \neq 1
\end{array}\right. .\] 
From (\ref{g_Z_i_explicite}), we deduce that for all $i, j, k \in \left\{1, \cdots, n \right\}$, for all $t > 0$, 
\begin{align*}
\partial_{r_{i}}g_{t, 0}^{Z^{(k)}}\left(\left(P^{-1}r \right)_{k} \right) & = - \frac{\left(P^{-1}\right)_{ki}\left(P^{-1}r\right)_{k}}{V_{k}(t)}g_{t,0}^{Z^{(k)}}\left(\left(P^{-1}r \right)_{k} \right), \\
\partial_{r_{j}r_{i}}^{2}g_{t, 0}^{Z^{(k)}}\left(\left(P^{-1}r \right)_{k} \right) & = \frac{\left(P^{-1} \right)_{kj}\left(P^{-1} \right)_{ki}}{V_{k}(t)} \left(\frac{\left(P^{-1}r \right)_{k}^{2}}{V_{k}(t)} - 1 \right)g_{t,0}^{Z^{(k)}}\left(\left(P^{-1}r \right)_{k} \right).
\end{align*}
Hence, for all $i, j \in \left\{1, \cdots, n \right\}$, for all $t > 0$, \begin{align*}
 \partial_{r_{j}r_{i}}^{2}g_{t, 0}^{X}\left(r\right) 
& = \sum\limits_{k\, = \, 1}^{n}{\frac{\left(P^{-1} \right)_{kj}\left(P^{-1} \right)_{ki}}{V_{k}(t)} \left(\frac{\left(P^{-1}r \right)_{k}^{2}}{V_{k}(t)} - 1 \right)g_{t,0}^{X}\left(r \right)} \\
& \hspace{1.25cm} + \sum\limits_{k\, = \, 1}^{n}{\sum\limits_{\substack{\ell \, = \, 1 \\ \ell \, \neq \, k}}^{n}{\frac{\left(P^{-1} \right)_{kj}\left(P^{-1} \right)_{\ell j}}{V_{k}(t)V_{\ell}(t)}\left(P^{-1}r \right)_{k}\left(P^{-1}r \right)_{\ell} g_{t,0}^{X}\left(r \right)} }
\end{align*}
Noting that for all $i, j, k \in \left\{1, \cdots, n \right\}$, $\left|\left(P^{-1} \right)_{ij} \right| \leqslant 1$ and $\left|\left(P^{-1}r \right)_{k} \right| \leqslant n \left\|r \right\|_{2}$, \[\left\|m_{t,x} - r \right\|_{2}^{2n} \leqslant 2^{2n-1}\left(\left\| r\right\|_{2}^{2n} +  2^{2n-1}n^{n}\left(2 - \exp\left(-2\gamma n t \right) \right)\left\|x \right\|_{2}^{2n}\right) , \] 
we deduce that for all $i, j \in \left\{1, \cdots, n \right\}$, there exists a constant $\widetilde{C}_{3}(t,n)>0$ locally bounded on $(0, + \infty) \times \N$ such that 
\begin{align*}
& \int_{\R^{n}}^{}{\left|\partial_{r_{i}r_{j}}^{2}g_{t,0}^{X}(r)\right|\left\|m_{t,x} - r \right\|_{2}^{2n}\dd r} \\ 
& \hspace{0.5cm}\leqslant \int_{\R^{n}}^{}{\left[\frac{1}{e_{4}(t)}\left(\frac{n^{2}\left\|r \right\|_{2}^{2}}{e_{4}(t)} + 1 \right) + \frac{n-1}{t}\left(\frac{n^{2}\left\|r \right\|_{2}^{2}}{t} + 1 \right) + n^{3}\left(\frac{1}{e_{4}(t)} + \frac{n-1}{t} \right)\left\|r \right\|_{2}^{2} \right]} \\
&  \hspace{2cm} \times g_{t,0}^{X}(r) \left\|m_{t,x} - r \right\|_{2}^{2n} \dd r \\
& \hspace{0.5cm}\leqslant \widetilde{C}_{3}(t,n)\left(1 + \left\|x \right\|_{2}^{2n} \right).
\end{align*}
The announced result \textbf{(2)} follows.
\end{proof}

\subsubsection{MILD formulation\label{Sous_section_6_1_2}}
In this section, we establish the MILD formulation associated with the martingale problem (\ref{PB_Mg_dual_FVr_Multi_d}) {  with $d = 1$}. 
\begin{Prop}
Let $\left(X_{t}\right)_{t\geqslant 0}$ be a stochastic process whose law $\P_{\mu}$ is solution to the martingale problem  {\rm{(\ref{PB_Mg_Z_Multi_d})}} with initial value $\mu$. Then, for all $f \in \CCCC^{2}_{b}(\R^{n}, \R)$, 
\begin{align*}
 \left\langle T^{(n)}(t_{0}-t)f, X_{t}^{n} \right\rangle & - \gamma \int_{0}^{t}{\sum\limits_{i\, = \, 1}^{n}{\sum\limits_{\substack{j \, = \, 1 \\ j \, \neq \, i}}^{n}{\left[\left\langle \Phi_{i,j}T^{(n)}(t_{0}-s)f, X_{s}^{n-1} \right\rangle - \left\langle T^{(n)}(t_{0}-s)f, X_{s}^{n} \right\rangle \right]\dd s} } } \\
&  - \gamma \int_{0}^{t}{\sum\limits_{i\, = \, 1}^{n}{\sum\limits_{\substack{j \, = \, 1 }}^{n}{\left\langle K_{i,j}T^{(n)}(t_{0}-s)f, X_{s}^{n+1} \right\rangle}\dd s } }  
\end{align*} 
is a $\P_{\mu}-$ martingale for $0\leqslant t \leqslant t_{0}$. 
\label{Prop_Eq_MILD}
\end{Prop}

\begin{proof} Let $t_{0} \geqslant 0$. {  Using (\ref{Eq_Fonction_polynome_mu}), let} $u, v, w : [0, t_{0}] \times \MM_{1}^{c,2}(\R) \times \widetilde{\Omega} \rightarrow \R$ be $\BB([0, t_{0}])\otimes \BB\left(\MM_{1}^{c,2}(\R)\right) \otimes \widetilde{\FF}-$measurable defined by 
\begin{center}
\begin{tabular}{ll}
$\bullet \ u(r,\mu) :=\left\langle T^{(n)}\left(t_{0} - r \right)f, \mu^{n} \right\rangle$,  & \hspace{+0.3cm} $\bullet \   v(r,\mu) := - \left\langle \partial_{t}T^{(n)}(t_{0}-r)f, \mu^{n} \right\rangle$, \vspace{0.15cm} \\
 $\bullet \ w(r, \mu) := { \LL_{\rm FVc}P_{T^{(n)}\left(t_{0}-r \right)f, n}\left(\mu\right)}$. \vspace{0.2cm} & \\ 
\end{tabular}
\end{center}
The expected result is a direct consequence of a version of {\color{blue} \cite[\color{black} Lemma 4.3.4]{Ethier_markov_1986}} where we replace the assumption of boundedness on $w$ by an assumption of domination{ : $\exists C >0, \forall t \in \left[0, t_{0} \right], \forall \mu \in \MM_{1}^{c,2}(\R), \ w(t,\mu) \leqslant C\left(1 + M_{2}(\mu) \right)$. Note that we have control on these moments (see Proposition \ref{Prop_moments_non_bornes_RECENTRE}). The verification of the assumptions of this lemma is left to the reader.} \qedhere  
\end{proof}

\subsection{Proof of Lemma \ref{Lemme_Theta_k} \label{Sous_section_6_2}}
Recall that, our goal is to prove{ , in dimension $1$,} that the stopping time $\theta_{k}$, defined by \begin{equation*}
\forall k \in \N, \qquad \theta_{k} := \inf{\left\{ t \geqslant 0\left. \phantom{1^{1^{1^{1}}}} \hspace{-0.6cm} \right| M(t)  \geqslant k \quad {\rm{or}} \quad \exists s \in [0,t], \ \left\langle  \xi_{s}, X_{t-s}^{M(s)}  \right\rangle \geqslant k \right\}}, 
\end{equation*}
satisfies $\lim_{k \to + \infty}{\theta_{k}} = + \infty $, $\P_{\left(\mu, \xi_{0} \right)}-$a.s. with $\mu \in \MM_{1}^{c,2}(\R)$ and $\xi_{0} \in \CCCC_{b}^{2}\left(\R^{M(0)}, \R \right)$. 
Before to prove Lemma \ref{Lemme_Theta_k}, we introduce the following lemma, whose proof will be given at Section \ref{Sous_section_6_2_2}. We denote by $S_{t}$ the number of jumps of the process $M$ on the time interval $[0,t]$.

\begin{Lem}  If $\xi_{0} \in \CCCC^{2}_{b}(\R^{n}, \R)$ then there exists a function 
$C_{0}$ on $\bigcup_{k \in \N}{(0, +\infty)^{k} \times \{k \}}$ to $\R_{+}$,  locally bounded,  {  such that for all $\left(t_{j} \right)_{j\in \N} \in (0,+\infty)^{\N}$, $k \mapsto C_{0}\left(\left(t_{i} \right)_{0\leqslant i \leqslant k}, k \right)$ is non-decreasing and satisfying
\[\forall t \in [0, T], \ \forall x \in \R^{M(t)}, \qquad \left|\xi_{t}(x) \right| \leqslant C_{0}\left(\tau_{1}, \tau_{2} - \tau_{1}, \cdots, \tau_{S_{T} + 1} - \tau_{S_{T}}, S_{T} \right)\left(1 + \left\|x \right\|_{2}^{2S_{T}}  \right). \]}
\label{Bornes_dual}
\end{Lem}

\vspace{-0.5cm}

The bound obtained above will only allow us to show that  $\theta_{k} \to +\infty$  $\P_{(\mu, \xi_{0})}-$a.s. under the assumption that the initial condition $X_{0}$ has all its finite moments. The following remark shows that we cannot expect that $\theta_{k} \to +\infty$ under weaker assumptions on the initial condition.

{  \begin{Rem} Let $\xi_{0} : x \mapsto \sin(x) \in \CCCC^{2}_{b}(\R, \R)$ and $\mu \in \MM_{1}^{c,2}(\R)$. Let us assume that $\left\langle \left|\id \right|^{4}, \mu  \right\rangle = + \infty$ and $\xi_{t}$ successively jumps at times $\tau_{1}$, $\tau_{2}$ and $\tau_{3}$ with respective jump operator $K_{11}$, $K_{11}$ and $\Phi_{13}$. If we denote by $\tau_{1,2} := \tau_{2} - \tau_{1}$, straightforward (but tedious) computations give  
\[ \xi_{\tau_{1}}(x,y) = K_{11}T^{(1)}(\tau_{1})\xi_{0}(x,y) = -\exp\left(-\frac{e_{4}(\tau_{1})}{2} - 4\gamma \tau_{1} \right)\sin\left(x\exp\left(-2\gamma \tau_{1}\right) \right)y^{2}\]
and the leading order term in $\xi_{\tau_{2}}(x,y,z)$ is of the form $\left(ax - by \right)^{2}z^{2}\sin\left(cx + dy \right)$.
Now, \[\xi_{\tau_{3}}(x,y) = \Phi_{13}T^{(3)}\left(\tau_{3} - \tau_{2} \right)\xi_{\tau_{2}}(x,y).\] If $\tau_{3} = \tau_{2}$, we obtain as leading order term in $\xi_{\tau_{3}}(x,y)$ the term $\left(ax - by \right)^{2}x^{2}\sin\left(cx + dy \right)$, which is not integrable with respect to $\mu^{2}(\dd x, \dd y)$. If $\tau_{3} > \tau_{2}$, one can check that the leading order term in $T^{(3)}(\tau_{3} - \tau_{2})\xi_{\tau_{2}}(x,y,z)$ is of the form $P_{4}(x,y,z)\sin\left(\widetilde{c}x + \widetilde{d}y + \widetilde{e}z  \right)$ where $P_{4}(X,Y,Z)$ is a homogeneous polynomial of degree $4$ such that $P_{4}(X,Y,Z) \to \left(a X - bY \right)^{2}Z^{2}$, $\widetilde{c} \to c$, $\widetilde{d} \to d$ and $\widetilde{e} \to 0$ when $\tau_{3} \to \tau_{2}$. Therefore, for $\tau_{3}$ close enough to $\tau_{2}$, $\xi_{\tau_{3}}(x,y)$ has a non-zero term proportional to $x^{4}\sin\left(\left[\widetilde{c} + \widetilde{e}\right]x + \widetilde{d}y \right)$ which is not compensated by another term. Hence, $\left\langle \left|\xi_{\tau_{3}}\right|, \mu^{2} \right\rangle = +\infty$ if $\tau_{3} - \tau_{2}$ is small enough, for any values of $\tau_{1}$ and $\tau_{2}$. Given $T$ large enough, we have proved that $\theta_{k} \leqslant \tau_{3} \leqslant T$ with positive probability.
 \label{Remarque_contre_exemple}
\end{Rem}
}

\subsubsection{Proof that Lemma \ref{Bornes_dual} implies Lemma \ref{Lemme_Theta_k} \label{Main_proof_theta_k}}
Note that $\theta_{k} = \widehat{\theta}_{k} \wedge \widetilde{\theta}_{k}$ where 
\begin{align*}
\widehat{\theta}_{k} & := \inf{\left\{ t \geqslant 0\left. \phantom{1^{1^{1^{1}}}} \hspace{-0.6cm} \right| M(t)  \geqslant k \right\}} \quad {\rm{and}} \quad \widetilde{\theta}_{k}  := \inf{\left\{ t \geqslant 0\left. \phantom{1^{1^{1^{1}}}} \hspace{-0.6cm} \right| \exists s \in [0,t], \ \left\langle  \xi_{s}, X_{t-s}^{M(s)}  \right\rangle \geqslant k  \right\}}.
\end{align*}
Thanks to {(\ref{Non-explosion})} it follows that $\widehat{\theta}_{k} \rightarrow +\infty$ when $k \rightarrow +\infty$. In order to prove that  $\widetilde{\theta}_{k} \rightarrow  +\infty$ when $k \rightarrow +\infty$, we rely on the control of the dual process obtained in {\rm{Lemma \ref{Bornes_dual}}}. So we need to control $\left\langle \left\|\cdot \right\|_{2}^{2S_{T}}, X_{t-s}^{M(s)} \right\rangle$. Let $T>0$ and $\varepsilon>0$ be arbitrary. From  {(\ref{Non-explosion})}, we choose $A:= A(T, \varepsilon) >0$ such that $\P_{(\mu, \xi_{0})}\left( S_{T} \leqslant A \right) \geqslant 1 - \varepsilon/3$. 
Then, using Proposition \ref{Prop_moments_non_bornes_RECENTRE}, we choose $B := B(T, \varepsilon, A)>0$ such that $\P_{(\mu, \xi_{0})}\left(\forall k \leqslant 2A, \ \forall t \leqslant T, \ \left\langle \left|\id\right|^{k}, X_{t} \right\rangle \leqslant B\right) \geqslant 1 - \varepsilon/3$. Finally, from Lemma \ref{Bornes_dual} we choose $\overline{C}_{0}:= \overline{C}_{0}\left(T, \varepsilon, A \right) >0$ such that $\P_{(\mu, \xi_{0})}\left(C_{0}\left(\left(\tau_{i+1} - \tau_{i} \right)_{0\leqslant i \leqslant A}, A \right) \leqslant \overline{C}_{0} \right) \geqslant 1- \varepsilon/3$.
We recall that for any $m\in \N^{\star}$, for all $x \in \R^{m}$, $\left(\sum_{i\, = \, 1}^{m}{x_{i}}  \right)^{n} \leqslant m^{n-1}\sum_{i\, = \, 1}^{m}{x_{i}^{n}}$.
Thus, the following inequality
\begin{align*}
\left\langle \left\|\cdot \right\|_{2}^{2S_{T}}, X_{t-s}^{M(s)} \right\rangle \leqslant M(s)^{S_{T}-1} \sum\limits_{i\, = \, 1}^{M(s)}{\int_{\R^{M(s)}}^{}{x_{i}^{2S_{T}}X_{t-s}^{M(s)}(\dd x)} } & = M(s)^{S_{T}}\left\langle \id^{2 S_{T}}, X_{t-s} \right\rangle  \\ 
& \leqslant \left(M(0) + A \right)^{A}B,
\end{align*}
takes place with probability $1-2\varepsilon/3$. Therefore, we deduce from {\rm{Lemma \ref{Bornes_dual}}} that for all $s \leqslant t \leqslant T$, \[ \left\langle \xi_{s}, X_{t-s}^{M(s)} \right\rangle  \leqslant \overline{C}_{0}\left(M(0) + A \right)^{A}B\] 
In particular, for $k \geqslant \overline{C}_{0}\left(1 +  \left(M(0) + A \right)^{A}B \right)$, it follows that 
\begin{align*}
\P_{\left(\mu, \xi_{0}\right)}\left(\widetilde{\theta}_{k}\geqslant T \right) & \geqslant \P_{\left(\mu, \xi_{0}\right)}\left(\left\{ S_{T} \leqslant A \right\} \cap \left\{\forall k \leqslant 2A, \ \forall t\leqslant T, \ \left\langle \id^{k}, X_{t} \right\rangle \leqslant B \right\}  \right. \\
& \hspace{3.45cm} \left. \cap \left\{C_{0}\left(\left(\tau_{i+1} - \tau_{i} \right)_{0\leqslant i \leqslant A}, A \right) \right\} \right) \\
&  \geqslant 1 - \varepsilon.
\end{align*}
The conclusion follows. \hfill $\square$ 

\subsubsection{Proof of Lemma \ref{Bornes_dual} \label{Sous_section_6_2_2}}

By mathematical induction on $k\in \N$, we prove the property 
 \begin{align*}
(\PP_{k}) : \quad &  \forall t \in \left[\tau_{k}, \tau_{k+1} \right[, \ \forall x \in \R^{M(t)}, \quad \left|\xi_{t}(x) \right| \leqslant C_{0}(\left(\tau_{i+1} - \tau_{i}\right)_{0\leqslant i\leqslant k}, k) \left(1 + \left\|x \right\|_{2}^{2k}  \right),
\end{align*} 
where $C_{0}$ is locally bounded on $\bigcup_{k\in \N}{\left(0, +\infty\right)^{k} \times \{ k\}}$. \\

\noindent  \textbf{Initial case.} For $k = 0$, $S_{0} = 0$ and $\xi_{0} \in \CCCC^{2}_{b}(\R^{n}, \R)$. Hence, the property $(\PP_{0})$ is satisfied. \vspace{0.1cm} \\
 \textbf{Inductive step.} We assume that, for $k \in \N^{\star}$, the property $(\PP_{k-1})$ is satisfied and prove that $(\PP_{k})$ is also. Let $t \in \left[\tau_{k}, \tau_{k+1} \right[$ and note that $M(t) = M\left(\tau_{k} \right)$. We make a partition of cases according to whether the dual process loses or gains a variable. {  Let $i, j \in \left\{1, \cdots, M\left(\tau_{k-1} \right) \right\}$ be fixed.} \\ 
 
\textbf{Step 1. Case $\Lambda_{k} = \Phi_{i,j}$ at the $k^{\rm{th}}$ jump.} In this case, $M(\tau_{k}) = M(\tau_{k-1}) - 1$ and we deduce from  the explicit expression{ (\ref{Processus_dual_explicite})} of the dual process that {  for all $x \in \R^{M\left(\tau_{k-1}  \right)-1}$,} 
\begin{align*}
\xi_{t}(x) & = T^{\left(M(\tau_{k-1}) - 1\right)}(t-\tau_{k})\Phi_{i,j}T^{\left(M(\tau_{k-1})\right)}(\tau_{k}-\tau_{k-1})\xi_{\tau_{k-1}}(x). 
\end{align*}
By using expression (\ref{Expression_Phi_ij}) of $\Phi_{i,j}$ and the property $\left(\PP_{k-1} \right)$, we deduce from {\rm{Corollary \ref{Bornes_T_n_t}} \textbf{(1)}} that for all $x \in \R^{M\left(\tau_{k-1} \right)-1}$, 
\begin{align*}
& \left|\Phi_{i,j}T^{\left(M\left(\tau_{k-1} \right) \right)}\left(\tau_{k} - \tau_{k-1} \right)\xi_{\tau_{k-1}}(x) \right| \\
& \hspace{2cm} \leqslant C_{2}\left(\tau_{k} - \tau_{k-1}, M\left(\tau_{k-1}\right) \right){C}_{0}\left(\left(\tau_{i+1} - \tau_{i} \right)_{0\leqslant i \leqslant k-1}, k-1 \right)\left(1 + \left\|x \right\|_{2}^{2(k-1)} \right),
\end{align*}
where $C_{2}{C}_{0}$ is locally bounded.
Using again {\rm{Corollary \ref{Bornes_T_n_t}} \textbf{(1)}} and the fact that $t\mapsto C_{2}\left(t, M\left(\tau_{k-1} \right)\right)$ is non-decreasing, we deduce the property $\left(\PP_{k} \right)$. \\

\textbf{Step 2. Case $\Lambda_{k} = K_{i,j}$ at the $k^{\rm{th}}$ jump.}
In this case, $M(\tau_{k}) = M(\tau_{k-1}) + 1$ and the explicit expression (\ref{Processus_dual_explicite}) of dual process that {  for all $x \in \R^{M\left(\tau_{k-1} \right) +1}$,}
\begin{align*}
\xi_{t}(x) & = T^{\left(M(\tau_{k-1})+1\right)}(t-\tau_{k})K_{i,j}T^{\left(M(\tau_{k-1})\right)}(\tau_{k}-\tau_{k-1})\xi_{\tau_{k-1}}(x).
\end{align*}

\noindent From the expression (\ref{Operateur_K_ij}) of $K_{i,j}$ and Theorem \ref{Thm_T_n_t} \textbf{(3)}, we have for all $x \in \R^{M\left(\tau_{k-1} \right) + 1}$,
\begin{align*}
& \left|K_{i,j}T^{\left(M\left(\tau_{k-1} \right) + 1 \right)}\left(\tau_{k} - \tau_{k-1} \right)\xi_{\tau_{k-1}}(x) \right| \\
& \hspace{0.25cm} = \left|\left(\partial_{x_{j}}m_{\tau_{k} - \tau_{k-1},\widetilde{x}} \right)^{t}\left[\left(\xi_{\tau_{k-1}} \ast \Hess\left(g_{\tau_{k} - \tau_{k-1}, 0}^{X}\right) \right)(m_{\tau_{k} - \tau_{k-1}, \widetilde{x}})\partial_{x_{i}}m_{\tau_{k} - \tau_{k-1}, \widetilde{x}}\right] \right|x_{M\left(\tau_{k-1} \right) + 1}^{2},
\end{align*}
where $\widetilde{x} = \left(x_{1}, \cdots, x_{M\left(\tau_{k-1} \right)} \right)^{t} \in \R^{M\left(\tau_{k-1} \right)}$. From the property $\left(\PP_{k-1} \right)$ and {\rm{Corollary \ref{Bornes_T_n_t}} \textbf{(2)}}, we deduce that  
\begin{align*}
& \left|K_{i,j}T^{\left(M\left(\tau_{k-1} \right) + 1 \right)}\left(\tau_{k} - \tau_{k-1} \right)\xi_{\tau_{k-1}}(x) \right| \\ 
& \hspace{2cm} \leqslant C_{3}\left(\tau_{k} - \tau_{k-1}, M\left(\tau_{k-1}\right) \right) C_{0}\left(\left(\tau_{i+1} - \tau_{i} \right)_{0\leqslant i \leqslant k-1}, k-1 \right)   \left(1 + \left\|x\right\|_{2}^{2k} \right),
\end{align*}
where $C_{3}{C}_{0}$ is locally bounded.
Using {\rm{Corollary \ref{Bornes_T_n_t}} \textbf{(1)}}, we deduce the property $\left(\PP_{k} \right)$. We conclude by the principle of induction. \hfill $\square$

\subsection{Proof of Theorem \ref{Thm_Identite_dualite_affaiblie} \label{Sous_section_6_3}}

Recall that $\left(X_{t}\right)_{t\geqslant 0}$ is a stochastic process whose law $\P_{\mu}$ is a solution of the martingale problem (\ref{PB_Mg_Z_Multi_d}) with $\mu \in \MM_{1}^{c,2}(\R)$ and $\left(\xi_{t} \right)_{t\geqslant 0}$ a dual process independent of $\left(X_{t}\right)_{t\geqslant 0}$ built on the same probability space. To simplify, we will note $\P = \P_{(\mu, \xi_{0})}$ the distribution of $\left(\left(X_{t}, \xi_{t}\right)\right)_{t\geqslant 0}$. As $\xi_{0} \in \CCCC^{2}_{b}(\R^{M(0)}, \R)$  and for the choice of the stopping time $\theta_{k}$ given by {\rm{(\ref{Theta_k})}}, the set of quantities, involved in the expectations of the weakened duality identity (\ref{Identite_dualite_affaiblie}), are bounded. \\ 

\textbf{Step 1. Approximation reasoning.} To establish the relation (\ref{Identite_dualite_affaiblie}) we introduce a increasing sequence $0 = t_{0}^{n} < t_{1}^{n} < \dots < t_{p_{n}}^{n} = t$ of subdivisions of $[0,t]$ such that $t_{i+1}^{n} = t_{i}^{n}  + h$ with $h$ tending to $0$. Note that 
\begin{align*}
& \E\left(\left\langle \xi_{t \wedge \theta_{k}} , X_{0}^{M(t\wedge \theta_{k})} \right\rangle\exp\left(\gamma \int_{0}^{t\wedge \theta_{k}}{M^{2}(u) \dd u} \right) \right) - \E\left(\left\langle \xi_{0} , X^{M(0)}_{t\wedge \theta_{k}} \right\rangle\right) \\
& \hspace{ 0.75cm} = \sum\limits_{i\, = \, 0}^{p_{n}-1}{\left[\E\left(\left\langle \xi_{t_{i+1}^{n}\wedge \theta_{k}}, X_{t\wedge \theta_{k} - t_{i+1}^{n}\wedge \theta_{k}}^{M\left(t_{i+1}^{n}\wedge \theta_{k}\right)} \right\rangle\exp\left(\gamma \int_{0}^{t_{i+1}^{n}\wedge \theta_{k}}{M^{2}(u)\dd u}  \right)  \right)   \right.} \\
& \hspace{2.5cm} - \left. \E\left(\left\langle \xi_{t_{i}^{n}\wedge \theta_{k}}, X_{t\wedge \theta_{k} - t_{i}^{n}\wedge \theta_{k}}^{M\left(t_{i}^{n}\wedge \theta_{k}\right)} \right\rangle\exp\left(\gamma \int_{0}^{t_{i}^{n}\wedge \theta_{k}}{M^{2}(u)\dd u}  \right)   \right)\right]. 
\end{align*} 
We are therefore interested in terms of the form 
\begin{equation}
\begin{aligned}
& \E\left(\left\langle \xi_{(s+h)\wedge \theta_{k}}, X_{t\wedge \theta_{k} - (s + h)\wedge \theta_{k}}^{M\left((s+h)\wedge \theta_{k}\right)} \right\rangle\exp\left(\gamma \int_{0}^{(s+h)\wedge \theta_{k}}{M^{2}(u)\dd u}  \right)  \right) \\
& \hspace{0.5cm} - \E\left(\left\langle \xi_{s\wedge \theta_{k}}, X_{t\wedge \theta_{k} - s\wedge \theta_{k}}^{M\left(s\wedge \theta_{k}\right)} \right\rangle\exp\left(\gamma \int_{0}^{s\wedge \theta_{k}}{M^{2}(u)\dd u}  \right)   \right), \quad {\rm{for}} \ s \in \left[0, t-h \right].
\end{aligned}
 \label{Difference_terme_dual}
\end{equation}
It is sufficient to prove that these quantities are $O\left(h^{2}\right)$. The procedure to be adopted is as follows. First of all, we consider separately the two terms which constitute (\ref{Difference_terme_dual}) in Steps 2 and 3. Then, we prove that the sum of these terms is $O\left(h^{2}\right)$ in Steps 4 and 5. Throughout this section, in order to simplify the writing, the following notations are introduced:
\[t_{k} := t\wedge \theta_{k}, \qquad s_{k} := s\wedge \theta_{k}, \qquad {\rm{and}} \qquad  s^{h}_{k} := (s+h)\wedge \theta_{k}.\]
We note respectively $\tau_{1}$, $\tau_{2}$ the first and second jump times after $s_{k}$ for the process $M$. We denote by $\tau_{1,k} := \tau_{1} \wedge \theta_{k}$ and  $\tau_{2,k} := \tau_{2} \wedge \theta_{k}$. \\

\textbf{Step 2. First term of 
 (\ref{Difference_terme_dual})}. We exploit the explicit expression (\ref{Processus_dual_explicite}) of the dual process and make the following partition:
\begin{itemize}
\item[\textbf{(a)}] If there has been \emph{no} jump of $M$ on the interval $\left[s_{k}, s^{h}_{k}\right]$.
\item[\textbf{(b)}] If there was \emph{only one} jump of $M$ on the interval $\left[s_{k}, s^{h}_{k}\right]$ and distinguish according to the events $\{\Lambda = \Phi_{i,j}\}$ and $\{\Lambda = K_{i,j}\}$ where $\Lambda$ is the first $\Lambda_{k}$ defined by (\ref{Saut_Phi_i_j}) and (\ref{Saut_K_i_j}) after $s_{k}$. 
\item[\textbf{(c)}] If there are \emph{two or more than two} jumps of $M$ on the interval $\left[s_{k}, s^{h}_{k}\right]$. 
\end{itemize}
Then 
\begin{align*}
& \E\left(\left\langle \xi_{s_{k}^{h}}, X_{t_{k}- s_{k}^{h}}^{M\left(s_{k}^{h}\right)} \right\rangle\exp\left(\gamma \int_{0}^{s_{k}^{h}}{M^{2}(u)\dd u}  \right) \left| \phantom{1^{1^{1^{1^{1^{1}}}}}} \hspace{-0.8cm} \widetilde{\FF}_{s_{k}} \right.  \right) \notag\\
& \hspace{0.5cm}  = \E\left(\left\langle \xi_{s^{h}_{k}}, X_{t_{k}- s^{h}_{k}}^{M\left(s^{h}_{k}\right)} \right\rangle\exp\left(\gamma \int_{0}^{s^{h}_{k}}{M^{2}(u)\dd u}  \right) \II_{\left\{\tau_{1,k} > s^{h}_{k} \right\}}  \left| \phantom{1^{1^{1^{1^{1^{1}}}}}} \hspace{-0.8cm} \widetilde{\FF}_{s_{k}} \right.  \right) \notag \\ 
&  \hspace{1cm} + \sum\limits_{\substack{i, j \, = \, 1 \\  i \, \neq \, j}}^{M(s_{k})}\E\left(\left\langle \xi_{s^{h}_{k}}, X_{t_{k}- s^{h}_{k}}^{M\left(s^{h}_{k}\right)} \right\rangle\exp\left(\gamma \int_{0}^{s^{h}_{k}}{M^{2}(u)\dd u}  \right)   \II_{\left\{\tau_{1,k} \leqslant s^{h}_{k}, \tau_{2,k} > s_{k}^{h},   \Lambda = \Phi_{i,j}\right\}}  \left| \phantom{1^{1^{1^{1^{1^{1}}}}}} \hspace{-0.8cm} \widetilde{\FF}_{s_{k}} \right.  \right) \notag \\ 
& \hspace{1cm} + \sum\limits_{\substack{ i , j \, = \, 1 }}^{M(s_{k})}\E\left(\left\langle \xi_{s^{h}_{k}}, X_{t_{k}- s^{h}_{k}}^{M\left(s^{h}_{k}\right)} \right\rangle\exp\left(\gamma \int_{0}^{s^{h}_{k}}{M^{2}(u)\dd u}  \right)  \II_{\left\{\tau_{1,k} \leqslant s^{h}_{k}, \tau_{2,k} > s^{h}_{k},  \Lambda = K_{i,j}\right\}}  \left| \phantom{1^{1^{1^{1^{1^{1}}}}}} \hspace{-0.8cm} \widetilde{\FF}_{s_{k}} \right.  \right) \notag \\ 
&  \hspace{1cm}+ \E\left(\left\langle \xi_{s^{h}_{k}}, X_{t_{k}- s^{h}_{k}}^{M\left(s^{h}_{k}\right)} \right\rangle\exp\left(\gamma \int_{0}^{s^{h}_{k}}{M^{2}(u)\dd u}  \right)\II_{\left\{\tau_{2,k} \leqslant s^{h}_{k} \right\}}  \left| \phantom{1^{1^{1^{1^{1^{1}}}}}} \hspace{-0.8cm} \widetilde{\FF}_{s_{k}} \right.   \right). \notag 
\end{align*}

{  We only deal with the second term, where there is a single jump with $\Lambda = \Phi_{i,j}$. The other terms are treated in a similar way.} \\

If there is a jump on $[s_{k}, s^{h}_{k}]$ and for $i \neq j \in \left\{1, 2, \cdots, M\left(s_{k}\right)  \right\}$ fixed, $\Lambda = \Phi_{i,j}$, then $M\left(s^{h}_{k}\right) = M(s_{k}) - 1$ and $\xi_{s^{h}_{k}}  = T^{\left(M(s_{k}) -1 \right)}\left(s^{h}_{k}-\tau_{1,k}\right)\Phi_{i,j} T^{\left(M(s_{k}) \right)}\left(\tau_{1,k} - s_{k} \right)\xi_{s_{k}}$. Thus,
\begin{align*}
& \E\left(\left\langle \xi_{s^{h}_{k}}, X_{t_{k}- s^{h}_{k}}^{M\left(s^{h}_{k}\right)} \right\rangle\exp\left(\gamma \int_{0}^{s^{h}_{k}}{M^{2}(u)\dd u}  \right)   \II_{\left\{\tau_{1,k} \leqslant s^{h}_{k}, \tau_{2,k} > s_{k}^{h},   \Lambda = \Phi_{i,j}\right\}}  \left| \phantom{1^{1^{1^{1^{1^{1}}}}}} \hspace{-0.8cm} \widetilde{\FF}_{s_{k}} \right.  \right)  \\
& \hspace{0.5cm} =\left\langle T^{\left(M\left(s_{k}\right) -1 \right)}\left(s^{h}_{k} -\tau_{1,k}\right)\Phi_{i,j} T^{\left(M\left(s_{k}\right) \right)}\left(\tau_{1,k} - s_{k}  \right)\xi_{s_{k}}, X_{t_{k} - s_{k}^{h}}^{M\left(s_{k}\right) - 1} \right\rangle \phantom{\exp\left(\gamma \int_{0}^{s}{M^{2}(u)\dd u} \right)} \\
& \hspace{1.5cm}  \times \exp\left(\gamma \int_{0}^{s_{k}}{M^{2}(u)\dd u} + \gamma \left[\tau_{1,k} - s_{k}\right] M^{2}(s_{k}) + \gamma\left[s^{h}_{k} -\tau_{1,k}\right]\left(M(s_{k}) - 1\right)^{2}  \right)   \\
& \hspace{1.5cm} \times   \II_{\left\{\tau_{1,k} - s_{k} \leqslant s^{h}_{k} - s_{k}\right\}}\E\left[ \left. \II_{\left\{\Lambda = \Phi_{i,j}\right\}} \II_{\left\{\tau_{2,k} -\tau_{1,k} > s^{h}_{k} - \tau_{1,k} \right\}}  \right| \sigma\left(\tau_{1,k} \right) \vee \widetilde{\FF}_{s_{k}} \right]. 
\end{align*}
Now, using that, given $\Lambda = \Phi_{i,j}$ and $\sigma\left(\tau_{1,k} \right) \vee \widetilde{\FF}_{s_{k}}$,  $\tau_{2,k} - \tau_{1,k}$ follows an exponential law of parameter $\gamma \left(M\left(s_{k}\right) - 1 \right)^{2} + \gamma\left(M\left(s_{k}\right) - 1 \right)\left(M\left(s_{k}\right) - 2 \right) $,
we deduce that
\begin{align*}
& \P\left(\left\{\Lambda = \Phi_{i,j} \right\} \cap  \left\{ \tau_{2,k} - \tau_{1,k}  > s_{k}^{h} - \tau_{1,k} \right\} \left| \phantom{1^{1^{1^{1}}}} \hspace{-0.7cm} \right. \sigma\left(\tau_{1,k} \right) \vee \widetilde{\FF}_{s_{k}} \right) \notag\\
&  \hspace{1cm} = \frac{\exp\left(-\gamma\left[\left(M\left(s_{k} \right) - 1 \right)^{2} + \left(M\left(s_{k} \right) - 1 \right)\left(M\left(s_{k}\right) - 2 \right)^{\phantom{2}} \hspace{-0.15cm} \right]\left[s_{k}^{h} -\tau_{1,k} \right] \right) }{M^{2}\left(s_{k} \right) + M\left(s_{k} \right)\left(M\left(s_{k} \right) - 1 \right)}.  \notag 
\end{align*} 
Then using that, given $\sigma\left(\tau_{1,k} \right) \vee \widetilde{\FF}_{s_{k}}$,  $\tau_{1,k} - s_{k}$ follows an exponential law of parameter $\gamma M^{2}\left(s_{k}\right) + \gamma M\left(s_{k}\right) \left(M\left(s_{k}\right)-1 \right)$, we deduce that 
\begin{align*}
 & \E\left(\left\langle \xi_{s^{h}_{k}}, X_{t_{k}- s^{h}_{k}}^{M\left(s^{h}_{k}\right)} \right\rangle\exp\left(\gamma \int_{0}^{s^{h}_{k}}{M^{2}(u)\dd u}  \right)   \II_{\left\{\tau_{1,k} \leqslant s^{h}_{k}, \tau_{2,k} > s_{k}^{h},   \Lambda = \Phi_{i,j}\right\}}  \left| \phantom{1^{1^{1^{1^{1^{1}}}}}} \hspace{-0.8cm} \widetilde{\FF}_{s_{k}} \right.  \right)  \\
 & \hspace{0.5cm} =  \gamma {\int_{s_{k}}^{s^{h}_{k}}{ \left\{\left\langle T^{\left(M(s_{k}) -1 \right)}\left(s^{h}_{k} - r \right)\Phi_{i,j}T^{\left(M(s_{k}) \right)}\left(r - s_{k}\right)\xi_{s_{k}}, X_{t_{k} - s_{k}^{h}}^{M\left(s_{k}\right) - 1} \right\rangle \exp\left(\gamma \int_{0}^{s_{k}}{M^{2}(u) \dd u} \right)  \phantom{\int_{0}^{s^{h}_{k} - s_{k}}{} } \right. } }  \\
& \hspace{+1.25cm} \times  \left. \exp\left(- 2\gamma \left(r-s_{k}\right)\left(M(s_{k}) - 1 \right) -  \gamma \left(s^{h}_{k} - s_{k}\right)\left(M(s_{k}) - 1 \right) \left(M(s_{k}) - 2 \right)\phantom{\int_{s_{k}}^{s^{h}_{k}}} \hspace{-0.7cm} \right) \right\}\dd r. 
\end{align*}

\textbf{Step 3. Second term of (\ref{Difference_terme_dual}).} It follows from the MILD formulation of {\rm{Proposition \ref{Prop_Eq_MILD}}} that 
\begin{align}
& \E\left(\left\langle \xi_{s_{k}}, X_{t_{k}-s_{k}}^{M(s_{k})} \right\rangle   \exp\left(\gamma\int_{0}^{s_{k}}{M^{2}(u)\dd u}  \right)\left| \phantom{1^{1^{1^{1^{1^{1}}}}}} \hspace{-0.8cm} \widetilde{\FF}_{s_{k}} \right. \right) \notag \\
 & \hspace{0.05cm} = \E\left(\left\langle T^{(M(s_{k}))}(s^{h}_{k}-s_{k})\xi_{s_{k}}, X_{t_{k}-s^{h}_{k}}^{M(s_{k})} \right\rangle \exp\left(\gamma\int_{0}^{s_{k}}{M^{2}(u)\dd u}  \right) \left| \phantom{1^{1^{1^{1^{1^{1}}}}}} \hspace{-0.8cm} \widetilde{\FF}_{s_{k}} \right.  \right)  \notag \\
& \hspace{0.3cm} + \gamma \sum\limits_{\substack{i, j \, = \, 1 \\ i \, \neq \, j}}^{M(s_{k})} \E\left(\int_{0}^{s^{h}_{k} - s_{k}}{\left(\left\langle \Phi_{i,j} T^{\left(M(s_{k}) \right)}(r)\xi_{s_{k}}, X_{t_{k}-s_{k}-r}^{M(s_{k})-1} \right\rangle \right.}  \right. \notag \\ 
& \hspace{1.8cm} - \left.\left.\left\langle T^{\left(M(s_{k}) \right)}(r) \xi_{s_{k}}, X_{t_{k}-s_{k}-r}^{M(s_{k})} \right\rangle \right)\dd r \exp\left(\gamma\int_{0}^{s_{k}}{M^{2}(u)\dd u}  \right) \left| \phantom{1^{1^{1^{1^{1^{1}}}}}} \hspace{-0.8cm} \widetilde{\FF}_{s_{k}} \right.  \right) \notag \\ 
& \hspace{0.3cm} + \gamma\sum\limits_{ i , j \, = \, 1}^{M(s_{k})} \E\left( \int_{0}^{s^{h}_{k} - s_{k}}{\left\langle K_{i,j}T^{(M(s_{k}))}(r)\xi_{s_{k}}, X_{t_{k}-s_{k}-r}^{M(s_{k})+1} \right\rangle \dd r}  \exp\left(\gamma\int_{0}^{s_{k}}{M^{2}(u)\dd u}  \right) \left| \phantom{1^{1^{1^{1^{1^{1}}}}}} \hspace{-0.8cm} \widetilde{\FF}_{s_{k}} \right. \hspace{-1cm} \phantom{\sum\limits_{ i , j \, = \, 1}^{M(s_{k})}} \right).  \notag 
\end{align}

\textbf{Step 4. Conclusion.}   Putting together all the previous equations, we deduce that
\begin{align*}
& \E\left(\left\langle \xi_{(s+h)\wedge \theta_{k}}, X_{t\wedge \theta_{k} - (s + h)\wedge \theta_{k}}^{M\left([s+h]\wedge \theta_{k}\right)} \right\rangle\exp\left(\gamma \int_{0}^{[s+h]\wedge \theta_{k}}{M^{2}(u)\dd u}  \right)  \right) \\
& \hspace{3.5cm} - \E\left(\left\langle \xi_{s\wedge \theta_{k}}, X_{t\wedge \theta_{k} - s\wedge \theta_{k}}^{M\left(s\wedge \theta_{k}\right)} \right\rangle\exp\left(\gamma \int_{0}^{s\wedge \theta_{k}}{M^{2}(u)\dd u}  \right)   \right) \\
& \hspace{0.5cm}  = {\textbf{\rm{\textbf{(A)}}}} + {\textbf{\rm{\textbf{(B)}}}} + {\textbf{\rm{\textbf{(C)}}}} + {\textbf{\rm{\textbf{(D)}}}} + {\textbf{\rm{\textbf{(E)}}}} + O(h^{2}),
\end{align*}
where 
\begin{align*}
{\textbf{\rm{\textbf{(A)}}}} & = \E\left(\left\langle T^{(M(s_{k}))}\left(s^{h}_{k} - s_{k} \right)\xi_{s_{k}}, X_{t_{k} - s^{h}_{k}}^{M(s_{k})} \right\rangle  \right. \\
& \hspace{0.5cm} \times \left. \exp\left(\gamma \int_{0}^{s_{k}}{M^{2}(u)\dd u}   -\gamma\left[s^{h}_{k} - s_{k} \right]M\left(s_{k} \right)\left(M(s_{k}) - 1 \right)  \right) \right)  \\ 
 & \hspace{+0.5cm} + \gamma \E\left(\int_{0}^{s^{h}_{k} - s_{k}}{\sum\limits_{\substack{i, j \, = \, 1 \\ i \, \neq \, j}}^{M(s_{k})}{\left\langle T^{(M(s_{k}))}(r)\xi_{s_{k}}, X_{t_{k} - s_{k} - r}^{M(s_{k})} \right\rangle}\dd r }\exp\left(\gamma \int_{0}^{s_{k}}{M^{2}(u) \dd u} \right)  \right), \\ 
 {\textbf{\rm{\textbf{(B)}}}} & =  \gamma \E\left(\sum\limits_{\substack{i, j \, = \, 1 \\ i \, \neq \, j}}^{M(s_{k})}{\int_{0}^{s^{h}_{k} - s_{k}}{\left\{\left\langle T^{(M(s_{k}) - 1)}\left(s^{h}_{k} - s_{k} - r \right)\Phi_{i,j}T^{(M(s_{k}))}(r)\xi_{s_{k}}, X_{t_{k} - s^{h}_{k}}^{M(s_{k}) - 1} \right\rangle \right.} }  \right.  \\
 & \hspace{-0.15cm} \left. \phantom{\int_{0}^{s_{k}}} - \left.   \left\langle T^{(M(s_{k}) - 1)}\left(0 \right)\Phi_{i,j}T^{(M(s_{k}))}(r)\xi_{s_{k}}, X_{t_{k} - s_{k} - r}^{M(s_{k}) - 1} \right\rangle \right\}\exp\left(\gamma \int_{0}^{s_{k}}{M^{2}(u) \dd u}  \right) \phantom{\sum\limits_{\substack{i, j \, = \, 1 \\ i \, \neq \, j}}^{M(s_{k})}} \hspace{-1cm}\right), \\
{\textbf{\rm{\textbf{(C)}}}} & =  \gamma \E\left(\sum\limits_{ i , j \, = \, 1}^{M(s_{k})}{\int_{0}^{s^{h}_{k} - s_{k}}{\left\{\left\langle T^{(M(s_{k}) + 1)}\left(s^{h}_{k} - s_{k} - r \right)K_{i,j}T^{(M(s_{k}))}(r)\xi_{s_{k}}, X_{t_{k} - s^{h}_{k}}^{M(s_{k}) + 1} \right\rangle \right.} }  \right.  \\
 & \hspace{-0.15cm} \left. \phantom{\int_{0}^{s_{k}}} - \left.   \left\langle T^{(M(s_{k}) + 1)}\left(0 \right)K_{i,j}T^{(M(s_{k}))}(r)\xi_{s_{k}}, X_{t_{k} - s_{k} - r}^{M(s_{k}) + 1} \right\rangle \right\}\exp\left(\gamma \int_{0}^{s_{k}}{M^{2}(u) \dd u}  \right) \phantom{\sum\limits_{ i , j \, = \, 1}^{M(s_{k})}} \hspace{-0.95cm} \right),  \\
 {\textbf{\rm{\textbf{(D)}}}} & =  \gamma \E\left(\sum\limits_{\substack{i, j \, = \, 1 \\ i \, \neq \, j}}^{M(s_{k})}{\int_{0}^{s^{h}_{k} - s_{k}}{\left\{\left\langle T^{(M(s_{k}) - 1)}\left(s^{h}_{k} - s_{k} - r \right)\Phi_{i,j}T^{(M(s_{k}))}(r)\xi_{s_{k}}, X_{t_{k} - s^{h}_{k}}^{M(s_{k}) - 1} \right\rangle \right.} }  \right.  \\
 & \hspace{-0.15cm} \left. \phantom{\int_{0}^{s_{k}}} \times \left. \exp\left(-2\gamma r \left[M(s_{k} ) - 1\right] - \gamma \left(s^{h}_{k} - s_{k}  \right)\left[M(s_{k}) -1  \right] \left[M(s_{k}) -2  \right]  \right) \phantom{1^{1^{1^{1^{1}}}}} \hspace{-0.8cm}\right\} \dd r \right),  \\ 
{\textbf{\rm{\textbf{(E)}}}} & =  \gamma \E\left(\sum\limits_{ i , j \, = \, 1}^{M(s_{k})}{\int_{0}^{s^{h}_{k} - s_{k}}{\left\{\left\langle T^{(M(s_{k}) + 1)}\left(s^{h}_{k} - s_{k} - r \right)K_{i,j}T^{(M(s_{k}))}(r)\xi_{s_{k}}, X_{t_{k} - s^{h}_{k}}^{M(s_{k}) +1} \right\rangle \right.} }  \right.  \\
 & \hspace{-0.15cm} \left. \phantom{\int_{0}^{s_{k}}} \times \left. \exp\left(-2\gamma r M(s_{k} )  - \gamma \left(s^{h}_{k} - s_{k}  \right)M(s_{k}) \left[M(s_{k}) -1  \right]  \right) \phantom{1^{1^{1}}} \hspace{-0.5cm}\right\} \dd r \right).  
\end{align*}

All these terms are $O\left(h^{2} \right)$ uniformly with respect to the other parameters which is sufficient to conclude the proof. This can be proved similarly for each term, so we only give the details for {\textbf{\rm{\textbf{(A)}}}}. We first notice that when $h \to 0$, 
\begin{align*}
{\textbf{\rm{\textbf{(A)}}}} & = \E\left(M(s_{k}) \left(M(s_{k}) - 1 \right)\int_{0}^{s_{k}^{h}-s_{k}}{\left\{\left\langle T^{(M(s_{k}))}\left(s^{h}_{k} - s_{k} \right)\xi_{s_{k}}, X_{t_{k} - s^{h}_{k}}^{M(s_{k})} \right\rangle \right. }  \right. \\
& \hspace{0.5cm} \left. \phantom{\int_{0}^{s_{k}}}  - \left. \left\langle T^{(M(s_{k}))}(r)\xi_{s_{k}}, X_{t_{k} - s_{k}^{\phantom{?}} - r}^{M(s_{k})} \right\rangle \right\} \dd r \exp\left(\gamma \int_{0}^{s_{k}}{M^{2}(u)\dd u}  \right) \right) + O\left(h^{2} \right).
\end{align*}
By conditioning with respect to $\widetilde{\FF}_{s_{k}}$ in the previous expression and using the MILD formulation  of the {\rm{Proposition \ref{Prop_Eq_MILD}}} again we obtain that 
\begin{align*}
& {\textbf{\rm{\textbf{(A)}}}} = \gamma\E\left(\exp\left(\gamma \int_{0}^{s_{k}}{M^{2}(u)\dd u}  \right) M(s_{k}) \left(M(s_{k}) - 1 \right)\int_{0}^{s^{h}_{k} - s_{k}}{}\left[\int_{t_{k} - s_{k} - r}^{t_{k} - s^{h}_{k}}{\left\{ \phantom{\sum\limits_{11}^{}{}} \right.} \right.\right. \\
& \hspace{0.25cm} \sum\limits_{\substack{i, j \, = \, 1 \\ i \, \neq \, j}}^{M(s_{k})}{\left[\left\langle \Phi_{i,j}T^{(M(s_{k}))}\left(t_{k} - s_{k} - v \right)\xi_{s_{k}}, X_{v}^{M(s_{k}) - 1} \right\rangle -\left\langle T^{(M(s_{k}))}\left(t_{k} - s_{k} - v \right)\xi_{s_{k}}, X_{v}^{M(s_{k})} \right\rangle \right] \hspace{-0.75cm}\phantom{\int_{0}^{s^{1}}} }  \\ 
& \hspace{0.25cm} + \left.\left.\left.\sum\limits_{ i , j \, = \, 1}^{M(s_{k})}{\left\langle K_{i,j}T^{(M(s_{k}))}(t_{k} - s_{k} - v) \xi_{s_{k}}, X_{v}^{M(s_{k})+1} \right\rangle}  \right\}\dd v\right]\dd r  \right) + O\left(h^{2} \right).
\end{align*}
Both integrals are on intervals of length at most $h$. Since all the quantities in the integrals are bounded by definition  of the stopping time $\theta_{k}$, we deduce that ${\textbf{\rm{\textbf{(A)}}}} \underset{h\to 0}{=} O(h^{2})$. \hfill $\square$

\appendix
\makeatletter
\def\@seccntformat#1{Annexe~\csname the#1\endcsname:\quad}
\@addtoreset{equation}{section}
  \renewcommand\theequation{\thesection.\arabic{equation}}
\makeatother

\section{Technical lemmas involved in the existence proof}

\subsection{Approximation lemma \label{Sous_sect_lemmes_preuve_existence}}

{ 
\begin{Lem}
Let $p, d \in \N^{\star}$, $\mu, \nu \in \MM_{1}\left(\R^{d}\right)$, $F \in \CCCC_{b}^{3}\left(\R^{p}, \R\right)$ and $g=(g_{1},\cdots,  g_{p}) \in \CCCC^{3}_{b}\left(\R^{d}, \R^{p}\right)$. Then, 
\begin{align*}
& F\left(\left\langle g\circ\tau_{-\left\langle \id, \mu \right\rangle}, \mu \right\rangle \right) - F\left(\left\langle g\circ\tau_{-\left\langle \id, \mu \right\rangle}, \nu \right\rangle \right) \\
& \hspace{0.5cm} = \sum_{i\, = \, 1}^{p}{\partial_{i}F\left(\left\langle g\circ\tau_{-\left\langle \id, \nu \right\rangle}, \nu \right\rangle \right)} \left\{\left\langle g_{i}\circ\tau_{-\left\langle \id, \nu \right\rangle}, \mu - \nu \right\rangle - \left\langle \nabla g_{i}\circ\tau_{-\left\langle \id, \nu \right\rangle}, \nu \right\rangle \cdot \left\langle \id, \mu-\nu \right\rangle \phantom{\frac{1}{2}} \right. \\
& \hspace{1.5cm} - \left.\left\langle \nabla g_{i}\circ\tau_{-\left\langle \id, \nu \right\rangle}, \mu - \nu \right\rangle \cdot \left\langle \id, \mu-\nu \right\rangle + \frac{1}{2}\left\langle \id, \mu - \nu \right\rangle^{t} \left\langle \Hess(g_{i})\circ\tau_{-\left\langle \id, \nu \right\rangle}, \nu \right\rangle \left\langle \id, \mu - \nu \right\rangle\right\} \\
& \hspace{0.75cm} + \frac{1}{2}\sum_{i,j \, = \, 1}^{p}{\partial_{ij}^{2}F\left(\left\langle g\circ\tau_{-\left\langle \id, \nu \right\rangle}, \nu \right\rangle \right)}\left\{\left\langle g_{i}\circ\tau_{-\left\langle \id, \nu \right\rangle}, \mu - \nu \right\rangle \left\langle g_{j}\circ\tau_{-\left\langle \id, \nu \right\rangle}, \mu - \nu \right\rangle \phantom{\frac{1}{2}} \right. \\
& \hspace{1.5cm} - \left\langle g_{i}\circ\tau_{-\left\langle \id, \nu \right\rangle}, \mu - \nu \right\rangle \left\langle \nabla g_{j}\circ\tau_{-\left\langle \id, \nu \right\rangle}, \nu \right\rangle \cdot \left\langle \id, \mu - \nu \right\rangle \\
& \hspace{1.5cm} - \left\langle g_{j}\circ\tau_{-\left\langle \id, \nu \right\rangle}, \mu - \nu \right\rangle \left\langle \nabla g_{i}\circ\tau_{-\left\langle \id, \nu \right\rangle}, \nu \right\rangle \cdot \left\langle \id, \mu - \nu \right\rangle \\
& \hspace{1.5cm} + \left.\left(\left\langle \nabla g_{i}\circ\tau_{-\left\langle \id, \nu \right\rangle}, \nu \right\rangle \cdot \left\langle \id, \mu - \nu \right\rangle   \right) \left(\left\langle \nabla g_{j}\circ\tau_{-\left\langle \id, \nu \right\rangle}, \nu \right\rangle \cdot \left\langle \id, \mu - \nu \right\rangle   \right) \phantom{\frac{1}{2}} \hspace{-0.35cm} \right\} \\
& \hspace{0.75cm} + O\left(\left\|\left\langle \id, \mu - \nu \right\rangle \right\|^{3} + \sum_{j\, = \, 1}^{p}{\sum_{k\, = \, 0}^{2}{\left\|D^{k}g_{j} \circ \tau_{-\left\langle \id, \nu \right\rangle}, \mu-\nu \right\|^{3}}}  \right)
  \end{align*} 
  where $D^{k}g_{j}$ denotes the differential of order $k$ of $g_{j}$.
  \label{Lem_FVr_p_variables}
\end{Lem}
}

\begin{proof} The general case $p, d \in \N^{\star}$ can be proved by a straightforward extension of the proof of the case $p =d =1$ which is the only case that we prove. Applying \textsc{Taylor}'s formula to $g\circ\tau_{- \left\langle  \id , \mu\right\rangle} = g\left(\cdot - \left\langle \id , \mu \right\rangle  \right)$, we obtain that
\begin{align*}
 & \left\langle g\circ\tau_{-\left\langle  \id, \mu \right\rangle} - g\circ\tau_{-\left\langle \id, \nu \right\rangle}, \mu \right\rangle \\
 &  \hspace{0.25cm}= \int_{\R}^{}{\mu(\dd x)\left[g'\left(x - \left\langle \id, \nu  \right\rangle  \right)\left\langle \id, \nu - \mu \right\rangle + \frac{1}{2}g''\left(x - \left\langle \id, \nu \right\rangle   \right) \left\langle \id, \nu - \mu \right\rangle^{2}  + O\left(\left|\left\langle \id, \nu - \mu \right\rangle\right|^{3} \right)\right]} \\
 &\hspace{0.25cm}  = - \left\langle \id, \mu - \nu \right\rangle\left[\left\langle g'\circ\tau_{-\left\langle \id, \nu \right\rangle}, \nu \right\rangle + \left\langle g'\circ\tau_{-\left\langle \id, \nu  \right\rangle}, \mu - \nu \right\rangle  \right] \\ 
 & \hspace{1cm}  + \frac{1}{2} \left\langle \id, \mu - \nu \right\rangle^{2}\left[\left\langle g''\circ\tau_{-\left\langle \id, \nu \right\rangle}, \nu \right\rangle + \left\langle g''\circ\tau_{-\left\langle \id, \nu  \right\rangle}, \mu - \nu  \right\rangle  \right] + O\left(\left|\left\langle \id, \mu - \nu \right\rangle\right|^{3}\right). 
\end{align*}
Therefore, we deduce the following approximation: 
\begin{align}
& \left\langle g\circ\tau_{-\left\langle \id, \mu \right\rangle}, \mu \right\rangle - \left\langle g\circ\tau_{-\left\langle \id , \nu\right\rangle}, \nu \right\rangle \notag \\
& \hspace{0.0cm}= \left\langle g\circ\tau_{-\left\langle \id, \nu \right\rangle}, \mu - \nu  \right\rangle  - \left\langle \id, \mu - \nu \right\rangle\left[\left\langle g'\circ\tau_{-\left\langle \id, \nu \right\rangle}, \nu \right\rangle + \left\langle g'\circ\tau_{-\left\langle \id, \nu \right\rangle}, \mu - \nu \right\rangle  \right] \label{Lemme_appendice_etape} \\
 & \hspace{0.15cm} + \frac{1}{2} \left\langle \id, \mu - \nu \right\rangle^{2}\left\langle g''\circ\tau_{-\left\langle \id, \nu \right\rangle}, \nu \right\rangle + O\left(\left|\left\langle \id, \mu - \nu \right\rangle\right|^{3} + \left\langle \id, \mu - \nu  \right\rangle^{2}\left|\left\langle g''\circ\tau_{-\left\langle \id, \nu \right\rangle}, \mu - \nu \right\rangle\right| \right). \notag
\end{align}
Applying \textsc{Taylor}'s formula to $F\left(\left\langle g\circ\tau_{-\left\langle \id, \mu \right\rangle}, \mu \right\rangle \right)$, we obtain that
\begin{align*}
& F\left(\left\langle g\circ\tau_{-\left\langle \id, \mu \right\rangle} , \mu \right\rangle \right) - F\left(\left\langle g\circ\tau_{-\left\langle \id, \nu \right\rangle}, \nu \right\rangle \right) \\
 & \hspace{1cm}=   F'\left(\left\langle g\circ\tau_{-\left\langle \id, \nu \right\rangle}, \nu \right\rangle \right) \times \left[\left\langle  g\circ\tau_{-\left\langle \id, \mu \right\rangle}, \mu \right\rangle  - \left\langle g\circ\tau_{-\left\langle \id, \nu \right\rangle}, \nu \right\rangle  \right] \\ 
&\hspace{2cm} + \frac{F''\left(\left\langle g\circ\tau_{-\left\langle \id, \nu \right\rangle}, \nu \right\rangle \right)}{2} \times \left[\left\langle g\circ\tau_{-\left\langle \id, \mu \right\rangle}, \mu  \right\rangle  - \left\langle g\circ\tau_{-\left\langle \id, \nu \right\rangle}, \nu \right\rangle  \right]^{2} \\ 
& \hspace{2cm} +  O \left(\left|\left\langle g\circ\tau_{-\left\langle \id, \mu \right\rangle}, \mu \right\rangle  - \left\langle g\circ\tau_{-\left\langle \id, \nu \right\rangle}, \nu \right\rangle \right|^{3} \right).
\end{align*}
Using (\ref{Lemme_appendice_etape}), we deduce that
\begin{align*}
\left[\left\langle g\circ\tau_{-\left\langle \id, \mu \right\rangle}, \mu \right\rangle  - \left\langle g\circ\tau_{-\left\langle \id, \nu \right\rangle}, \nu \right\rangle  \right]^{2} &  = \left\langle g\circ\tau_{-\left\langle \id, \nu \right\rangle}, \mu - \nu \right\rangle^{2} + \left\langle \id, \mu - \nu \right\rangle^{2} \left\langle g'\circ\tau_{-\left\langle \id, \nu \right\rangle}, \nu \right\rangle^{2} \\
 & \hspace{0.5cm} - 2\left\langle g\circ\tau_{-\left\langle \id, \nu \right\rangle}, \mu - \nu \right\rangle\left\langle \id, \mu - \nu \right\rangle\left\langle g'\circ\tau_{-\left\langle \id, \nu \right\rangle}, \nu \right\rangle  \\
& \hspace{0.5cm} + O\left(\left|\left\langle \id, \mu - \nu \right\rangle\right|^{3}  +   \left\langle \id, \mu - \nu \right\rangle^{2}\left[\left|\left\langle g\circ\tau_{-\left\langle \id, \nu \right\rangle}, \mu - \nu \right\rangle\right| \phantom{\frac{1}{2}} \right.  \right. \\
& \hspace{0.85cm}  \left. \phantom{\frac{1}{2}} + \left|\left\langle g'\circ\tau_{-\left\langle \id, \nu \right\rangle}, \mu - \nu \right\rangle\right|   + \left|\left\langle g''\circ\tau_{-\left\langle \id, \nu \right\rangle}, \mu - \nu \right\rangle\right|\right] \\
& \hspace{0.85cm}  \left. \phantom{\frac{1}{2}} + \left|\left\langle \id, \mu - \nu \right\rangle\left\langle g\circ\tau_{-\left\langle \id, \nu \right\rangle}, \mu - \nu \right\rangle\left\langle g'\circ\tau_{-\left\langle \id, \nu \right\rangle}, \mu - \nu \right\rangle\right| \hspace{-0.3cm}\phantom{\frac{1}{2}} \right), 
\end{align*}
and  \begin{align*}
& O\left(\left|\left\langle g\circ\tau_{-\left\langle \id, \mu \right\rangle}, \mu \right\rangle  - \left\langle g\circ\tau_{-\left\langle \id, \nu \right\rangle}, \nu \right\rangle \right|^{3} \right) \\
&  \hspace{0.25cm} = O\left(\left|\left\langle \id, \mu - \nu \right\rangle\right|^{3} + \left\langle \id, \mu - \nu \right\rangle^{2}\left[\left|\left\langle g\circ\tau_{-\left\langle \id, \nu \right\rangle}, \mu - \nu \right\rangle\right| + \left|\left\langle g'\circ\tau_{-\left\langle \id, \nu \right\rangle}, \mu - \nu \right\rangle\right| \phantom{\frac{1}{2}} \hspace{-0.35cm} \right]\right. \\
& \hspace{0.75cm} \left. + \left|\left\langle \id, \mu - \nu \right\rangle\right|\left[\left\langle g\circ\tau_{-\left\langle \id, \nu \right\rangle}, \mu - \nu \right\rangle^{2} + \left\langle g'\circ\tau_{-\left\langle \id, \nu \right\rangle}, \mu - \nu \right\rangle^{2} \right] + \left|\left\langle g\circ\tau_{-\left\langle \id, \nu \right\rangle}, \mu - \nu \right\rangle\right|^{3} \right).
\end{align*}
The announced result follows from \textsc{Young}'s inequalities. \qedhere
\end{proof} 

\subsection{Lemma of convergence \label{Sous_sect_lemmes_preuve_existence}}

\begin{Lem} {  Let $k \in \N^{\star}$ and assume that $Y_0\in{\cal M}_1(\R)$ satisfies
$\left\langle \id^{k}, Y_{0} \right\rangle < \infty$.} We consider for $t \in [0,T]$, an increasing sequence  $0 = t_{0}^{n} < t_{1}^{n} < \dots < t_{p_{n}}^{n} = T$ of subdivisions of $[0,T]$ whose mesh tends to $0$. Then, for all 
$h_{1} \in \{1, \id \}$ and $h_{2} \in \CCCC^{2}_{b}(\R, \R)$, we obtain that, in $\P_{\nu}^{\rm FV}-$probability, 
\begin{align*}
{\rm{\textbf{\rm{\textbf{(1)}}}}} \quad & \lim_{n\to +\infty}\sum\limits_{i \, = \, 0}^{p_{n} - 1}{\int_{t_{i}^{n}\wedge t}^{t_{i+1}^{n}\wedge t}{\left|\left\langle h_{1} \times \left[h_{2}\circ\tau_{-  \left\langle \id, Y_{t_{i}^{n}\wedge t} \right\rangle  } - h_{2}\circ\tau_{-\left\langle   \id, Y_{s} \right\rangle  } \right], Y_{s} \right\rangle \right|^{k}}  \dd s}  = 0, \\
{\rm{\textbf{\rm{\textbf{(2)}}}}} \quad & \lim_{n\to +\infty}\sum\limits_{i \, = \, 0}^{p_{n} - 1}{\int_{t_{i}^{n}\wedge t}^{t_{i+1}^{n}\wedge t}{\left|\left\langle h_{1}, Y_{s} \right\rangle \right|^{k}\left|\left\langle h_{2}\circ\tau_{-\left\langle \id, Y_{t_{i}^{n}\wedge t} \right\rangle  } - h_{2}\circ\tau_{-\left\langle   \id, Y_{s} \right\rangle  }, Y_{s} \right\rangle \right|^{k}}  \dd s}  = 0.
\end{align*}
\label{Lem_convergence_proba}
\end{Lem}

\begin{proof}
The two properties can be proved similarly. We only prove the first one. Thanks to {\rm{Lemma} \ref{Lem_Mart_id_id_bien_def} \textbf{(1)(a)}} and using that $h_{2}$ is \textsc{Lipschitz}, there exists a constant $C_{{\rm{Lip}}}$ such that
\begin{align*}
& \int_{t_{i}^{n}\wedge t}^{t_{i+1}^{n}\wedge t}{\left|\left\langle  h_{1} \times \left[h_{2}\circ\tau_{-\left\langle  \id, Y_{t_{i}^{n}\wedge t}  \right\rangle } - h_{2}\circ\tau_{-\left\langle   \id, Y_{s}  \right\rangle } \right], Y_{s} \right\rangle \right|^{k}\dd s} \\
& \hspace{2.5cm} \leqslant \int_{t_{i}^{n}\wedge t}^{t_{i+1}^{n}\wedge t}{\left|\left\langle h_{1}, Y_{s} \right\rangle \right|^{k} \left[ C_{{\rm{Lip}}}^{k}\left|M_{s}^{\id}(\id) - M_{t_{i}^{n}\wedge t}^{\id}(\id) \right|\wedge \left(2 \left\|h_{2} \right\|_{\infty} \right)^{k}\right] \dd s}.  
\end{align*}
If $h_{1} = 1$, the dominated convergence theorem allows us to conclude the proof. If $h_{1} = \id$, using that $\sup_{t\in[0,T]}{\left|\left\langle \id, Y_{t} \right\rangle \right|} < \infty$ $\P_{\nu}^{\rm FV}-$a.s. by Lemma \ref{Lem_Mart_id_id_bien_def} \textbf{(1)(b)}, we can also apply the dominated convergence. \qedhere
\end{proof}

\subsection{Control of error terms}

\begin{Lem} Let $t>0$ be fixed and assume that {  $Y_0\in{\cal M}_1(\R)$ satisfies $\left\langle \id^{2}, Y_{0} \right\rangle < \infty$}. Let $j\in \{0, 1, 2 \}$ and $g \in \CCCC_{b}^{4}(\R, \R)$ fixed. The sequences 
\begin{align*}
{\rm{\textbf{\rm{\textbf{(1)}}}}} \quad & \left(\sum\limits_{i\, = \, 0}^{p_{n} - 1}{ \left|M_{t_{i+1}^{n} \wedge t}^{\id}\left(\id\right) - M_{t_{i}^{n} \wedge t}^{\id}\left(\id\right) \right|^{3}} \right)_{n\in \N} \\
{\rm{\textbf{\rm{\textbf{(2)}}}}}  \quad & \left(\sum\limits_{i\, = \, 0}^{p_{n} - 1}{ \left|M_{t_{i+1}^{n} \wedge t}^{\id}\left(g^{(j)}\circ\tau_{-\left\langle \id, Y_{t_{i}^{n}\wedge t} \right\rangle}\right) - M_{t_{i}^{n} \wedge t}^{\id}\left(g^{(j)}\circ\tau_{-\left\langle \id, Y_{t_{i}^{n}\wedge t} \right\rangle}\right) \right|^{3}} \right)_{n\in \N}
\end{align*}
converge to $0$ in $\P_{\nu}^{\rm FV}-$probability.
\label{Lem_controle_erreur}
\end{Lem}

\begin{proof} 
\textbf{Step 1. Proof of (1).} Let $\varepsilon >0$ and $t\geqslant 0$ fixed. Let $A>0$ to be determined later. We introduce the stopping time \[\tau_{A} := \inf{\left\{ t \geqslant 0\left. \phantom{1^{1^{1^{1}}}} \hspace{-0.6cm} \right| \left\langle \id^{2}, Y_{t} \right\rangle - \left\langle \id, Y_{t} \right\rangle^{2} \geqslant A  \right\}}\] which satisfies almost surely $\lim_{A\to + \infty}{\tau_{A}} = +\infty$ by Lemma \ref{Lem_Mart_id_id_bien_def}. Then, using \textsc{Markov}'s inequality, we obtain that \begin{align*}
& \P_{\nu}^{\rm FV}\left(\sum\limits_{i \, = \, 0}^{p_{n} - 1}{\left|M_{t_{i+1}^{n} \wedge t}^{\id}(\id) - M_{t_{i}^{n}\wedge t}^{\id}(\id)\right|^{3}} > \varepsilon \right) \\
& \hspace{0.2cm}\leqslant \P_{\nu}^{\rm FV}\left(\tau_{A} \leqslant t \right) + \P_{\nu}^{\rm FV}\left(\left\{\tau_{A} > t \right\} \cap \left\{\sum\limits_{i \, = \, 0}^{p_{n} - 1}{\left|M_{t_{i+1}^{n} \wedge t}^{\id}(\id) - M_{t_{i}^{n}\wedge t}^{\id}(\id)\right|^{3}} > \varepsilon \right\} \right) \\
& \hspace{0.2cm}\leqslant \P_{\nu}^{\rm FV}\left(\tau_{A} \leqslant t \right) + \frac{1}{\varepsilon} \sum\limits_{i\, = \, 0}^{p_{n}-1}{\E\left(\left|M^{\id}_{t_{i+1}^{n}\wedge t\wedge \tau_{A}}(\id) - M^{\id}_{t_{i}^{n}\wedge t}(\id)  \right|^{3} \right)  }.
\end{align*}
From the definition of $\tau_{A}$, we obtain for all $s \in  \left[t_{i}^{n} \wedge t, t_{i+1}^{n} \wedge t \right] $, $\left|M^{\id}_{s\wedge \tau_{A}}(\id) - M^{\id}_{t_{i}^{n}\wedge t}(\id) \right|^{3}$ is bounded and $\left\langle M^{\id}(\id) \right\rangle_{t\wedge \tau_{A}}$ is $2\gamma A^{2}$-\textsc{Lipschitz}. Thanks to the \textsc{Burkolder-Davis-Gundy} inequality and {\rm{Lemma} \ref{Lem_Mart_id_id_bien_def} \textbf{(2)}}, there exists a constant $C_{1}$ such that 
\begin{align*}
\E\left(\left|M^{\id}_{t_{i+1}^{n}\wedge t\wedge \tau_{A}}(\id) - M^{\id}_{t_{i}^{n}\wedge t}(\id)  \right|^{3} \right)& \leqslant C_{1} \E\left(\left[\left\langle M^{\id}(\id) \right\rangle_{t_{i+1}^{n}\wedge t \wedge \tau_{A}} - \left\langle M^{\id}(\id) \right\rangle_{t_{i}^{n}\wedge t}  \right]^{\frac{3}{2}} \right) \\
&  \leqslant C_{1}\left(2\gamma A^{2} \right)^{\frac{3}{2} } \left(t_{i+1}^{n}\wedge t - t_{i}^{n}\wedge t \right)^{\frac{3}{2} }.
\end{align*} 
Therefore, if we choose $A$ such that $\P_{\nu}^{\rm FV}\left(\tau_{A} \leqslant t \right) \leqslant \frac{\varepsilon}{2}$, and $n_{0} \in \N$ such that for all $n \geqslant n_{0}$, \[\sqrt{\sup_{i\in \llbracket 0, p_{n} -1\rrbracket}{\left|t_{i+1}^{n}\wedge t - t_{i}^{n_{0}}\wedge t \right|}} \leqslant \frac{\varepsilon^{2}}{2C_{1}\left(2\gamma A^{2} \right)^{\frac{3}{2}}t},\] we obtain that
\begin{align*}
& \P_{\nu}^{\rm FV}\left(\sum\limits_{i \, = \, 0}^{p_{n} - 1}{\left|M_{t_{i+1}^{n} \wedge t}^{\id}(\id) - M_{t_{i}^{n}\wedge t}^{\id}(\id)\right|^{3}} > \varepsilon \right) \\
& \hspace{2.5cm} \leqslant \P_{\nu}^{\rm FV}\left(\tau_{A} \leqslant t \right) + \frac{C_{1}\left(2\gamma A^{2} \right)^{\frac{3}{2} }}{\varepsilon} t \sqrt{\sup_{i\in \llbracket 0, p_{n} -1\rrbracket}{\left|t_{i+1}^{n}\wedge t - t_{i}^{n}\wedge t \right|}} \\
& \hspace{2.5cm} \leqslant \varepsilon,
\end{align*}
and the first announced result follows. \\

\textbf{Step 2. Proof of (2).} In similar way as previously, we obtain that 
\begin{align*}
&  \P_{\nu}^{\rm FV}\left(\sum\limits_{i\, = \, 0}^{p_{n} - 1}{ \left|M_{t_{i+1}^{n} \wedge t}^{\id}\left(g^{(j)}\circ\tau_{-\left\langle \id, Y_{t_{i}^{n}\wedge t} \right\rangle}\right) - M_{t_{i}^{n} \wedge t}^{\id}\left(g^{(j)}\circ\tau_{-\left\langle \id, Y_{t_{i}^{n}\wedge t} \right\rangle}\right) \right|^{3}}>\varepsilon \right)  \\
& \hspace{1cm} \leqslant \frac{1}{\varepsilon} \sum\limits_{i\, = \, 0}^{p_{n} - 1}{\E\left(\left|M^{\id}_{t_{i+1}^{n} \wedge t}\left(g^{(j)}\circ\tau_{-\left\langle \id, Y_{t_{i}^{n}\wedge t} \right\rangle} \right) - M^{\id}_{t_{i}^{n}\wedge t}\left(g^{(j)}\circ\tau_{-\left\langle \id, Y_{t_{i}^{n}\wedge t} \right\rangle} \right)  \right|^{3} \right)} \\
& \hspace{1cm} \leqslant \frac{C_{1}\left(2\gamma \left\|g^{(j)} \right\|_{\infty}^{2} \right)^{\frac{3}{2} }}{\varepsilon}t \sqrt{\sup_{i\in \llbracket 0, p_{n} -1\rrbracket}{\left|t_{i+1}^{n}\wedge t - t_{i}^{n}\wedge t \right|}},
\end{align*}
which converges to 0 when $n \to + \infty$.
\end{proof}

{ 

\section{Extensions to the centered \textsc{Fleming-Viot} martingale~problem\label{Sous_section_2_3}}

The goal of this section is to give some extensions to the martingale problem {\rm{(\ref{PB_Mg_Z_Multi_d})}} which are equivalent and it is usual to switch from one to the other. 

\subsection{Extension to multiple variables}

We firstly introduce the version of the centered \textsc{Fleming-Viot} martingale problem with $p \in \N^{\star}$ variables. We will then provide the heuristic for the martingale problem for polynomials 
given in the next section. 

\begin{Def}
The probability measure $\P_{\mu} \in \MM_{1}\left(\widetilde{\Omega}_{d} \right)$ is said to solve the  \emph{centered \textsc{Fleming-Viot} martingale problem with $p$ variables}, {  with resampling rate $\gamma$}  and  with initial condition $\mu \in \MM_{1}^{c,2}\left(\R^{d}\right)$, if the canonical process $\left(X_{t} \right)_{t\geqslant 0}$ on $\widetilde{\Omega}_{d}$ satisfies  $\P_{\mu}(X_{0} = \mu) = 1$ and for each $F\in \CCCC^{2}(\R^{p}, \R)$ and $g \in \CCCC_{b}^{2}(\R^{d}, \R^{p})$, 
\begin{equation}
\begin{aligned}
\widehat{M}_{t}^{F,d}(g) & := F\left(\left\langle g, X_{t} \right\rangle \right) - F\left(\left\langle g, X_{0} \right\rangle \right) \\
& \hspace{-0.5cm} -  \int_{0}^{t}{\sum_{j\, = \, 1}^{p}{\partial_{j}F\left(\left\langle g, X_{s} \right\rangle \right)\left(\left\langle\frac{\Delta g_{j}}{2}, X_{s} \right\rangle + \gamma \left[\left\langle \id^{t}\left\langle \Hess\left(g_{j} \right), X_{s} \right\rangle \id, X_{s} \right\rangle - 2\left\langle \nabla g_{j} \cdot \id, X_{s} \right\rangle \right] \right)}\dd s} \\
& \hspace{-0.5cm} - \gamma\int_{0}^{t}\sum_{i,j\, = \, 1}^{p}{\partial_{ij}^{2}F\left(\left\langle g, X_{s} \right\rangle\right)\left(\left\langle g_{i}g_{j}, X_{s} \right\rangle - \left\langle g_{i}, X_{s} \right\rangle\left\langle g_{j}, X_{s} \right\rangle + \left\langle \nabla g_{i}, X_{s} \right\rangle^{t}\left\langle \id \, \id^{t}, X_{s} \right\rangle \left\langle \nabla g_{j}, X_{s} \right\rangle  \phantom{\frac{1}{2}}\right.} \\
& \hspace{2cm} - \left\langle \nabla g_{i}, X_{s} \right\rangle \cdot \left\langle g_{j} \times \id, X_{s} \right\rangle
- \left. \left\langle \nabla g_{j}, X_{s} \right\rangle \cdot \left\langle g_{i} \times \id, X_{s} \right\rangle \phantom{\frac{1}{2}} \hspace{-0.35cm} \right)\dd s
\end{aligned}
\label{PB_Mg_Z_p_variables}
\end{equation}
is a continuous $\P_{\mu}-$martingale in $L^{2}\left(\widetilde{\Omega}_{d} \right)$. 
\label{Def_PB_Mg_Z_p_variables}
\end{Def}

\begin{Thm} For all $\mu \in \MM_{1}^{c,2}\left(\R^{d}\right)$, the probability measure $\P_{\mu}$ constructed in {\rm{Theorem \ref{Prop_PB_Mg_Z}}}, satisfies the martingale problem of {\rm{Definition \ref{Def_PB_Mg_Z_p_variables}}}.
\label{Prop_PB_Mg_Z_p_variables}
\end{Thm}

\begin{proof} We can deduce the result from the original \textsc{Fleming-Viot} martingale problem with $p$ variables {\rm{\color{blue} \cite{dawson_wandering_1982}}} given by {\rm{(\ref{PB_Mg_Y_p_variables})}} below, following exactly the same method as for the proof of {\rm{Theorem \ref{Prop_PB_Mg_Z}}}. We recall that, the probability measure $\P_{\nu}^{\rm FV} \in \MM_{1}\left(\Omega_{d} \right)$ is said to solve the  \emph{original \textsc{Fleming-Viot} martingale problem} with $p$ variables {  with resampling rate $\gamma$} and  with initial condition $\nu \in \MM_{1}\left(\R^{d}\right)$, if the canonical process $\left(Y_{t} \right)_{t\geqslant 0}$ on $\Omega_{d}$ satisfies $\P_{\nu}^{\rm FV}(Y_{0} = \nu) = 1$ and for each $F\in \CCCC^{2}(\R^{p}, \R)$ and $g \in \CCCC_{b}^{2}(\R^{d}, \R^{p})$, 
\begin{equation}
    \begin{aligned}
    M_{t}^{F,d}(g) & := F\left(\left\langle g, Y_{t} \right\rangle \right) - F\left(\left\langle g, Y_{0} \right\rangle \right) - \int_{0}^{t}{\sum_{k\, = \, 1}^{p}{\partial_{k}F\left(\left\langle g, Y_{s} \right\rangle\right)\left\langle \frac{\Delta g_{k}}{2}, Y_{s} \right\rangle}\dd s} \\
    & \hspace{1cm} - \gamma \int_{0}^{t}{\sum_{i, j \, = \, 1}^{p}{\partial_{ij}^{2}F\left(\left\langle g, Y_{s} \right\rangle \right) \left[ \left\langle g_{i}g_{j}, Y_{s} \right\rangle - \left\langle g_{i}, Y_{s} \right\rangle \left\langle g_{j}, Y_{s} \right\rangle  \right]}\dd s} 
\end{aligned}
\label{PB_Mg_Y_p_variables}
\end{equation}
is a  $\P_{\nu}^{\rm FV}-$martingale. \qedhere 
\end{proof}

\subsection{Heuristics leading to the martingale problem for polynomials}

Our goal here is to study the \textsc{Doob} semi-martingale decomposition of polynomial functions of the centered \textsc{Fleming-Viot} process. {  To obtain the expression of $\LL_{\rm FVc}^{d}$ for polynomial functions $P_{f, n}$ with $f \in \CCCC_{b}^{2}((\R^{d})^{n}, \R)$ (see (\ref{Eq_Fonction_polynome_mu})),
we first look for this expression when $f$ has product form.} The previous martingale problem {\rm{(\ref{PB_Mg_Z_p_variables})}} gives the following result: for the choice of $F(\alpha_{1}, \cdots, \alpha_{n}) := \prod\limits_{i\, = \, 1}^{n}{\alpha_{i}} $ where $\alpha_{i} \in \R$, {  noting that $F\left(\left\langle g_{1}, \mu \right\rangle, \cdots, \left\langle g_{n}, \mu \right\rangle \right) = P_{f, n}(\mu)$ with $f\left(x_{1}, \cdots, x_{n} \right) := \prod_{i\, = \, 1}^{n}{g_{i}(x_{i})}$ and $g_{i} \in \CCCC^{2}_{b}\left(\R^{d}, \R\right)$, $i\in \{1, \cdots, n \}$}, we deduce that for all $\mu \in \MM_{1}^{c,2}\left(\R^{d}, \R\right)$,
\begin{align*}
    \LL_{\rm FVc}^{d}P_{f,n}(\mu) & = \sum_{i\, = \, 1}^{n}{\left(\left\langle \frac{\Delta g_{i}}{2}, \mu \right\rangle -2\gamma \left\langle \nabla g_{i} \cdot \id, \mu \right\rangle + \gamma\left\langle \id^{t} \left\langle \Hess(g_{i}), \mu \right\rangle \id, \mu \right\rangle \right)\prod_{\substack{k\, = \, 1 \\ k \, \neq \, i}}^{n}{\left\langle g_{k}, \mu \right\rangle}} \\
    & \hspace{0.5cm}  + \gamma \sum_{\substack{i, j \, = \, 1 \\ j \, \neq \, i}}^{n}{\left(\left\langle g_{i} g_{j}, \mu \right\rangle - \left\langle g_{i}, \mu \right\rangle \left\langle g_{j}, \mu \right\rangle + \left\langle \id^{t} \left\langle \nabla g_{i}, \mu \right\rangle \left\langle \nabla g_{j}, \mu \right\rangle^{t} \id, \mu \right\rangle \phantom{\frac{1}{2}}\right.} \\
    & \hspace{1.5cm}  \left. \phantom{\frac{1}{2}} -\left\langle \nabla g_{j}, \mu \right\rangle \cdot \left\langle g_{i} \times \id \right\rangle  - \left\langle \nabla g_{i}, \mu \right\rangle \cdot \left\langle g_{j} \times \id \right\rangle \right)\prod_{\substack{k\, = \, 1 \\ k \, \neq \, i,j}}^{n}{\left\langle g_{k}, \mu \right\rangle}
\end{align*}

The previous relation leads us to introduce, for each $n, d \in \N^{\star}$ and for all $f \in \CCCC_{b}^{2}((\R^{d})^{n}, \R)$, the  operator $B^{(n), d}$ defined by (\ref{Eq_Bn_Multi_d}). Indeed, {  again} for the choice $f(x_{1}, \cdots, x_{n}) := \prod_{i\, = \, 1}^{n}{g_{i}(x_{i})}$ with $g_{i} \in \CCCC^{2}_{b}\left(\R^{d}, \R\right)$, $i\in \{1, \cdots, n \}$, we obtain  
\begin{align*}
B^{(n), d}f(x_{1}, \cdots, x_{n}) & = \sum\limits_{i\, = \, 1}^{n}{\left(\frac{\Delta g_{i}(x_{i)}}{2} - 2\gamma \nabla g_{i}(x_{i}) \cdot \id(x_{i})\right)\prod\limits_{\substack{k\, = \, 1 \\ k \, \neq \, i}}^{n}{g_{k}(x_{k})}  } \\
& \quad - \gamma \sum\limits_{\substack{i,j\, = \, 1 \\ j \, \neq \, i}}^{n}{\left(\left(\nabla g_{j}(x_{j}) \cdot \id\left(x_{i} \right) g_{i}(x_{i}) \right) + \left(\nabla g_{i}(x_{i}) \cdot \id\left(x_{j} \right) g_{j}(x_{j}) \right) \right)\prod\limits_{\substack{k\, = \, 1 \\ k \, \neq \, i,j}}^{n}{g_{k}(x_{k})}}. 
\end{align*}
Note that, {  for all $\mu \in \MM_{1}^{c,2}(\R^{d})$,}
\begin{equation*}
\begin{aligned}
\LL_{\rm FVc}^{d}{  P_{f, n}(\mu)} &  = \left\langle B^{(n), d}f, \mu^{n} \right\rangle + \gamma\sum\limits_{i\, = \, 1}^{n}{\left\langle \id^{t} \left\langle \Hess(g_{i}), \mu \right\rangle \id, \mu \right\rangle  \prod\limits_{\substack{k\, = \, 1 \\ k \, \neq \, i}}^{n}{\left\langle g_{k}, \mu \right\rangle } } \\ 
& \hspace{0.4cm} + \gamma \sum\limits_{i\, = \, 1}^{n}{\sum\limits_{\substack{j\, = \, 1 \\ j \, \neq \, i}}^{n}{\left(\left\langle  g_{i}g_{j}, \mu \right\rangle - \left\langle g_{i}, \mu \right\rangle\left\langle g_{j}, \mu \right\rangle  + \left\langle \id^{t} \left\langle \nabla g_{i}, \mu \right\rangle \left\langle \nabla g_{j}, \mu \right\rangle^{t} \id, \mu \right\rangle  \right) } }\prod\limits_{\substack{k\, = \, 1 \\ k \, \neq \, i, j}}^{n}{\left\langle g_{k}, \mu \right\rangle }. 
\end{aligned}
\label{Eq_Heuristic_polynom}
\end{equation*}

This leads us to introduce Definition (\ref{Def_PB_Mg_dual_FVr_Multi_d}) of the centered \textsc{Fleming-Viot} martingale problem for polynomials.

\section{Regular conditional probabilities and Lemma \ref{Lem_Q_omega_Martingale}\label{Appendix_Regular_Conditional_Proba}}

\subsection{Technical result for Lemma \ref{Lem_Q_omega_Martingale} \label{Etape_9_Existence}}
As the filtered probability space $\left(\Omega, \FF, \left( \FF_{t}\right)_{t\geqslant 0} \right)$ is Polish (see Section \ref{Sous_section_2_1}), we deduce from {\color{blue} \cite[\color{black} Theorem 3.18 of Section 5.3 (p 307)]{Kar}} there exists, for all $\nu \in \MM_{1}(\R)$, a unique family $\left(\Q_{\omega} \right)_{\omega\in \Omega}$ of regular conditional probability of $\P_{\nu}^{\rm FV}$ given $\FF_{t^{\star}}$ and a $\P_{\nu}^{\rm FV}-$null event  $N \in \FF_{t^{\star}}$ such that for all $\omega \in \Omega \, \backslash \, N$, 
\begin{equation}
\Q_{\omega}\left(\left\{ \widetilde{\omega} \in \Omega \left| \phantom{1^{1^{1^{1}}}} \hspace{-0.6cm} \right. \widetilde{\omega}_{t^{\star}} = \omega_{t^{\star}} \right\} \right) = 1.
 \label{Q_omega}
\end{equation} 
The following Theorem \ref{Proba_regulieres_condi_FV} ensures that time shifts of regular conditional probabilities of $\P_{\nu}^{\rm FV}$ remain solutions to the \textsc{Fleming-Viot} martingale problem (\ref{PB_Mg_Y}). The proof of this result is given hereafter and is based on the proof of {\color{blue} \cite[\color{black} Lemma 4.19 of Section 5.4 (p 321)]{Kar}}. We introduce, for $\omega \in \Omega$, the time-shift operator $\theta$ defined by  \[\left[\theta_{s}\omega\right]_{t} := \omega_{s+t}, \qquad 0\leqslant t < +\infty, \quad s\geqslant 0.\]

\begin{Thm} Let $t^{\star} \in \R_{+}$ be a deterministic time. Then there exists  a $\P_{\nu}^{\rm FV}-$null event $N \in \FF_{t^{\star}}$ such that, for every $\omega \in \Omega \, \backslash  \, N$, the probability measure 
\begin{equation}
\P_{\omega}\left(\dd \widetilde{\omega} \right) := \theta_{t^{\star}} \sharp \, \Q_{\omega}\left(\dd \widetilde{\omega} \right) \label{PPP_omega}
\end{equation}
solves the martingale problem {\rm{(\ref{PB_Mg_Y})}} with $\nu := \omega_{t^{\star}}$.
\label{Proba_regulieres_condi_FV}
\end{Thm}

\noindent{\textit{Proof}.} \textbf{Step 0. Preliminary results.} We denote by $\CCCC^{2}_{K}(\R,\R)$ the space of real functions of class $\CCCC^{2}(\R,\R)$ with compact support. It is well-known that the formulation of the martingale problem (\ref{PB_Mg_Y}) for $F, g \in \CCCC^{2}_{b}(\R, \R)$ is equivalent to the one for $F, g \in \CCCC^{2}_{K}(\R, \R)$ {\color{blue} \cite{dawson_wandering_1982}}. The space $\CCCC^{2}_{K}(\R, \R)$ equipped with the norm $\|f \|_{W_{0}^{2,\infty}} := \left\| f \right\|_{\infty} + \left\| f' \right\|_{\infty} + \left\| f'' \right\|_{\infty}$ is separable. So, we can choose a dense countable family $\BB \subset \CCCC^{2}_{K}(\R, \R)$, for the topology associated to the norm, that is to say
\begin{align*}
& \forall F, g \in \CCCC^{2}_{K}(\R, \R), \quad \exists \left(F_{n}\right)_{n \in \N},  \left(g_{n}\right)_{n \in \N} \in \BB^{\N}, \qquad  F_{n} \xrightarrow[n \to + \infty ]{ { \|\cdot \|_{W_{0}^{2,\infty}}} } F, \quad g_{n} \xrightarrow[n \to + \infty ]{ { \|\cdot \|_{W_{0}^{2,\infty}}} } g.
\end{align*} 
Hence, if we denote $\LL_{\rm FV} := \LL_{\rm FV}^{{ 1}}$ 
we deduce that  $\LL_{\rm FV}{\left(F_{n}\right)}_{g_{n}} \xrightarrow[n \to + \infty ]{ {\rm{  \|. \|_{\infty} }} } \LL_{\rm FV} F_{g}$.
\\

\textbf{Step 1. Reformulation of the goal.} Let $\nu \in \MM_{1}(\R)$.  From (\ref{Q_omega}), it follows that $\P_{\omega}\left( \widetilde{\omega}_{0} = \nu \right)  =1$ is satisfied with $\nu := \omega_{t^{\star}}$. The rest of the proof is devoted to construct a $\P_{\nu}^{\rm FV}-$null event $N_{4}$ such that 
\[\E_{\P_{\nu}^{\rm FV}}\left[F_{g}\left(\omega_{t}\right) - F_{g}\left(\omega_{s}\right) - \int_{s}^{t}{\LL_{\rm FV}F_{g}(\omega_{r})\dd r} \left| \phantom{1^{1^{1^{1^{1^{1}}}}}} \hspace{-0.9cm} \right. \FF_{s} \right] = 0,\] is satisfied for all $\omega \in \Omega \, \backslash \, N_{4}$. This means that for all $0\leqslant s< t < \infty$, $A \in \FF_{s}$, $F, g \in \CCCC^{2}_{K}(\R, \R)$, 
\begin{equation}
\forall w \in \Omega \, \backslash \, N_{4}, \qquad  \int_{\Omega}^{}{\left[M_{t}^{F_{g}}\left(\widetilde{\omega} \right) - M_{s}^{F_{g}}\left(\widetilde{\omega} \right) \right]\II_{A}\left(\widetilde{\omega} \right)\P_{\omega}\left(\dd \widetilde{\omega} \right)} = 0,
\label{Propriete_evenement_nul}
\end{equation}
where \[M_{t}^{F_{g}}\left(\widetilde{\omega} \right) := F_{g}\left(\widetilde{\omega}_{t} \right) - F_{g}\left(\widetilde{\omega}_{0} \right) - \int_{0}^{t}{\LL_{\rm FV} F_{g}\left(\widetilde{\omega}_{r} \right)\dd r}. \]
Let $\omega \in \Omega$, $0\leqslant s< t < \infty$, $A \in \FF_{s}$, $F, g \in \CCCC^{2}_{K}(\R, \R)$ be fixed. \\

\textbf{Step 2. Property (\ref{Propriete_evenement_nul}) satisfied except on a $\P_{\nu}^{\rm FV}-$null event $N_{1}(s,t,A,F,g) \in~\FF_{t^{\star}}$.}  As $\LL_{\rm FV} F_{g} \in \CCCC^{2}_{b}(\R, \R)$, the random variable $M_{t}^{F_{g}} - M_{s}^{F_{g}}$ is bounded. Note that,
\begin{align*}
\int_{\Omega}^{}{\left[M_{t}^{F_{g}}\left(\widetilde{\omega} \right) - M_{s}^{F_{g}}\left(\widetilde{\omega} \right) \right]\II_{A}\left(\widetilde{\omega} \right)\P_{\omega}\left(\dd \widetilde{\omega} \right)} 
& = \E_{\Q_{\omega}\left(\dd \hat{\omega} \right)}\left(\left[M_{t}^{F_{g}} - M_{s}^{F_{g}}\right]\circ \theta_{t^{\star}}\left(\widehat{\omega} \right)\II_{\theta_{t^{\star}}^{-1}A}\left( \widehat{\omega} \right) \right)  \\
& = \E_{\P_{\nu}^{\rm FV}}\left(\left[M_{t}^{F_{g}} - M_{s}^{F_{g}} \right] \circ \theta_{t^{\star}}\II_{\theta_{t^{\star}}^{-1}A} \left| \phantom{1^{1^{1^{1}}}} \hspace{-0.7cm}  \right. \FF_{t^{\star}} \right)(\omega) \\
& =  \E_{\P_{\nu}^{\rm FV}}\left[ \E_{\P_{\nu}^{\rm FV}}\left(\left[M_{t}^{F_{g}} - M_{s}^{F_{g}} \right] \circ \theta_{t^{\star}}\II_{\theta_{t^{\star}}^{-1}A} \left| \phantom{1^{1^{1^{1}}}} \hspace{-0.7cm}  \right. \FF_{t^{\star} + s} \right) \left| \phantom{1^{1^{1^{1^{1^{1}}}}}} \hspace{-1cm}  \right. \FF_{t^{\star}} \right](\omega) \\ 
& =  \E_{\P_{\nu}^{\rm FV}}\left[\II_{\theta_{t^{\star}}^{-1}A} \E_{\P_{\nu}^{\rm FV}}\left(\left[M_{t}^{F_{g}} - M_{s}^{F_{g}} \right] \circ \theta_{t^{\star}} \left| \phantom{1^{1^{1^{1}}}} \hspace{-0.7cm}  \right. \FF_{t^{\star} + s} \right) \left| \phantom{1^{1^{1^{1^{1^{1}}}}}} \hspace{-1cm}  \right. \FF_{t^{\star}} \right](\omega) \\ 
& = 0,
\end{align*}
where the last equality follows from martingale property (\ref{PB_Mg_Y}). This chain of equalities shows that the random variable $\omega \mapsto \int_{A}^{}{\left[M_{t}^{F_{g}}\left(\widetilde{\omega} \right) - M_{s}^{F_{g}}\left(\widetilde{\omega} \right) \right]\P_{\omega}\left(\dd \widetilde{\omega} \right)}$ is null except on a $\P_{\nu}^{\rm FV}-$null event $N_{1}(s,t,A,F,g) \in \FF_{t^{\star}}$  which depends on $s, t, A, F$ and $g$. \\ 

\textbf{Step 3. Property (\ref{Propriete_evenement_nul}) satisfied except on a $\P_{\nu}^{\rm FV}-$null event $N_{2}(s,t,F,g) \in~\FF_{t^{\star}}$.} We consider a countable subcollection $\EE$ of $\FF_{s}$ which generates $\FF_{s}$ {\color{blue} \cite[\color{black} Definition 3.17 of Section 5.3 (p 306)]{Kar}} and a $\P_{\nu}-$null event $N_{2}(s,t,F,g) \in \FF_{t^{\star}}$ such that for $\omega \in \Omega \, \backslash \, N_{2}(s,t,F,g)$,  \[
\forall A \in \EE, \qquad \int_{A}^{}{\left[M_{t}^{F_{g}}\left(\widetilde{\omega} \right) - M_{s}^{F_{g}}\left(\widetilde{\omega} \right) \right]\P_{\omega}\left(\dd \widetilde{\omega} \right) } = 0.\]
Therefore, the two measures \[v_{\omega}^{+}(A) := \int_{A}^{}{\left[M_{t}^{F_{g}} - M_{s}^{F_{g}} \right]^{+}\left( \widetilde{\omega}\right)\P_{\omega}\left(\dd \widetilde{\omega} \right)}  \quad {\rm{and}} \quad  v_{\omega}^{-}(A) := \int_{A}^{}{\left[M_{t}^{F_{g}} - M_{s}^{F_{g}} \right]^{-}\left( \widetilde{\omega}\right)\P_{\omega}\left(\dd \widetilde{\omega} \right)},\]
coincide on $\EE$, hence on $\FF_{s}$. Therefore, for $\omega \in \Omega \, \backslash \, N_{2}(s,t,F,g)$, we have proved that for all $A \in \FF_{s}$,  $\E_{\P_{\omega}}\left(\II_{A}\left[M_{t}^{F_{g}} - M_{s}^{F_{g}} \right] \right) = 0$. \\

\textbf{Step 4. Property (\ref{Propriete_evenement_nul}) satisfied except on a $\P_{\nu}^{\rm FV}-$null event $N_{3}(F,g) \in~\FF_{t^{\star}}$.} We may set now, the $\P_{\nu}^{\rm FV}-$null event 
\[N_{3}(F,g) := \bigcup_{\substack{s, t \, \in \, \Q \\ 0\leqslant s < t < \infty} }{N_{2}(s,t,F,g)}. \]
Due to the boundedness and continuity of $t \mapsto M_{t}^{F_{g}}$, it follows from the dominated convergence theorem that for $\omega \in \Omega \, \backslash \, N_{3}(F,g)$ \[ \forall s < t, \forall A \in \FF_{s}, \quad  \E_{\P_{\omega}}\left(\II_{A}\left[M_{t}^{F_{g}} - M_{s}^{F_{g}} \right] \right) = 0, \] 
in other words, for all $\omega \in \Omega \, \backslash \, N_{3}(F,g)$, $\left(M_{t}^{F_{g}}\left(\widetilde{\omega} \right)\right)_{t\geqslant 0}$ is a $\left(\FF_{s}, \P_{\omega}\left(\dd \widetilde{\omega} \right) \right)-$martingale. \\

\textbf{Step 5. Conclusion.} Now we define the $\P_{\nu}^{\rm FV}-$null event
\[N_{4} := \bigcup_{F, g \, \in \,  \BB}{N_{3}(F,g)} \] 
From the {\rm{Step 4}}, we have for all $s \leqslant t$, 
\begin{align*}
& \forall \omega \in \Omega \, \backslash \, N_{4}, \quad \forall A \in \FF_{s}, \quad  \forall F, g \in \BB, \qquad  \E_{\P_{\omega}}\left[\II_{A}\left(M_{t}^{F_{g}}  - M_{s}^{F_{g}} \right) \right] = 0.
\end{align*}
From Step 0, for all $F, g \in \CCCC^{2}_{K}(\R, \R)$, there exist two sequences $\left(F_{n}\right)_{n \in \N}, \left(g_{n}\right)_{n \in \N} \in\BB^{\N}$ such that  \[F_{n} \xrightarrow[n \to + \infty ]{ {  \|.\|_{W_{0}^{2, \infty}} } } F, \quad g_{n} \xrightarrow[n \to + \infty ]{ { \|.\|_{W_{0}^{2, \infty}} } } g, \qquad {\rm{and}} \qquad  \LL_{\rm FV}{\left(F_{n}\right)}_{g_{n}} \xrightarrow[n \to + \infty ]{ {\rm{  \|.\|_{\infty}}} } \AA F_{g}.\]
By the dominated convergence theorem, we deduce that for all $\omega \in \Omega$, $s \leqslant t$, and $A \in \FF_{s}$, 
\begin{align*}
 \E_{\P_{\omega}}\left[\II_{A}\left(M_{t}^{F_{g}}  - M_{s}^{F_{g}} \right) \right]  = \lim_{n \to + \infty}{ \E_{\P_{\omega}}\left[\II_{A}\left(M_{t}^{\left(F_{n}\right)_{g_{n}}}  - M_{s}^{\left(F_{n}\right)_{g_{n}}} \right) \right]}  =   0.
\end{align*}
which concludes the proof.  \hfill $\square$

\subsection{Proof of Lemma \ref{Lem_Q_omega_Martingale} \label{Preuve_lemme_fondamental}} 

 By abuse of notation, we note $h\left(\omega_{|_{\left[0,t^{\star}\right]}} \right)  = h(\omega)$. We want to prove that for all $0\leqslant s \leqslant t$, for all $\FF_{s}-$measurable bounded random variable $Z$, \[\E_{\P_{\nu}^{\rm FV}\left(\dd \tilde{\omega} \right)}\left(\left[\MM_{t}\left(\widetilde{\omega} \right) - \MM_{s}\left(\widetilde{\omega} \right) \right]Z\left(\widetilde{\omega} \right) \right) = 0. \]
Using {\color{blue} \cite[ {\color{black} Definition 3.2 (iii)' of Section 1}]{Ike}}
we deduce that
\begin{align*}
& \E_{\P_{\nu}^{\rm FV}\left(\dd \tilde{\omega} \right)}\left(\left[\MM_{t}\left(\widetilde{\omega} \right) - \MM_{s}\left(\widetilde{\omega} \right) \right]Z\left(\widetilde{\omega} \right) \right)  = \E_{\P_{\nu}\left(\dd \omega \right)}\left[\E_{\Q_{\omega}\left(\dd \tilde{\omega} \right)}\left(\left[\MM_{t}\left(\widetilde{\omega} \right) - \MM_{s}\left(\widetilde{\omega} \right) \right]Z\left(\widetilde{\omega} \right) \right)\right].
\end{align*}
Thus, it is sufficient to prove that for $\P_{\nu}^{\rm FV}-$almost every $\omega\in \Omega$, $\left(\MM_{t}\left(\widetilde{\omega} \right)\right)_{0\leqslant t \leqslant T}$ is a $\Q_{\omega}\left(\dd \widetilde{\omega} \right)-$ martingale and this is what we propose to establish in the rest of this proof. \vspace{0.3cm} \\
For fixed $\omega$, the function $h(\omega) \in \CCCC_{b}^{2}(\R, \R)$ can be considered as deterministic. 
We deduce from {\rm{Theorem} \ref{Proba_regulieres_condi_FV}} that there exists a  $\P_{\nu}^{\rm FV}-$null event $N\in \FF_{t^{\star}}$ such that for all $\omega \in \Omega \, \backslash \, N$, $\left(M_{t}^{\id}(h(\omega))\left(\widetilde{\omega} \right)\right)_{0\leqslant t \leqslant T}$ is a $\P_{\omega}\left(\dd \widetilde{\omega}\right)-$martingale. We deduce from {\color{blue} \cite[ {\color{black} Theorem 3.18 of Section 5.3 (p 307)}]{Kar}} that $\P_{\nu}^{\rm FV}-$almost every $\omega \in \Omega$, 
\begin{equation}
\widetilde{\omega}_{|_{[0, t^{\star}]}} = \omega_{|_{[0, t^{\star}]}}, \qquad \Q_{\omega}\left(\dd\widetilde{\omega}\right)\!-\!{\rm{a.s.}}
 \label{omega_omega_tilde}
\end{equation} 
This implies that, $\Q_{\omega}\left(\dd \widetilde{\omega} \right)-$almost surely, 
\begin{align*}
\MM_{t}\left(\widetilde{\omega} \right) = M^{\id}_{t}\left(h(\omega) \right)\left(\widetilde{\omega} \right) - M^{\id}_{t\wedge t^{\star}}\left(h(\omega) \right)\left(\widetilde{\omega} \right) & = \left\{ \begin{array}{lcc}
M^{\id}_{t-t^{\star}}\left(h\left(\omega\right)\right)\left(\theta_{t^{\star}}\left(\widetilde{\omega}\right) \right) & {\rm{if}} &  t > t^{\star} \\
0 & {\rm{if}} & t \leqslant t^{\star}
\end{array} \right. \\
& = M^{\id}_{\left(t-t^{\star}\right)^{+}}\left(h\left(\omega\right)\right)\left(\theta_{t^{\star}}\left(\widetilde{\omega}\right) \right)
\end{align*}
{  where $(a)^{+}$ designates the non-negative part of $a \in \R$}.
Let $n\in \N^{\star}$ and $0\leqslant s \leqslant T$. To prove the  martingale property for all $\FF_{s}-$measurable bounded random variable $Z$, it is sufficient to prove it on elementary events. Then, we consider a random variable $Z$ of the form \[Z(\omega) := \II_{\left\{\omega_{t_{1}} \in \Gamma_{1}, \cdots, \omega_{t_{n}} \in \Gamma_{n}  \right\} } \] where for all $i \in \{1, \cdots, n\}, t_{i} \leqslant s$ and $\Gamma_{i} \subset \MM_{1}(\R)$ measurable.  We define \[\widetilde{Z}\left(\omega, \widetilde{\omega}\right) := \II_{\left\{\omega_{t_{i}} \in \Gamma_{i}, \forall i \in \{1, \cdots, n \}  \ {\rm{such \ that}} \  t_{i} \leqslant t^{\star}  \right\} } \II_{\left\{\tilde{\omega}_{t_{j}} \in \Gamma_{j}, \forall j \in \{1, \cdots, n \}  \ {\rm{such \ that}} \  t_{j} >  t^{\star}  \right\} }.\]
By (\ref{omega_omega_tilde}), $\widetilde{Z}\left(\omega, \widetilde{\omega} \right) = Z\left(\omega \right)$, $\Q_{\omega}\left(\dd \widetilde{\omega} \right)-{\rm{a.s.}}$ Therefore, for $\P_{\nu}^{\rm FV}-$almost every $\omega \in \Omega$,
\begin{align*}
& \E_{\Q_{\omega}}\left(\left[\MM_{t} - \MM_{s} \right]Z \right) \\ 
& \hspace{0.5cm} =  \E_{\Q_{\omega}\left(\dd \widetilde{\omega} \right)}\left(\left[M_{\left(t-t^{\star}\right)^{+}}^{\id}\left(h\left(\omega\right)\right)\left(\theta_{t^{\star}}\left(\widetilde{\omega} \right)\right) - M_{\left(s-t^{\star}\right)^{+}}^{\id}\left(h\left(\omega\right)\right)\left(\theta_{t^{\star}}\left(\widetilde{\omega} \right)\right) \right]\widetilde{Z}\left(\omega, \widetilde{\omega} \right) \right) \\
& \hspace{0.5cm} = \II_{\left\{\omega_{t_{i}} \in \Gamma_{i}, \forall i \in \{1, \cdots, n \}  \ {\rm{such \ that}} \  t_{i} \leqslant t^{\star}  \right\} } \times \\
& \hspace{2cm} \E_{\Q_{\omega}\left(\dd \tilde{\omega} \right)}\left(\left[M_{t-t^{\star}}^{\id}\left(h\left(\omega\right)\right)\left(\theta_{t^{\star}}\left(\widetilde{\omega} \right)\right) - M_{s-t^{\star}}^{\id}\left(h\left(\omega\right)\right)\left(\theta_{t^{\star}}\left(\widetilde{\omega} \right)\right) \right]^{\phantom{1}} \right. \\
& \hspace{3.5cm}  \times \left. \II_{\left\{\left[\theta_{t^{\star}}\left(\tilde{\omega}\right)\right]_{t_{j}-t^{\star}} \in \Gamma_{j},   \forall j \in \{1, \cdots, n\} \ {\rm{such \ that}} \ t_{j} > t^{\star}  \right\} } \right)  \\
& \hspace{0.5cm} \overset{{(\ref{PPP_omega})}}{=} \II_{\left\{\omega_{t_{i}} \in \Gamma_{i}, \forall i \in \{1, \cdots, n \}  \ {\rm{such \ that}} \  t_{i} \leqslant t^{\star}  \right\} }  \E_{\P_{\omega}\left(\dd \hat{\omega} \right)}\left(\left[M_{t-t^{\star}}^{\id}\left(h\left(\omega\right)\right)\left(\widehat{\omega}\right) - M_{s-t^{\star}}^{\id}\left(h\left(\omega\right)\right)\left(\widehat{\omega}\right) \right]^{\phantom{1}} \right. \\
& \hspace{3.5cm}  \times \left. \II_{\left\{\hat{\omega}_{t_{j}-t^{\star}} \in \Gamma_{j}, \forall j \in \{1, \cdots, n\} \ {\rm{such \ that}} \ t_{j} > t^{\star}  \right\} } \right)  \\
& \hspace{0.5cm} = 0,
\end{align*}
using that $\left(M_{t-t^{\star}}^{\id}\left(h\left(\omega\right)\right)\left(\widehat{\omega}\right) \right)_{t^{\star}\leqslant t \leqslant T+t^{\star}}$ is a $\P_{\omega}\left(\dd \widehat{\omega} \right)-$martingale if the internal indicator is non zero. Thus, for $\P_{\nu}^{\rm FV}-$almost every $\omega\in \Omega$, $\left(\MM_{t}\left(\widetilde{\omega} \right)\right)_{0\leqslant t \leqslant T}$ is a $\Q_{\omega}\left(\dd \widetilde{\omega} \right)-$martingale which completes the first part of this proof. In similar way, we can prove \[M^{\id^{2}}_{t}\left( h\left(\widetilde{\omega} \right)\right)\left(\widetilde{\omega} \right) - M^{\id^{2}}_{t\wedge t^{\star}}\left( h\left(\widetilde{\omega} \right)\right)\left(\widetilde{\omega} \right) \] is a $\P_{\nu}^{\rm FV}-$martingale. Applying \textsc{It\^{o}}'s formula to compute $\left\langle h\left(\widetilde{\omega}_{t} \right), \widetilde{\omega}_{t} \right\rangle^{2}$ and comparing it to the previous result, we obtain the announced result. \hfill $\square$ } \\

\textbf{Funding.} This work was partially funded by the Chair ``Mod\'{e}lisation Math\'{e}matique et Biodiversit\'{e}'' of VEOLIA-Ecole Polytechnique-MNHN-F.X. N.C. was partially funded by the European Union (ERC, SINGER, 101054787).
Views and opinions expressed are however those of the author(s) only and do not necessarily reflect those of the European Union or the European Research Council. Neither the European Union nor the granting authority can be held responsible for them. \\

{  \textbf{Acknowledgments.} We thank the two anonymous referees for their useful comments.}


\bibliographystyle{plain} 
\small \bibliography{These_biblio} \normalsize


\end{document}